\documentclass[10pt]{article}

\title{Well-posedness, magnetic helicity conservation, inviscid limit and asymptotic stability for the generalized Navier-Stokes-Maxwell equations}
\author{Kyungkeun Kang\thanks{
		Yonsei University. E-mail address: \url{kkang@yonsei.ac.kr}}
	\and
	Jihoon Lee\thanks{
		Chung-Ang University. E-mail address: \url{jhleepde@cau.ac.kr}}
	\and
	Dinh Duong Nguyen\thanks{
		Yonsei University and Chung-Ang University. E-mail address:  \url{dinhduongnguyen.math@gmail.com}
	}
}

\usepackage{mathtools}
\usepackage{amssymb,bbm,amsmath,amsfonts,amsthm}
\usepackage{vmargin} 
\usepackage[linktocpage,colorlinks=true,raiselinks]{hyperref}
\hypersetup{urlcolor=black, citecolor=red, linkcolor=blue}
\usepackage{cleveref}
\usepackage{tkz-euclide}
\usepackage{multirow}
\usepackage{pgfplots}
\pgfplotsset{compat=newest}
\usepgfplotslibrary{fillbetween}
\usepackage{capt-of}

\usepackage{geometry}
\numberwithin{equation}{section}

\DeclareMathOperator*{\esssup}{ess\,sup}

\newtheorem{theorem}{Theorem}[section]

\newtheorem{lemma}{Lemma}[section]
\newtheorem{proposition}{Proposition}[section]

\theoremstyle{definition}
\newtheorem{remark}{Remark}[section]
\newenvironment{AMS}{}{}
\newenvironment{keywords}{}{}

\begin{document}
	\maketitle
	
	\begin{abstract}
		This paper is devoted to studying the well-posedness, conservation of magnetic helicity, inviscid limit and asymptotic stability of the generalized Navier-Stokes-Maxwell (NSM) equations with the standard Ohm's law in $\mathbb{R}^d$ for $d \in \{2,3\}$. More precisely, the global well-posedness is established in case of fractional Laplacian velocity $(-\Delta)^\alpha v$ with $\alpha = \frac{d}{2}$ for suitable data. In addition, the local well-posedness in the inviscid case is also provided for sufficient smooth data, which allows us to study the inviscid limit of associated positive viscosity solutions in the case $\alpha = 1$, where an explicit bound on the difference is given. Furthermore, in three dimensions if the initial data satisfies futher suitable conditions then magnetic helicity is conserved as the electric conductivity goes to infinity. On the other hand, in the case $\alpha = 0$ the stability near a magnetohydrostatic equilibrium with a constant (or equivalently bounded) magnetic field  is also obtained in which nonhomogeneous Sobolev norms of the velocity and electric fields, and for $p \in (2,\infty]$ the $L^p$ norm of the magnetic field converge to zero as time goes to infinity with an implicit rate. In this velocity damping case, the situation is different both in case of the two and a half, and three-dimensional (Hall)-magnetohydrodynamics ((H)-MHD) system, where an explicit rate of convergence in infinite time is computed for both the velocity and magnetic fields in nonhomogeneous Sobolev norms. Therefore, it seems that there is a gap between NSM and MHD in terms of the norm convergence of the magnetic field and the rate of decaying in time, even the latter equations can be proved as a limiting system of the former one in the sense of distributions as the speed of light tends to infinity.
	\end{abstract}
	
	
	
	\tableofcontents
	\allowdisplaybreaks 
	
	%
	\section{Introduction}
	%
	
	%
	\subsection{The systems}
	%
	
	Let us consider the one fluid incompressible Navier-Stokes-Maxwell equations with the standard Ohm's law, which are given in the following form\footnote{Here, the usual fractional Laplacian operator is defined in terms of Fourier transform, i.e., for $\alpha \in \mathbb{R}$
	\begin{equation*}
		\mathcal{F}((-\Delta)^\alpha(f))(\xi) := |\xi|^{2\alpha}\mathcal{F}(f)(\xi) \qquad \text{where}\qquad \mathcal{F}(f)(\xi) := \int_{\mathbb{R}^d} \exp\{-i\xi\cdot x\} f(x) \,dx \quad \text{for } \xi \in \mathbb{R}^d.
	\end{equation*}
	In the case $\alpha = 0$, we use the standard convention that $(-\Delta)^0$ is the identity operator.}
	\begin{equation} \label{NSM} \tag{NSM}
		\left\{
		\begin{aligned}
			\partial_t v + v \cdot \nabla v + \nabla \pi &= -\nu(-\Delta)^\alpha v + j \times B, 
			\\
			\frac{1}{c}\partial_t E - \nabla \times B &= -j,
			\\
			\frac{1}{c}\partial_t B + \nabla \times E &= 0,
			\\
			\sigma(cE + v \times B) &= j,
			\\
			\textnormal{div}\, v = \textnormal{div}\, B &= 0,
		\end{aligned}
		\right.
	\end{equation}
	where $\alpha \geq 0$, for $d \in \{2,3\}$, $(v,E,B,j) : \mathbb{R}^d \times (0,\infty) \to \mathbb{R}^3$ and $\pi : \mathbb{R}^d \times (0,\infty) \to \mathbb{R}$ are the velocity, electric, magnetic and electric current fields,  and scalar pressure of the fluid, respectively. The positive constants $\nu,\sigma$ and $c$ denote in order the viscosity, electric conductivity and  speed of light. We will denote the initial data by $(v,E,B)_{|_{t=0}} = (v_0,E_0,B_0)$. We note that the case $d = 2$ is also known as the $2\frac{1}{2}$-dimensional version. Let us also quickly recall the standard meaning of the above system. In \eqref{NSM}, through the Lorentz force $j \times B$ (under quasi-neutrality assumptions) and the electric current field $j$ the Navier-Stokes equations (the first line) are coupled to the Maxwell system, where the latter one consists of the Ampere's equations with Maxwell's correction (the second line) and the Faraday's law (the third line). In addition, the fourth line is the usual Ohm's law and the last one stands for the incompressiblity of the velocity and magnetic fields. It can be seen that if the term $\frac{1}{c}\partial_t E$ is neglected formally for either large $c$ or time-independent $E$, then \eqref{NSM} reduces to the usual $2\frac{1}{2}$-dimensional fractional magnetohydrodynamics (MHD) equations, i.e., \eqref{HMHD} with $\beta = 1$ and $\kappa = 0$ (for more physical introduction to the magnetohydrodynamics, see \cite{Biskamp_1993,Davidson_2017}). Therefore,  \eqref{NSM} with $\alpha = 1$ is also known as the full MHD system.
	
	In fact, by ignoring thermal effects, \eqref{NSM} with $\alpha = 1$ can be derived from kinetic equations (see \cite{Jang-Masmoudi_2012}). By considering a solenoidal Ohm’s law\footnote{In this case, $j =  \sigma(-\nabla \bar{\pi} + cE + v \times B)$ with $\textnormal{div}\, j = \textnormal{div}\, E = 0$ and for some additional electromagnetic pressure $\bar{\pi}$, see \eqref{NSM-SO}.} instead, it also can be formally obtained as a limiting system of a two-fluid incompressible Navier–Stokes–Maxwell system by taking the momentum transfer coefficient $\epsilon > 0$ tends to zero (see \cite{Arsenio-Ibrahim-Masmoudi_2015}). More precisely, if $v^+$ and $v^-$ denote the cations and anions velocities, respectively, with the same viscosity $\mu > 0$ and the corresponding thermal pressures $\pi^+$ and $\pi^-$, then the scaled two-fluid incompressible Navier–Stokes–Maxwell equations were proposed in \cite{Giga-Yoshida_1984} and will be written in the following form\footnote{In fact, the authors in \cite{Giga-Yoshida_1984} suggested a more general model with different coefficients appearing in the equations.} (we use the same notation for the electric and magnetic fields as previously)
	\begin{equation*} \tag{2-NSM} \label{2NSM}
		\left\{
		\begin{aligned}
			\partial_t v^+ + v^+ \cdot \nabla v^+ +  \frac{1}{2\sigma\epsilon^2}(v^+-v^-) &= \mu \Delta v^+ - \nabla \pi^+ + \frac{1}{\epsilon}(cE + v^+ \times B),
			\\
			\partial_t v^- + v^- \cdot \nabla v^- - \frac{1}{2\sigma\epsilon^2}(v^+-v^-) &= \mu \Delta v^- - \nabla \pi^- - \frac{1}{\epsilon}(cE + v^- \times B),
			\\
			\frac{1}{c}\partial_t E - \nabla \times B &= -\frac{1}{2\epsilon}(v^+-v^-),
			\\
			\frac{1}{c}\partial_t B + \nabla \times E &= 0,
			\\
			\textnormal{div}\, v^+ = \textnormal{div}\, v^- = \textnormal{div}\, E = \textnormal{div}\, B &= 0,
		\end{aligned}
		\right.
	\end{equation*}
	which models the motion of a plasma of positively (cations) and negatively (anions) charged
	particles under the assumption of equal masses. In the above system, the condition $\textnormal{div}\, E = 0$, which is known as a degenerate Gauss’s law (see \cite{Arsenio-Ibrahim-Masmoudi_2015}) and follows from the charge neutrality and the incompressibility of the plasma (see \cite{Giga-Yoshida_1984}), and the third term on the right-hand side of the second equation presents the momentum transfer between the two fluids. The existence and uniqueness of global energy solutions to \eqref{2NSM} (for more general coefficients) have recently obtained in \cite{Giga-Ibrahim-Shen-Yoneda_2018} in two dimensions. In the three-dimensional case, they also showed the existence of global energy solutions and local well-posedness (LWP) for initial data $(v^{\pm}_0,E_0,B_0) \in H^\frac{1}{2} \times L^2 \times L^2$, and this local solution can be globally extended for small $v^{\pm}_0$ in the $\dot{H}^\frac{1}{2}$ norm. It can be seen that the energy equality to \eqref{2NSM} formally reads\footnote{The notation $\|(f_1,...,f_n)\|^m_X := \sum^n_{i=1}\|f_i\|^m_X$ will be used throughout the paper for $n,m \in \mathbb{N}$ and some functional space $X$.}
	\begin{equation*}
		\frac{1}{2} \frac{d}{dt}\left\|\left(\frac{v^{\pm}}{\sqrt{2}},E,B\right)\right\|^2_{L^2} + \mu\left\|\frac{1}{\sqrt{2}} \nabla v^{\pm}\right\|^2_{L^2} + \frac{1}{\sigma}\left\|\frac{1}{2\epsilon}(v^+-v^-)\right\|^2_{L^2} = 0,
	\end{equation*}
	which suggests us to consider a system, which is satisfied by the following quantities
	\begin{equation*}
		k := \frac{1}{2\epsilon}(v^+-v^-),\quad u := \frac{1}{2}(v^++v^-),\quad p := \frac{1}{2}(\pi^++\pi^-) \quad \text{and} \quad  \bar{p} := \frac{\epsilon}{2}(\pi^+-\pi^-)
	\end{equation*}
	and is given by rewriting \eqref{2NSM} as follows
	\begin{equation*} 
		\left\{
		\begin{aligned}
			\partial_t u + u \cdot \nabla u + \epsilon^2 k \cdot \nabla k  &= \mu\Delta u - \nabla p + k \times B, 
			\\
			\epsilon^2(\partial_t k + u \cdot \nabla k + k \cdot \nabla u) + \frac{1}{\sigma} k &= \epsilon^2 \mu \Delta k - \nabla \bar{p} + cE + u \times B, 
			\\	
			\frac{1}{c}\partial_t E - \nabla \times B &= -k,
			\\
			\frac{1}{c}\partial_t B + \nabla \times E &= 0,
			\\
			\textnormal{div}\, u = \textnormal{div}\, k = \textnormal{div}\, E = \textnormal{div}\, B &= 0.
		\end{aligned}
		\right.
	\end{equation*}
	As $\epsilon \to 0$, the above system formally converges to \eqref{NSM} with solenoidal Ohm’s law
	instead of the usual one (see \cite{Arsenio-Ibrahim-Masmoudi_2015,Arsenio-SaintRaymond_2016}), i.e., the following system (for simplicity we will replace $u,k,p,\bar{p}$ and $\mu$ by $v,j,\pi,\bar{\pi}$ and $\nu$, respectively)
	\begin{equation*} \label{NSM-SO} \tag{NSM-SO}
		\left\{
		\begin{aligned}
			\partial_t v + v \cdot \nabla v + \nabla \pi &= \nu\Delta v + j \times B, 
			\\	
			\frac{1}{c}\partial_t E - \nabla \times B &= -j,
			\\
			\frac{1}{c}\partial_t B + \nabla \times E &= 0,
			\\
			\sigma(-\nabla \bar{\pi} + cE + v \times B) &= j, 
			\\
			\textnormal{div}\, v = \textnormal{div}\, j = \textnormal{div}\, E &= \textnormal{div}\, B = 0,
		\end{aligned}
		\right.
	\end{equation*}
	that shares a similar structure and mathematical difficulties to those of \eqref{NSM} with $\alpha = 1$. In fact, we will list known results to \eqref{NSM} and it is possible to obtain similar ones to \eqref{NSM-SO}. The rigorous proof of the convergence from \eqref{2NSM} to \eqref{NSM-SO} as $\epsilon \to 0$ does not seem to be known for $L^2$ initial data. In \cite{Arsenio-Ibrahim-Masmoudi_2015}, the authors established the limit as first $c \to \infty$ and then $\epsilon \to 0$, where \eqref{2NSM} converges weakly to the standard $2\frac{1}{2}$-dimensional MHD system, i.e., \eqref{HMHD} with $\alpha = 1$ and $\kappa = 0$. They also pointed out that the same result also holds in the case $c \to \infty$ and $\epsilon \to 0$ at the same time, but with additional conditions on the relation between $\epsilon$ and $c$ in which $\epsilon$ is considered as a function of  $c$ satisfying further assumptions. However, the other order of taking limit has not been confirmed yet, i.e., the limit as $\epsilon \to 0$ first and then $c \to \infty$, where the limiting system is the same as the previous case. In addition, it is also very interested and much more complicated to consider  \eqref{NSM} with a generalized Ohm's law, which can be derived from either the two-fluid Navier-Stokes-Maxwell equations or kinetic models with different masses (see \cite{Acheritogaray-Degond-Frouvelle-Liu_2011,Jang-Masmoudi_2012,Peng-Wang-Xu_2022}) for $\alpha = 1$, in particular, the new system takes a more general form as follows 
	\begin{equation*} \label{NSM-GO} \tag{NSM-GO}
		\left\{
		\begin{aligned}
			\partial_t v + v \cdot \nabla v + \nabla \pi &= -\nu(-\Delta)^\alpha v + j \times B, 
			\\	
			\frac{1}{c}\partial_t E - \nabla \times B &= -j,
			\\
			\frac{1}{c}\partial_t B + \nabla \times E &= 0,
			\\
			\sigma(-\nabla \bar{\pi} + cE + v \times B)  &= j + \kappa j \times B, 
			\\
			\textnormal{div}\, v = \textnormal{div}\, j = \textnormal{div}\, E &= \textnormal{div}\, B = 0,
		\end{aligned}
		\right.
	\end{equation*}
	which takes into account of the Hall effect for some nonnegative constant $\kappa$. This new constant is corresponding to the magnitude of the Hall effect
	compared to the typical fluid length scale. Furthermore, by taking the limit as $c \to \infty$ formally or ignoring the term $\frac{1}{c}\partial_t E$ (for example, $E$ is time-independent), \eqref{NSM-GO} reduces to the following fractional Hall-magnetohydrodynamics equations (with $\beta = 1$), i.e.,
	\begin{equation} \label{HMHD} \tag{H-MHD}
		\left\{
		\begin{aligned}
			\partial_t v + v \cdot \nabla v + \nabla \pi &= -\nu(-\Delta)^\alpha v + (\nabla \times B) \times B, 
			\\
			\partial_t B - \nabla \times (v \times B) &= -\frac{1}{\sigma}(-\Delta)^\beta B - \frac{\kappa}{\sigma} \nabla \times ((\nabla \times B) \times B),
			\\
			\textnormal{div}\, v = \textnormal{div}\, B &= 0.
		\end{aligned}
		\right.
	\end{equation}
	Indeed, the Hall term in \eqref{HMHD} plays an important role in magnetic reconnection, which can not be explained by using \eqref{MHD}, and is derived from either two-fluid models or kinetic equations in \cite{Acheritogaray-Degond-Frouvelle-Liu_2011}. The systematical study of the above equations was initiated in \cite{Lighthill_1960} long time ago. However, even in the case that $d = 2$ and $\alpha = 1$, the global regularity issue of \eqref{HMHD} has not been fully established for general initial data. In this case, the existence of global energy solutions has been provided in \cite{Chae-Degond-Liu_2014} in both two and three dimensions, but it is not the case for \eqref{NSM} and \eqref{NSM-GO} as mentioned before. In addition, by using the convex integration framework, the author in \cite{Dai_2021} proved the nonuniqueness of weak solutions in the Leray-Hopf class for $d = 3$. In the case of without the resistivity, illposedness results around shear-type flows are also obtained in \cite{Jeong-Oh_2022}. Furthermore, global small initial data solutions in both cases $d = 2$ and $d = 3$ have been provided in \cite{Bae-Kang_2023,Chae-Degond-Liu_2014,Chae-Lee_2014,Danchin-Tan_2021,Danchin-Tan_2022,Liu-Tan_2021,Tan_2022,Wan_2015,Wan-Zhou_2019,Wan-Zhou_20191,Ye_2022}. In the stationary case, regularity and partial regularity of weak solutions can be found in \cite{Chae-Wolf_2017} and \cite{Chae-Wolf_2015} in the two and three-dimensional cases, respectively. As mentioned previously, \eqref{HMHD} can be obtained formally from \eqref{NSM-GO}.  Therefore, it is reasonable to consider the conditional global well-posedness (GWP) for \eqref{NSM-GO}, for instance, under smallness assumptions of initial data. This issue will be considered in \cite{KLN_2024_2}, which will allows us to have a connection between \eqref{NSM-GO} and \eqref{HMHD}.
	
	Finally, it is also convenient to write down the standard fractional MHD system as follows
	\begin{equation} \label{MHD} \tag{MHD}
		\left\{
		\begin{aligned}
			\partial_t v + v \cdot \nabla v + \nabla \pi &= -\nu (-\Delta)^\alpha v + B \cdot \nabla B, 
			\\
			\partial_t B + v \cdot \nabla B &= -\frac{1}{\sigma} (-\Delta)^\beta B + B \cdot \nabla v,
			\\
			\textnormal{div}\, v = \textnormal{div}\, B &= 0,
		\end{aligned}
		\right.
	\end{equation}
	where $(v,B) : \mathbb{R}^d \times (0,\infty) \to \mathbb{R}^d$ and $\pi : \mathbb{R}^d \times (0,\infty) \to \mathbb{R}$ for $d \in \{2,3\}$ and the magnetic resistivity constant $\mu > 0$. In three dimensions, it is well-known that by using vector identities,  \eqref{HMHD} with $\kappa = 0$ and \eqref{MHD} are equivalent to each other up to a modified pressure.
	
	%
	\subsection{The state of the art}
	%
	
	\textbf{A. The case $d = 2$.} Let us give a quick review on the study of \eqref{NSM} in two dimensions with $\alpha = 1$. Formally, its energy balance is given by (the same for \eqref{NSM-SO} and \eqref{NSM-GO} in both cases $d = 2$ and $d = 3$)
	\begin{equation*}
		\frac{1}{2} \frac{d}{dt}\|(v,E,B)\|^2_{L^2} + \nu\|\nabla v\|^2_{L^2} + \frac{1}{\sigma}\|j\|^2_{L^2} = 0.
	\end{equation*}
	Thus, similar to in the case of the usual Navier-Stokes equations, we could expect the existence of global energy solutions (see \cite{Leray_1933,Leray_1934}). However, it seems that this energy equality is not enough to obtain the existence of $L^2$ weak solutions, which is different to that of \eqref{2NSM} in the two and three-dimentional cases as mentioned previously. The main difficulty is the lack of compactness, due to the hyperbolicity 
	of the Maxwell equations, which is needed to pass to the limit as $n \to \infty$ of the term $j^n \times B^n$, especially for the one $E^n \times B^n$, where $n$ is the regularization parameter of a standard approximate system to \eqref{NSM} (for example, see the proof of Theorem \ref{theo1}). Therefore, higher regular data should be considered on the GWP issue. The first GWP result to \eqref{NSM} was obtained in \cite{Masmoudi_2010} in the case where
	\begin{equation*}
		(v_0,E_0,B_0) \in L^2 \times H^s \times H^s \quad \text{ for } s \in (0,1).
	\end{equation*}
	In addition, higher regular estimates are also provided in \cite{Masmoudi_2010} in the case where $(v_0,E_0,B_0) \in H^\delta \times H^s \times H^s$ for $\delta \geq 0$, $s \geq 1$, $s - 2 < \delta \leq s$\footnote{In the statement, the author assumed that $\delta \geq 0$ and $s \geq 1$. However, it seems to us that he used conditions $\delta > 0$ and $s > 1$ during the proof.}, see also \cite{Kang-Lee_2013} for another proof\footnote{By using the standard Brezis-Gallouet inequality, the authors in \cite{Kang-Lee_2013} considered the case where $\delta = s = 2$ and  all the third components is assumed to be zero, i.e., $v,E,B,j: \mathbb{R}^2 \times (0,\infty) \to \mathbb{R}^2$. However, the pure 2D flow assumption can be removed and the assumption on the initial data can be improved when we consider \eqref{NSM}, see Theorem \ref{theo1} below.}
	and \cite{Fan-Ozawa_2020} for the case of bounded domains. The GWP is also obtained in \cite{Ibrahim-Keraani_2012} for small initial  data satisfying\footnote{The space $\dot{B}^0_{2,1}$ is the usual homogeneous Besov space (see Appendix A) and $L^2_{\textnormal{log}}$ is the set of tempered distributions $f$ satisfying
	\begin{equation*}
		\|f\|^2_{L^2_{\textnormal{log}}} := \sum_{q \in \mathbb{Z}, q \leq 0} \|\dot{\Delta}_q f\|^2_{L^2} + \sum_{q \in \mathbb{Z}, q > 0} q \|\dot{\Delta}_q f\|^2_{L^2} < \infty,
	\end{equation*}
	where for each $q \in \mathbb{Z}$, $\dot{\Delta}_q$ is the homogeneous dyadic block (see  Appendix A).} 
	\begin{equation*}
		(v_0,E_0,B_0) \in \dot{B}^0_{2,1} \times L^2_{\textnormal{log}} \times L^2_{\textnormal{log}},
	\end{equation*}
	where we have the following relations $\dot{B}^0_{2,1} \subset L^2$ and $\cup_{s>0} H^s \subset L^2_{\textnormal{log}} \subset L^2$. However, the LWP has not been contructed for the above arbitrary large initial data. After that, the authors in \cite{Germain-Ibrahim-Masmoudi_2014} have been considered mild solutions to \eqref{NSM} and they obtained the LWP 
	for possibly large initial data and the GWP for small initial data under the assumption
	\begin{equation*}
		(v_0,E_0,B_0) \in L^2 \times L^2_{\textnormal{log}} \times L^2_{\textnormal{log}}.
	\end{equation*}
	Here, in the two previous results,  in order to estimate the term $E \times B$ coming from $j \times B$ in some homogeneous Besov spaces, the authors used the paraproduct estimate \eqref{paraproduct}, and it is critical in two dimentions, thus the extra logarithmic regularity of $(E_0,B_0)$ is needed. Recently, the authors in \cite{Arsenio-Gallagher_2020} revisited \eqref{NSM} in the case where $(v_0,E_0,B_0) \in L^2 \times H^s \times H^s$ for $s \in (0,1)$, as considered previously in \cite{Masmoudi_2010}, with providing further improvements, which include some $c$-independent estimates of $(v,E,B)$. That allowed them to investigate the asymptotic behavior as $c \to \infty$ by proving the convergence of solutions to \eqref{NSM} to that of the standard $2\frac{1}{2}$-dimensional MHD equations, i.e., \eqref{HMHD} with $\alpha = \beta = 1$ and $\kappa = 0$, in the sense of distributions.
	\\
	
	\noindent
	\textbf{B. The case $d = 3$.} As mentioned previously, the existence of energy solutions is unknown so far. We will shortly recall some results to \eqref{NSM} in the three-dimensional case with $\alpha = 1$. One of the first results was given in \cite{Ibrahim-Keraani_2012}, where the authors constructed global small solutions with initial data
	\begin{equation*}
		(v_0,E_0,B_0) \in \dot{B}^\frac{1}{2}_{2,1} \times \dot{H}^\frac{1}{2} \times \dot{H}^\frac{1}{2}.
	\end{equation*}
	For large initial data in some $\ell^1$ weighted space in Fourier side, the authors in \cite{Ibrahim-Yoneda_2012} have been provided the local in time existence of mild solutions. Moreover, by using the fact that the damped-wave operator does not have any smoothing
	effect, they also showed these local solutions lost regularity in some finite time. Later on, the above result in \cite{Ibrahim-Keraani_2012} was extended in \cite{Germain-Ibrahim-Masmoudi_2014} in which either local large initial data solutions or global small intial data ones was provided for initial data in the following space 
	\begin{align*}
		(v_0,E_0,B_0) \in \dot{H}^\frac{1}{2} \times \dot{H}^\frac{1}{2} \times \dot{H}^\frac{1}{2}.
	\end{align*}
	Recently,  the existence of weak solutions was built in \cite{Arsenio-Gallagher_2020} for the initial data in the form of 
	\begin{equation*}
		(v_0,E_0,B_0) \in L^2 \times H^s \times H^s \qquad \text{for} \quad s \in \left[\frac{1}{2},\frac{3}{2}\right),
	\end{equation*} 
	under the smallness assumption of $(E_0,B_0)$ in the $\dot{H}^s$ norm (the smallness assumption is related to only the $L^2$ norm of $(v_0,E_0,B_0)$ and $\dot{H}^s$ norm of $(E_0,B_0)$).
	
	We also note that time-periodic small solutions and their asymptotic stability were investiagted in \cite{Ibrahim-Lemarie-Rieusset-Masmoudi_2018}. For further results to \eqref{NSM} (and also to \eqref{NSM-SO}) such as GWP for small data and LWP for possibly large data, loss of regularity, asymptotic behaviors, existence of global weak solutions with small data, global regularity criteria, time periodic solutions and so on, we refer the reader to \cite{Arsenio-Gallagher_2020,Arsenio-Ibrahim-Masmoudi_2015,Germain-Ibrahim-Masmoudi_2014,Ibrahim-Keraani_2012,Ibrahim-Lemarie-Rieusset-Masmoudi_2018,Ibrahim-Yoneda_2012,Jiang-Luo-Tang_2020,Kang-Lee_2013,Peng-Wang-Xu_2022,Wen-Ye_2020,Yue-Zhong_2016}.
	
	%
	\subsection{Main results}
	%
	
	For the reader's convenience, before going to the detailed statements, let us first summarize the main results in the present paper as follows:
	\begin{enumerate}
		\item[1.] The GWP of \eqref{NSM} for $\nu > 0$, $\alpha = \frac{d}{2}$ with $d = 2$ and $d = 3$ in Theorems \ref{theo1} and \ref{theo1-3d};
		
		\item[2.] The conservation of magnetic helicity of \eqref{NSM} as $\sigma \to \infty$ with $\nu > 0$, $\alpha = \frac{3}{2}$ and $d = 3$ in Theorem \ref{theo-MH};
		
		\item[3.] The LWP for $\nu = 0$ and the inviscid limit of \eqref{NSM} for $d \in \{2,3\}$ in Theorem \ref{theo-inviscid};
		
		\item[4.] The stability near a magnetohydrostatic equilibrium with a constant (or equivalently bounded) magnetic field of \eqref{NSM} and \eqref{HMHD} for $\alpha = 0, \beta = 1, \nu > 0, \kappa \geq 0$ and $d \in \{2,3\}$ in Theorems \ref{theo3} and \ref{pro-HMHD};
	\end{enumerate}
	which will be precisely presented in the following subsubsections, respectively.
	
	%
	\subsubsection{Global well-posedness}
	%
	
	Our first result is aiming to obtain higher regular solutions to \eqref{NSM} in two dimensions compared to those of in \cite{Arsenio-Gallagher_2020,Masmoudi_2010} with a direct and sightly different proof, which is stated as follows.
	
	\begin{theorem}[Higher regular solutions in two dimensions] \label{theo1}
		Let $d = 2$, $\alpha = 1$, $c,\nu,\sigma > 0$ and $(v_0,E_0,B_0) \in (H^\delta \times H^s \times H^s)(\mathbb{R}^2)$ with $\textnormal{div}\, v_0 = \textnormal{div}\, B_0 = 0$ and $\delta, s \in [0,\infty)$. 
		\begin{enumerate}
			
			\item[(i).] If $(\delta,s)$ satisfies one of the following assumptions 
			\begin{enumerate}
				\item $0 < \delta \leq s \leq 1$;
				
				\item $0 < s < 1$ and $s \leq \delta \leq 2s$;
				
				\item $s = 1$ and $1 \leq \delta < 2$;
				
				\item $s > 1$ and $s \leq \delta \leq s + 1$;
				
				\item $s > 1$ and $s - 1 \leq \delta < s$;
				
				\item $s \in (0,1)$ and $\delta = 0$;
			\end{enumerate}
			 then there exists a unique global solution $(v,E,B)$ to \eqref{NSM} satisfying for any $T \in (0,\infty)$ 
			\begin{equation*}
				v \in L^\infty(0,T;H^\delta) \cap L^2(0,T;H^{\delta+1}) \cap L^2(0,T;L^\infty) \quad \text{and} \quad (E,B) \in L^\infty(0,T;H^s),
			\end{equation*}
			and for $t \in (0,T)$
			\begin{align*}
				\|v(t)\|^2_{H^\delta} + \|(E,B)(t)\|^2_{H^s} + \int^t_0 \|v\|^2_{H^{\delta+1}} + \|v\|^2_{L^\infty} + \|j\|^2_{H^s} \,d\tau &\leq C(T,\delta,\nu,\sigma,s,v_0,E_0,B_0).
			\end{align*}
			
			\item[(ii).] If $\delta = 0$ and $s \in (0,1)$ then for any $t \in (0,T)$, in addition to Part (i)-(f), there holds
			\begin{align*}
				&v \in L^\infty(t,T;H^{\delta'}) \cap L^2(t,T;H^{\delta'+1}) \qquad\text{for} \quad \delta' \in 
				\begin{cases}
					[0,s] &\text{if} \quad s \in (0,1),
					\\
					[0,1] &\text{if} \quad s \in [\frac{1}{2},1).
				\end{cases}
			\end{align*}
			\item[(iii).] If $\delta = 0$ and $s = 1$ then we have the same properties as Part (i)-(f) and for any $t \in (0,T)$ and $\delta' \in [0,1]$
			\begin{align*}
				v \in L^\infty(t,T;H^{\delta'}) \cap L^2(t,T;H^{\delta'+1}) \qquad \text{and}\qquad (E,B) \in L^\infty(t,T;H^1).
			\end{align*}
		\end{enumerate}
		Furthermore, $v \in C([0,T];H^\delta)$ and $(E,B) \in C_{\textnormal{weak}}([0,T];H^s)$.
	\end{theorem}
	
	\begin{remark} \label{re1} We add some comments to Theorem \ref{theo1}:
		\begin{enumerate}
			\item Strategy of proof: The proof mainly based on the usual energy method with using the Brezis-Gallouet-Wainger inequalty \eqref{BGW}, a logarithmic Gronwall inequality in Lemma \ref{lem-gronwall}, some well-known commutator estimates and a carefully treated in each case. In addition, in the case of $(i)-(f)$, we also borrow the idea from \cite{Arsenio-Gallagher_2020} with a slightly different velocity decomposition. Furthermore, in order to obtain the uniqueness, we use the idea in \cite{Masmoudi_2010} with a slightly different proof. The idea here will also be applied to the three-dimensional case in Theorem \ref{theo1-3d}.
			
			\item The range of the initial data in Theorem \ref{theo1} consists of the dark and darker regions (the region $A$). We also provide a new proof for the region $A \cap B$. The results obtained in \cite{Masmoudi_2010} are the segment from $(0,0)$ to $(1,0)$ without the end points (also in \cite{Arsenio-Gallagher_2020}) and the darker and darkest areas (the region $B$) without the line $\delta = s-2$ and without the end point $(2,0)$ as well (and it seems also without the segments from $(1,0)$ to $(1,1)$ and from $(1,0)$ to $(2,0)$ excluding the end points, as mentioned previously).  
			\begin{center}
				\begin{tikzpicture}  
					\draw[step=1,help lines,black!20] (0,0) grid (5,5);
					
					\draw[very thin, black,->] (-0.5,0) -- (5.5,0) node[below]{$s$};
					\draw[very thin,black,->] (0,-0.5) -- (0,5.5) node[left]{$\delta$};
					
					\foreach \x in {1,2,...,5}
					\draw[shift={(\x,0)}] (0pt,2pt) -- (0pt,-2pt) node[below] {\footnotesize $\x$};
					
					\foreach \y in {1,2,...,5}
					\draw[shift={(0,\y)}] (2pt,0pt) -- (-2pt,0pt) node[left] {\footnotesize $\y$};
					
					\fill[fill=black!15] (0,0) to (1,0) to (1,1) to (0,0);
					
					\fill[fill=black!25] (1,0) to (5,4) to (5,5) to (1,1) to (1,0);
					
					\fill[fill=black!35] (1,0) to (2,0) to (5,3) to (5,4) to (1,0);
					
					\fill[fill=black!15] (0,0) to (5,5) to (4,5) to (1,2) to (0,0);
					
					\draw[fill=black] (0,0)  node[below left] {$0$};
					
					\draw[fill=black] (0.7,0.4)  node[] {$A$};
					\draw[fill=black] (2.7,2.2)  node[rotate=45] {$A \cap B$};
					
					\draw[thick,color=black] (0,0) -- (5,5); 
					\draw[] (2.8,2.8) node[above=0.25cm, left, rotate=45] {$\delta = s$}; 
					
					\draw[dashed,color=black] (2,0) -- (5,3); 
					\draw[] (4.2,2.15) node[below=0.25cm, left, rotate=45] {$\delta = s-2$}; 
					
					\draw[fill=black] (3,1.5)  node[] {$B$};
					
					\draw[thick,color=black] (0,0)--(1,2); 

					\draw[dashed,color=black] (0,1)--(1,2);
					\draw[thick,color=black] (1,2)--(4,5); \draw[] (2.8,4.3) node[below=0.25cm, left, rotate=45] {$\delta = s+1$};
					
					\draw[fill=black] (1.4,1.8)  node[] {$A$};
					\draw[fill=black] (3.4,3.8)  node[] {$A$};
					
					\draw[thick,color=black] (0,0) -- (1,0);
					\draw[dashed,color=black] (1,0) -- (2,0);
					\draw[dashed,color=black] (0,0) -- (0.5,1);
					
					\draw[color=black] (0,0)  circle[radius=1.2pt] node[below left] {};
					\draw[color=black] (2,0)  circle[radius=1.2pt] node[below left] {};
					\draw[color=black] (1,0)  circle[radius=1.2pt] node[below left] {};
					
					\draw[color=black] (1,2)  circle[radius=1.2pt] node[below left] {};
					
					\draw[color=black] (0,1)  circle[radius=1.2pt];
				\end{tikzpicture}
				
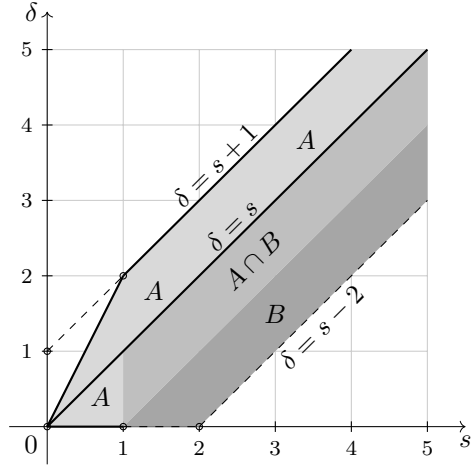
\captionof{figure}{The relation between $s$ and $\delta$ in Theorem \ref{theo1} and \cite{Arsenio-Gallagher_2020,Masmoudi_2010}.}
			\end{center}
			
			\item As mentioned previously in the introduction, similar results as in Theorems \ref{theo1}, \ref{theo1-3d}, \ref{theo-MH}, \ref{theo-inviscid} and \ref{theo3} can be easily obtained to \eqref{NSM-SO} by using mainly the divergence-free condition of $(E,j)$. 

			\item As it will be seen later that the estimates in Theorem \ref{theo1} are $c$-independent, from \cite[Corollary 1.3]{Arsenio-Gallagher_2020}) we can prove that \eqref{NSM} converges to \eqref{HMHD} with $\alpha = \beta = 1$ and $\kappa = 0$ as $c \to \infty$ in the sence of distributions, see also the proof of Theorem \ref{theo1-3d}-$(ii)$. In addition, Theorem \ref{theo1} also holds in the case that $-\nu \Delta v$ is replaced by $(\nu_2\partial_{22}v_1,\nu_1\partial_{11}v_2,\nu_3\Delta v_3)$ in \eqref{NSM} for any positive constants $\nu_1,\nu_2$ and $\nu_3$, by using the divergence-free condition of $v$, for example see \cite{KLN_2024}.
			
			\item It seems to be not clear to us how to obtain a priori estimates for initial data in the following cases: a) the triangle $(0,0)-(1,2)-(0,1)$ without the segment from $(0,0)$ to $(1,2)$ including the end points; b) the segment from $(1,0)$ to $(2,0)$ including the end points; c) the line $\delta = s-2$; and d) the domain is either above the line $\delta = s+1$ or under the one $\delta = s-2$. 
			
			\item Note that in Part $(iii)$, we are not able to close the estimate of $(v,E,B)$ in the whole time interval $(0,T)$, but only in $(t,T)$ for any $t \in (0,T)$. Furthermore, higher regularity for $(v,E,B)$ after the initial time as Parts $(ii)$ and $(iii)$ can be obtained to the cases from Part $(i)-(a)$ to Part $(i)-(e)$.
		\end{enumerate}	
	\end{remark}
	
	Our second result focuses on the three-dimensional case, where we obtain the GWP of \eqref{NSM} for possibly critical exponent fractional Laplacian. More precisely, it is given as follows.
	
	\begin{theorem}[Possibly critical exponent in three dimensions] \label{theo1-3d}
		Let $d = 3$, $\alpha = \frac{3}{2}$ and $(v_0,E_0,B_0)  \in (H^\delta \times H^s \times H^s)(\mathbb{R}^3)$ with $\delta,s \in [0,\infty)$. 
		\begin{enumerate}
			\item[(i)] \textnormal{(Global well-posedness).} If $(\delta,s)$ satisfies one of the following conditions 
			\begin{enumerate}
				\item $\delta = 0$ and $s \in (0,\frac{3}{2})$;
						
				\item $0 < \delta \leq s \leq \frac{3}{2}$;
			
				\item $0 < s < \frac{3}{2}$ and $s \leq \delta \leq 2s$;
			
				\item $s = \frac{3}{2}$ and $\frac{3}{2} \leq \delta < 3$;
			
				\item $s > \frac{3}{2}$ and $s - \frac{3}{2} \leq \delta \leq s + \frac{3}{2}$;
			\end{enumerate}
			then there exists a unique global solution $(v,E,B)$ to \eqref{NSM} satisfying for any $T \in (0,\infty)$
			\begin{equation*}
				v \in L^\infty(0,T;H^\delta) \cap L^2(0,T;H^{\delta+\frac{3}{2}}) \cap L^2(0,T;L^\infty) \quad\text{and}\quad (E,B) \in L^\infty(0,T;H^s),
			\end{equation*}
			and for $t \in (0,T)$
			\begin{equation*}
				\|v(t)\|^2_{H^\delta} + \|(E,B)(t)\|^2_{H^s} + \int^t_0 \|v\|^2_{H^{\delta+\frac{3}{2}}} + \|v\|^2_{L^\infty} + \|j\|^2_{H^s} \,d\tau \leq C(T,\delta,\nu,\sigma,s,v_0,E_0,B_0).
			\end{equation*}
			In addition, $v \in C([0,T];H^\delta)$ and $(E,B) \in C_{\textnormal{weak}}([0,T];H^s)$.
			
			\item[(ii)] \textnormal{(The limit as $c \to \infty$).} Let $c > 0$ and $(v^c_0,E^c_0,B^c_0) \in L^2 \times H^s \times H^s$ with $s \in (0,\frac{3}{2})$ satisfying $\textnormal{div}\, v^c_0 = \textnormal{div}\, B^c_0 = 0$ and as $c \to \infty$
			\begin{equation*}
				(v^c_0,E^c_0,B^c_0) \rightharpoonup (\bar{v}_0,\bar{E}_0,\bar{B}_0) \qquad \text{in} \quad L^2 \times H^s \times H^s
			\end{equation*}
			for some $(\bar{v}_0,\bar{E}_0,\bar{B}_0)$ with $\textnormal{div}\, \bar{v}_0 = \textnormal{div}\, \bar{B}_0 = 0$. Then, there exists a sequence of global solutions $(v^c,E^c,B^c)$ to \eqref{NSM} with $\alpha = \frac{3}{2}$ and $(v^c,E^c,B^c)_{|_{t=0}} = (v^c_0,E^c_0,B^c_0)$ given as in Part (i). In addition,  up to an extraction of a subsequence, $(v^c,B^c)$ converges to $(v,B)$ in the sense of distributions as $c \to \infty$, where
			$(v,B)$ satisfies \eqref{HMHD} with $\alpha = \frac{3}{2}$, $\beta = 1$, $\kappa = 0$ and $(v,B)_{|_{t=0}} = (\bar{v}_0,\bar{B}_0)$. The same conclusion can be obtained for the initial data given by one of the parts from $(i)-(b)$ to $(i)-(e)$.
		\end{enumerate}
	\end{theorem}

	\begin{remark} \label{re2} We add some comments to Theorem \ref{theo1-3d}: 
		\begin{enumerate}
			\item As mentioned previously, strategy of proof is similar to that of Theorem \ref{theo1} with using in addition some homogeneous Besov-type maximal regularity estimate for the fractional heat equation, see Proposition \ref{pro-F_heat} in Appendix D in Section \ref{sec:app}. Similar to Theorem \ref{theo1}, for the reader's convenience, we will summarize the conditions of $(\delta,s)$ as follows:
			\begin{center}
				\begin{tikzpicture}  
					\draw[step=1,help lines,black!20] (0,0) grid (5,5);
					
					\draw[very thin, black,->] (-0.5,0) -- (5.5,0) node[below]{$s$};
					\draw[very thin,black,->] (0,-0.5) -- (0,5.5) node[left]{$\delta$};
					
					\foreach \x in {1,2,...,5}
					\draw[shift={(\x,0)}] (0pt,2pt) -- (0pt,-2pt) node[below] {\footnotesize $\x$};
					
					\foreach \y in {1,2,...,5}
					\draw[shift={(0,\y)}] (2pt,0pt) -- (-2pt,0pt) node[left] {\footnotesize $\y$};

					\fill[fill=black!15] (0,0) to (3/2,0) to (5,7/2) to (5,5) to (7/2,5) to (3/2,3) to  (0,0);

					\draw[fill=black] (0,0)  node[below left] {$0$};
					
					\draw[thick,color=black] (0,0) -- (5,5); 
					\draw[] (2.7,2.7) node[above=0.25cm, left, rotate=45] {$\delta = s$}; 
					
					\draw[thick,color=black] (3/2,0) -- (5,7/2); 
					\draw[] (4.1,2.5) node[below=0.25cm, left, rotate=45] {$\delta = s-3/2$}; 
					
					\draw[thick,color=black] (0,0)--(3/2,3);

					\draw[dashed,color=black] (0,3/2)--(3/2,3);
					\draw[thick,color=black] (3/2,3)--(7/2,5); \draw[] (3,5) node[below=0.25cm, left, rotate=45] {$\delta = s+3/2$};
					
					\draw[thick,color=black] (0,0) -- (3/2,0);

					\draw[color=black] (0,0)  circle[radius=1.2pt] node[below left] {};
					\draw[color=black] (3/2,0)  circle[radius=1.2pt] node[below left] {};

					\draw[color=black] (3/2,3)  circle[radius=1.2pt] node[below left] {};

					\draw[color=black] (0,3/2)  circle[radius=1.2pt] node[left] {$\frac{3}{2}$};

					\draw[color=black] (3/2,0)  circle[radius=1.2pt] node[below] {$\frac{3}{2}$};
					
				\end{tikzpicture}
				
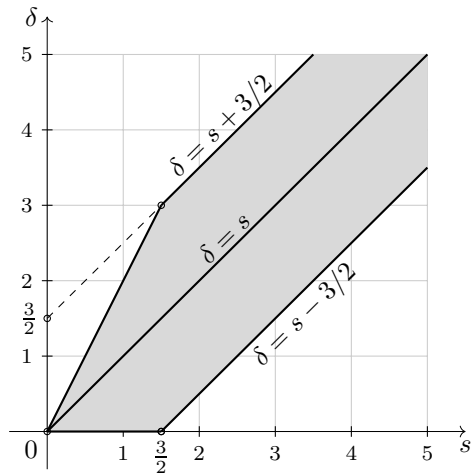
\captionof{figure}{The relation between $s$ and $\delta$ in Theorem \ref{theo1-3d}.}
			\end{center}
			In addition, as it can be seen from the proof below, similar results as in Parts $(i)$ and $(ii)$ also hold in the case $\alpha > \frac{3}{2}$ with a modified range of the initial data, for example in the case of Parts $(i)-(a)$, which should be replaced by $\delta = 0$ and $s \in (0,\alpha)$. Similar notes as Remark \ref{re1}-4 and 5 are also applied here.
			
			\item We first explain why should we choose $\alpha = \frac{3}{2}$ as a critical case. It is well-known that the fractional Navier-Stokes (formally  setting $E = B = 0$ in \eqref{NSM}) and \eqref{MHD} (in the case $\alpha = \beta$ and $\nu = \mu$) equations  have the following scaling property
			\begin{equation*}
				(v_\lambda,B_\lambda,\pi_\lambda)(x,t) \mapsto (\lambda^{2\alpha-1} v, \lambda^{2\alpha-1} B, \lambda^{2(2\alpha-1)}\pi)(\lambda x,\lambda^{2\alpha} t) \qquad \text{for} \quad \lambda > 0.
			\end{equation*}
			In addition, it can be seen formally that from the energy inequality 
			\begin{equation*}
				\mathcal{E}(v,B) := \esssup_{t \in (0, \infty)}\|(v,B)(t)\|^2_{L^2(\mathbb{R}^d)} + 2\nu\|(\Lambda^\alpha v,\Lambda^\alpha B)\|^2_{L^2(0,\infty;L^2(\mathbb{R}^d))} \leq \|(v,B)(0)\|^2_{L^2(\mathbb{R}^d)},
			\end{equation*}
			which yields $\mathcal{E}(v_\lambda,B_\lambda) = \lambda^{4\alpha-2-d}\mathcal{E}(v,B)$. That is why it suggests to take $\alpha = \frac{d+2}{4}$ with $\alpha = 1$ (see \cite{Leray_1933}) and $\alpha = \frac{5}{4}$ (see \cite{Lions_1969,Wu_2003}) in the cases of two and three dimentions, repsectively, to obtain the existence and uniqueness of global in time weak solutions. A similar observation also holds for the Hall equations \eqref{H}, where $\mathcal{E}(B_\lambda) = \lambda^{4\alpha-4-d}\mathcal{E}(B)$ by using the scaling invariance $B_\lambda(x,t) \mapsto \lambda^{2\beta-2} B(\lambda x,\lambda^{2\beta} t)$, and thus it is suggested to take $\beta = \frac{3}{2}$ (see Proposition \ref{pro-H}) and $\beta = \frac{7}{4}$ (see \cite{Wan_2015} and also Proposition \ref{pro-H}) in the cases of $d = 2$ and $d = 3$, respectively. Unfortunately, it does not seem to be the case to \eqref{NSM}, where a similar scaling as above seems does not exist mainly due to the appearance of the electric field. It seems to us that the most difficult term in \eqref{NSM} is the Lorentz force one $j \times B = \sigma cE \times B + \sigma (v \times B) \times B$, which drives the fluid. Compared to the usual fractional Navier-Stokes system, we have two new terms $\sigma c E \times B$ and $\sigma (v \times B) \times B$. While the latter one satisfies the usual scaling property, we have not known any similar thing for the former one, since no scaling information of $E$ has been provided yet. To have a better understanding the situation, it is natural to focus more carefully on the $(E,B)$ system, i.e., the Maxwell equations\footnote{Under suitable assumptions on $v$, the existence and uniqueness of $L^2$ weak solutions $(E,B)$ to \eqref{M} can be established, see Lemma \ref{lem-M}.} 
			\begin{equation} \label{M} \tag{M}
				\left\{
				\begin{aligned}
					\frac{1}{c} \partial_t E - \nabla \times B &= -j = -\sigma (c E + v \times B),
					\\
					\frac{1}{c} \partial_t B + \nabla \times E &= 0, \quad \textnormal{div}\, B = 0.
				\end{aligned}
				\right.
			\end{equation}
			Similar to \eqref{NSM}, we do not have any scaling property to \eqref{M} even in the case $v = 0$. However, if we formally drop out the electric current field, i.e., the term on the right-hand side of the first equation (it can be done formally either by taking $\sigma = 0$ or by setting $v = 0$ and ignoring the electric damping term $-\sigma c E$) then in these cases  
			\eqref{M} is rewritten by 
			\begin{equation} \label{M'} \tag{M'}
				\left\{
				\begin{aligned}
					\frac{1}{c} \partial_t E - \nabla \times B &= 0,
					\\
					\frac{1}{c} \partial_t B + \nabla \times E &= 0, \quad \textnormal{div}\, B = 0,
				\end{aligned}
				\right.
			\end{equation}
			which is invariant under the scaling $(E_\lambda,B_\lambda) \mapsto \lambda^\beta(E,B)(\lambda x,\lambda t)$ for any real number $\beta$. Furthermore, by defining\footnote{The existence and uniqueness of $L^2$ weak solutions $(E,B)$ to \eqref{M'} can be found in Lemma \ref{lem-M}.}
			\begin{equation*}
				\mathcal{E}(E,B) := \esssup_{t \in (0,\infty)} \|(E,B)(t)\|^2_{L^2(\mathbb{R}^d)} \leq \|(E,B)(0)\|^2_{L^2(\mathbb{R}^d)},
			\end{equation*}
			it follows that $\mathcal{E}(E_\lambda,B_\lambda) = \lambda^{2\beta-d} \mathcal{E}(E,B)$, which suggests us to choose $\beta = \frac{d}{2}$ with $\beta = \frac{3}{2}$ in the three-dimensional case. Coming back to the Lorentz force $j \times B$, if we scale $v_\lambda(t,x) \mapsto \lambda^\gamma v(\lambda x,\lambda t)$ for some real number $\gamma$ to be determined later and use the scaling property of \eqref{M'} for $(E,B)$ then this force is invariant under choosing $\gamma = 0$. Thus, in order to control the Lorentz force term by using the fractional Laplacian one with the scaling $(v_\lambda,E_\lambda,B_\lambda)(x,t) \mapsto (v,\lambda^\beta E,\lambda^\beta B)(\lambda x,\lambda t)$, it suggests us to take $\alpha \geq \beta = \frac{d}{2}$. Therefore, in two dimensions this also explains the GWP result given in Theorem \ref{theo1} in the case $\alpha = 1$ with a slightly stronger assumption on the initial data, i.e., $(v_0,E_0,B_0) \in L^2 \times H^s \times H^s$ for any $s \in (0,1)$ and probably raises a difficult problem in the case $\alpha \in (0,1)$. We should mention here that the critical case $\alpha = \frac{3}{2}$ to \eqref{NSM} can be compared to the results in \cite{Wu_2011} in the three-dimensional case, where the author proved the GWP of \eqref{MHD} for $\alpha \geq \frac{5}{4}, \beta > 0$ and $\alpha + \beta \geq \frac{5}{2}$ (in fact the author provided  general results in $d$ dimensions, for similar partial fractional dissipation results, see also  \cite{Yang-Jiu-Wu_2019}), so if we choose $\beta = 1$ then we should take $\alpha \geq \frac{3}{2}$. Moreover, the case $(\alpha,\beta) = (\frac{3}{2},1)$ can be obtained by taking the limit as $c \to \infty$ in which \eqref{NSM} with $\alpha = \frac{3}{2}$ converges to \eqref{HMHD} with $\alpha = \frac{3}{2}$ and $\kappa = 0$, in the sense of distributions, see Theorem \ref{theo1-3d}-$(ii)$. 
			
			\item In the case $\delta = s > 0$, we should remark that a more general result has been obtained in \cite{Wen-Ye_2020}. More precisely, the GWP of \eqref{NSM} is provided with replacing $-(-\Delta)^\alpha v$ by $-L^2 v$, where for some nondecreasing symmetric function $g \geq 1$, the operator $L$ is defined via Fourier transform as follows
			\begin{equation*}
				\mathcal{F}(Lu)(\xi) := \frac{|\xi|^\frac{d}{2}}{g(\xi)} \mathcal{F}(u)(\xi) \qquad \text{with} \qquad \int^\infty_e \frac{1}{s\log(s)g^2(s)} \,ds = \infty.
			\end{equation*}
			The above stronger conditions are inspirited by the similar weaker ones for the supercritical hyperdissipative Navier-Stokes equations given in \cite{Barbato-Morandin-Romito_2014,Tao_2009}, where the first paper did not need the above logarithmic term and improved the result in the second one, which also did not assume the logarithmic term but requiring $g^4$ instead of $g^2$. As mentioned previously, the critical case for the usual fractional Navier-Stokes equations is $\alpha = \frac{d+2}{4}$. By choosing $g = 1$, $\delta = s > 0$ and either $d = 2$ or $d = 3$, the result  in \cite{Wen-Ye_2020} reduces to Theorem \ref{theo1} or Theorem  \ref{theo1-3d}. However, they have not been explained about the choice of the exponent $\frac{d}{2}$ in the definition of $L$ and have not been considered lower regularity data cases, for instance $\delta = 0$ and $s \in (0,\frac{d}{2})$. In addition, it is possible to obtain (at least) the existence of global weak solutions to \eqref{NSM} for $\alpha \in (1,\frac{3}{2})$ and for small data by adapting the technique provided in \cite[Theorem 1.1]{Arsenio-Gallagher_2020}. Furthermore, it is also natural to ask the two following questions: 1) Can the above logarithmic term be dropped out as  in \cite{Barbato-Morandin-Romito_2014,Tao_2009}? and 2) Can the regularity of $v_0$ be reduced in \cite{Wen-Ye_2020}, namely, $(v_0,E_0,B_0) \in L^2 \times H^s \times H^s$ for $s \in (0,\frac{d}{2})$ for $d \in \{2,3\}$?
			
			\item Theorem  \ref{theo1-3d}-$(ii)$ also says that the hyperbolicity of \eqref{NSM} (due to the Maxwell equations) is weakly transferred into the parabolicity of \eqref{HMHD} with $\kappa = 0$ as $c \to \infty$. See also Lemma \ref{lem-M}, where under suitable assumptions on the velocity, a similar result is obtained for the Maxwell equations \eqref{M}, even for $L^2$ initial data. For more general estimates on \eqref{M}, we refer to \cite{Arsenio-Gallagher_2020,Germain-Ibrahim-Masmoudi_2014,Giga-Ibrahim-Shen-Yoneda_2018,Ibrahim-Keraani_2012,Ibrahim-Lemarie-Rieusset-Masmoudi_2018}. 
		\end{enumerate}
	\end{remark}
	
	As mentioned previously in Remark \ref{re2}, for the sake of completeness we will summarize the GWP of the Hall system (i.e., \eqref{HMHD} without the fluid equations) in the two and three-dimensional cases as follows. This system is also known as the electron MHD equations.
	
	\begin{proposition} \label{pro-H} Let $d \in \{2,3\}$, $B_0 \in H^s(\mathbb{R}^d)$ with $s \in [0,\infty)$, $\kappa,\sigma \in (0,\infty)$ and $T \in (0,\infty)$. Assume that $\beta \geq \frac{3}{2}$ if $d = 2$ and $\beta \geq \frac{7}{4}$ if $d = 3$. Then the Hall system
		\begin{equation} \label{H} \tag{H}
			\partial_t B = -\frac{1}{\sigma}(-\Delta)^\beta B - \frac{\kappa}{\sigma} \nabla \times ((\nabla \times B) \times B)
			\qquad \text{and}\qquad
			\textnormal{div}\, B = 0,
		\end{equation}
		has a unique global solution $B \in L^\infty(0,T;H^s) \cap L^2(0,T;H^{s+\alpha})$ satisfying for $t \in (0,T)$
		\begin{equation*}
			\|B(t)\|^2_{H^s} + \int^t_0 \|B\|^2_{H^{s+\beta}} \,d\tau \leq C(T,\beta,\kappa,\sigma,s,B_0).
		\end{equation*} 
	\end{proposition}
	
	\begin{remark} We add some remarks to Proposition \ref{pro-H} as follows: It can be seen from the proof given in Appendix F in Section \ref{sec:app} that similar results as Proposition \ref{pro-H} (i.e., for initial data $(v_0,B_0) \in H^s(\mathbb{R}^d)$ with $s \geq 0$) can be obtained when we couple \eqref{H} together with the fractional Navier-Stokes equations for the velocity fractional Laplacian $-(-\Delta)^\alpha v$ for $\alpha \geq \frac{d+2}{4}$ (as  \cite{Wan_2015} for $d = 3$ and for $s > \frac{5}{2}$). It seems to us the technique in the proof of Proposition \ref{pro-H}, which can also be adapted to obtain the GWP of \eqref{MHD} with initial data $(v_0,B_0) \in H^s(\mathbb{R}^d)$ for $s \geq 0$ in the case either $\alpha \geq \frac{d+2}{4}, \beta > 0$ and $\alpha + \beta \geq \frac{d+2}{2}$ (as  \cite{Wu_2011} for $s > \frac{d}{2} + 1$), or $\alpha \geq \frac{d+2}{4}$ and $\beta \geq \frac{d+2}{4}$ (as  \cite{Wu_2003} for $s \geq \max\{2\alpha,2\beta\}$). Furthermore, we do not investigate the large-time behavior here, but it can be easily obtained by adapting the Fourier splitting method provided in \cite{Chae-Schonbek_2013,Schonbek_1985,Schonbek_1986}, see also the proof of Theorem \ref{pro-HMHD}. In addition, we should also remark that the local existence of strong solutions to the inviscid \eqref{HMHD} (i.e., $\nu = 0$) has been provided in \cite{Chae-Wan-Wu_2015} for $\beta > \frac{1}{2}$. Finally, the GWP can be established in the critical case $\beta = \frac{1}{2}$ for small data (see \cite{Jeong-Oh_2022}).
	\end{remark}
	
		%
	\subsubsection{Magnetic helicity conservation}
	%
	
	Our next result visits the issue of conservation of magnetic helicity to \eqref{NSM} in three  dimensions as the electric conductivity goes to infinity as follows.
	
	\begin{theorem}[Magnetic helicity conservation as $\sigma \to \infty$] \label{theo-MH} Let $\alpha = \frac{3}{2}$, $s > \frac{3}{2}$ and $c,\nu,\sigma > 0$. Assume that $(v^{\sigma}_0, E^{\sigma}_0, B^{\sigma}_0) \in H^s(\mathbb{R}^3)$ with its $L^2$ norm is uniformly bounded in terms of $\sigma$ 
	and
	$\textnormal{div}\, v^{\sigma}_0 = \textnormal{div}\, B^{\sigma}_0 = 0$.
	For any $T \in (0,\infty)$, \eqref{NSM} has a unique global solution $(v^{\sigma}, E^{\sigma}, B^{\sigma})$ in $(0,T)$ with $(B^{\sigma},E^{\sigma},B^{\sigma})_{|_{t=0}} = (v^{\sigma}_0,E^{\sigma}_0,B^{\sigma}_0)$  given as in Theorem \ref{theo1-3d}. 
		Furthermore, if $(E^{\sigma}_0,B^{\sigma}_0) \in \dot{H}^{-1}(\mathbb{R}^3)$ and $B^{\sigma}_0 \to B_0$ in $\dot{H}^{-1}(\mathbb{R}^3)$ for some $B_0 \in (H^s \cap \dot{H}^{-1})(\mathbb{R}^3)$ with $\textnormal{div}\, B_0 = 0$ then 
		\begin{equation*}
			\lim_{\sigma \to \infty} \int_{\mathbb{R}^3} A^{\sigma}(t) \cdot B^{\sigma}(t) \,dx 
			=  \int_{\mathbb{R}^3} A_0 \cdot B_0\,dx \quad \text{for a.e. }   t \in (0,T),
		\end{equation*}
		where  $\nabla \times  f = g$ and $\textnormal{div}\,f = 0$ for $f \in \{A^{\sigma},A_0\}$ corresponding to $g \in \{B^{\sigma},B_0\}$. In addition, if the initial magnetic helicity is positive then there exists an absolute positive constant $C$ such that
		\begin{equation*}
			\liminf_{t\to \infty} \liminf_{\sigma \to \infty} \|B^{\sigma}(t)\|^2_{\dot{H}^{-\frac{1}{2}}(\mathbb{R}^3)} \geq  C\int_{\mathbb{R}^3} A_0 \cdot B_0 \,dx > 0.
		\end{equation*}
	\end{theorem}
	
	\begin{remark} We add some comments to Theorem \ref{theo-MH} as follows: Our motivation to investigate the conservation of magnetic helicity is inspired by Theorem \ref{theo1-3d}-$(ii)$, where \eqref{NSM} with $\alpha = \frac{3}{2}$ converges weakly to \eqref{HMHD} with $\alpha = \frac{3}{2}$, $\beta = 1$ and $\kappa = 0$. We note that the additional condition of the initial data in $\dot{H}^{-1}$ will be used to ensure the well-defined for the magnetic helicity on the whole space $\mathbb{R}^3$. If we are either in bounded domains or on the three-dimensional torus, then this assumption is not needed, but we have to study the GWP of \eqref{NSM} in this case, which does not seem to be known. 
	Since we are in the whole space $\mathbb{R}^3$ then the magnetic helicity does not depend on the choice of potential vector $A^\sigma$. We note that Theorem \ref{theo-MH} can be considered as a similar version in $\mathbb{R}^3$ of the Taylor’s conjecture on magnetic helicity conservation, which was solved recently  in \cite{Faraco-Lindberg_2020,Faraco-Lindberg_2022}.
	\end{remark}
	
	%
	\subsubsection{Local well-posedness}
	%
	
	Our next result is concerned the LWP of \eqref{NSM} in the inviscid case. That will allow us to study further either the inviscid limit as $\nu \to 0$ or the limit as $c \to \infty$ in suitable frameworks. More precisely, the statement is given as follows.
	
	\begin{theorem}[Local well-posedness, inviscid limit and the limit as $c \to \infty$] \label{theo-inviscid}
		Let $d \in \{2,3\}$, $c,\sigma > 0$ and  $(v_0,E_0,B_0) \in H^s(\mathbb{R}^d)$ with $\textnormal{div}\, v_0 = \textnormal{div}\, B_0 = 0$ and $s \in \mathbb{R}, s > \frac{d}{2} + 1$. 
		\begin{enumerate}
			\item[(i)] \textnormal{(Local well-posedness).} There exists a unique local solution $(v,E,B)$ to \eqref{NSM} with $\nu = 0$ and $(v,E,B)_{|_{t=0}} = (v_0,E_0,B_0)$ in $(0,T_0)$ for some $T_0 = T_0(d,\sigma,s,v_0,E_0,B_0) > 0$ such that $(v,E,B) \in L^\infty(0,T_0;H^s)$ satisfying for $t \in (0,T_0)$
			\begin{align*}
				\|(v,E,B)(t)\|^2_{H^s} + \int^t_0 \|j\|^2_{H^s}\,d\tau &\leq C(T_0,d,\sigma,s,v_0,E_0,B_0).
			\end{align*}
			
			\item[(ii)] \textnormal{(Inviscid limit).} Let $\alpha = 1$ and $\nu > 0$. Then there exists a sequence of solutions $(v^\nu,E^\nu,B^\nu)$ to \eqref{NSM} with  $(v^\nu,E^\nu,B^\nu)_{|_{t=0}} = (v_0,E_0,B_0)$ given globally as in Theorem \ref{theo1} if $d = 2$ and locally as in Part (i) if $d = 3$, respectively. Moreover, for $t \in (0,T_0)$ and for $s' \in [0,s)$
			\begin{equation*}
				\|(v^\nu-v,E^\nu-E,B^\nu-B)(t)\|_{H^{s'}} \leq \nu^\frac{s-s'}{s}C(T_0,d,\sigma,s,v_0,E_0,B_0),
			\end{equation*}
			where $(v,E,B)$ is the unique local solution to \eqref{NSM} with $\nu = 0$ and $(v,E,B)_{|_{t=0}} = (v_0,E_0,B_0)$ given as in Part (i).
			
			\item[(iii)] \textnormal{(The limit as $c \to \infty$).} Let $c > 0$ and $(v^c_0,E^c_0,B^c_0) \in H^s$ satisfying $\textnormal{div}\, v^c_0 = \textnormal{div}\, B^c_0 = 0$ and as $c \to \infty$
			\begin{equation*}
				(v^c_0,E^c_0,B^c_0) \rightharpoonup (\bar{v}_0,\bar{E}_0,\bar{B}_0) \qquad \text{in} \quad H^s
			\end{equation*}
			for some $(\bar{v}_0,\bar{E}_0,\bar{B}_0)$ with $\textnormal{div}\, \bar{v}_0 = \textnormal{div}\, \bar{B}_0 = 0$. Then there exists a sequence of solutions $(v^c,E^c,B^c)$ to \eqref{NSM} with $\nu = 0$ and $(v^c,E^c,B^c)_{|_{t=0}} = (v^c_0,E^c_0,B^c_0)$ given as in Part (i) in $(0,T_0)$ for some $T_0 > 0$. In addition,  up to an extraction of a subsequence, $(v^c,B^c)$ converges to $(v,B)$ in the sense of distributions as $c \to \infty$, where
			$(v,B)$ satisfies \eqref{HMHD} with $\beta = 1$, $\nu = \kappa = 0$ and $(v,B)_{|_{t=0}} = (\bar{v}_0,\bar{B}_0)$.
		\end{enumerate}
	\end{theorem}
	
	\begin{remark} We add some comments to Theorem \ref{theo-inviscid} as follows: The proofs of Parts $(i)$ and $(ii)$ share the same ideas as those of the LWP of Euler equations and the invicid limit from the Navier-Stokes equations to the Euler system. The proof of Part $(iii)$ follows the ideas from \cite{Arsenio-Gallagher_2020} and Theorem \ref{theo1-3d}-$(ii)$.
	\end{remark}
		
	%
	\subsubsection{Stability and large-time behavior}
	%
	
	\textbf{A. The case of \eqref{NSM}.} Let us now focus on the stability issue of \eqref{NSM} around its stationary states. In this case, if we look for a zero-velocity steady solution, i.e., $(v^* = 0,E^*,B^*,\pi^*)$ then it should satisfy
	\begin{equation} \label{S-NSM} \tag{S-NSM}
		\nabla \pi^* = \sigma c E^* \times B^*, 
		\quad
		\nabla \times B^* = j^* = \sigma c E^*,
		\quad
		\nabla \times E^* = 0
		\qquad \text{and}\qquad
		\textnormal{div}\, B^* = 0. 
	\end{equation}
	Indeed, by using the following well-known identity
	\begin{equation*}
		j^* \times B^* = (\nabla \times B^*) \times B^* = B^* \cdot \nabla B^* - \frac{1}{2} \nabla |B^*|^2,
	\end{equation*}
	it follows that $B^*$ also satisfies the following stationary Euler-type equations, which is also known as the magnetohydrostatic system 
	\begin{equation} \label{MHS} \tag{MHS}
		B^* \cdot \nabla B^* + \nabla p^* = 0 \qquad \text{and}\qquad \textnormal{div}\, B^* = 0 \qquad \text{where}\qquad p^* := -\frac{1}{2}|B^*|^2 - \pi^*.
	\end{equation}
	In three dimensions, solutions to \eqref{MHS} either in bounded domains or on the torus are recently constructed in \cite{Constantin-Pasqualotto_2023} as infinite time limits of Voigt approximations\footnote{That means $(\partial_t v,\partial_t B)$ is replaced by $(\partial_t (-\Delta)^{\alpha_0} v, \partial_t (-\Delta)^{\beta_0} B)$ for some $\alpha_0,\beta_0 > 0$.} of viscous and non-resistive \eqref{MHD} (i.e., with $\alpha = 1$ and $\mu = 0$). It is also believed that \eqref{MHS} plays an important role in connection to the design of nuclear fusion devices such as tokamaks and stellarators. There are several examples of $(v^*,E^*,B^*,\pi^*)$ to either \eqref{S-NSM} or \eqref{MHS} such as for $x \in \mathbb{R}^d$
	\begin{align*}
		v^* &= E^* = 0, \quad  B^*(x) = \text{constant vector in } \mathbb{R}^3 \qquad \text{and}\qquad \pi^* = \text{constant};
		\\
		v^* &= E^* = 0, \quad B^*(x) \in \{(-x_1,x_2,0),(x_2,x_1,0),(0,x_3,x_2),...\} \qquad \text{and}\qquad \pi^* = \text{constant}.
	\end{align*}
	By setting
	\begin{equation*}
		\bar{v} := v + v^* = v, \quad \bar{E} := E + E^*, \quad \bar{B} := B + B^*  \qquad \text{and} \qquad \bar{\pi} := \pi + \pi^*,
	\end{equation*}
	it can be seen from \eqref{NSM} for $(\bar{v},\bar{E},\bar{B},\bar{\pi},\bar{j})$ and \eqref{S-NSM} that the perturbation $(v,E,B,\pi)$ satisfies
	\begin{equation} \label{NSM*} \tag{NSM*}
		\left\{
		\begin{aligned}
			\partial_t v + v \cdot \nabla v + \nabla (\pi + \pi^*) &= -\nu (-\Delta)^\alpha v + (j + j_*) \times (B + B^*), 
			\\
			\frac{1}{c}\partial_t E - \nabla \times (B + B^*)  &= -(j + j_*),
			\\
			\frac{1}{c}\partial_t B + \nabla \times E &= 0,
			\\
			\sigma(c(E + E^*) + v \times (B + B^*)) &= j + j_* = \bar{j}, 
			\\
			\sigma(cE + v \times B) &= j,
			\\
			\sigma(cE^* + v \times B^*) &= j_*,
			\\
			\textnormal{div}\, v = \textnormal{div}\, B &= 0,
		\end{aligned}
		\right.
	\end{equation}
	with the initial data is denoted by $(v,E,B)_{|_{t=0}} = (v_0,E_0,B_0)$. We are now going to the statement, which is given as follows.
	
	\begin{theorem}[Velocity damping effect on the stability near a constant magnetic field $B^*$] \label{theo3} Let $d \in \{2,3\}$, $\alpha = 0$ and $c,\nu,\sigma > 0$. Assume that $(v_0,E_0,B_0) \in H^s(\mathbb{R}^d)$ with $\textnormal{div}\, v_0 = \textnormal{div}\, B_0 = 0$ and $s \in \mathbb{R},s > \frac{d}{2} + 1$. Suppose that $B^*$ is a constant vector in $\mathbb{R}^3$ with $\epsilon_* := \|B^*\|_{L^\infty}$. Then the following properties hold. 
	\begin{enumerate}
		\item[(i)]\textnormal{(Stability around a constant magnetic field $B^*$).}  There exists a constant $\epsilon_0 = \epsilon_0(\nu,\sigma,s) > 0$ such that if $\|(v_0,E_0,B_0)\|^2_{H^s} \leq \epsilon^2_0$ then there is a unique global solution $(v,E,B)$ to \eqref{NSM*} satisfying for $t > 0$
		\begin{align*}
			&\|(v,E,B)(t)\|^2_{H^s} + \int^t_0 \nu\|v\|^2_{H^s} + \frac{1}{\sigma}\|\bar{j}\|^2_{H^s} \,d\tau \leq 2\epsilon^2_0,
			\\
			&\int^t_0 \|E\|^2_{\dot{H}^{s''}} + \|B\|^2_{\dot{H}^{s'}} \,d\tau \leq \epsilon^2_0
			\begin{cases}
				C(c,\nu,\sigma,s) &\text{if} \quad c \in (0,1),
				\\
				C(\nu,\sigma,s) &\text{if} \quad c \geq 1.
			\end{cases}
		\end{align*}
		In addition, for any $s' \in [0,s)$, $s'' \in [1,s)$, $p \in (2,\infty]$, $q \in [1,\infty]$ and for some constant $b_0 \in [0,\epsilon_0)$ as $t \to \infty$
		\begin{equation*}
			\|(v,E,\bar{j},j)(t)\|_{H^{s'}},  \|B(t)\|_{L^p}, \|B(t)\|_{\dot{H}^{s''}}, \|B(t)\|_{L^q_{\textnormal{loc}}} \to 0
			\qquad \text{and}\qquad 
			\|B(t)\|_{L^2} \to b_0.
		\end{equation*}

		\item[(ii)] \textnormal{(The limit as $c \to \infty$).} 
		Let $c > 0$ and $(v^c_0,E^c_0,B^c_0) \in H^s$ satisfying $\textnormal{div}\, v^c_0 =  \textnormal{div}\, B^c_0 = 0$, 
		\begin{equation*}
				\|(v^c_0,E^c_0,B^c_0)\|^2_{H^s} \leq \epsilon^2_1 \qquad\text{and}\qquad (v^c_0,E^c_0,B^c_0) \rightharpoonup (\bar{v}_0,\bar{E}_0,\bar{B}_0) \quad \text{in } \, H^s \text{ as } c \to \infty
		\end{equation*}
		for some small $\epsilon_1 = \epsilon_1(\nu,\sigma,s) > 0$ and  for some $(\bar{v}_0,\bar{E}_0,\bar{B}_0) \in H^s$ with $\textnormal{div}\, \bar{v}_0 = \textnormal{div}\, \bar{B}_0 = 0$. Then, there exists a sequence of global solutions $(v^c,E^c,B^c)$ to \eqref{NSM*} with $\alpha = 0$ and $(v^c,E^c,B^c)_{|_{t=0}} = (v^c_0,E^c_0,B^c_0)$ given as in Part (i). In addition,  up to an extraction of a subsequence, $(v^c,B^c)$ converges to $(v,B)$ in the sense of distributions as $c \to \infty$, where
		$(v,B)$ satisfies \eqref{HMHD*} with $\kappa = 0$ and $(v,B)_{|_{t=0}} = (\bar{v}_0,\bar{B}_0)$.
		\end{enumerate}
	\end{theorem}

	\begin{remark} We add some comments to Theorem \ref{theo3}:
		\begin{enumerate}
			\item Strategy of proof: The proof is based on the energy method with using some nice cancellation properties, which related to the constant vector $B^*$, which allow us to define a suitable energy form. We then obtain a bound for this energy form locally in time in which by using the smallness of the initial data and a bootstrap argument, the global in time estimate can be established. Then, the large-time behavior can be ontained by using the damping structure of the system.
						
			\item It seems to us that Theorem \ref{theo3} is the first stability result to \eqref{NSM}. The case $\alpha = 0$ means that we have a velocity damping term. Moreover, it can be seen from the last three relations in \eqref{S-NSM} that $\Delta B^* = 0$, and furthermore by Liouville's theorem (see \cite{Evans_2010}), if $B^*$ is bounded then $B^*$ is a constant vector. Thus, the boundedness of $B^*$ is equivalent to the constant one. Note that if $B^*$ is a constant vector in $\mathbb{R}^3$ then $j^* = E^* = \nabla \pi^* = 0$.  If we choose $\epsilon_0$ even smaller then the upper bound on the right-hand side of the main inequality in Part $(i)$ can be replaced by $\epsilon^2_0$. 
			
			\item The case that $B^*$ is not a constant (unbounded) vector given  in the previous examples, which is much more complicated and will be considered in a forthcoming work. Similar things happen in the case $\alpha = 1$ even in the case that $B^*$ is a constant vector in $\mathbb{R}^3$. The main difficulty in these cases is the control of either the weighted term $j \times B^*$ or $\|v\|_{L^2_tL^\infty_x}$ in which at the moment it seems not clear to us.
			
			\item How to obtain an explicit rate of convergence as $t \to \infty$ is not clear to us in this case, which is different to the case of \eqref{HMHD} in which under additional assumptions on the initial data, a logarithmic rate is obtained, see Theorem \ref{pro-HMHD} below.
			
			\item As it can be seen from the proof of Theorem \ref{theo3} that we also obtain a similar bound  in Part $(i)$ as replacing $\bar{j}$ by $j$ with a slight different unper bound such as $C(c,\sigma,\epsilon_*)\epsilon^2_0$ instead of $2\epsilon^2_0$. In addition, for $r \in [0,s-1)$
			\begin{equation*}
				\|(\partial_t v,\partial_t E,\partial_tB)(t)\|_{H^r} \to 0 \quad \text{as} \quad t \to \infty.
			\end{equation*}
			Moreover, if $\partial_t B$ decays sufficiently fast in $H^r$ (with an explicit rate of convergence, for example $t^{-\gamma}$ for some $\gamma > 1$) then we can conclude by using the fundamental theorem of calculus
			in time that $B \to b$ strongly in $H^r$ as $t \to \infty$ for some $b$, see \cite{Beekie-Friedlander-Vicol_2022}.
		\end{enumerate}
	\end{remark}
	
	\noindent
	\textbf{B. The case of \eqref{HMHD}.} Next, we will study the stability of \eqref{HMHD} around its zero-velocity stationary solutions with a constant magnetic field $B^*$. In addition, we also provide the large-time behavior of the corresponding perturbation $(v,B)$ in $L^2$ norm under suitable assumptions on the intial data. It is inspired by Theorem \ref{theo3} and also by the so-called magnetic relaxation phenomena of the non-resistive \eqref{MHD} system (i.e., with $\mu = 0$). Indeed, it is given formally in \cite{Moffat_1985} as follows: \textit{If $(v,B)$ is a smooth solution to \eqref{HMHD} without magnetic diffusion and with $\kappa = 0$ then $\|v(t)\|_{L^2} \to 0$ and $B$ converges to a stationary Euler flow as $t \to \infty$}. Recent related results in this direction are obtained either on $d$-dimensional torus  or in bounded domains in \cite{Beekie-Friedlander-Vicol_2022,Constantin-Pasqualotto_2023}. It can be compared to \eqref{NSM} in Theorem \ref{theo3}, where the time limit of the perturbation $B$ in $L^2$ norm as $t \to \infty$ is given by a constant $b_0 \in [0,\epsilon_0)$. It is not clear to us, even in addition $(v_0,E_0,B_0) \in L^1$, that whether or not $b_0 = 0$. However, the $L^\infty$ norm of $B$ converges to zero at infinite time, but without an explicit rate of decaying. Therefore, it seems to us that there is a gap between the "magnetic relaxation" of \eqref{NSM} and that of \eqref{HMHD} in the case $\alpha = 0$, even in the latter case we should assume an extra condition $(v_0,B_0) \in L^1$, but with an explicit asymptotic behavior. If $(v^*,B^*,p^*)$ is a stationary solution to \eqref{HMHD} with $\beta = 1$ and $v^* = 0$ then for $j^* := \nabla \times B^*$
	\begin{equation} \label{S-H-MHD} \tag{S-H-MHD}
		\nabla \pi^* = j^* \times B^*,
		\quad
		\frac{1}{\sigma}\Delta B^* = \frac{\kappa}{\sigma} \nabla \times (j^* \times B^*)
		\qquad \text{and}\qquad
		\textnormal{div}\, B^* = 0.
	\end{equation}
	As mentioned previously, if $B^*$ is a solution to \eqref{S-H-MHD} then $B^*$ also satisfies \eqref{MHS}. Note that the examples in Case A also satisfy \eqref{S-H-MHD}. Moreover, it follows from \eqref{HMHD} for $(\bar{v},\bar{B},\bar{\pi},\bar{j})$ with $\alpha = 0$ and $\beta = 1$, and \eqref{S-H-MHD} that the perturbation $(v := \bar{v}-v^*,B := \bar{B}-B^*,\pi := \bar{\pi}-\pi^*,j := \bar{j}-j^*)$ with $j := \nabla \times B$ satisfies
	\begin{equation} \label{HMHD*} \tag{H-MHD*}
		\left\{
		\begin{aligned}
			\partial_t v + v \cdot \nabla v + \nabla \pi &= -\nu v + j \times (B+B^*) + j^* \times B,
			\\
			\partial_t B - \nabla \times (v \times (B+B^*)) &= \frac{1}{\sigma}\Delta B - \frac{\kappa}{\sigma} \nabla \times (j \times (B+B^*)) - \frac{\kappa}{\sigma} \nabla \times (j^* \times B),
			\\
			\textnormal{div}\, v = \textnormal{div}\, B &= 0,
		\end{aligned}
		\right.
	\end{equation}
	in which the initial data is given by $(v,B)_{|_{t=0}} = (v_0,B_0)$. In Theorem \ref{theo3}-$(ii)$, we prove that \eqref{NSM*} converges to \eqref{HMHD*} with $k = 0$ as $c \to \infty$ in the sense of distributions. However, we are not able to prove the convergence of $B$ to zero in $L^2$ norm, but in $L^\infty$ one, and the rate of decaying in time of $(v,B)$ is implicit. The next result shows that we can obtain an explicit rate of convergence as $t \to \infty$ for $(v,B)$, which satisfies \eqref{HMHD*}, under an additional assumption of initial data, even in the case $\kappa \geq 0$.
	
	\begin{theorem}[A counterpart of Theorem \ref{theo3}] \label{pro-HMHD}		
		Let $d \in \{2,3\}$, $\alpha = 0$, $\kappa \geq 0$, $\nu,\sigma > 0$ and $(v_0,B_0) \in H^s(\mathbb{R}^d)$ with $s \in \mathbb{R}$,  $s > \frac{d}{2} + 1$. 
		Assume that $B^*$ is a constant verctor in $\mathbb{R}^3$ with $\epsilon_* := \|B^*\|_{L^\infty}$. 
		There exists a constant $\epsilon_0 = \epsilon_0(\kappa,\nu,\sigma,s) > 0$ such that if $\|(v_0,B_0)\|^2_{H^s} \leq \epsilon^2_0$ then there is a unique global solution $(v,B)$ to \eqref{HMHD*} such that for $t > 0$
		\begin{equation*}
			\|(v,B)(t)\|^2_{H^s} + \int^t_0 \nu\|v\|^2_{H^s} + \frac{1}{\sigma}\|\nabla B\|^2_{H^s} \,d\tau \leq 2 \epsilon^2_0,
		\end{equation*}
		and for $s' \in [0,s)$, $s'' \in [1,s)$, $s''' \in [0,s-2)$, $p \in (2,\infty]$ and $q \in [1,\infty]$
		\begin{equation*}
			\|v(t)\|_{H^{s'}}, \|B(t)\|_{L^p},  \|B(t)\|_{L^q_{\textnormal{loc}}}, \|B(t)\|_{\dot{H}^{s''}}, \|(\partial_t v, \partial_t B)(t)\|_{H^{s'''}} \to 0 \quad \text{as} \quad t \to \infty.
		\end{equation*}
		In addition, if $(v_0,B_0) \in L^1$ then for $t > 0$, $s' \in [0,s)$ and for each $m \in \mathbb{N}$ with $m \geq 2$
		\begin{equation*}
			\|(v,B)(t)\|_{H^{s'}} \leq C(\epsilon_0,\epsilon_*,\kappa,\nu,m,\sigma,s,v_0,B_0) 	\log^{-\frac{(m-1)(s-s')}{2s}}(e+m^{-1}\nu t). 
		\end{equation*}		
	\end{theorem}

	\begin{remark} We add some comments to Theorem \ref{pro-HMHD}:
		\begin{enumerate}
			\item Strategy of proof: The proof is similar to that of Theorem \ref{theo3}. In addition, to obtain explicit rate of decaying in time, we can apply the Fourier splitting method provided in \cite{Chae-Schonbek_2013,Schonbek_1985,Schonbek_1986} with an additional assumption on the initial data. However, there are new terms retaled to $B^*$, which should be controlled in a different way.
			
			\item In two dimensions, it is well-known that for either \eqref{MHD} with $\alpha = 0$ and $\beta = 1$ or \eqref{HMHD*} with $B^* = 0$ and $\kappa = 0$, the GWP for large initial data has not been established yet. For large initial data and $\nu = 0$, the authors in \cite{Cao-Wu-Yuan_2014} have been provided the existence and uniqueness of global solutions in $H^s(\mathbb{R}^2)$ for $s \in \mathbb{R}, s > 2$, to \eqref{MHD} for $\beta > 1$. Their idea can be adapted to \eqref{HMHD*} in the case $B^* = 0$, $\nu = \kappa = 0$, $d = 2$, and with replacing $\Delta B$ by $-(-\Delta)^\beta B$ for $\beta > 1$.
			
			\item There are also several stability results to \eqref{MHD} in two dimentions. In this case, the authors in \cite{Lin-Xu-Zhang_2015,Wu-Wu-Xu_2015} studied the stability without the magetic diffusion term, and with either Laplacian or damping velocity. In these two papers, the authors considered the system of $(v,\phi)$ instead of $(v,B)$, where $B = (\partial_2 \phi,-\partial_1\phi)$, and they investigated the stability of $(v,\phi)$ around $(v^*,\phi^*) = (0,x_2)$ or equivalently of $(v,B)$ near $(v^*=0,B^* = (1,0))$. Recently, the authors in \cite{Jo-Kim-Lee_2022} improved the result in \cite{Wu-Wu-Xu_2015} by considering lower regular data. More precisely, they proved the stability near $B^* = (1,0)$ for the initial data in a rougher space than $H^4 \cap L^1$ and large-time behavior in $L^2$ norm with an optimal decay rate for $H^5 \cap W^{2,1}$ initial data.
			 
			\item Note that the authors in \cite{Chae-Schonbek_2013} have been provided the temporal decay in time of energy solutions and also of higher regular ones to \eqref{HMHD} with $\alpha = 1$ and $d = 3$. Here, since we do not focus on obtaining an optimal decay then the rate of decay can be improved in one way or another. In addition, in Remark \ref{re-3D} we point out that it is difficult to obtain polynomial decay rate when using the Fourier splitting method to \eqref{HMHD*} in three dimensions.
		\end{enumerate}
	\end{remark}

	The rest of the paper is organized as follows: The proofs of Theorems \ref{theo1}, \ref{theo1-3d}, \ref{theo-inviscid}, \ref{theo3} and \ref{pro-HMHD}, and Proposition \ref{pro-H} will be provided in Sections \ref{sec:theo1}, \ref{sec:theo1-3d}, \ref{sec:theo-MH}, \ref{sec:inv}, \ref{sec:theo3}, \ref{sec:HMHD} and \ref{sec:app}, respectively. Some technical tools are also recalled and proved in the appendix given in Section \ref{sec:app}.
	
	%
	\section{Proof of Theorem \ref{theo1}} \label{sec:theo1}
	%
	
	In this section, we will provide a quite simple proof of Theorem \ref{theo1}, which is mainly relied on the standard energy method with using the usual Brezis-Gallouet-Wainger inequality in the case $\delta > 0$ to bound the norm $\|v\|_{L^2_tL^\infty_x}$. We will also revisit the case $\delta = 0$, i.e., $(v_0,E_0,B_0) \in L^2 \times H^s \times H^s$ for $s \in (0,1)$ by taking the idea from \cite{Arsenio-Gallagher_2020} with using a slightly different decomposition of the velocity to obtain a bound on $\|v\|_{L^2_tL^\infty_x}$.
	
	\begin{proof}[Proof of Theorem \ref{theo1}-(i)] The proof consists of three parts with several substeps in each part as follows.
		
		\textbf{Part I: Approximate system and local existence.} Let us fix $n \in \mathbb{N}$. Assume that $(v_0,E_0,B_0) \in H^\delta \times H^s \times H^s$ with $\text{div}\,v_0 = \text{div}\,B_0 = 0$ and $\delta,s \geq 0$. An approximate system of \eqref{NSM} is taken by the following form
		\begin{equation} \label{NSM_app}  
			\left\{
			\begin{aligned}
				\frac{d}{dt} \left(v^n,E^n,B^n\right) &=
				(F^n_1,F^n_2,F^n_3)(v^n,E^n,B^n),
				\\
				\text{div}\,v^n = \text{div}\,B^n &= 0,
				\\
				(v^n,E^n,B^n)_{|_{t=0}} &= T_n(v_0,E_0,B_0),
			\end{aligned}
			\right.
		\end{equation}
		where for $j^n = \sigma(cE^n + T_n(v^n \times B^n))$ and $i\in \{1,2,3\}$, $F^n_i$ are given by
		\begin{align*}
			F^n_1 &:= \nu\Delta v^n -T_n(\mathbb{P}(v^n \cdot \nabla v^n)) + T_n(\mathbb{P}(j^n \times B^n)), \quad F^n_2 := c(\nabla \times B^n -j^n) \qquad \text{and}\qquad F^n_3 := - c\nabla \times E^n.
		\end{align*}	
		Here, $T_n$ and $\mathbb{P}$ are the usual Fourier truncation operator and Leray projection\footnote{As usual, the operators $T_n$ and $\mathbb{P}$ are defined by
		\begin{align*}
			\mathcal{F}(T_n(f))(\xi) &:= \mathbbm{1}_{B_n}(\xi) \mathcal{F}(f)(\xi) \quad \text{for } n \in \mathbb{R}, n > 0,\xi \in \mathbb{R}^d,
			\\
			\mathbb{P}(f) &:= f + \nabla(-\Delta)^{-1} \text{div} \,f.
		\end{align*} 
		Here, $\mathbbm{1}_{B_n}$ is the characteristic function of $B_n$, where $B_n$ is the ball of radius $n$ centered at the origin.}, respectively. 
		For $\delta, s \in \mathbb{R}$ with $\delta, s \geq 0$, we define the following functional spaces
		\begin{align*}
			H^s_n &:= \left\{h \in H^s : \text{supp}(\mathcal{F}(h)) \subseteq B_n\right\}, \\ 
			V^s_n &:= \left\{h \in H^s_n : \text{div}\, h = 0\right\},
		\end{align*}
		and the mapping
		\begin{align*}
			F^n : V^\delta_n \times H^s_n \times V^s_n &\to V^\delta_n \times H^s_n \times V^s_n
			\\
			(v^n,E^n,B^n) &\mapsto F^n(v^n,E^n,B^n) := (F^n_1,F^n_2,F^n_3).
		\end{align*}
		The space $V^\delta_n \times H^s_n \times V^s_n$ is equipped with the following norm\footnote{For $s \in \mathbb{R}$ and $\xi \in \mathbb{R}^d$, $\mathcal{F}(J^s(f))(\xi) := (1 + |\xi|^2)^{\frac{s}{2}} \mathcal{F}(f)(\xi)$ and $\|f\|_{H^s} := \|J^s f\|_{L^2}$ with $H^0 \equiv L^2$.}
		\begin{equation*}
			\|(v^n,E^n,B^n)\|^2_{\delta,s} := \|v^n\|^2_{H^\delta} + \|(E^n,B^n)\|^2_{H^s}. 
		\end{equation*}
		It can be checked that $F^n$ is well-defined and is locally Lipschitz continuous as well. Then the Picard theorem (see \cite[Theorem 3.1]{Majda_Bertozzi_2002}) implies that there exists a unique solution $(v^n,E^n,B^n) \in C^1([0,T^n_*),V^\delta_n \times H^s_n \times V^s_n)$ to \eqref{NSM_app} for some $T^n_* > 0$. In addition, if $T^n_* < \infty$ then (see \cite[Theorem 3.3]{Majda_Bertozzi_2002})
		\begin{equation*}
			\lim_{t\to T^n_*} \left(\|v^n(t)\|^2_{H^\delta} + \|(E^n,B^n)(t)\|^2_{H^s}\right) = \infty.
		\end{equation*}
		
		\textbf{Part II: Global existence and uniform bound.} In the following steps (from Step 1 to Step 13), in order to verify that $T^n_* = \infty$, we will assume that $T^n_* < \infty$ and prove the following inequality
		\begin{equation*}
			\esssup_{t \in (0,T^n_*)} \left(\|v^n(t)\|^2_{H^\delta} + \|(E^n,B^n)(t)\|^2_{H^s}\right) < \infty,
		\end{equation*}
		which leads to a contradiction with the analysis in Part I. 
		
		\textbf{Step 1: The case $0 < \delta \leq s < 1$.} It can be checked that if $(v^n,E^n,B^n) \in V^\delta_n \times H^s_n \times V^s_n$ then $T_n(v^n,E^n,B^n,j^n) = (v^n,E^n,B^n,j^n)$ in the $L^2$ sense. In the sequel, we will write only $(v,E,B,j)$ instead of $(v^n,E^n,B^n,j^n)$ for simplicity. The standard energy inequality to \eqref{NSM_app} is given by
		\begin{equation*}
			\frac{1}{2} \frac{d}{dt}\|(v,E,B)\|^2_{L^2} + \nu\|\nabla v\|^2_{L^2} + \frac{1}{\sigma}\|j\|^2_{L^2} = 0,
		\end{equation*}
		which yields for $t \in (0,T^n_*)$
		\begin{equation*}
			\|(v,E,B)(t)\|^2_{L^2} + 2\int^t_0 \nu\|\nabla v\|^2_{L^2} + \frac{1}{\sigma} \|j\|^2_{L^2} \,d\tau  \leq \|T_n(v_0,E_0,B_0)\|^2_{L^2} \leq \|(v_0,E_0,B_0)\|^2_{L^2} =: \mathcal{E}^2_0.
		\end{equation*}
		Moreover, the $\dot{H}^\delta$-$\dot{H}^s$ estimate is given by\footnote{The usual fractional derivative operator is given by $\mathcal{F}(\Lambda^s(f))(\xi) := |\xi|^s \mathcal{F}(f)(\xi)$ for $\xi \in \mathbb{R}^2, s \in \mathbb{R}$. Recall that  $\|f\|_{\dot{H}^s} := \|\Lambda^s f\|_{L^2}$ and for $s \geq 0$, $\|f\|^2_{H^s} \approx \|\Lambda^s f\|^2_{L^2} + \|f\|^2_{L^2}$.}
		\begin{align*}
			\frac{1}{2} \frac{d}{dt} \left(\|v\|^2_{\dot{H}^\delta} + \|(E,B)\|^2_{\dot{H}^s}\right) + \nu\|v\|^2_{\dot{H}^{\delta+1}} + \frac{1}{\sigma}\|j\|^2_{\dot{H}^s} 
			=: \sum^3_{k=1} I_k,
		\end{align*}
		where for some $\epsilon \in (0,1)$, since $s,\delta \in (0,1)$ with $\delta \leq s$ 
		\begin{align*}
			I_1 &= \int_{\mathbb{R}^2} (j \times B) \cdot \Lambda^{2\delta} v \,dx 
			\leq \frac{\epsilon\nu}{2}\|v\|^2_{\dot{H}^{2\delta + 1-s}} + C(\epsilon,\nu,s)\|j\|^2_{L^2}\|B\|^2_{\dot{H}^s} ;
			\\
			I_2 &= - \int_{\mathbb{R}^2} v \cdot \nabla v \cdot \Lambda^{2\delta} v \,dx \leq 			 \frac{\epsilon\nu}{2}\|v\|^2_{\dot{H}^{\delta+1}} + C(\epsilon,\delta,\nu)\|\nabla v\|^2_{L^2}\|v\|^2_{\dot{H}^\delta};
			\\
			I_3 &= \int_{\mathbb{R}^2} \Lambda^s j \cdot \Lambda^s (v \times B) \,dx
			\\
			&\leq C(s)\|j\|_{\dot{H}^s}\|B\|_{\dot{H}^s}\left(\|\nabla v\|_{L^2} + \|v\|_{L^\infty}\right)
			\\
			&\leq \frac{\epsilon}{\sigma}\|j\|^2_{\dot{H}^s} +  C(\epsilon,\sigma,s)\left(\|\nabla v\|^2_{L^2} + \|v\|^2_{L^\infty}\right)\|B\|^2_{\dot{H}^s},
		\end{align*}
		where we used the well-known inequalities (see \cite{Bahouri-Chemin-Danchin_2011})
		\begin{align*}
			&&\|f\|_{L^{p_0}} &\leq C(p_0,s_0) \|f\|_{\dot{H}^{s_0}}  & &\text{for}\quad s_0 \in [0,1), p_0 = \frac{2}{1-s_0},&&
			\\
			&&\|f\|_{\dot{H}^{s_1}} &\leq C(s_1,s_2) \|f\|^{\alpha_0}_{L^2} \|f\|^{1-\alpha_0}_{\dot{H}^{s_2}} &
			&\text{for}\quad s_1,s_2 \in (0,\infty), s_1 < s_2, \alpha_0 = 1 - \frac{s_1}{s_2},&&
		\end{align*}
		and the following homogeneous Kato-Ponce type inequality (see \cite{Grafakos-Oh_2014})
		for $1 < p_i,q_i \leq \infty$, $i \in \{1,2\}$, $s_0 > 0$ and  $\frac{1}{p_i} + \frac{1}{q_i} = \frac{1}{2}$
		\begin{equation*} 
			\|\Lambda^{s_0}(fg)\|_{L^2} \leq C(s_0,p_i,q_i)\left(\|\Lambda^{s_0} f\|_{L^{p_1}}\|g\|_{L^{q_1}} + \|f\|_{L^{p_2}}\|\Lambda^{s_0} g\|_{L^{q_2}} \right).
		\end{equation*}
		Therefore, since $2\delta + 1- s \leq \delta + 1$ and by choosing $\epsilon = \frac{1}{2}$
		\begin{align*}
			\frac{d}{dt} \left(\|v\|^2_{\dot{H}^\delta} + \|(E,B)\|^2_{\dot{H}^s}\right) + \nu\|v\|^2_{\dot{H}^{\delta+1}} + \frac{1}{\sigma}\|j\|^2_{\dot{H}^s} &\leq C(\sigma,s)\left(\|\nabla v\|^2_{L^2} + \|v\|^2_{L^\infty} + \|j\|^2_{L^2}\right)\|B\|^2_{\dot{H}^s} 
			\\
			&\quad+ \nu\|v\|^2_{L^2} + C(\delta,\nu)\|\nabla v\|^2_{L^2}\|v\|^2_{\dot{H}^\delta}.
		\end{align*}
		
		\textbf{Step 2: The case $\delta = s = 1$.} Similar to the previous case, we obtain
		\begin{align*}
			\frac{1}{2} \frac{d}{dt} \|(v,E,B)\|^2_{\dot{H}^1} + \nu\|v\|^2_{\dot{H}^2} + \frac{1}{\sigma}\|j\|^2_{\dot{H}^1} =: \sum^3_{k=1} I_k,
		\end{align*}
		where $I_2 = 0$, $I_1$ and $I_3$ are estimated as follows\footnote{Here, $A:B := \sum_{1 \leq i,j \leq 3} a_{ij}b_{ij}$ for two matrices $A = a_{ij}$ and $B = b_{ij}$.}
		\begin{align*}
			I_1 &= -\int_{\mathbb{R}^2} (j \times B) \cdot \Delta v \,dx
			\\
			&\leq C\|j\|^\frac{1}{2}_{L^2}\|\nabla j\|^\frac{1}{2}_{L^2}\|B\|^\frac{1}{2}_{L^2}\|\nabla B\|^\frac{1}{2}_{L^2}\|\Delta v\|_{L^2}
			\\
			&\leq \frac{\epsilon}{\sigma}\|j\|^2_{\dot{H}^1} + \epsilon \nu\|v\|^2_{\dot{H}^2} +  C(\epsilon,\nu,\sigma)\|j\|^2_{L^2}\|B\|^2_{L^2}\|\nabla B\|^2_{L^2};
			\\
			I_3 &= \int_{\mathbb{R}^2} \nabla j : \nabla (v \times B) \,dx
			\\
			&\leq C \|\nabla j\|_{L^2}\left(\|\nabla v\|^\frac{1}{2}_{L^2}\|\Delta v\|^\frac{1}{2}_{L^2}\|B\|^\frac{1}{2}_{L^2}\|\nabla B\|^\frac{1}{2}_{L^2} + \|v\|_{L^\infty}\|\nabla B\|_{L^2}\right)
			\\
			&\leq \frac{\epsilon}{\sigma}\|j\|^2_{\dot{H}^1} +  \epsilon \nu \|v\|^2_{\dot{H}^2} +  C(\epsilon,\nu,\sigma)\left(\|\nabla v\|^2_{L^2}\|B\|^2_{L^2} + \|v\|^2_{L^\infty}\right)\|B\|^2_{\dot{H}^1},
		\end{align*}
		which yields by choosing $\epsilon = \frac{1}{4}$
		\begin{align*}
			\frac{d}{dt} \|(v,E,B)\|^2_{\dot{H}^1} + \nu\|v\|^2_{\dot{H}^2} + \frac{1}{\sigma}\|j\|^2_{\dot{H}^1} \leq C(\nu,\sigma)\left(1 + \|B\|^2_{L^2}\right)\left(\|j\|^2_{L^2} + \|\nabla v\|^2_{L^2} + \|v\|^2_{L^\infty}\right)\|B\|^2_{\dot{H}^1}.
		\end{align*}
		
		\textbf{Step 3: Conclusion of Steps 1 and 2.} Since $0 < \delta \leq s \leq 1$, by using the energy inequality, we collect the main estimates in the two previous steps as follows 
		\begin{align*}
			\frac{d}{dt} Y_{\delta,s} + \nu\|v\|^2_{H^{\delta+1}} + \frac{1}{\sigma}\|j\|^2_{H^s} \leq C_1G Y_{\delta,s} + \left[\frac{1}{2}C_2\|B\|_{H^s}\|v\|_{H^1} \left(1 + \log^\frac{1}{2}\left(\frac{\|v\|_{H^{\delta+1}}}{\|v\|_{H^1}}\right)\right)\right]^2,
		\end{align*}
		where $C_1(\delta,\nu,\sigma,s), C_2(\delta,\sigma,s) > 0$, and for $t \in (0,T^n_*)$
		\begin{equation*}
			Y_{\delta,s}(t) := \|v(t)\|^2_{H^\delta} + \|(E,B)(t)\|^2_{H^s}
			\quad \text{and} \quad
			G(t) := \left(1 + \|B(t)\|^2_{L^2} \right)\left(1 + \|j(t)\|^2_{L^2} + \|\nabla v(t)\|^2_{L^2}\right).
		\end{equation*}
		Here, in order to to bound the norm $\|v\|_{L^\infty}$, we also used the well-known  Brezis-Gallouet-Wainger inequality in the following form
		(for example, see \cite{Brezis-Gallouet_1980} for $s_0 = 1$ and $d = 2$; see \cite{Brezis-Wainger_1980} for $s_0 > \max\{\frac{d}{2} - 1,0\}$ and $d \geq 1$, and see \cite{Hu-Kukavica-Ziane_2015} for $s_0 \in (0,1)$ and $d = 2$) with $s_0 = \delta$ and $d = 2$
		\begin{equation} \label{BGW}
			\|f\|_{L^\infty(\mathbb{R}^d)} \leq C(s_0)\|f\|_{H^\frac{d}{2}(\mathbb{R}^d)}\left(1 + \log^\frac{1}{2}\left(1+\frac{\|f\|_{H^{s_0+1}(\mathbb{R}^d)}}{\|f\|_{H^\frac{d}{2}(\mathbb{R}^d)}}\right)\right)\qquad f \neq 0.
		\end{equation}
		By applying the following inequality (see \cite{Hu-Kukavica-Ziane_2015}) for $\alpha,\beta,\gamma > 0$ and $\log_{+}(a) := \max\{\log(a),0\}, a > 0$
		\begin{equation*}
			\beta(1 + \log_{+}(\gamma))^\frac{1}{2} \leq \alpha\gamma + \beta\left(1 + \log\left(1 + \frac{\beta}{\alpha}\right)\right)^\frac{1}{2} 
		\end{equation*}
		to the case where
		\begin{align*}
			\alpha = \frac{\sqrt{\nu}}{2}\|v\|_{H^1},\quad \beta = C_2\|B\|_{H^s}\|v\|_{H^1} \quad \text{and} \quad \gamma = \frac{\|v\|_{H^{\delta+1}}}{\|v\|_{H^1}},
		\end{align*}
		we find that 
		\begin{align*}
			R &:= \frac{C_2}{2}\|B\|_{H^s}\|v\|_{H^1}\left(1 + \log^\frac{1}{2}\left(\frac{\|v\|_{H^{\delta+1}}}{\|v\|_{H^1}}\right)\right) 
			\\
			&\leq C_2\|B\|_{H^s}\|v\|_{H^1}\left(1+ \log\left(\frac{\|v\|_{H^{\delta+1}}}{\|v\|_{H^1}}\right)\right)^\frac{1}{2}
			\\
			&\leq \frac{\sqrt{\nu}}{2}\|v\|_{H^{\delta+1}} + C_2\|B\|_{H^s}\|v\|_{H^1}\left(1 + \log\left(1 + \frac{2C_2}{\sqrt{\nu}}\|B\|_{H^s}\right)\right)^\frac{1}{2},
		\end{align*}
		which yields
		\begin{align*}
			R^2 \leq \frac{\nu}{2}\|v\|^2_{H^{\delta+1}} + 2C_2^2\|B\|^2_{H^s}\|v\|^2_{H^1}\left(1 + \log\left(1 + \frac{C_2\nu}{\sqrt{2}}\|B\|_{H^s}\right)\right)
		\end{align*}
		and 
		\begin{align*}
			\frac{d}{dt} Y_{\delta,s} + \frac{\nu}{2}\|v\|^2_{H^{\delta+1}} + \frac{1}{\sigma}\|j\|^2_{H^s} \leq C_1G Y_{\delta,s} + C(\delta,\nu,\sigma,s)\|v\|^2_{H^1}\left(1 + \log\left(1 + Y_{\delta,s}\right)\right)Y_{\delta,s}.
		\end{align*}
		Therefore, it follows from Lemma \ref{lem-gronwall} that
		\begin{equation*}
			Y_{\delta,s}(t) \leq \exp\{(\log(e+Y_{\delta,s}(0)) + (1+T^n_*)C(\mathcal{E}_0,\nu,\sigma,s))\exp\{(1+T^n_*)C(\mathcal{E}_0,\delta,\nu,\sigma,s)\}\},
		\end{equation*}
		which gives us the conclusion in Steps 1 and 2. In addition, since $v \in L^2_tH^{\delta+1}_x$ for $\delta > 0$, it implies that $v \in L^2_tL^\infty_x$. 
		We should remark here that in Steps 1 and 2, we obtain the double exponential bound in time, i.e., in the form of $C\exp\{CT^n_*\exp\{CT^n_*\}\}$ for some constant $C$ depending on the parameters and the intial data. However, if we use directly Step 14 below in these two steps then the bound can be given in the form of either $C\exp\{CT^n_*\}$ or $C(T^n_*)^C$. 
				
		\textbf{Step 4: The case $\delta \in (0,1)$ and $s = 1$.} By applying Step 1 (the case $\delta = s$), we are able to close the $H^\delta$ estimate of $(v,E,B)$, in particular 
		\begin{equation*}
			\|v\|_{L^2_tH^{\delta+1}_x} \leq C(T^n_*,\delta,\nu,\sigma,v_0,E_0,B_0).
		\end{equation*}
		It remains to obtain the $H^1$ estimate of $(E,B)$. It can be seen that
		\begin{equation*}
			\frac{1}{2} \frac{d}{dt}  \|(E,B)\|^2_{H^1} + \frac{1}{\sigma}\|j\|^2_{H^1} =: I_{31} + I_{32},
		\end{equation*}
		where for some $\epsilon \in (0,1)$, since $\delta \in (0,1)$
		\begin{align*}
			I_{31} &= \int_{\mathbb{R}^2} j \cdot (v \times B) \,dx \leq \frac{\epsilon}{\sigma}\|j\|^2_{L^2} + C(\epsilon,\sigma)\|B\|^2_{L^2}\|\nabla B\|^2_{L^2} +  C(\epsilon,\sigma)\|v\|^2_{L^2}\|\nabla v\|^2_{L^2};
			\\
			I_{32} &= \int_{\mathbb{R}^2} \nabla j : \nabla (v \times B) \,dx
			\\
			&\leq C(\delta)\|\nabla j\|_{L^2}\left(\|\nabla v\|_{L^\frac{2}{1-\delta}}\|B\|_{L^\frac{2}{\delta}} + \|v\|_{L^\infty}\|\nabla B\|_{L^2}\right)
			\\
			&\leq C(\delta)\|\nabla j\|_{L^2}\left(\|\Lambda^{\delta+1} v\|_{L^2}\|\Lambda^{1-\delta} B\|_{L^2} + \|v\|_{H^{\delta+1}}\|\nabla B\|_{L^2}\right) 
			\\
			&\leq \frac{\epsilon}{\sigma}\|\nabla j\|^2_{L^2} + C(\epsilon,\sigma)\|v\|^2_{H^{\delta+1}} \|B\|^2_{H^1}.
		\end{align*}
		By choosing $\epsilon = \frac{1}{2}$, it follows that
		\begin{equation*}
			\frac{d}{dt}  \|(E,B)\|^2_{H^1} + \frac{1}{\sigma}\|j\|^2_{H^1} \leq  C(\sigma)\|v\|^2_{L^2}\|\nabla v\|^2_{L^2} + C(\sigma)\left(\|B\|^2_{L^2} + \|v\|^2_{H^{\delta+1}}\right)\|B\|^2_{H^1},
		\end{equation*}
		which is closable. Thus, the conclusion follows.
		
		\textbf{Step 5: The case $\frac{1}{2} \leq s < 1$ and $s < \delta \leq 1$.} We first focus on obtaining the $H^s$ estimate for $(E,B)$. Since $\delta > s$, as in Step 1 (for the case $\delta = s$) we are able to bound the norms
		\begin{equation*}
			\|v\|_{L^\infty_tH^s_x \cap L^2_tH^{s+1}_x}, \quad \|(E,B)\|_{L^\infty_tH^s_x}  \quad \text{and}\quad  \|j\|_{L^2_tH^s_x}.
		\end{equation*} 
		It remains to bound the norm $\|v\|_{L^\infty_tH^\delta_x \cap L^2_tH^{\delta+1}_x}$. It can be seen that
		\begin{equation*}
			\frac{1}{2} \frac{d}{dt} \|v\|^2_{\dot{H}^\delta}  + \nu\|v\|^2_{\dot{H}^{\delta+1}} =: I_1 + I_2,
		\end{equation*}
		where $I_2$ is bounded as in Step 1 (for $\delta \in (0,1)$) and $I_2 = 0$ (for $\delta = 1$), and since $s \in [\frac{1}{2},1)$ and $\delta \leq 1$, for some $\epsilon \in (0,1)$
		\begin{align*}
			I_1 &= \int_{\mathbb{R}^2} (j \times B) \cdot \Lambda^{2\delta} v \,dx 
			\\
			&\leq C(s) \|j\|_{\dot{H}^{1-s}} \|B\|_{\dot{H}^s}\|\Lambda^{2\delta} v\|_{L^2}
			\\
			&\leq C(\epsilon,\nu,s) \|j\|^2_{H^{1-s}} \|B\|^2_{H^s} + \epsilon\nu\left(\|v\|^2_{\dot{H}^{\delta+1}} + \|v\|^2_{L^2}\right)
			\\
			&\leq C(\epsilon,\nu,s) \|j\|^2_{H^s} \|B\|^2_{H^s} + \epsilon\nu\left(\|v\|^2_{\dot{H}^{\delta+1}} + \|v\|^2_{L^2}\right).
		\end{align*}
		It implies the closable of the $H^\delta$ estimate of $v$ by choosing $\epsilon = \frac{1}{4}$.
		
		\textbf{Step 6a: The case $\frac{1}{2} < s < 1$ and $1 < \delta \leq 2s$.} In this case, we can estimate $(E,B)$ exactly as in Step 5. We now focus on the estimates of $I_1$ and $I_2$. Firstly, since $s \in (\frac{1}{2},1)$ and $\delta \in (1,2s]$  
		\begin{align*}
			I_1 &= \int_{\mathbb{R}^2} \Lambda^{\delta-1}(j \times B) \cdot \Lambda^{\delta+1} v \,dx
			\\
			&\leq C(\delta)\left(\|j\|_{\dot{H}^{\delta-s}}\|B\|_{\dot{H}^s} + \|j\|_{\dot{H}^s}\|B\|_{\dot{H}^{\delta-s}}\right)\|v\|_{\dot{H}^{\delta+1}}
			\\
			&\leq \epsilon\nu \|v\|^2_{\dot{H}^{\delta+1}} + C(\epsilon,\delta,\nu) \|j\|^2_{H^s}\|B\|^2_{H^s}.
		\end{align*}
		Secondly, for $\sigma' := \delta -1 \in (0,1)$ we find that 
		\begin{align*}
			I_2 &= -\int_{\mathbb{R}^2} \Lambda^{\sigma'}(v \cdot \nabla v) \cdot \Lambda^{\delta+1} v \,dx
			\\
			&\leq C(\delta) \|\Lambda^{\delta+1}v\|_{L^2} \times
			\begin{cases}
				\|\Lambda^{\sigma'} v\|_{L^4}\|\nabla v\|_{L^4} + \|v\|_{L^\frac{4}{1-2\sigma'}}\|\Lambda^{\sigma'}\nabla v\|_{L^\frac{4}{1+2\sigma'}} \quad &\text{if} \quad \sigma' \in (0,\frac{1}{2}),
				\\
				\|\Lambda^{\sigma'} v\|_{L^4}\|\nabla v\|_{L^4} + \|v\|_{L^6}\|\Lambda^{\sigma'} \nabla v\|_{L^3} \quad &\text{if}\quad \sigma' \in [\frac{1}{2},1).
			\end{cases}
		\end{align*}
		Moreover, 
		\begin{align*}
			\|\Lambda^{\sigma'} v\|_{L^4}, \|v\|_{L^\frac{4}{1-2\sigma'}} &\leq C(\delta)\|\Lambda^{\sigma' + \frac{1}{2}} v\|_{L^2} \leq C(\delta)\|v\|^\frac{3}{2(\delta+1)}_{L^2} \|\Lambda^{\delta+1} v\|^\frac{2s-1}{2(\delta+1)}_{L^2},
			\\
			\|\nabla v\|_{L^4}, \|\Lambda^{\sigma'} \nabla v\|_{L^\frac{4}{1+2\sigma'}} &\leq C(\delta)\|\Lambda^\frac{1}{2} \nabla v\|_{L^2} \leq C(\delta)\|\nabla v\|^\frac{2\delta-1}{2\delta}_{L^2} \|\Lambda^{\delta+1} v\|^\frac{1}{2\delta}_{L^2},
			\\
			\|v\|_{L^6} &\leq C\|v\|^\frac{1}{3}_{L^2} \|\nabla v\|^\frac{2}{3}_{L^2}, 
			\\
			\|\Lambda^{\sigma'} \nabla v\|_{L^3} &\leq C(\delta)\|\Lambda^{\sigma' + \frac{1}{3}} \nabla v\|_{L^2} \leq  C(\delta)\|\nabla v\|^\frac{2}{3\delta}_{L^2} \|\Lambda^{\delta+1} v\|^\frac{3\delta-2}{3\delta}_{L^2},
		\end{align*} 
		which yields
		\begin{align*}
			I_2 &\leq C(\delta)\|v\|^\frac{3}{2(\delta+1)}_{L^2} \|\nabla v\|^\frac{2\delta-1}{2\delta}_{L^2} \|\Lambda^{\delta+1} v\|^\frac{4\delta^2+2\delta+1}{2(\delta+1)\delta}_{L^2} + C(\delta)\|v\|^\frac{1}{3}_{L^2} \|\nabla v\|^\frac{2\delta+2}{3\delta}_{L^2} \|\Lambda^{\delta+1} v\|^\frac{6\delta-2}{3\delta}_{L^2}.
		\end{align*}
		In addition, since $1 < \delta$ then an application of Step 5 (with $\delta = 1$) gives us the bound on $\|\nabla v\|_{L^\infty_tL^2_x}$. It can be seen that since $\delta > 1$
		\begin{equation*}
			\frac{4\delta^2+2\delta+1}{2(\delta+1)\delta}, \frac{6\delta-2}{3\delta} < 2,
		\end{equation*}
		which implies the closable of the $H^\delta$ estimate as in Step 5 by using Young inequality with $\epsilon = \frac{1}{6}$.
		
		\textbf{Step 6b: The case $0 < s < \frac{1}{2}$ and $s < \delta \leq 2s$.} Similar to the previous case, we only need to focus on the estimate of $v$. Indeed, $I_2$ can be bounded as in Step 1 since $\delta \in (0,1)$. In addition, for some $\epsilon \in (0,1)$, since $s \in (0,1)$, $s < \delta \leq 2s$ and $1-(\delta-s), \frac{\delta-s}{2} \in (0,1)$ 
		\begin{align*}
			I_1 &= \int_{\mathbb{R}^2} [(\Lambda^s(j \times B) - \Lambda^s j \times B - j \times \Lambda^s B) + \Lambda^s j \times B + j \times \Lambda^s B] \cdot \Lambda^{2\delta-s} v\,dx =: \sum^3_{k=1} I_{1k},
			\\
			I_{11} &\leq \|\Lambda^s(j \times B) - \Lambda^s j \times B - j \times \Lambda^s B\|_{L^\frac{2}{2+s-\delta}} \|\Lambda^{2\delta-s} v\|_{L^\frac{2}{\delta-s}}
			\\
			&\leq C(\delta,s)\|\Lambda^{\frac{4s-\delta}{4}} j\|_{L^\frac{4}{2+s-\delta}}\|\Lambda^{\frac{\delta}{4}} B\|_{L^\frac{4}{2+s-\delta}} \|v\|_{\dot{H}^{\delta+1}}
			\\
			&\leq \epsilon\nu \|v\|^2_{\dot{H}^{\delta+1}} + C(\epsilon,\delta,\nu,s)\|\Lambda^\frac{2s+\delta}{4}j\|^2_{L^2}\|\Lambda^\frac{3\delta-2s}{4}B\|^2_{L^2}
			\\
			&\leq \epsilon\nu \|v\|^2_{\dot{H}^{\delta+1}} + C(\epsilon,\delta,\nu,s)\|j\|^2_{H^s}\|B\|^2_{H^s},
			\\
			I_{12} &\leq \|j\|_{\dot{H}^s}\|B\|_{L^\frac{2}{1-(\delta-s)}}\|\Lambda^{2\delta-s} v\|_{L^\frac{2}{\delta-s}} 
			\\
			&\leq \epsilon\nu \|v\|^2_{\dot{H}^{\delta+1}} + C(\epsilon,\delta,s)\|j\|^2_{\dot{H}^s}\|B\|^2_{\dot{H}^s},
			\\
			I_{13} &\leq \epsilon\nu \|v\|^2_{\dot{H}^{\delta+1}} + C(\epsilon,\delta,s)\|j\|^2_{\dot{H}^s}\|B\|^2_{\dot{H}^s},
		\end{align*}
		where we used the well-known Kenig-Ponce-Vega commutator estimate (see \cite{Kenig-Ponce-Vega_1993})
		\begin{equation} \label{KPV}
			\|\Lambda^{s_0}(fg) - g\Lambda^{s_0} f  - f\Lambda^{s_0} g\|_{L^{p_0}(\mathbb{R}^d)} \leq C(d,p_0,p_{01},p_{02},s_0,s_{01},s_{02})\|\Lambda^{s_{01}} f\|_{L^{p_{01}}(\mathbb{R}^d)} \|\Lambda^{s_{02}} g\|_{L^{p_{02}}(\mathbb{R}^d)},
		\end{equation}
		for $0 < s_0,s_{01},s_{02} < 1$ and $p_0,p_{01},p_{02} \in (1,\infty)$ satisfying $s_0 = s_{01} + s_{02}$ and $\frac{1}{p_0} = \frac{1}{p_{01}} + \frac{1}{p_{02}}$.
		
		\textbf{Step 7: The case $s = 1$ and $1 < \delta < 2$.} In this case, we only need to bound $I_1$, other terms can be done exactly as in Step 5 (the estimate of $(E,B)$) and Step 6a (the estimate of $I_2$). Indeed, similar to Step 6a, since $\delta \in (1,2)$ with $\delta - 1 \in (0,1)$ for some $\epsilon \in (0,1)$
		\begin{align*}
			I_1 &= \int_{\mathbb{R}^2} \Lambda^{\delta-1}(j \times B) \cdot \Lambda^{\delta+1} v \,dx
			\\
			&\leq C(\delta)\left(\|\Lambda^{\delta-1}j\|_{L^{\frac{2}{\delta-1}}}\|B\|_{L^{\frac{2}{2-\delta}}} + \|j\|_{L^{\frac{2}{\delta-1}}}\|\Lambda^{\delta-1}B\|_{L^{\frac{2}{2-\delta}}}\right)\|v\|_{\dot{H}^{\delta+1}}
			\\
			&\leq C(\delta)\left(\|j\|_{\dot{H}^1}\|B\|_{\dot{H}^{\delta-1}} + \|j\|_{\dot{H}^{\delta-1}}\|B\|_{\dot{H}^1}\right)\|v\|_{\dot{H}^{\delta+1}}
			\\
			&\leq \epsilon\nu \|v\|^2_{\dot{H}^{\delta+1}} + C(\epsilon,\delta,\nu) \|j\|^2_{H^1}\|B\|^2_{H^1}.
		\end{align*}
		
		\textbf{Step 8: The case $\delta = s > 1$ (revisited).} This case has been treated in \cite{Masmoudi_2010} with a different proof, but to make the present work self-contained, we revisit this case with providing our simple proof. The proof of this case is also useful for later use, for instance in Steps 9, 10, 12 and 13 below, which are also included in the main aims of this paper. It can be seen that
		\begin{equation*}
			\frac{1}{2} \frac{d}{dt} \|(v,E,B)\|^2_{\dot{H}^s}  + \nu\|v\|^2_{\dot{H}^{s+1}} + \frac{1}{\sigma}\|j\|^2_{\dot{H}^s} =: \sum^3_{k=1} I_k,
		\end{equation*}
		where for some $\epsilon \in (0,1)$, since $s > 1$
		\begin{align*}
			I_1 &= \int_{\mathbb{R}^2} \Lambda^s(j \times B) \cdot \Lambda^s v \,dx
			\\
			&\leq C(s)\left(\|j\|_{L^\infty}\|B\|_{\dot{H}^s} + \|j\|_{\dot{H}^s}\|B\|_{L^\infty}\right)\|v\|_{\dot{H}^s}
			\\
			&\leq C(s)\left(\|j\|^\frac{s-1}{s}_{L^2}\|j\|^\frac{1}{s}_{\dot{H}^s}\|B\|_{\dot{H}^s} + \|j\|_{\dot{H}^s}\|B\|^\frac{s-1}{s}_{L^2}\|B\|^\frac{1}{s}_{\dot{H}^s}\right)\|\nabla v\|^\frac{1}{s}_{L^2}\|v\|^\frac{s-1}{s}_{\dot{H}^{s+1}}
			\\
			&\leq \frac{\epsilon}{\sigma}\|j\|^2_{H^s} + \frac{2\epsilon\nu}{3} \|v\|^2_{\dot{H}^{s+1}} + C(\epsilon,\nu,\sigma,s)\left(\|j\|^2_{L^2} + \|\nabla v\|^2_{L^2} + \|B\|^{2(s-1)}_{L^2}\|\nabla v\|^2_{L^2}\right)\|B\|^2_{H^s};
			\\
			I_2 &= -\int_{\mathbb{R}^2} \Lambda^s(v \cdot \nabla v) \cdot \Lambda^s v \,dx
			\\
			&\leq C(s)\left(\|\Lambda^s v\|_{L^4}\|\nabla v\|_{L^4} + \|v\|_{L^\infty}\|\Lambda^{s+1}v\|_{L^2}\right)\|\Lambda^s v\|_{L^2}
			\\
			&\leq C(s)\left(\|\Lambda^{s-\frac{1}{2}} \nabla v\|_{L^2}\|\Lambda^\frac{1}{2}\nabla v\|_{L^2} + \|v\|_{L^\infty}\|\Lambda^{s+1}v\|_{L^2}\right)\|\Lambda^s v\|_{L^2}
			\\
			&\leq C(s)\left(\|\Lambda^{s+1} v\|_{L^2}\|\nabla v\|_{L^2} + \|v\|_{L^\infty}\|\Lambda^{s+1}v\|_{L^2}\right)\|\Lambda^s v\|_{L^2}
			\\
			&\leq \frac{2\epsilon\nu}{3}\|v\|^2_{\dot{H}^{s+1}} + C(\epsilon,s,\nu)\left(\|\nabla v\|^2_{L^2} + \|v\|^2_{L^\infty}\right)\|v\|^2_{\dot{H}^s};
			\\
			I_3 &= \int_{\mathbb{R}^2} \Lambda^s j \cdot \Lambda^s (v \times B) \,dx
			\\
			&\leq C(s)\|j\|_{\dot{H}^s}\left(\|B\|_{\dot{H}^s}\|v\|_{L^\infty} + \|B\|_{L^\infty}\|v\|_{\dot{H}^s}\right)
			\\
			&\leq C(s)\|j\|_{\dot{H}^s}\|v\|_{L^\infty}\|B\|_{\dot{H}^s} + C(s)\|j\|_{\dot{H}^s}\|B\|^\frac{s-1}{s}_{L^2}\|B\|^\frac{1}{s}_{\dot{H}^s}\|\nabla v\|^\frac{1}{s}_{L^2}\|v\|^\frac{s-1}{s}_{\dot{H}^{s+1}}
			\\
			&\leq \frac{\epsilon}{\sigma}\|j\|^2_{H^s} + \frac{2\epsilon\nu}{3}\|v\|^2_{\dot{H}^{s+1}} +  C(\epsilon,\sigma,s)\| v\|^2_{L^\infty}\|B\|^2_{H^s} + C(\epsilon,\sigma,s)\|B\|^{2(s-1)}_{L^2}\|\nabla v\|^2_{L^2}\|B\|^2_{H^s},
		\end{align*}
		here we used the following Agmon-type inequality (its simple proof can be found in Appendix B in Section \ref{sec:app})
		\begin{equation} \label{L-infty}
			\|f\|_{L^\infty} \leq C(s_0)\|f\|^\frac{s_0-1}{s_0}_{L^2}\|f\|^\frac{1}{s_0}_{\dot{H}^{s_0}} \qquad \text{for}\quad s_0 > 1.
		\end{equation}
		By choosing $\epsilon = \frac{1}{4}$ and using the energy estimate, it follows that 
		\begin{equation*}
			\frac{d}{dt} Y_s  + \nu\|v\|^2_{H^{s+1}} + \frac{1}{\sigma}\|j\|^2_{H^s} 
			\leq C(\nu,\sigma,s)G_s Y_s + C(\nu,\sigma,s)\|v\|^2_{H^1}\left(1 + \log\left(1 + \frac{Y_s}{\|v\|^2_{H^1}}\right)\right)Y_s,
		\end{equation*}
		where for $t \in (0,T^n_*)$
		\begin{equation*}
			Y_s(t) := \|(v,E,B)(t)\|^2_{H^s}
			\quad \text{and}\quad
			G_s(t) :=  \left(1+\|B(t)\|^{2(s-1)}_{L^2}\right)\left(1 + \|j(t)\|^2_{L^2} + \|\nabla v(t)\|^2_{L^2}\right),
		\end{equation*}
		and here in order to bound $\|v\|_{L^\infty}$, we also used \eqref{BGW} with $s_0 = s-1 > 0$ and $d = 2$. It can be seen from the above estimate of $Y_s$ that
		\begin{align*}
			\frac{d}{dt}Y_s &\leq C(\nu,\sigma,s)G_sY_s + C(\nu,\sigma,s)\|v\|^2_{H^1}Y_s\left(1 + \log(\|v\|^2_{H^1}+Y_s) - \log(\|v\|^2_{H^1})\right)
			\\
			&\leq C(\nu,\sigma,s)G_sY_s + C(\nu,\sigma,s)\|v\|^2_{H^1}Y_s(1 + \log(1+Y_s)),
		\end{align*}
		where we used the fact that $|x\log(x)| \leq \exp\{-1\}$ for $x \in (0,1)$. Therefore, for $t \in (0,T^n_*)$ Lemma \ref{lem-gronwall} gives us
		\begin{equation*}
			Y_s(t) \leq \exp\{(\log(e+Y_s(0)) + (1+T^n_*)C(\mathcal{E}_0,\nu,\sigma,s))\exp\{(1+T^n_*)C(\mathcal{E}_0,\nu,\sigma,s)\}\}.
		\end{equation*}
		
		\textbf{Step 9: The case $s > 1$ and $s < \delta < s+1$.} Since $\delta > s$, we are able to close the $H^s$ estimate of $(v,E,B)$ as in Step 8. It remains to focus on the $H^\delta$ estimate of $v$. Moreover, it can be seen that $I_2$ can be bounded exactly as  the previous step with replacing $s$ by $\delta$, i.e., 
		\begin{equation*}
			I_2 = -\int_{\mathbb{R}^2} \Lambda^\delta(v \cdot \nabla v) \cdot \Lambda^\delta v \,dx
			\leq \epsilon\nu\|v\|^2_{\dot{H}^{\delta+1}} + C(\epsilon,\delta,\nu)\left(\|\nabla v\|^2_{L^2} + \|v\|^2_{L^\infty}\right)\|v\|^2_{\dot{H}^\delta}.
		\end{equation*}
		We continue with the bound $I_1$ as follows. Since $\delta \in (s,s+1)$, we can write $\delta = s + \epsilon_0$ for $\epsilon_0 \in (0,1)$ and find that since $\delta > 1$
		\begin{align*}
			I_1 &= \int_{\mathbb{R}^2} \Lambda^{\delta-1}(j \times B) \cdot \Lambda^{\delta+1} v \,dx
			\\
			&\leq C(\epsilon_0)\left(\|\Lambda^{\delta-1}j\|_{L^\frac{2}{\epsilon_0}}\|B\|_{L^\frac{2}{1-\epsilon_0}} + \|j\|_{L^\frac{2}{1-\epsilon_0}}\|\Lambda^{\delta-1}B\|_{L^\frac{2}{\epsilon_0}}\right)\|\Lambda^{\delta+1}v\|_{L^2}
			\\
			&\leq C(\epsilon_0)\left(\|\Lambda^{\delta-\epsilon_0}j\|_{L^2}\|\Lambda^{\epsilon_0}B\|_{L^2} + \|\Lambda^{\epsilon_0}j\|_{L^2}\|\Lambda^{\delta-\epsilon_0}B\|_{L^2}\right)\|\Lambda^{\delta+1}v\|_{L^2}
			\\
			&\leq C(\epsilon_0)\left(\|j\|_{\dot{H}^{\delta-\epsilon_0}}\|B\|_{\dot{H}^{\epsilon_0}} + \|j\|_{\dot{H}^{\epsilon_0}}\|B\|_{\dot{H}^{\delta-\epsilon_0}}\right)\|v\|_{\dot{H}^{\delta+1}}
			\\
			&\leq \epsilon\nu\|v\|^2_{\dot{H}^{\delta+1}} + C(\epsilon_0,\epsilon,\nu) \|j\|^2_{H^s}\|B\|^2_{H^s}.
		\end{align*}
		Therefore, by choosing $\epsilon = \frac{1}{4}$ as  in Step 5, the conclusion follows.
		
		\textbf{Step 10: The case $s > 1$ and $\delta = s+1$.} This case is very similar to Step 9. We only need to bound $I_1$ as follows
		\begin{align*}
			I_1 &= \int_{\mathbb{R}^2} \Lambda^s(j \times B) \cdot \Lambda^{\delta+1} v \,dx
			\\
			&\leq C(s)\left(\|\Lambda^s j\|_{L^2}\|B\|_{L^\infty} + \|j\|_{L^\infty}\|\Lambda^s B\|_{L^2} \right)\|\Lambda^{\delta+1}v\|_{L^2}
			\\
			&\leq \epsilon\nu \|v\|^2_{\dot{H}^{\delta+1}} + C(\epsilon,\nu,s) \|j\|^2_{H^s}\|B\|^2_{H^s}.
		\end{align*}
		Thus, the conclusion follows. 
		
		In the next few steps, we will consider a domain for $(\delta,s)$ with $s > 1$ and $s - 1 \leq \delta < s$, which has been provided in \cite{Masmoudi_2010} with a different proof. We aim to revisit this domain of initial data with a new proof.
		
		\textbf{Step 11: The case $s \in (1,2]$ and $s - 1 \leq \delta \leq 1$ (revisited).} Similar to the previous case, we find that
		\begin{align*}
			\frac{1}{2} \frac{d}{dt}\left(\|v\|^2_{\dot{H}^\delta} + \|(E,B)\|^2_{\dot{H}^s}\right) + \nu\|v\|^2_{\dot{H}^{\delta+1}} + \frac{1}{\sigma}\|j\|^2_{\dot{H}^s}
			&=: \sum^3_{k=1} I_k,
		\end{align*}
		where for some $\epsilon \in (0,1)$, since $\delta \in [s-1,1]$ and $s \in (1,2]$
		\begin{align*}
			I_1 &= \int_{\mathbb{R}^2} (j \times B) \cdot \Lambda^{2\delta} v \,dx 
			\\
			&\leq C(\delta)
			\begin{cases}
				\|j\|_{L^2}\|B\|_{\dot{H}^\delta}\|\Lambda^{2\delta}v\|_{L^\frac{2}{1-\delta}} &\text{if} \quad \delta \in [s-1,1),
				\\
				\|j\|^\frac{1}{2}_{L^2}\|\nabla j\|^\frac{1}{2}_{L^2}\|B\|^\frac{1}{2}_{L^2}\|\nabla B\|^\frac{1}{2}_{L^2}\|\Delta v\|_{L^2} &\text{if} \quad \delta = 1,
			\end{cases}
			\\
			&\leq 
			\begin{cases}
				\epsilon\nu \|v\|^2_{\dot{H}^{\delta+1}} + C(\epsilon,\delta,\nu)\|j\|^2_{L^2}\|B\|^2_{\dot{H}^\delta}  &\text{if} \quad \delta \in [s-1,1),
				\\
				\frac{\epsilon}{\sigma}\|j\|^2_{\dot{H}^1} + \epsilon \nu\|v\|^2_{\dot{H}^2} +  C(\epsilon,\nu,\sigma)\|j\|^2_{L^2}\|B\|^2_{L^2}\|B\|^2_{\dot{H}^1} &\text{if} \quad \delta = 1;
			\end{cases}
			\\
			I_2 &= - \int_{\mathbb{R}^2} (v \cdot \nabla v) \cdot \Lambda^{2\delta} v \,dx 
			\leq
			\begin{cases}
				\epsilon\nu\|v\|^2_{\dot{H}^{\delta+1}} + C(\epsilon,\delta,\nu)\|\nabla v\|^2_{L^2}\|v\|^2_{\dot{H}^\delta} &\text{if} \quad \delta \in [s-1,1),
				\\
				0 &\text{if} \quad \delta = 1;
			\end{cases}
			\\
			I_3 &= \int_{\mathbb{R}^2} \Lambda^s j \cdot \Lambda^s (v \times B) \,dx
			\\
			&\leq 
			\begin{cases}
				C(s)\|j\|_{\dot{H}^s}\left(\|\Lambda^s v\|_{L^\frac{2}{1-(\delta+1-s)}}\|B\|_{L^\frac{2}{\delta+1-s}} + \|v\|_{L^\infty}\|B\|_{\dot{H}^s}\right) &\text{if} \quad \delta > s - 1,
				\\
				C(s)\|j\|_{\dot{H}^s}\left(\|\Lambda^{\delta+1} v\|_{L^2}\|B\|_{L^\infty} + \|v\|_{L^\infty}\|B\|_{\dot{H}^s}\right) &\text{if} \quad \delta = s - 1,
			\end{cases}
			\\
			&\leq
			\begin{cases}
				C(s)\|j\|_{\dot{H}^s}\left(\|\Lambda^{\delta+1} v\|_{L^2}\|\Lambda^{s-\delta}B\|_{L^2} + \|v\|_{L^\infty}\|B\|_{\dot{H}^s}\right) &\text{if} \quad \delta > s - 1,
				\\
				C(s)\|j\|_{\dot{H}^s}\left(\|\Lambda^{\delta+1} v\|_{L^2}\|B\|_{L^\infty} + \|v\|_{L^\infty}\|B\|_{\dot{H}^s}\right) &\text{if} \quad \delta = s - 1,
			\end{cases}
			\\
			&\leq \frac{\epsilon}{\sigma}\|j\|^2_{\dot{H}^s} + \epsilon\nu \|v\|^2_{\dot{H}^{\delta+1}} + C(\epsilon,\nu,s)\|j\|^2_{\dot{H}^s}\left(\|B\|^2_{\dot{H}^{s-\delta}} + \|B\|^2_{L^\infty}\right) + C(\epsilon,\sigma,s)\|v\|^2_{L^\infty}\|B\|^2_{\dot{H}^s}.
		\end{align*}
		Therefore, an application of Step 3 gives us the conclusion.
		
		\textbf{Step 12: The case $s \in (1,2]$ and $1 < \delta < s$ (revisited).} In this case, we apply Step 8 to obtain $(v,E,B) \in L^\infty_tH^\delta_x \cap L^2_tH^{\delta+1}_x \times L^\infty_tH^\delta_x \times L^\infty_tH^\delta_x$. It remains to obtain the $H^s$ estimate of $(E,B)$. We only need to bound $I_3$. Since $s \in (1,2]$ and $1 < \delta < s$ with $\delta + 1 - s \in (0,1)$, $I_3$ is given and bounded as the previous step.
		
		\textbf{Step 13: The case $s > 2$ and $s-1 \leq \delta < s$ (revisited).} Similar to the previous step, an application of Step 8 gives us $v \in  L^\infty_tH^\delta_x \cap L^2_tH^{\delta+1}_x$. We now focus on $I_3$ by using $s \leq \delta + 1$
		\begin{align*}
			I_3 &= \int_{\mathbb{R}^2} \Lambda^s j \cdot \Lambda^s (v \times B) \,dx 
			\\
			&\leq C(s)\|j\|_{H^s}\|v\|_{H^{\delta+1}}\|B\|_{H^s} 
			\\
			&\leq \frac{\epsilon}{\sigma} \|j\|^2_{H^s} + C(\epsilon,\sigma,s)\|v\|^2_{H^{\delta+1}}\|B\|^2_{H^s}.
		\end{align*}
		In addition, from Step 8 to Step 13, it can be seen that since $\delta,s > 1$, for $t \in (0,T^n_*)$
		\begin{equation*}
			\|j(t)\|_{H^{\min\{\delta,s\}}} \leq C(c,\sigma)\left(\|E(t)\|_{H^s} + \|v(t)\|_{H^\delta}\|B(t)\|_{H^s}\right).
		\end{equation*} 
		
		\textbf{Step 14: The bound of $\|v\|_{L^2_tL^\infty_x}$ (revisited).} We revisit this case with a slightly different decomposition as it has been considered in \cite{Arsenio-Gallagher_2020,Masmoudi_2010} before. In addition, the idea here will be applied to the proof of Theorem \ref{theo1-3d} below. Assume that $v_0 \in L^2$ and $(E_0,B_0) \in H^s$ for some $s \in (0,1)$. In the previous steps, to close the main estimates, we need to assume that $v_0 \in H^\delta$ for some $\delta > 0$. Thus, the same argument might not work in the case $\delta = 0$ with $v \in L^\infty_tL^2_x \cap L^2_tH^1_x$, which is not enough to bound the norm $\|v\|_{L^2_tL^\infty_x}$. In the two-dimensional case, there is a way to overcome this difficulty in which the idea comes from the recent result in \cite{Arsenio-Gallagher_2020}, where the authors have introduced a suitable decomposition of the velocity, which is useful for obtaining first an estimate of $\|v\|_{L^2_tL^\infty_x}$ in terms of  $\|B\|_{L^\infty_t\dot{H}^s_x}$ for $s \in (0,1)$ and then closing the estimate of $\|(E,B)\|_{L^\infty_t\dot{H}^s_x}$. As a consequence, they are able to bound the norm $\|v\|_{L^2_tL^\infty_x}$. More precisely, they decomposed the velocity and pressure by $v = \bar{v}^1 + \bar{v}^2 + \bar{v}^3$ and $\pi = \bar{\pi}_1 + \bar{\pi}_2$, where $\bar{v}^1$, $\bar{v}^2$ and $\bar{v}^3$ are solutions of the following heat equation and Stokes systems
		\begin{align*}
			&&\partial_t \bar{v}^1 -\nu \Delta \bar{v}^1 &= 0, 
			&&\textnormal{div}\, \bar{v}^1 = 0,
			&&\bar{v}^1_{|_{t=0}} = v_0,&&
			\\
			&&\partial_t \bar{v}^2 -\nu \Delta \bar{v}^2 + \nabla \bar{\pi}_1 &= -v \cdot \nabla v,
			&&\textnormal{div}\, \bar{v}^2 = 0,
			&&\bar{v}^2_{|_{t=0}} = 0,&&
			\\
			&&\partial_t \bar{v}^3 -\nu \Delta \bar{v}^3 + \nabla \bar{\pi}_2 &= j \times B,
			&&\textnormal{div}\, \bar{v}^3 = 0,
			&&\bar{v}^3_{|_{t=0}} = 0.&&
		\end{align*}
		That allows them to study each part of the decomposition separately, where the most difficult part is dealing with $\bar{v}^3$ in which they overcomed this issue by introducing a suitable iteration. In fact, it is possible to combine $\bar{v}^1$ and $\bar{v}^2$ parts together by decomposing $v^n = v^{n,1} + v^{n,2}$ in such a way that $v^{n,i} \in L^2_n$ for $i \in \{1,2\}$. In the sequel, we will write $(v^1,v^2)$ instead of $(v^{n,1},v^{n,2})$ for simplicity. Indeed, we first define $v^1$ be a divergence-free vector with $\text{supp}(\mathcal{F}(v^1)) \subseteq B_n$ and be a solution of the first equation below. It can be seen that from the properties of $v$ that such a $v^1 \in L^2$ exists (see the estimate below). Then, we set $v^2 = v-v^1$, which leads to $v^2 \in L^2$, $\text{supp}(\mathcal{F}(v^2)) \subseteq B_n$ and $\text{div}\,v^2 = 0$. It follows from \eqref{NSM_app} that 
		\begin{align*}
			&&\partial_t v^1 -\nu \Delta v^1 
			&= -\mathbb{P}(T_n(v \cdot \nabla v)),
			&&\textnormal{div}\, v^1 = 0,
			&&v^1_{|_{t=0}} = T_n(v_0),&&
			\\
			&&\partial_t v^2 -\nu \Delta v^2 
			&= \mathbb{P}(T_n(j \times B)),
			&&\textnormal{div}\, v^2 = 0,
			&&v^2_{|_{t=0}} = 0.&&
		\end{align*}
		In the sequel, we firstly explain how $\|v\|_{L^2_tL^\infty_x}$ can be controlled only in terms of $\|v^2\|_{L^2_tL^\infty_x}$ and secondly use the technique in \cite{Arsenio-Gallagher_2020} to bound $\|v^2\|_{L^2_tL^\infty_x}$. We first focus on obtaining estimates of $v^1$. Its energy estimate is given by
		\begin{align*}
			\frac{d}{dt}\|v^1\|^2_{L^2} + \nu\|\nabla v^1\|^2_{L^2} 
			&\leq C(\nu) \|v\|^2_{L^2}\|\nabla v\|^2_{L^2},
		\end{align*}
		which implies for $t \in (0,T^n_*)$ by using the energy estimate of $v$
		\begin{equation*}
			\|v^1(t)\|^2_{L^2} + \nu\int^t_0 \|\nabla v^1\|^2_{L^2} \,d\tau \leq \|v^1(0)\|^2_{L^2} + C(\nu) \|v\|^2_{L^\infty_tL^2_x} \int^t_0 \|\nabla v\|^2_{L^2} \,d\tau \leq C(\mathcal{E}_0,\nu).
		\end{equation*}
		Moreover, it follows that for $t \in (0,T^n_*)$ and $q \in \mathbb{Z}$
		\begin{equation*}
			\frac{1}{2}\frac{d}{dt}\|\Delta_q v^1\|^2_{L^2} + \nu \|\Delta_q \nabla v^1\|^2_{L^2} 
			\leq \|\Delta_q v^1\|_{L^2}\|\Delta_q (v \cdot \nabla v)\|_{L^2} .
		\end{equation*}
		It can be seen from the definition of nonhomogeneous dyadic blocks (see Appendix A in Section \ref{sec:app}) that for $q \in \mathbb{Z}$ with $q \geq 0$ 
		\begin{equation*}
			\|\Delta_q \nabla v^1\|^2_{L^2} 
			\geq C 2^{2q}\|\Delta_q v^1\|^2_{L^2},
		\end{equation*}
		which yields 
		\begin{align*}
			\esssup_{t \in (0,T^n_*)} \|\Delta_q v^1(t)\|^2_{L^2} + C(\nu) \left(\int^{T^n_*}_0 2^{2q}\|\Delta_q v^1\|_{L^2} \,d\tau\right)^2
			&\leq  R^2_q \qquad \text{if} \quad q \geq 0,
			\\
			\esssup_{t \in (0,T^n_*)} \|\Delta_q v^1(t)\|^2_{L^2}  
			&\leq R^2_q \qquad \text{if} \quad q \geq -1,
		\end{align*}
		where
		\begin{equation*}
			R_q := \int^{T^n_*}_0 \|\Delta_q (v \cdot \nabla v)\|_{L^2}  \,d\tau + \|\Delta_q v(0)\|_{L^2}.
		\end{equation*}
		Furthermore,
		\begin{equation*}
			\sum_{q \geq -1} \esssup_{t \in (0,T^n_*)} \|\Delta_q v^1(t)\|^2_{L^2} + C(\nu) \sum_{q \geq 0} \left(\int^{T^n_*}_0 2^{2q}\|\Delta_q v^1\|_{L^2} \,d\tau\right)^2
			\leq  2\sum_{q \geq -1} R^2_q.
		\end{equation*}
		We now estimate the right-hand side as follows
		\begin{align*}
			\sum_{q \geq -1} R^2_q &\leq C \sum_{q \geq -1} \left(\int^{T^n_*}_0 \|\Delta_q (v \cdot \nabla v)\|_{L^2} \,d\tau\right)^2 + C\|v_0\|^2_{L^2}
			\\
			&\leq C\left(\int^{T^n_*}_0 \left(\sum_{q \geq -1} \|\Delta_q (v \cdot \nabla v)\|^2_{L^2} \right)^\frac{1}{2} \,d\tau\right)^2 + C\|v_0\|^2_{L^2}
			\\
			&\leq C\left(\|v^1\|^2_{L^2_tL^\infty_x} + \|v^2\|^2_{L^2_tL^\infty_x}\right) \|\nabla v\|^2_{L^2_tL^2_x} + C\|v_0\|^2_{L^2},
		\end{align*}
		where we used the fact that $L^1(0,T^n_*;B^0_{2,2}(\mathbb{R}^2)) \subset \tilde{L}^1(0,T^n_*;B^0_{2,2}(\mathbb{R}^2))$, see the definition of this functional space in Appendix A in Section \ref{sec:app}. We first focus on obtaining a bound on the norm $\|v^1\|_{L^2_tL^\infty_x}$. The Littlewood–Paley decomposition gives us
		\begin{align*}
			\int^{T^n_*}_0 \|v^1\|^2_{L^\infty} \,d\tau 
			&\leq C \int^{T^n_*}_0 \sum_{q \geq -1} 2^q \|\Delta_q v^1\|_{L^2}  \sum_{k \geq -1}  2^k \|\Delta_k v^1\|_{L^2} \,d\tau,
			\\
			&\leq C \int^{T^n_*}_0 \left(\sum_{|q-k| \leq N} + \sum_{|q-k| > N}\right) 2^q \|\Delta_q v^1\|_{L^2}   2^k \|\Delta_k v^1\|_{L^2} \,d\tau =: \bar{R}_1 + \bar{R}_2,
		\end{align*}
		where $N \in \mathbb{N}$ to be determined later and we used the following Bernstein-type estimate (see \cite{Bahouri-Chemin-Danchin_2011}) for $1 \leq q_0 \leq p_0 \leq \infty$ 
		\begin{equation*}
			\|f\|_{L^{p_0}(\mathbb{R}^d)} \leq C(p_0,q_0,d) \lambda_0^{d\left(\frac{1}{q_0}-\frac{1}{p_0}\right)} \|f\|_{L^{q_0}(\mathbb{R}^d)} \qquad \text{if} \quad \text{supp}(\mathcal{F}(f)) \subset \left\{\xi \in \mathbb{R}^d : |\xi| \leq \lambda_0 \right\}.
		\end{equation*}		
		The terms on the right-hand side can be bounded as follows
		\begin{align*}
			\bar{R}_1 &= C\int^{T^n_*}_0 \sum_{q \geq -1} 2^q \|\Delta_q v^1\|_{L^2}   \sum_{q-N \leq k \leq q+N} 2^k \|\Delta_k v^1\|_{L^2} \,d\tau
			\\
			&\leq C\int^{T^n_*}_0 \sum_{q \geq -1}  \left(2^{2q}\|\Delta_q v^1\|^2_{L^2}  + \sum_{q-N \leq k \leq q+N} 2^{2k} \|\Delta_k v^1\|^2_{L^2} \right) \,d\tau
			\\
			&\leq CN \left(\|\nabla v^1\|^2_{L^2_tL^2_x} + \|v^1\|^2_{L^2_tL^2_x}\right),
		\end{align*}
		and by using Young inequality for sequences
		\begin{align*}
			\bar{R}_2 &= C2^{1-N} \sum^\infty_{q = N} \int^{T^n_*}_0  2^{2q} \|\Delta_q v^1\|_{L^2}  \sum_{k = -1} ^{q-N-1} 2^{k-(q-N)} \|\Delta_k v^1\|_{L^2} \,d\tau
			\\
			&\leq C2^{1-N} \left(\sum^\infty_{q = N} \left(\int^{T^n_*}_0  2^{2q} \|\Delta_q v^1\|_{L^2} \,d\tau \right)^2\right)^\frac{1}{2}   \left(\sum^\infty_{k_0=0} \left(\sum_{k = -1} ^{k_0-1} 2^{-(k_0-k)} \esssup_{t \in (0,{T^n_*})} \|\Delta_k v^1\|_{L^2} \right)^2\right)^\frac{1}{2}
			\\
			&\leq C2^{1-N} \left(\sum^\infty_{q = 0} \left(\int^{T^n_*}_0  2^{2q} \|\Delta_q v^1\|_{L^2} \,d\tau \right)^2\right)^\frac{1}{2}   \left(\sum^\infty_{k=-1} \esssup_{t \in (0,{T^n_*})} \|\Delta_k v^1\|^2_{L^2} \right)^\frac{1}{2} \sum^\infty_{k = -1} 2^{-k}.
		\end{align*}
		Therefore,
		\begin{align*}
			\|v^1\|^2_{L^2_tL^\infty_x}  &\leq C(\nu) 2^{-N} \left(\|v^1\|^2_{L^2_tL^\infty_x} + \|v^2\|^2_{L^2_tL^\infty_x}\right) \|\nabla v\|^2_{L^2_tL^2_x} 
			+ CN\left(\|\nabla v^1\|^2_{L^2_tL^2_x} + \|v^1\|^2_{L^2_tL^2_x}\right)
			+ C \|v_0\|^2_{L^2},
		\end{align*}
		which by choosing\footnote{Here $\lceil \cdot \rceil$ denotes the usual ceiling function.}
		\begin{equation*}
			N = \left\lceil\log_2\left(4 + 2C(\nu)\|\nabla v\|^2_{L^2_tL^2_x}\right)\right\rceil
		\end{equation*}
		and using the energy estimates of $v^1$ and $v$ yields
		\begin{align*}
			\|v^1\|^2_{L^2_tL^\infty_x} &\leq 
			C(\mathcal{E}_0,\nu)\left((\|\nabla v^1\|^2_{L^2_tL^2_x} + \|v^1\|^2_{L^2_tL^2_x})(1+\|\nabla v\|^2_{L^2_tL^2_x}) + \|v^2\|^2_{L^2_tL^\infty_x} + 1\right)
			\\
			&\leq 
			C(\mathcal{E}_0,\nu)(1 + T^n_*)\left(\|v^2\|^2_{L^2_tL^\infty_x} + 1\right).
		\end{align*}
		In addition,
		\begin{equation*}
			\|v\|^2_{L^2_tL^\infty_x} \leq C\left(\|v^1\|^2_{L^2_tL^\infty_x} + \|v^2\|^2_{L^2_tL^\infty_x}\right) \leq C(\mathcal{E}_0,\nu)(1+T^n_*)\left(\|v^2\|^2_{L^2_tL^\infty_x} + 1\right).
		\end{equation*}
		It remains to bound the term $\|v^2\|_{L^2_tL^\infty_x}$. It can be seen that
		\begin{equation*}
			\|v^2\|^2_{L^2_t\dot{H}^1_x} \leq C\left(\|v\|^2_{L^2_t\dot{H}^1_x} + \|v^1\|^2_{L^2_t\dot{H}^1_x}\right) \leq C(\mathcal{E}_0,\nu).
		\end{equation*}
		Let us now summarize how the authors in \cite{Arsenio-Gallagher_2020} obtained a bound on $\|v^2\|_{L^2_tL^\infty_x}$ in terms of $\|B\|_{L^2_t\dot{H}^s_x}$ and use this relation to close the estimate of $\|(E,B)\|_{L^\infty_t\dot{H}^s_x}$ itself. By decomposing $v^2$ into high and low frequencies in \cite[Lemmas 7.3 and 7.4]{Arsenio-Gallagher_2020}, respectively, they proved that for $0 \leq t_0 < t < {T^n_*}$
		\begin{equation*}
			\|v^2\|^2_{L^2(t_0,t;L^\infty)} \leq C(\mathcal{E}_0,\nu)\log(e+t-t_0) +  C(\sigma,s)\|v^2\|^2_{L^2(t_0,t;\dot{H}^1)}\log\left(e + \frac{\mathcal{E}^2_0\|B\|^2_{L^\infty(t_0,t;\dot{H}^s)}}{\|v^2\|^2_{L^2(t_0,t;\dot{H}^1)}}\right).
		\end{equation*}
		Moreover, as the estimate of $I_3$ in Step 1 (consider only the equations of $(E,B)$), we find that\footnote{There is a slightly different here, where the constant $C(\sigma,s)$ does not depend on $c$. However, it seems not to be the case as in \cite{Arsenio-Gallagher_2020}, where the authors used the relation $j = \sigma(cE + v \times B)$ and considered $\sigma cE$ as a damping term on the left-hand side in the equation of $E$.}
		\begin{align*}
			\frac{1}{2}\frac{d}{dt}\|(E,B)\|^2_{\dot{H}^s} + \frac{1}{\sigma}\|j\|^2_{\dot{H}^s} 
			\leq \frac{\epsilon}{\sigma}\|j\|^2_{\dot{H}^s} + C(\epsilon,\sigma,s)\left(\|\nabla v\|^2_{L^2} + \|v\|^2_{L^\infty}\right)\|B\|^2_{\dot{H}^s}
		\end{align*}
		and then by choosing $\epsilon = \frac{1}{2}$ for $0 \leq t_0 < t < {T^n_*}$
		\begin{equation*}
			\|(E,B)(t)\|^2_{\dot{H}^s} 
			\leq \|(E,B)(t_0)\|^2_{\dot{H}^s} \exp\left\{C(\sigma,s)\int^t_{t_0} \|\nabla v\|^2_{L^2} + \|v\|^2_{L^\infty} \,d\tau\right\},
		\end{equation*}
		and by using the bound on $\|v\|_{L^2_tL^\infty_x}$ in terms of $\|v^2\|_{L^2(t_0,t;L^\infty)}$ and those of $\|(v,v^2)\|_{L^2_t\dot{H}^1_x}$, an iteration as in \cite[The proof of Theorem 1.2]{Arsenio-Gallagher_2020} can be applied to the above bound of $(E,B)$, which gives us\footnote{Here, the authors in \cite{Arsenio-Gallagher_2020} used a suitable time decomposition of the whole time interval $(0,\infty)$ based on the fact that $\|v^2\|_{L^2_t\dot{H}^1_x} < \infty$, which allowed them to set up an iteration and obtain the bound of $\|(E,B)\|_{L^\infty_t\dot{H}^s_x}$ on each small time interval, then they obtained the bound on the whole time interval by using the continuous in time of regularized solutions. Here, we only need to change $(0,\infty)$ into $(0,T^n_*)$.}
		\begin{equation*}
			\mathcal{E}^2_0\|(E,B)\|^2_{L^\infty(0,t;\dot{H}^s)} \leq \left(e + \mathcal{E}^2_0\|(E_0,B_0)\|^2_{\dot{H}^s} + 
			t\right)^{C(T^n_*,\mathcal{E}_0,\nu,\sigma,s)} \qquad\text{for} \quad t \in (0,T^n_*), 
		\end{equation*}
		which by using the increasing of the function $z \mapsto z\log(e+\frac{C}{z})$ for $z > 0$, implies that
		\begin{equation*}
			\|v^2\|^2_{L^2(0,t;L^\infty)} +
			\|v\|^2_{L^2(0,t;L^\infty)} 
			\leq C(T^n_*,\mathcal{E}_0,\nu,\sigma,s)\left(\mathcal{E}^2_0 + \left(e + \mathcal{E}^2_0\|(E_0,B_0)\|^2_{\dot{H}^s} + t\right)\right)^{C(T^n_*,\mathcal{E}_0,\nu,\sigma,s)},
		\end{equation*} 
		where $C(T^n_*,\mathcal{E}_0,\nu,\sigma,s) = (1+T^n_*)C(\mathcal{E}_0,\nu,\sigma,s)$.
		
		\textbf{Step 15: Conclusion of Part II.} Collecting the main estimates from Step 1 to Step 14, we find that $T^n_* = \infty$. Moreover, by replacing $T^n_*$ by any given (does not depend on $n$) $T \in (0,\infty)$ and repeating  the above calculations, it follows that for $t \in (0,T)$ and for the same range of $(s,\delta)$ from Step 1 to Step 14
		\begin{equation*}
			\|v^n(t)\|^2_{H^\delta} + \|(E^n,B^n)(t)\|^2_{H^s} + \int^t_0 \|v^n\|^2_{H^{\delta+1}} + \|v^n\|^2_{L^\infty} + \|j^n\|^2_{H^s} \,d\tau \leq C(T,\delta,\nu,\sigma,s,v_0,E_0,B_0).
		\end{equation*} 
		
		\textbf{Part III: Pass to the limit and uniqueness.} Although this part is quite standard, it has not been given in details in \cite{Arsenio-Gallagher_2020,Masmoudi_2010}. Our aim in this part is to fulfill this gap for the sake of completeness and also for later use in the proofs of next theorems. Firstly, by using the ideas in \cite{Fefferman-McCormick-Robinson-Rodrigo_2014,KLN_2024,Majda_Bertozzi_2002}, we prove that $(v^n,E^n,B^n)$ and $(\nabla v^n,j^n)$ are Cauchy sequences in $L^\infty(0,T;L^2(\mathbb{R}^2))$ and $L^2(0,T;L^2(\mathbb{R}^2))$, respectively, for any $T \in (0,\infty)$ and $\delta,s > 1$, which allows us to pass to the limit from \eqref{NSM_app} in a stronger sense than the usual one (the sense of distributions) in the case either $\delta \in [0,1]$ or $s \in (0,1]$. Secondly, the uniqueness can be obtained by a carefully analysis. 
		
		\textbf{Step 16: Cauchy sequence.} Assume that $(v^n,E^n,B^n)$ and $(v^m,E^m,B^m)$ for $m,n \in \mathbb{N}$ with $m > n$ are two solutions to \eqref{NSM_app} with the same initial data. It follows that 
		\begin{equation*}
			\frac{1}{2}\frac{d}{dt} \|(v^n-v^m,E^n-E^m,B^n-B^m)\|^2_{L^2} + \nu\|\nabla (v^n-v^m)\|^2_{L^2} + \frac{1}{\sigma}\|j^n-j^m\|^2_{L^2} =: \sum^{8}_{k=4} I_k,
		\end{equation*}
		where for some $\epsilon \in (0,1)$, since $\delta,s > 1$
		\begin{align*}
			I_4 &= \int_{\mathbb{R}^2} \mathbb{P}(-T_n(v^n \cdot \nabla v^n) + T_m(v^m \cdot \nabla v^m)) \cdot (v^n-v^m)\,dx =: \sum^{3}_{k=1} I_{4k},
			\\
			I_{41} &= -\int_{\mathbb{R}^2} (T_n-T_m)(v^n \cdot \nabla v^n) \cdot (v^n-v^m)\,dx
			\leq C(\delta)n^{-(\delta-1)}\|v^n\|^2_{H^\delta}\|v^n-v^m\|_{L^2},
			\\
			I_{42} &= -\int_{\mathbb{R}^2} T_m((v^n-v^m) \cdot \nabla v^n) \cdot (v^n-v^m)\,dx
			\leq C(\epsilon,\nu)\|\nabla v^n\|^2_{L^2}\|v^n-v^m\|^2_{L^2} + 
			\epsilon\nu\|\nabla(v^n-v^m)\|^2_{L^2},
			\\
			I_{43} &= -\int_{\mathbb{R}^2} T_m(v^m \cdot \nabla (v^n-v^m)) \cdot (v^n-v^m)\,dx = 0;
			\\
			I_5 &= \int_{\mathbb{R}^2} \mathbb{P}(T_n(j^n \times B^n) - T_m(j^m \times B^m)) \cdot (v^n-v^m)\,dx =: \sum^{3}_{k=1} I_{5k},
			\\
			I_{51} &= \int_{\mathbb{R}^2} (T_n-T_m)(j^n\times B^n) \cdot (v^n-v^m)\,dx  
			\leq C(s)n^{-2s}\|j^n\|^2_{H^s} + \|B^n\|^2_{H^s}\|v^n-v^m\|^2_{L^2},
			\\
			I_{52} &= \int_{\mathbb{R}^2} T_m((j^n-j^m)\times B^n) \cdot (v^n-v^m)\,dx,
			\\
			I_{53} &= \int_{\mathbb{R}^2} T_m(j^m \times (B^n-B^m)) \cdot (v^n-v^m)\,dx
			\\
			&\leq C(\epsilon,\nu)\|j^m\|^2_{H^s} \|B^n-B^m\|^2_{L^2} + 
			\epsilon\nu \left(\|v^n-v^m\|^2_{L^2} + \|\nabla(v^n-v^m)\|^2_{L^2}\right);
			\\
			I_6 &= -\int_{\mathbb{R}^2} (j^n - j^m) \cdot (-T_n(v^n \times B^n) + T_m(v^m \times B^m))\,dx =: \sum^{3}_{k=1} I_{6k},
			\\
			I_{61} &= \int_{\mathbb{R}^2} (j^n-j^m) \cdot (T_n-T_m)(v^n \times B^n)\,dx 
			\leq \frac{\epsilon}{\sigma}\|j^n-j^m\|^2_{L^2} + C(\epsilon,\delta,\sigma,s) n^{-2\min\{\delta,s\}}\|v^n\|^2_{H^\delta}\|B^n\|^2_{H^s},
			\\
			I_{62} &= \int_{\mathbb{R}^2} (j^n-j^m) \cdot T_m((v^n-v^m) \times B^n)\,dx = - I_{52},
			\\
			I_{63} &= \int_{\mathbb{R}^2} (j^n-j^m) \cdot T_m(v^m \times (B^n-B^m))\,dx \leq \frac{\epsilon}{\sigma}\|j^n-j^m\|^2_{L^2} + C(\epsilon,\sigma)\|v^n\|^2_{H^\delta}\|B^n-B^m\|^2_{L^2};
			\\
			I_7 &= \int_{\mathbb{R}^2} \nabla \times (B^n-B^m) \cdot (c(E^n-E^m)) \,dx;
			\\
			I_8 &= -\int_{\mathbb{R}^2} \nabla \times (E^n-E^m) \cdot (c(B^n-B^m)) \,dx = - I_7;
		\end{align*}
		here we used the fact that $H^{s_0}(\mathbb{R}^2)$ is an algebra for $s_0 > 1$ and the following inequality (see \cite{Fefferman-McCormick-Robinson-Rodrigo_2014})
		\begin{align*}
			\|T_n(f) - f\|_{H^{s_1}} &\leq n^{-s_2} \|f\|_{H^{s_1+s_2}} \qquad \text{for} \quad s_1,s_2 \in \mathbb{R}, s_2 \geq 0.
		\end{align*}
		Therefore, by choosing $\epsilon = \frac{1}{4}$
		\begin{multline*}
			\frac{d}{dt}E^{nm} + \nu\|\nabla(v^n-v^m)\|^2_{L^2} + \frac{1}{\sigma}\|j^n-j^m\|^2_{L^2} \leq C(\nu,\sigma)\left(1 +
			\|v^n\|^2_{H^\delta} + \|j^m\|^2_{H^s} + \|B^n\|^2_{H^s}\right)E^{nm}
			\\
			+ C(\delta,\sigma,s)\left(n^{-(\delta-1)}\|v^n\|^2_{H^\delta}\|v^n-v^m\|_{L^2} + n^{-2s}\|j^n\|^2_{H^s} +  n^{-2\min\{\delta,s\}}\|v^n\|^2_{H^\delta}\|B^n\|^2_{H^s}\right).
		\end{multline*}
		By denoting for $t \in (0,T)$ 
		\begin{equation*}
			E^{nm}(t) := \|(v^n-v^m,E^n-E^m,B^n-B^m)(t)\|^2_{L^2}
		\end{equation*}
		and using $E^{nm}(0) = 0$ and Step 15, it follows that for $C = C(T,\delta,\nu,\sigma,s,v_0,E_0,B_0)$
		\begin{equation*}
			E^{nm}(t) + \int^t_0 \|\nabla(v^n-v^m)\|^2_{L^2} + \|j^n-j^m\|^2_{L^2} \,d\tau \leq C\max\left\{n^{-(\delta-1)},n^{-2s},n^{-2\min\{\delta,s\}}\right\},
		\end{equation*}
		which ends the proof by letting $n \to \infty$.

		\textbf{Step 17: Pass to the limit.} There are two substeps in this step as follows.
		
		\textbf{Step 17a: The case $\delta,s > 1$.} We use the notation $\to$, $\rightharpoonup$ and $\overset{\ast}{\rightharpoonup}$ to denote the usual strong, weak and weak-star convergences, respectively. From the previous step, there exists $(v,E,B,j)$ such that as $n \to \infty$
		\begin{align*}
			&&(v^n,E^n,B^n) &\to (v,E,B) &&\text{in} \quad L^\infty(0,T;L^2(\mathbb{R}^2)),&&
			\\
			&&(\nabla v^n,j^n) &\to (\nabla v,j)  &&\text{in} \quad L^2(0,T;L^2(\mathbb{R}^2)),&&
		\end{align*}
		which implies by using interpolation inequalities and Step 15 that for all $s' \in (1,\min\{\delta,s\})$ as $n \to \infty$
		\begin{align*}
			&&(v^n,E^n,B^n) &\to (v,E,B) &&\text{in} \quad L^\infty(0,T;H^{s'}(\mathbb{R}^2)),&&
			\\
			&&(\nabla v^n,j^n) &\to (\nabla v,j)  &&\text{in} \quad L^2(0,T;H^{s'}(\mathbb{R}^2)),&&
			\\
			&&(\Delta v^n, \nabla \times E^n,\nabla \times B^n) &\to (\Delta v, \nabla \times E, \nabla \times B)  &&\text{in} \quad L^2(0,T;H^{s'-1}(\mathbb{R}^2)).&&
		\end{align*}
		Moreover, for the nonlinear terms as $n \to \infty$ 
		\begin{align*}
			&&T_n(v^n \cdot \nabla v^n,v^n \times B^n) &\to (v \cdot \nabla v,v \times B) &&\text{in} \quad L^\infty(0,T;H^{s'-1}(\mathbb{R}^2) \times H^{s'}(\mathbb{R}^2)),&&
			\\
			&&T_n(j^n \times B^n) &\to j \times B &&\text{in} \quad L^2(0,T;H^{s'}(\mathbb{R}^2)),&&
		\end{align*}
		since  
		\begin{align*}
			\textnormal{N}_1 &:= \|T_n(v^n \cdot \nabla v^n) - v \cdot \nabla v\|_{H^{s'-1}} 
			\leq n^{s'-\delta} \|v^n\|^2_{H^\delta}  + \|v^n-v\|_{H^{s'}} \left(\|v^n\|_{H^\delta} + \|v\|_{H^\delta}\right);
			\\
			\textnormal{N}_2 &:= \|T_n(v^n \times B^n) - v \times B\|_{H^{s'}} 
			\\
			&\leq n^{s'-\min\{\delta,s\}}\|v^n\|_{H^\delta}\|B^n\|_{H^s} + \|v^n-v\|_{H^{s'}}\|B^n\|_{H^s} + \|v\|_{H^\delta}\|B^n-B\|_{H^{s'}};
			\\
			\textnormal{N}_3 &:= \int^T_0 \|T_n(j^n \times B^n) - j \times B\|^2_{H^{s'}} \,dt
			\\
			&\leq  C\int^T_0 n^{s'-s}\|j^n\|^2_{H^s}\|B^n\|^2_{H^s} + \|j^n-j\|^2_{H^{s'}}\|B^n\|^2_{H^s}  + \|B^n-B\|^2_{H^{s'}}\|j\|^2_{H^s} \,dt.
		\end{align*}
		In addition, \eqref{NSM_app} gives us for $t \in (0,T)$, $\sigma \in \{s'-1,\delta-1\}$ and $\sigma' \in \{s'-1,s-1\}$
		\begin{align*}
			\int^t_0\|\partial_t v^n\|^2_{H^\sigma} \,d\tau &\leq 
			C\int^t_0 \|(v^n \cdot \nabla v^n,  \nu \Delta v^n,j^n \times B^n)\|^2_{H^\sigma} \,d\tau,
			\\
			\int^t_0\left\|\frac{1}{c}(\partial_t E^n,\partial_t B^n)\right\|^2_{H^{\sigma'}} \,d\tau &\leq 
			C\int^t_0 \|(\nabla \times E^n,\nabla \times B^n,j^n)\|^2_{H^{\sigma'}} \,d\tau,
		\end{align*}
		which together with Step 15 and the above strong convergences leads to there exists a subsequence denoted by $(v^{n_k},E^{n_k},B^{n_k})$ such that as $n_k \to \infty$
		\begin{align*}
			&&(\partial_t v^{n_k},\frac{1}{c}\partial_t E^{n_k},\frac{1}{c}\partial_t B^{n_k}) &\rightharpoonup (\partial_t v,\frac{1}{c}\partial_t E,\frac{1}{c}\partial_t B) &&\text{in} \quad L^2(0,T;H^{\delta-1}(\mathbb{R}^2) \times H^{s-1}(\mathbb{R}^2) \times H^{s-1}(\mathbb{R}^2)),&&
			\\
			&&(\partial_t v^{n_k},\frac{1}{c}\partial_t E^{n_k},\frac{1}{c}\partial_t B^{n_k}) &\to (\partial_t v,\frac{1}{c}\partial_t E,\frac{1}{c}\partial_t B) &&\text{in} \quad L^2(0,T;H^{s'-1}(\mathbb{R}^2)).&&
		\end{align*}
		In addition, the above strong convergences and \eqref{NSM_app} imply that in $L^2(0,T;H^{s'-1}(\mathbb{R}^2))$
		\begin{equation} \label{NSM_P}
			\left\{
			\begin{aligned}
				\partial_t v + \mathbb{P}(v \cdot \nabla v) &= \nu \Delta v +  \mathbb{P}(j \times B), 
				\\
				\frac{1}{c}\partial_t E - \nabla \times B  &=  -j
				\\
				\frac{1}{c}\partial_t B + \nabla \times E &= 0
				\\
				\sigma(cE + v \times B) &= j,
				\\
				\textnormal{div}\, v = \textnormal{div}\, B &= 0.
			\end{aligned}
			\right.
		\end{equation}
		Indeed, it can be checked that as $n \to \infty$
		\begin{align*}
			\textnormal{div}\, v^n &\to \textnormal{div}\, v  &&\text{in}\quad L^2(0,T;H^{s'}(\mathbb{R}^2)),
			\\
			\textnormal{div}\, B^n &\to \textnormal{div}\, B  &&\text{in}\quad L^2(0,T;H^{s'-1}(\mathbb{R}^2)),
			\\
			(v^n,E^n,B^n)_{|_{t=0}} = T_n(v_0,E_0,B_0) &\to (v_0,E_0,B_0) & &\text{in}\quad H^\delta(\mathbb{R}^2) \times H^{s}(\mathbb{R}^2) \times H^{s}(\mathbb{R}^2),
		\end{align*}
		which leads to $\textnormal{div}\, v = \textnormal{div}\, B = 0$ and $(v,E,B)_{|_{t=0}} = (v_0,E_0,B_0)$. Then, the theorem de Rham (see \cite{Temam_2001}) ensures the existence of a scalar function $\pi$ such that $(v,E,B,\pi)$ satisfies  \eqref{NSM} at least in the sense of distributions. From the uniform bound in Step 15, we also have as $n_k \to \infty$ 
		\begin{align*}
			&&(v^{n_k},E^{n_k},B^{n_k}) &\overset{\ast}{\rightharpoonup} (v,E,B)  &&\text{in} \quad L^\infty(0,T;H^\delta(\mathbb{R}^2) \times H^{s}(\mathbb{R}^2) \times H^{s}(\mathbb{R}^2)), &&
			\\
			&&(\nabla v^{n_k},j^{n_k}) &\rightharpoonup  (\nabla v,j)  &&\text{in} \quad L^2(0,T;H^\delta(\mathbb{R}^2) \times H^s(\mathbb{R}^2)), &&
		\end{align*}
		which implies that for $\delta, s > 1$, 
		\begin{equation*}
			(v,E,B) \in L^\infty(0,T;H^\delta \times H^{s} \times H^s) \quad \text{and} \quad (v,j) \in  L^2(0,T;H^{\delta+1} \times H^{s})
		\end{equation*}
		satisfying for $t \in (0,T)$
		\begin{equation*}
			\|v(t)\|^2_{H^\delta} + \|(E,B)(t)\|^2_{H^{s}} + \int^t_0 \|v\|^2_{H^{\delta+1}} + \|j\|^2_{H^s} \,d\tau \leq C(T,\delta,\nu,\sigma,s,v_0,E_0,B_0).
		\end{equation*}
		In fact, after possibly being redefined on a set of measure zero,  $v \in C([0,T];H^\delta(\mathbb{R}^2))$  (see \cite{Evans_2010,Temam_2001}) since $v \in  L^2(0,T;H^{\delta+1}(\mathbb{R}^2))$ and $\partial_t v \in L^2(0,T;H^{\delta-1}(\mathbb{R}^2))$. Furthermore, we find that from the uniform bound in Step 15, $(E,B)$ is weak continuous in time with values in $H^{s}(\mathbb{R}^2)$. 
		
		\textbf{Step 17b: The case either $\delta \in [0,1]$ or $s \in (0,1]$.} It is enough to focus on the case where $\delta = 0$ and $s \in (0,1)$, which is considered in Step 14 above. Other cases follow as consequences. It follows from the uniform bound in Step 15 that 
		\begin{equation*}
			(v^n,E^n,B^n) \quad \text{is uniformly bounded in} \quad L^\infty(0,T;L^2(\mathbb{R}^2) \times H^s(\mathbb{R}^2) \times H^s(\mathbb{R}^2))
		\end{equation*}
		satisfying for $t \in (0,T)$
		\begin{align*}
			\|v^n(t)\|^2_{L^2} + \|(E^n,B^n)(t)\|^2_{H^s} + \int^t_0 \|v^n\|^2_{H^1} + \|j^n\|^2_{H^s}\,d\tau &\leq C(T,\nu,\sigma,s,v_0,E_0,B_0).
		\end{align*}
		In addition, for $\phi \in H^1(K;\mathbb{R}^3)$ for any compact set $K \subset \mathbb{R}^2$ with $\|\phi\|_{H^1(K)} \leq 1$, it yields for $\tau \in (0,T)$\footnote{Here, $(\cdot,\cdot)$ is the standard $L^2$ inner product.}
		\begin{align*}
			&\int^{\tau}_0 |(\partial_tv^n,\phi)|^2 \,dt 
			\leq C(\nu)\int^{\tau}_0 \left(1+\|v^n\|^2_{L^2}\right)\|\nabla v^n\|^2_{L^2} + \|j^n\|^2_{L^2} \|B^n\|^2_{H^s}\,dt 
			\leq C(T,\nu,\sigma,s,v_0,E_0,B_0),
			\\
			&\frac{1}{c}\int^{\tau}_0 \left|(\partial_tE^n,\phi)\right|^2 + \left|(\partial_tB^n,\phi)\right|^2 \,dt \leq \int^{\tau}_0 \|E^n\|^2_{L^2} + \|B^n\|^2_{L^2} + \|j^n\|^2_{L^2}\,dt 			\leq C(T,\nu,\sigma,v_0,E_0,B_0),
		\end{align*}
		which implies that\footnote{As usual, for $s \in \mathbb{R}$ with $s > 0$, the space $H^{-s}(\mathbb{R}^2)$ can be considered as the dual space of $H^s(\mathbb{R}^2)$, see \cite{Bahouri-Chemin-Danchin_2011}.}
		\begin{equation*}
			(\partial_tv^n,\partial_t E^n,\partial_t B^n) \quad \text{is uniformly bounded in} \quad L^2(0,T;H^{-1}(K)).
		\end{equation*}
		Therefore, there exists a subsequence (still denoted by) $(v^n,E^n,B^n,j^n)$ and $(v,E,B,j)$ such that as $n \to \infty$
		\begin{align*}
			&&(v^n,E^n,B^n) &\overset{\ast}{\rightharpoonup} (v,E,B)  &&\text{in} \quad L^\infty(0,T;L^2(\mathbb{R}^2) \times H^s(\mathbb{R}^2) \times H^s(\mathbb{R}^2)), &&
			\\
			&&(v^n,j^n) &\rightharpoonup  (v,j)  &&\text{in} \quad L^2(0,T;H^1(\mathbb{R}^2) \times H^s(\mathbb{R}^2)), &&
			\\
			&&(\partial_tv^n,\partial_t E^n,\partial_t B^n) &\rightharpoonup  (\partial_t v,\partial_t E, \partial_t B) &&\text{in} \quad L^2(0,T;H^{-1}(K)). &&
		\end{align*}
		Recall that the injections $H^1 \hookrightarrow L^2 \hookrightarrow H^{-1}$ and $H^s \hookrightarrow L^2 \hookrightarrow H^{-1}$ for $s \in (0,1)$ are locally compact by using the Rellich–Kondrachov and Schauder theorems (see \cite{Brezis_2011,Leoni_2023}) then an application of the Aubin-Lions lemma (see \cite{Boyer-Fabrie_2013}) implies that as $n \to \infty$
		\begin{align*}
			(v^n,E^n,B^n) \to (v,E,B) \qquad \text{ in} \quad L^2(0,T;L^2_{\textnormal{loc}}(\mathbb{R}^2)).
		\end{align*}
		Furthermore, it can be seen from \eqref{NSM_app} that $(v^n,E^n,B^n)$ satisfies
		\begin{align*}
			a)\quad &\int^T_0 \int_{\mathbb{R}^2} v^n \cdot \partial_t\phi - \mathbb{P}(T_n(v^n \cdot \nabla v^n)) \cdot \phi - \nu \nabla v^n : \nabla \phi + \mathbb{P}(T_n(j^n \times B^n)) \cdot \phi \,dxdt = -\int_{\mathbb{R}^2} v^n(0) \cdot \phi(0) \,dx,
			\\
			b)\quad &\int^T_0 \int_{\mathbb{R}^2} \frac{1}{c}E^n \cdot \partial_t\varphi + B^n \cdot (\nabla \times \varphi) - j^n \cdot \varphi \,dxdt = -\int_{\mathbb{R}^2} \frac{1}{c}E^n(0) \cdot \varphi(0) \,dx,
			\\
			c)\quad &\int^T_0 \int_{\mathbb{R}^2} \frac{1}{c}B^n \cdot \partial_t\varphi - E^n \cdot (\nabla \times \varphi) \,dxdt  = -\int_{\mathbb{R}^2} \frac{1}{c}B^n(0) \cdot \varphi(0) \,dx,
		\end{align*}
		where $\phi,\varphi \in C^\infty_0([0,T) \times \mathbb{R}^2;\mathbb{R}^3)$ with $\textnormal{div}\,\phi = 0$. By using the above weak and strong convergences as $n \to \infty$, we can pass to the limit for the linear terms easily. It remains to check the convergence of the nonlinear terms. Moreover, we find that\footnote{Here, $v \otimes u := (v_iu_j)_{1\leq i,j\leq 3}$ for $v = (v_1,v_2,v_3)$ and $u = (u_1,u_2,u_3)$.}
		\begin{align*}
			\textnormal{NL}_1 &:= \left|\int^T_0 \int_{\mathbb{R}^2} \mathbb{P}(T_n(v^n \cdot \nabla v^n)- v \cdot \nabla v) \cdot \phi \,dxdt\right| 
			\\
			&\leq \left|\int^T_0 \int_{\mathbb{R}^2} (T_n(v^n \otimes v^n) - v^n \otimes v^n) : \nabla \phi \,dxdt\right|  
			+ \left|\int^T_0 \int_{\mathbb{R}^2} ((v^n-v) \otimes v^n) : \nabla \phi \,dxdt\right| 
			\\
			&\quad + \left|\int^T_0 \int_{\mathbb{R}^2} (v \otimes (v^n-v)) : \nabla \phi \,dxdt\right|
			\\
			&\leq \|T_n(v^n \otimes v^n) - v^n \otimes v^n\|_{L^2_tH^{-1}_x} \|\nabla \phi\|_{L^2_tH^1_x}
			+\|v^n-v\|_{L^2_{t,x}(\textnormal{supp}(\phi))} \|v^n\|_{L^\infty_tL^2_x}  \|\nabla \phi\|_{L^2_tL^\infty_x}
			\\
			&\quad + \|v^n-v\|_{L^2_{t,x}(\textnormal{supp}(\phi))} \|v\|_{L^\infty_tL^2_x} \|\nabla \phi\|_{L^2_tL^\infty_x} 
			\\
			& \to 0 \quad \text{as} \quad n \to \infty
		\end{align*}
		by using the strong convergence of $v^n$, and Steps 14 and 15 with
		\begin{equation*}
			\|T_n(v^n \otimes v^n) - v^n \otimes v^n\|_{L^2_tH^{-1}_x}  \leq 
			\frac{1}{n}\|v^n \otimes v^n\|_{L^2_tL^2_x} \leq \frac{1}{n} \|v^n\|_{L^\infty_tL^2_x} \|v^n\|_{L^2_tL^\infty_x}.
		\end{equation*}
		Similarly,
		\begin{align*}
			\textnormal{NL}_2 &:= \left|\int^T_0 \int_{\mathbb{R}^2} \mathbb{P}(T_n(j^n \times B^n) - j \times B) \cdot \phi \,dxdt\right| 
			\\
			&\leq \left|\int^T_0 \int_{\mathbb{R}^2} (T_n(j^n \times B^n) - j^n \times B^n) \cdot \phi \,dxdt\right|
			+ \left|\int^T_0 \int_{\mathbb{R}^2} (j \times (B^n-B)) \cdot \phi \,dxdt\right| 
			\\
			&\quad +  \left|\int^T_0 \int_{\mathbb{R}^2} ((j^n-j) \times B^n) \cdot \phi \,dxdt\right| 
			\\
			&\leq \|T_n(j^n \times B^n) - j^n \times B^n\|_{L^2_tH^{-1}_x}  \|\phi\|_{L^2_tH^1_x} +  \|B^n-B\|_{L^2_{t,x}(\textnormal{supp}(\phi))} \|j\|_{L^2_{t,x}} \|\phi\|_{L^\infty_{t,x}}
			\\
			&\quad + \|B^n-B\|_{L^2_{t,x}(\textnormal{supp}(\phi))} \|j^n-j\|_{L^2_{t,x}} \|\phi\|_{L^\infty_{t,x}} + \left|\int^T_0 \int_{\mathbb{R}^2} ((j^n-j) \times B) \cdot \phi \,dxdt\right| 
			\\
			& \to 0 \quad \text{as} \quad n \to \infty
		\end{align*}
		by using the strong and weak convergences of $B^n$ and $j^n$, Steps 14 and 15, and \eqref{So-paraproduct} with
		\begin{equation*}
			\|T_n(j^n \times B^n) - j^n \times B^n\|_{L^2_tH^{-1}_x}  \leq \frac{1}{n^{2s}} \|j^n \times B^n\|_{L^2_tH^{2s-1}_x} \leq \frac{1}{n^{2s}} \|B^n\|_{L^\infty_tH^s_x} \|j^n\|_{L^2_tH^s_x}.
		\end{equation*}
		Furthermore, we also use $j^n = \sigma(cE^n + T_n(v^n \times B^n))$ with
		\begin{align*}
			\textnormal{NL}_3 &:= \left|\int^T_0 \int_{\mathbb{R}^2} (T_n(v^n \times B^n) - v \times B) \cdot \varphi \,dxdt\right| 
			\\
			&\leq \left|\int^T_0 \int_{\mathbb{R}^2} (T_n(v^n \times B^n) - v^n \times B^n) \cdot \varphi \,dxdt\right|
			+ \left|\int^T_0 \int_{\mathbb{R}^2} ((v^n-v) \times B^n) \cdot \varphi \,dxdt\right| 
			\\
			&\quad + \left|\int^T_0 \int_{\mathbb{R}^2} (v \times (B^n-B)) \cdot \varphi \,dxdt\right| 
			\\
			&\leq \|T_n(v^n \times B^n) - v^n \times B^n\|_{L^2_tH^{-1}_x}  \|\varphi\|_{L^2_tH^1_x} +  \|v^n-v\|_{L^2_{t,x}(\textnormal{supp}(\varphi))} \|B^n\|_{L^\infty_tL^2_x} \|\varphi\|_{L^2_tL^\infty_x}
			\\
			&\quad + \|B^n-B\|_{L^2_{t,x}(\textnormal{supp}(\varphi))} \|v\|_{L^\infty_tL^2_x}  \|\varphi\|_{L^2_tL^\infty_x} 
			\\
			&\to 0 \quad \text{as} \quad n \to \infty.
		\end{align*}
		It can be seen that $\textnormal{div}\, v = \textnormal{div}\, B = 0$ in the sense of distributions and $(v,E,B)_{|_{t=0}} = (v_0,E_0,B_0)$. It shows that $(v,E,B,j)$ satisfies \eqref{NSM_P} in the sense of distributions (similar to those of $a)$, $b)$ and $c)$ without $n$) with $(v,E,B)_{|_{t=0}} = (v_0,E_0,B_0)$. In addition, $v \in C([0,T];L^2)$, $(E,B)$ is weak continuous in time with values in $H^s$, and $(v,E,B)$ shares the same bounds in the case either $\delta \in [0,1]$ or $s \in (0,1]$ as that of $(v^n,E^n,B^n)$ given in Step 15. As mentioned in the previous case, a scalar pressure $\pi$ can be recovered such that $(v,E,B,\pi)$ satisfies \eqref{NSM} in the sense of distributions.

		\textbf{Step 18: Uniqueness.} Although the uniqueness has been considered in \cite{Masmoudi_2010} 
		with a different functional space and has not been mentioned in \cite{Arsenio-Gallagher_2020}, we adapt the idea in \cite{Masmoudi_2010} by providing a slightly different proof, which will take the advantage of the bound $\|v\|_{L^2_tL^\infty_x}$ given in Step 14 compared to \cite{Masmoudi_2010}. We note that our modified proof can also be useful in the three-dimensional case in Theorem \ref{theo1-3d}. Assume that $(v,E,B,j,\pi)$ and $(\bar{v},\bar{E},\bar{B},\bar{j},\bar{\pi})$ are two solutions to \eqref{NSM} with the same initial data $(v_0,E_0,B_0) \in L^2 \times H^s \times H^s$ for $s \in (0,1)$. It is worth mentioning that we can not use the usual energy method here due to the lack of smoothness of $(E,B)$. It can be seen that the difference satisfies
		\begin{equation*} 
			\left\{
			\begin{aligned}
				\partial_t (v-\bar{v}) + (v-\bar{v}) \cdot \nabla v + \bar{v} \cdot \nabla (v-\bar{v}) + \nabla (\pi - \bar{\pi}) &=  \nu\Delta(v-\bar{v}) + (j-\bar{j}) \times B + \bar{j} \times (B-\bar{B}), 
				\\
				\frac{1}{c}\partial_t (E-\bar{E}) - \nabla \times (B-\bar{B}) &= -(j-\bar{j}),
				\\
				\frac{1}{c}\partial_t (B-\bar{B}) + \nabla \times (E-\bar{E}) &= 0,
				\\
				\sigma(cE + v \times B) &= j,
				\\
				\sigma(c\bar{E} + \bar{v} \times \bar{B}) &= \bar{j},
				\\
				\frac{1}{\sigma}(j-\bar{j}) - (v \times B) + (\bar{v} \times \bar{B}) &= c(E-\bar{E}) ,
				\\
				\textnormal{div}\, v = \textnormal{div}\, B = \textnormal{div}\, \bar{v} = \textnormal{div}\, \bar{B} &= 0,
			\end{aligned}
			\right.
		\end{equation*}
		which gives us by using the continuity in time of both $v$ and $\bar{v}$ that
		\begin{equation*}
			\frac{1}{2}\frac{d}{dt}\|v-\bar{v}\|^2_{L^2} + \nu \|v-\bar{v}\|^2_{\dot{H}^1} =: \sum^3_{k=1} \bar{I}_k,
		\end{equation*}
		where for some $\epsilon \in (0,1)$ and for any $s' \in (0,s]$ 
		\begin{align*}
			\bar{I}_1 &= - \int_{\mathbb{R}^2} (v-\bar{v}) \cdot \nabla v \cdot (v-\bar{v}) \,dx  \leq  \epsilon\nu \|v-\bar{v}\|^2_{\dot{H}^1} +  C(\epsilon,\nu)\|\nabla v\|^2_{L^2}\|v-\bar{v}\|^2_{L^2}; 
			\\
			\bar{I}_2 &= \int_{\mathbb{R}^2} (j-\bar{j}) \times B \cdot (v-\bar{v}) \,dx 
			\leq C(s')\|j-\bar{j}\|_{L^2}\|B\|_{\dot{H}^{s'}}\|v-\bar{v}\|_{\dot{H}^{1-s'}}
			\\
			&\leq C(s')\|j-\bar{j}\|_{L^2}\|B\|_{\dot{H}^{s'}}\|v-\bar{v}\|^{s'}_{L^2}\|v-\bar{v}\|^{1-s'}_{\dot{H}^1}
			\\
			&\leq \epsilon\nu \|v-\bar{v}\|^2_{\dot{H}^1} + C(\epsilon,\nu,s') \|B\|^\frac{2}{s'+1}_{\dot{H}^{s'}} \|j-\bar{j}\|^\frac{2}{s'+1}_{L^2} \|v-\bar{v}\|^\frac{2s'}{s'+1}_{L^2};
			\\
			\bar{I}_3 &= \int_{\mathbb{R}^2} \bar{j} \times (B-\bar{B}) \cdot (v-\bar{v})\,dx \leq C(s')\|\bar{j}\|_{L^2}\|B-\bar{B}\|_{\dot{H}^{s'}}\|v-\bar{v}\|_{\dot{H}^{1-s'}}
			\\
			&\leq C(s')\|\bar{j}\|_{L^2}\|B-\bar{B}\|_{\dot{H}^{s'}}\|v-\bar{v}\|^{s'}_{L^2}\|v-\bar{v}\|^{1-s'}_{\dot{H}^1}
			\\
			&\leq \epsilon\nu \|v-\bar{v}\|^2_{\dot{H}^1} +  C(\epsilon,\nu,s')\|\bar{j}\|^\frac{2}{s'+1}_{L^2}\left(\|B-\bar{B}\|^2_{\dot{H}^{s'}} + \|v-\bar{v}\|^2_{L^2}\right).
		\end{align*}
		Therefore, by choosing $\epsilon = \frac{1}{6}$ and taking $T_* \in (0,T]$
		\begin{equation*}
			\|v-\bar{v}\|^2_{L^\infty(0,T_*;L^2)} + \nu \int^{T_*}_0  \|v-\bar{v}\|^2_{\dot{H}^1} \,d\tau \,d\tau \leq \sum^3_{k=1} \bar{J}_k,
		\end{equation*}
		where 
		\begin{align*}
			\bar{J}_1 &:= C(\nu) \int^{T_*}_0 \|\nabla v\|^2_{L^2}\|v-\bar{v}\|^2_{L^2} \,d\tau \leq C(\nu) \|v\|^2_{L^2(0,T_*;\dot{H^1})} \|v-\bar{v}\|^2_{L^\infty(0,T_*;L^2)};
			\\
			\bar{J}_2 &:= C(\nu,s') \int^{T_*}_0  \|B\|^\frac{2}{s'+1}_{\dot{H}^{s'}} \|j-\bar{j}\|^\frac{2}{s'+1}_{L^2} \|v-\bar{v}\|^\frac{2s'}{s'+1}_{L^2} \,d\tau
			\\
			&\leq C(c,\nu,\sigma,s') \int^{T_*}_0  \|B\|^\frac{2}{s'+1}_{\dot{H}^{s'}} \left(\|E-\bar{E}\|^\frac{2}{s'+1}_{L^2} + \|(v-\bar{v}) \times B\|^\frac{2}{s'+1}_{L^2} + \|\bar{v} \times (B-\bar{B})\|^\frac{2}{s'+1}_{L^2}\right) \|v-\bar{v}\|^\frac{2s'}{s'+1}_{L^2} \,d\tau
			\\
			&=: \sum^3_{k=1} \bar{J}_{2k},
			\\
			\bar{J}_{21} &= C(c,\nu,\sigma,s') \int^{T_*}_0  \|B\|^\frac{2}{s'+1}_{\dot{H}^{s'}} \|E-\bar{E}\|^\frac{2}{s'+1}_{L^2}  \|v-\bar{v}\|^\frac{2s'}{s'+1}_{L^2} \,d\tau
			\\
			&\leq C(c,\nu,\sigma,s')T_*\|B\|^\frac{2}{s'+1}_{L^\infty(0,T_*;\dot{H}^{s'})}\left(\|E-\bar{E}\|^2_{L^\infty(0,T_*;L^2)} + \|v-\bar{v}\|^2_{L^\infty(0,T_*;L^2)}\right),
			\\
			\bar{J}_{22} &= C(c,\nu,\sigma,s') \int^{T_*}_0  \|B\|^\frac{2}{s'+1}_{\dot{H}^{s'}}  \|(v-\bar{v}) \times B\|^\frac{2}{s'+1}_{L^2}  \|v-\bar{v}\|^\frac{2s'}{s'+1}_{L^2} \,d\tau
			\\
			&\leq C(c,\nu,\sigma,s') \int^{T_*}_0  \|B\|^\frac{2}{s'+1}_{\dot{H}^{s'}}  \|v-\bar{v}\|^\frac{2}{s'+1}_{L^\frac{2}{s'}}  \|B\|^\frac{2}{s'+1}_{L^\frac{2}{1-s'}}  \|v-\bar{v}\|^\frac{2s'}{s'+1}_{L^2} \,d\tau
			\\
			&\leq C(c,\nu,\sigma,s') \int^{T_*}_0  \|B\|^\frac{2}{s'+1}_{\dot{H}^{s'}}  \|v-\bar{v}\|^\frac{2}{s'+1}_{\dot{H}^{1-s'}}  \|B\|^\frac{2}{s'+1}_{\dot{H}^{s'}}  \|v-\bar{v}\|^\frac{2s'}{s'+1}_{L^2} \,d\tau
			\\
			&\leq C(c,\nu,\sigma,s') \int^{T_*}_0  \|B\|^\frac{4}{s'+1}_{\dot{H}^{s'}}  \|v-\bar{v}\|^\frac{4s'}{s'+1}_{L^2} \|v-\bar{v}\|^\frac{2(1-s')}{s'+1}_{\dot{H}^1}   \,d\tau
			\\
			&\leq C(c,\nu,\sigma,s') T^\frac{2s'}{s'+1}_* \|B\|^\frac{4}{s'+1}_{L^\infty(0,T_*;\dot{H}^{s'})}\left(\|v-\bar{v}\|^2_{L^\infty(0,T_*;L^2)} +  \nu\|v-\bar{v}\|^2_{L^2(0,T_*\dot{H}^1)}\right),
			\\
			\bar{J}_{23} &= C(c,\nu,\sigma,s') \int^{T_*}_0  \|B\|^\frac{2}{s'+1}_{\dot{H}^{s'}}  \|\bar{v} \times (B-\bar{B})\|^\frac{2}{s'+1}_{L^2} \|v-\bar{v}\|^\frac{2s'}{s'+1}_{L^2} \,d\tau
			\\
			&\leq  C(c,\nu,\sigma,s') \int^{T_*}_0  \|B\|^\frac{2}{s'+1}_{\dot{H}^{s'}}  \|\bar{v}\|^\frac{2s'}{s'+1}_{L^2}\|\bar{v}\|^\frac{2(1-s')}{s'+1}_{\dot{H}^1} \|B-\bar{B}\|^\frac{2}{s'+1}_{\dot{H}^{s'}} \|v-\bar{v}\|^\frac{2s'}{s'+1}_{L^2} \,d\tau
			\\
			&\leq C(c,\nu,\sigma,s')T^\frac{2s'}{s'+1}_*\|B\|^\frac{2}{s'+1}_{L^\infty(0,T_*;\dot{H}^{s'})}\|\bar{v}\|^\frac{2s'}{s'+1}_{L^\infty(0,T_*;L^2)} \|\bar{v}\|^\frac{2(1-s')}{s'+1}_{L^2(0,T_*;\dot{H}^1)}
			\\
			&\quad \times\left(\|B-\bar{B}\|^2_{L^\infty(0,T_*;\dot{H}^{s'})} + \|v-\bar{v}\|^2_{L^\infty(0,T_*;L^2)}\right);
			\\
			\bar{J}_3 &:= C(\nu,s') \int^{T_*}_0 \|\bar{j}\|^\frac{2}{s'+1}_{L^2}\left(\|B-\bar{B}\|^2_{\dot{H}^{s'}} + \|v-\bar{v}\|^2_{L^2}\right) \,d\tau 
			\\
			&\leq C(\nu,s')T^\frac{s'}{s'+1}_* \|\bar{j}\|^\frac{2}{s'+1}_{L^2(0,T_*;L^2)} \left(\|B-\bar{B}\|^2_{L^\infty(0,T_*;\dot{H}^{s'})} + \|v-\bar{v}\|^2_{L^\infty(0,T_*;L^2)}\right).
		\end{align*}
		In addition, by using Lemma \ref{lem_M0}, it follows that
		\begin{align*}
			\|(E-\bar{E},B-\bar{B})\|^2_{L^\infty(0,T_*;H^{s'})} \leq C(c)\|j-\bar{j}\|^2_{L^1(0,T_*;H^{s'})} 
			=: \sum^6_{k=4}\bar{J}_k,
		\end{align*}
		where for any $s' \in (0,s)$
		\begin{align*}
			\bar{J}_4 &= C(c,\sigma)\|E-\bar{E}\|^2_{L^1(0,T_*;H^{s'})} \leq  C(c,\sigma)T^2_*\|E-\bar{E}\|^2_{L^\infty(0,T_*;H^{s'})};
			\\
			\bar{J}_5 &=  C(c,\sigma)\|(v-\bar{v}) \times B\|^2_{L^1(0,T_*;H^{s'})} \leq \bar{J}_{51} + \bar{J}_{52},
			\\
			\bar{J}_{51} &:= C(c,\sigma) \|(v-\bar{v}) \times B\|^2_{L^1(0,T_*;L^2)} 
			\\
			&\leq C(c,\sigma,s')T^{s'+1}_*\|B\|^2_{L^\infty(0,T_*;\dot{H}^{s'})}\|v-\bar{v}\|^{2s'}_{L^\infty(0,T_*;L^2)}\|v-\bar{v}\|^{2(1-s')}_{L^2(0,T_*;\dot{H}^1)},
			\\
			&\leq C(c,\nu,\sigma,s')T^{s'+1}_*\|B\|^2_{L^\infty(0,T_*;\dot{H}^{s'})} \left(\|v-\bar{v}\|^2_{L^\infty(0,T_*;L^2)} + \nu\|v-\bar{v}\|^2_{L^2(0,T_*;\dot{H}^1)}\right),
			\\
			\bar{J}_{52} &:= C(c,\sigma) \|(v-\bar{v}) \times B\|^2_{L^1(0,T_*;\dot{H}^{s'})} 
			\\
			&\leq C(c,\nu,\sigma,s')T_*\|B\|^2_{L^\infty(0,T_*;\dot{H}^{s'})} \nu\|v-\bar{v}\|^2_{L^2(0,T_*;\dot{H}^1)} 
			\\
			&\quad+ C(c,\sigma,s,s')T^{s-s'+1}_* \|B\|^2_{L^\infty(0,T_*;\dot{H}^s)}\|v-\bar{v}\|^{2(s-s')}_{L^\infty(0,T_*;L^2)}\|v-\bar{v}\|^{2(1-(s-s'))}_{L^2(0,T_*;\dot{H}^1)}
			\\
			&\leq C(c,\nu,\sigma,s')T_*\|B\|^2_{L^\infty(0,T_*;\dot{H}^{s'})} \nu\|v-\bar{v}\|^2_{L^2(0,T_*;\dot{H}^1)} 
			\\
			&\quad+ C(c,\nu,\sigma,s,s')T^{s-s'+1}_* \|B\|^2_{L^\infty(0,T_*;\dot{H}^s)} \left(\|v-\bar{v}\|^2_{L^\infty(0,T_*;L^2)} +  \nu\|v-\bar{v}\|^2_{L^2(0,T_*;\dot{H}^1)}\right);
			\\
			\bar{J}_6 &= C(c,\sigma)\|\bar{v} \times (B-\bar{B})\|^2_{L^1(0,T_*;H^{s'})} \leq \bar{J}_{61} + \bar{J}_{62},
			\\
			\bar{J}_{61} &:= C(c,\sigma)\|\bar{v} \times (B-\bar{B})\|^2_{L^1(0,T_*;L^2)}
			\\
			&\leq C(c,\sigma,s') T^{s'+1}_* \|\bar{v}\|^{2s'}_{L^\infty(0,T_*;L^2)} \|\bar{v}\|^{2(1-s')}_{L^2(0,T_*;\dot{H}^1)}  \|B-\bar{B}\|^2_{L^\infty(0,T_*;\dot{H}^{s'})},
			\\
			\bar{J}_{62} &:= C(c,\sigma)\|\bar{v} \times (B-\bar{B})\|^2_{L^1(0,T_*;\dot{H}^{s'})}
			\\
			&\leq C(c,\sigma,s')T_* \left(\|\bar{v}\|^2_{L^2(0,T_*;\dot{H}^1)} + \|\bar{v}\|^2_{L^2(0,T_*;L^\infty)}\right) \|B-\bar{B}\|^2_{L^\infty(0,T_*;H^{s'})}.
		\end{align*}
		Combining all the above estimates and using Step 18, we find that for sufficiently small $T_*$ (depending on the parameters and initial data)
		\begin{align*}
			A(v-\bar{v},E-\bar{E},B-\bar{B}) &:= \|v-\bar{v}\|^2_{L^\infty(0,T_*;L^2)} + \nu \|v-\bar{v}\|^2_{L^2(0,T_*;\dot{H}^1)} +	\|(E-\bar{E},B-\bar{B})\|^2_{L^\infty(0,T_*;H^{s'})} \\
			&\leq \frac{1}{2} A(v-\bar{v},E-\bar{E},B-\bar{B}),
		\end{align*}
		which yields $v = \bar{v},E = \bar{E}$ and $B = \bar{B}$ in $(0,T_*)$. By repeating this process, we obtain the conclusion in the whole time interval $(0,T)$. Finally, we note that only the estimate of $\bar{J}_{52}$ needs $s' < s$ and other ones hold for $s' = s$ as well.
	\end{proof}

	\begin{proof}[Proof of Theorem \ref{theo1}-(ii) and (iii)] 
		In this part, by applying the previous one, we obtain more regularity for $(v,E,B)$.
		
		\textbf{Step 19: Higher regularity.} In the case $\delta = 0$ and $s \in (0,1)$, an application of Step 14, which allows us to bound $\|v^n\|_{L^2_tL^\infty_x}$ and $\|(E^n,B^n)\|_{L^\infty_tH^s_x}$. In addition, it follows from the energy estimate that  $v^n(t') \in H^1$ for a.e $t' \in (0,T)$ and for any $T \in (0,\infty)$. Thus, for any $t' \in (0,T)$ there exists $t_* \in (0,t')$ such that $v^n(t_*) \in H^1$. By fixing $t_*$, we define for $t \in [0,T-t_*)$, $u^n(t) := v^n(t_*+t)$ with $u^n_{|_{t=0}} := T_n(v^n(t_*)) \in H^{\delta'}$ for $\delta' \in (0,1]$ and consider \eqref{NSM_app} in $\mathbb{R}^2 \times (0,T-t_*)$ by replacing $u^n$ by $v^n$ with the initial data $(u^n,E^n,B^n)_{|_{t=0}} \in H^{\delta'} \times H^s \times H^s$ for $\delta' \in (0,1]$ and $s \in (0,1)$. An application of Steps 1, 5 and 17b in the proof of Part $(i)$, which allows us to bound $\|u^n\|_{L^\infty_tH^{\delta'}_x \cap L^2_tH^{\delta'+1}_x}$ and pass to the limit as $n \to \infty$. Furthermore, a similar argument can be applied to the case where $\delta = 0$ and $s = 1$ by using in addition Steps 2, 4 and 17b in the proof of Part $(i)$, we skip further details. We note that this step can be applied to other cases to gain more regularity for $(v,E,B)$ after the initial time, but we will not investigate here.  Thus, the proof of this part is complete.
	\end{proof}

	%
	\section{Proof of Theorem \ref{theo1-3d}} \label{sec:theo1-3d}
	%
	
	In this section, we will provide a standard proof of Theorem \ref{theo1-3d}, which follows the idea as that of Theorem \ref{theo1}.
	
	\begin{proof}[Proof of Theorem \ref{theo1-3d}-$(i)$] The proof is divided into several steps as follows.
		
		\textbf{Step 1: Local existence.} We will use exactly the approximate system \eqref{NSM_app} with replacing $\nu \Delta v^n$ by $-\nu(-\Delta)^\frac{3}{2} v^n$. Then there exists a unique solution $(v^n,E^n,B^n) \in C^1([0,T^n_*);V^\delta_n \times H^s_n \times V^s_n)$ for some $T^n_* > 0$. In what follows, we will assume that $T^n_* < \infty$. In the sequel, it is sufficient to focus on the case $\alpha = \frac{3}{2}$. The case $\alpha > \frac{3}{2}$ can be done similarly for more general initial data. 
		 
		\textbf{Step 2: The case $\delta = 0$ and $s \in (0,\frac{3}{2})$.} Similar to the two-dimensional case, the  energy estimate is given for $t \in (0,T^n_*)$ by
		\begin{equation*}
			\|(v^n,E^n,B^n)(t)\|^2_{L^2} + \int^t_0  \nu\|v^n\|^2_{\dot{H}^\frac{3}{2}} + \frac{1}{\sigma}\|j^n\|^2_{L^2} \,d\tau \leq \|(v_0,E_0,B_0)\|^2_{L^2} =: \mathcal{E}^2_0.
		\end{equation*}
		In addition, the $\dot{H}^s$ estimate of $(E^n,B^n)$ reads
		\begin{equation*}
			\frac{1}{2}\frac{d}{dt}\|(E^n,B^n)\|^2_{\dot{H}^s} + \frac{1}{\sigma} \|j^n\|^2_{\dot{H}^s} = \int_{\mathbb{R}^3} \Lambda^s(v^n \times B^n) \cdot \Lambda^s j^n \,dx,
		\end{equation*}
		where the right-hand side can be bounded as follows
		\begin{align*}
			\textnormal{RHS} &\leq C(s)\left(\|\Lambda^s v^n\|_{L^\frac{6}{2s}}\|B^n\|_{L^\frac{6}{3-2s}} + \|v^n\|_{L^\infty}\|\Lambda^s B^n\|_{L^2}\right)\|j^n\|_{\dot{H}^s}
			\\
			&\leq \frac{\epsilon}{\sigma}\|j^n\|^2_{\dot{H}^s} + C(\epsilon,\sigma,s)\left(\|v^n\|^2_{\dot{H}^\frac{3}{2}} + \|v^n\|^2_{L^\infty}\right)\|B^n\|^2_{\dot{H}^s},
		\end{align*}
		here we used the following homogeneous Sobolev inequality (see \cite{Bahouri-Chemin-Danchin_2011})
		\begin{equation*}
			\|f\|_{L^{p_0}} \leq C(s_0,p_0)\|f\|_{\dot{H}^{s_0}} \qquad \text{for} \quad s_0 \in \left[0,\frac{3}{2}\right), p_0 = \frac{6}{3-2s_0}.
		\end{equation*}
		Therefore, by choosing $\epsilon = \frac{1}{2}$, we obtain for $0 \leq t_0 < t \leq T^n_*$
		\begin{equation*}
			\|(E^n,B^n)(t)\|^2_{\dot{H}^s} \leq \|(E^n,B^n)(t_0)\|^2_{\dot{H}^s} \exp\left\{C(\sigma,s)\int^t_{t_0} \|v^n\|^2_{\dot{H}^\frac{3}{2}} + \|v^n\|^2_{L^\infty} \,d\tau \right\}.
		\end{equation*}
		It remains to control $\|v^n\|_{L^2_tL^\infty_x}$.  Defining $v^1$ and $v^2$ (we only write $(v^1,v^2)$ instead of $(v^{n,1},v^{n,2})$ for simplicity) as Step 14 in the proof of Theorem \ref{theo1} with
		\begin{align*}
			&&\partial_t v^1 + \nu (-\Delta)^\frac{3}{2} v^1 
			&= -\mathbb{P}(T_n(v^n \cdot \nabla v^n)),
			&&\textnormal{div}\, v^1 = 0,
			&&v^1_{|_{t=0}} = T_n(v_0),&&
			\\
			&&\partial_t v^2 + \nu (-\Delta)^\frac{3}{2} v^2 
			&= \mathbb{P}(T_n(j^n \times B^n)),
			&&\textnormal{div}\, v^2 = 0,
			&&v^2_{|_{t=0}} = 0.&&
		\end{align*}
		It follows that for $t \in (0,T^n_*)$
		\begin{equation*}
			\|v^1(t)\|^2_{L^2} + \nu\int^t_0 \|v^1\|^2_{\dot{H}^\frac{3}{2}} \,d\tau \leq \|v^1(0)\|^2_{L^2} + C(\nu)t^\frac{1}{3} \|v^n\|^\frac{8}{3}_{L^\infty(0,t;L^2)} \|v^n\|^\frac{4}{3}_{L^2(0,t;\dot{H}^\frac{3}{2})}  \leq C(\mathcal{E}_0,\nu)\left(1 +  (T^n_*)^\frac{1}{3}\right).
		\end{equation*}
		Moreover, for $t \in (0,T^n_*)$ and $q \in \mathbb{Z}$
		\begin{equation*}
			\frac{1}{2}\frac{d}{dt}\|\Delta_q v^1\|^2_{L^2} + \nu \|\Delta_q (-\Delta)^\frac{3}{4} v^1\|^2_{L^2} 
			\leq \|\Delta_q v^1\|_{L^2}\|\Delta_q (v^n \cdot \nabla v^n)\|_{L^2},
		\end{equation*}
		and by using for $q \geq 0$ 
		\begin{equation*}
			\|\Delta_q (-\Delta)^\frac{3}{4} v^1\|^2_{L^2} 
			\geq C 2^{3q}\|\Delta_q v^1\|^2_{L^2},
		\end{equation*}
		which yields (we use the same notation as the two-dimensional case)
		\begin{equation*}
			\sum_{q \geq -1} \esssup_{t \in (0,T^n_*)} \|\Delta_q v^1(t)\|^2_{L^2} + C(\nu) \sum_{q \geq 0} \left(\int^{T^n_*}_0 2^{3q}\|\Delta_q v^1\|_{L^2} \,d\tau\right)^2
			\leq  2\sum_{q \geq -1} R^2_q,
		\end{equation*}
		and
		\begin{align*}
			\sum_{q \geq -1} R^2_q 
			\leq C\left(\|v^1\|^2_{L^2_tL^\infty_x} + \|v^2\|^2_{L^2_tL^\infty_x}\right) \|\nabla v\|^2_{L^2_tL^2_x} + C\|v_0\|^2_{L^2}.
		\end{align*}
		The Littlewood–Paley decomposition and Bernstein-type estimate give us
		\begin{align*}
			\int^{T^n_*}_0 \|v^1\|^2_{L^\infty} \,d\tau 
			\leq C \int^{T^n_*}_0 \left(\sum_{|q-k| \leq N} + \sum_{|q-k| > N}\right) 2^\frac{3q}{2} \|\Delta_q v^1\|_{L^2}   2^\frac{3k}{2} \|\Delta_k v^1\|_{L^2} \,d\tau =: \bar{R}_1 + \bar{R}_2.
		\end{align*}
		The terms on the right-hand side can be bounded as follows
		\begin{align*}
			\bar{R}_1 &= C\int^{T^n_*}_0 \sum_{q \geq -1} 2^\frac{3q}{2} \|\Delta_q v^1\|_{L^2}   \sum_{q-N \leq k \leq q+N} 2^\frac{3k}{2} \|\Delta_k v^1\|_{L^2} \,d\tau
			\\
			&\leq C\int^{T^n_*}_0 \sum_{q \geq -1}  \left(2^{3q}\|\Delta_q v^1\|^2_{L^2}  + \sum_{q-N \leq k \leq q+N} 2^{3k} \|\Delta_k v^1\|^2_{L^2} \right) \,d\tau
			\\
			&\leq CN \left(\|v^1\|^2_{L^2_t\dot{H}^\frac{3}{2}_x} + \|v^1\|^2_{L^2_tL^2_x}\right);
			\\
			\bar{R}_2 &= C2^{1-\frac{3}{2}N} \sum^\infty_{q = N} \int^{T^n_*}_0  2^{3q} \|\Delta_q v^1\|_{L^2}  \sum_{k = -1} ^{q-N-1} 2^{\frac{3}{2}(k-(q-N))} \|\Delta_k v^1\|_{L^2} \,d\tau
			\\
			&\leq C2^{1-\frac{3}{2}N} \left(\sum^\infty_{q = N} \left(\int^{T^n_*}_0  2^{3q} \|\Delta_q v^1\|_{L^2} \,d\tau \right)^2\right)^\frac{1}{2}   \left(\sum^\infty_{k_0=0} \left(\sum_{k = -1} ^{k_0-1} 2^{-\frac{3}{2}(k_0-k)} \esssup_{t \in (0,{T^n_*})} \|\Delta_k v^1\|_{L^2} \right)^2\right)^\frac{1}{2}
			\\
			&\leq C2^{1-\frac{3}{2}N} \left(\sum^\infty_{q = 0} \left(\int^{T^n_*}_0  2^{3q} \|\Delta_q v^1\|_{L^2} \,d\tau \right)^2\right)^\frac{1}{2}   \left(\sum^\infty_{k=-1} \esssup_{t \in (0,{T^n_*})} \|\Delta_k v^1\|^2_{L^2} \right)^\frac{1}{2} \sum^\infty_{k = -1} 2^{-\frac{3}{2}k}.
		\end{align*}
		Therefore,
		\begin{align*}
			\|v^1\|^2_{L^2_tL^\infty_x}  &\leq C(\nu) 2^{-N} \left(\|v^1\|^2_{L^2_tL^\infty_x} + \|v^2\|^2_{L^2_tL^\infty_x}\right) \|\nabla v^n\|^2_{L^2_tL^2_x} 
			+ CN \left(\|v^1\|^2_{L^2_t\dot{H}^\frac{3}{2}_x} + \|v^1\|^2_{L^2_tL^2_x}\right)
			+ C \|v_0\|^2_{L^2},
		\end{align*}
		which by choosing 
		\begin{equation*}
			N = \left\lceil\log_2\left(4 + 2C(\nu)\|\nabla v^n\|^2_{L^2_tL^2_x}\right)\right\rceil
		\end{equation*}
		and using the energy estimates of $v^1$ and $v^n$ yields
		\begin{align*}
			\|v^1\|^2_{L^2_tL^\infty_x} &\leq 
			C(\nu)\left(\left(\|v^1\|^2_{L^2_t\dot{H}^\frac{3}{2}_x} + \|v^1\|^2_{L^2_tL^2_x}\right)(1+\|\nabla v^n\|^2_{L^2_tL^2_x}) + \|v^2\|^2_{L^2_tL^\infty_x} + \mathcal{E}^2_0\right)
			\\
			&\leq C(\mathcal{E}_0,\nu)\left(1 + (T^n_*)^\frac{1}{3} + (T^n_*)^\frac{5}{3}\right) + C(\nu)\|v^2\|^2_{L^2_tL^\infty_x}.
		\end{align*}
		In addition,
		\begin{equation*}
			\|v\|^2_{L^2_tL^\infty_x} \leq \|v^1\|^2_{L^2_tL^\infty_x} + \|v^2\|^2_{L^2_tL^\infty_x} \leq
			C(\mathcal{E}_0,\nu)\left(1 + (T^n_*)^\frac{1}{3} + (T^n_*)^\frac{5}{3}\right) + C(\nu)\|v^2\|^2_{L^2_tL^\infty_x}.
		\end{equation*}
		It remains to bound the term $\|v^2\|_{L^2_tL^\infty_x}$. It can be seen that for $0 \leq t_0 < t \leq T^n_*$
		\begin{align*}
			\|v^2\|^2_{L^\infty(t_0,t;L^2)} + \|v^2\|^2_{L^2(t_0,t;\dot{H}^\frac{3}{2})} 
			&\leq \|(v^n,v^1)\|^2_{L^\infty(t_0,t;L^2)} 
			+ \|(v^n,v^1)\|^2_{L^2(t_0,t;\dot{H}^\frac{3}{2})}  \leq C(\mathcal{E}_0,\nu)\left(1 + (T^n_*)^\frac{1}{3}\right).
		\end{align*}
		We will use the following Besov-type maximal regularity estimate for the forced fractional heat equation of $v^2$ in which its proof will be provided in Appendix D (see Section \ref{sec:app}) for $t \in (0,T^n_*]$
		\begin{equation} \label{Maximal_Regularity}
			\|v^2\|_{L^2(0,t;\dot{B}^{s+\frac{3}{2}}_{2,1})} \leq C(s)\|j^n \times B^n\|_{L^2(0,t;\dot{B}^{s-\frac{3}{2}}_{2,1})} \leq C(s) \|j^n\|_{L^2(0,t;L^2)}\|B^n\|_{L^\infty(0,t;\dot{H}^s)},
		\end{equation}
		where we used the paraproduct rule \eqref{paraproduct} in the second inequality. An application of \cite[Lemmas 7.3 and 7.4]{Arsenio-Gallagher_2020} yields for $s' > \frac{3}{2}$ and $0 \leq t_0 < t \leq T^n_*$
		\begin{align*}
			\|(\textnormal{Id}-\dot{S}_0)v^2\|_{L^2(t_0,t;L^\infty)} 
			&\leq C(s')\|v^2\|_{L^2(t_0,t;\dot{H}^\frac{3}{2})} \log^\frac{1}{2}\left(e + \frac{\|v^2\|_{L^2(t_0,t;\dot{B}^{s'}_{2,1})}}{\|v^2\|_{L^2(t_0,t;\dot{H}^\frac{3}{2})}}\right),
			\\
			\|\dot{S}_0v^2\|_{L^2(t_0,t;L^\infty)} 
			&\leq C\|v^2\|_{L^2(t_0,t;\dot{H}^\frac{3}{2})} \log^\frac{1}{2}\left(e + \frac{\|v^2\|_{L^2(t_0,t;L^2)}}{\|v^2\|_{L^2(t_0,t;\dot{H}^\frac{3}{2})}}\right).
		\end{align*}
		By choosing $s' = s+\frac{3}{2}$, using \eqref{Maximal_Regularity}, the estimate of $v^2$ and the increasing of the function $z \mapsto z\log(e+\frac{C}{z})$ for $z > 0$, we find that
		\begin{align*}
			\|v^2\|^2_{L^2(t_0,t;L^\infty)}  
			\leq C(\mathcal{E}_0,\nu,s)\left(1 + (T^n_*)^\frac{1}{3}\right) \log(e + t-t_0) + 
			C(\sigma,s)\|v^2\|^2_{L^2(t_0,t;\dot{H}^\frac{3}{2})} \log\left(e + \frac{\mathcal{E}^2_0\|B^n\|^2_{L^2(t_0,t;\dot{H}^s)}}{\|v^2\|^2_{L^2(t_0,t;\dot{H}^\frac{3}{2})}}\right).
		\end{align*}
		By using an upper bound on $\|v\|_{L^2_tL^\infty_x}$ in terms of $\|v^2\|_{L^2_tL^\infty_x}$ and that of
		$\|v^2\|_{L^2_t\dot{H}^\frac{3}{2}_x}$, 
		an iteration as \cite[The proof of Theorem 1.2]{Arsenio-Gallagher_2020} can be applied (replacing $\|u_e\|_{L^2_t\dot{H}^1_x}$ by $\|v^2\|_{L^2_t\dot{H}^\frac{3}{2}_x}$) to the $\dot{H}^s$ estimate of $(E^n,B^n)$, which yields for $t \in (0,T^n_*]$
		\begin{align*}
			\mathcal{E}^2_0\|(E^n,B^n)\|^2_{L^\infty(0,t;\dot{H}^s)} &\leq \left(e + \mathcal{E}^2_0\|(E_0,B_0)\|^2_{\dot{H}^s} +
			t\right)^{C(\mathcal{E}_0,\nu,\sigma,s)\left(1 + (T^n_*)^\frac{1}{3} + (T^n_*)^\frac{5}{3}\right)},
		\end{align*}
		and for some $C(\mathcal{E}_0,\nu,\sigma,s) > 1$ 
		\begin{equation*}
			\|v^n\|^2_{L^2(0,t;L^\infty)} \leq C(\mathcal{E}_0,\nu,\sigma,s)\left(1 + (T^n_*)^\frac{1}{3} + (T^n_*)^\frac{5}{3}\right)\left(e + \mathcal{E}^2_0\|(E_0,B_0)\|^2_{\dot{H}^s} + t\right)^{C(\mathcal{E}_0,\nu,\sigma,s)\left(1 + (T^n_*)^\frac{1}{3} + (T^n_*)^\frac{5}{3}\right)}.
		\end{equation*}
		
		\textbf{Step 3: The case $\delta = s$.} This step includes three substeps as follows.
		
		\textbf{Step 3a: The case $\delta = s \in (0,\frac{3}{2})$.} Similarly, it can be seen that
		\begin{equation*}
			\frac{1}{2}\frac{d}{dt}\left(\|v^n\|^2_{\dot{H}^\delta} + \|(E^n,B^n)\|^2_{\dot{H}^s}\right) + \nu \|v^n\|^2_{\dot{H}^{\delta+\frac{3}{2}}} + \frac{1}{\sigma} \|j^n\|^2_{\dot{H}^s} =: \sum^3_{k=1} I_k,
		\end{equation*}
		where for some $\epsilon \in (0,1)$, since $\delta = s$
		\begin{align*}
			I_1 &= \int_{\mathbb{R}^3} \Lambda^\delta(j^n \times B^n) \cdot \Lambda^\delta v^n\,dx
			\\
			&\leq \|j^n\|_{L^2}\|B^n\|_{L^\frac{6}{3-2\delta}}\|\Lambda^{2\delta} v^n\|_{L^\frac{6}{2\delta}}
			\\
			&\leq \frac{\epsilon\nu}{2} \|v^n\|^2_{\dot{H}^{\delta+\frac{3}{2}}} + C(\epsilon,\delta,\nu)\|j^n\|^2_{L^2}\|B^n\|^2_{\dot{H}^\delta};
			\\
			I_2 &= -\int_{\mathbb{R}^3} \Lambda^\delta(v \cdot \nabla v^n) \cdot \Lambda^\delta v^n \,dx 
			\\
			&\leq \|v^n\|_{L^\frac{6}{3-2\delta}}\|\nabla v^n\|_{L^2}\|\Lambda^{2\delta} v^n\|_{L^\frac{6}{2\delta}} 
			\\
			&\leq \frac{\epsilon \nu}{2} \|v^n\|^2_{\dot{H}^{\delta+\frac{3}{2}}} + C(\epsilon,\delta,\nu)\left(\|v^n\|^2_{L^2} + \|v^n\|^2_{\dot{H}^\frac{3}{2}}\right) \|v^n\|^2_{\dot{H}^\delta};
			\\
			I_3 &= \int_{\mathbb{R}^3} \Lambda^s(v \times B^n) \cdot \Lambda^s j^n \,dx 
			\\
			&\leq C(s)\left(\|\Lambda^s v^n\|_{L^\frac{6}{2s}}\|B^n\|_{L^\frac{6}{3-2s}} + \|v^n\|_{L^\infty}\|\Lambda^s B^n\|_{L^2}\right)\|j^n\|_{\dot{H}^s}
			\\
			&\leq C(s)\left(\|v^n\|_{\dot{H}^\frac{3}{2}}\|B^n\|_{\dot{H}^s} + \|v^n\|_{L^\infty}\|B^n\|_{\dot{H}^s}\right)\|j^n\|_{\dot{H}^s}
			\\
			&\leq \frac{\epsilon}{\sigma}\|j^n\|^2_{\dot{H}^s} + C(\epsilon,\sigma,s)\|v^n\|^2_{\dot{H}^\frac{3}{2}}\|B^n\|^2_{\dot{H}^s} + C(\epsilon,\sigma,s)\|v^n\|^2_{L^\infty}\|B^n\|^2_{\dot{H}^s}.
		\end{align*}
		Therefore, by choosing $\epsilon = \frac{1}{2}$ and using \eqref{BGW} with $d = 3$ and $s_0 = s + \frac{1}{2}$
		\begin{multline*}
			\frac{d}{dt}\left(\|v^n\|^2_{H^\delta} + \|(E^n,B^n)\|^2_{H^s}\right) + \nu \|v^n\|^2_{H^{\delta+\frac{3}{2}}} + \frac{1}{\sigma} \|j^n\|^2_{H^s} 
			\leq C(\delta,\nu,\sigma)G\|(v^n,B^n)\|^2_{H^s} 
			+ \nu \|v^n\|^2_{L^2} 
			\\+ 
			\left[\frac{1}{2}C(\sigma,s)\|B^n\|_{H^s}\|v^n\|_{H^\frac{3}{2}}\left(1 + \log^\frac{1}{2}\left(\frac{\|v^n\|_{H^{s+\frac{3}{2}}}}{\|v^n\|_{H^\frac{3}{2}}}\right)\right)\right]^2,
		\end{multline*}
		where 
		\begin{equation*}
			G(t) := \|v^n\|^2_{L^2} + \|v^n\|^2_{\dot{H}^\frac{3}{2}} + \|j^n\|^2_{L^2},
		\end{equation*}
		which yields the conclusion as that of Step 3 in the proof of Theorem \ref{theo1}.
		
		\textbf{Step 3b: The case $\delta = s = \frac{3}{2}$.}
		In this case, we find that for some $\epsilon_0 \in (0,\frac{3}{2})$
		\begin{align*}
			I_1 &\leq \|j^n\|_{L^6}\|B^n\|_{L^3}\|\Lambda^{2\delta} v^n\|_{L^2} \leq \epsilon\nu \|v^n\|^2_{\dot{H}^{\delta+\frac{3}{2}}} + C(\epsilon,\delta,\nu)\|j^n\|^2_{H^1}\|B^n\|^2_{H^1};
			\\
			I_2 &\leq \|v^n\|_{L^6}\|\nabla v^n\|_{L^3}\|\Lambda^{2\delta} v^n\|_{L^2} \leq \epsilon\nu\|v^n\|^2_{\dot{H}^\frac{3}{2}} + C(\epsilon,\nu,\delta)\left(\|v^n\|^2_{L^2} + \|v^n\|^2_{\dot{H}^\frac{3}{2}}\right)\|v^n\|^2_{\dot{H}^\frac{3}{2}};
			\\
			I_3 &\leq C(s)\left(\|\Lambda^s v^n\|_{L^\frac{6}{2\epsilon_0}}\|B^n\|_{L^\frac{6}{3-2\epsilon_0}} + \|v^n\|_{L^\infty}\|\Lambda^s B^n\|_{L^2}\right)\|j^n\|_{\dot{H}^s}
			\\
			&\leq \frac{\epsilon}{\sigma}\|j^n\|^2_{\dot{H}^s} + C(\epsilon,\sigma,s)\|v^n\|^2_{\dot{H}^{s -\epsilon_0 + \frac{3}{2}}}\|B^n\|^2_{\dot{H}^{\epsilon_0}} + C(\epsilon,\sigma,s)\|v^n\|^2_{L^\infty}\|B^n\|^2_{\dot{H}^s},
		\end{align*}
		which together with Step 3a (to bound $\|v^n\|_{L^2_tH^{s-\epsilon_0 + \frac{3}{2}}_x}, \|v^n\|_{L^2_tL^\infty_x}$ and $\|j^n\|_{L^2_tH^1_x}$) closes the $H^s$ estimate.
		
		\textbf{Step 3c: The case $\delta = s > \frac{3}{2}$.} In this case, it follows that for some $\epsilon \in (0,1)$
		\begin{align*}
			I_1 &\leq C(\delta)\left(\|j^n\|_{\dot{H}^\delta}\|B^n\|_{L^\infty} + \|j^n\|_{L^\infty}\|B^n\|_{\dot{H}^\delta}\right)\|v^n\|_{\dot{H}^\delta} =: I_{11} + I_{12},
			\\
			I_{11} &\leq C(\delta)\|j^n\|_{\dot{H}^\delta} \|B^n\|^\frac{2\delta-3}{2\delta}_{L^2}\|B^n\|^\frac{3}{2\delta}_{\dot{H}^\delta}\|v^n\|_{\dot{H}^\frac{3}{2}}^\frac{3}{2\delta}\|v^n\|^\frac{2\delta-3}{2\delta}_{\dot{H}^{\delta+\frac{3}{2}}}
			\\
			&\leq \frac{\epsilon}{\sigma}\|j^n\|^2_{\dot{H}^\delta}
			 + C(\epsilon,\delta,\sigma)\|B^n\|^\frac{2\delta-3}{\delta}_{L^2}\|B^n\|^\frac{3}{\delta}_{\dot{H}^\delta}\|v^n\|_{\dot{H}^\frac{3}{2}}^\frac{3}{\delta}\|v^n\|^\frac{2\delta-3}{\delta}_{\dot{H}^{\delta+\frac{3}{2}}}
			\\
			&\leq \frac{\epsilon}{\sigma}\|j^n\|^2_{\dot{H}^\delta} + \epsilon \nu \|v^n\|^2_{\dot{H}^{\delta+\frac{3}{2}}}  + C(\epsilon,\delta,\sigma)\|B^n\|^\frac{2(2\delta-3)}{3}_{L^2}\|B^n\|^2_{\dot{H}^\delta}\|v^n\|^2_{\dot{H}^\frac{3}{2}},
			\\
			I_{12} &\leq C(\delta)\|j^n\|^\frac{2\delta-3}{2\delta}_{L^2}\|j^n\|^\frac{3}{2\delta}_{\dot{H}^\delta}\|B^n\|_{\dot{H}^\delta}\|v^n\|^\frac{3}{2\delta}_{\dot{H}^\frac{3}{2}}\|v^n\|^\frac{2\delta-3}{2\delta}_{\dot{H}^{\delta+\frac{3}{2}}}
			\\ 
			&\leq \frac{\epsilon}{\sigma}\|j^n\|^2_{\dot{H}^\delta} + C(\epsilon,\delta,\sigma)\|j^n\|^\frac{2(2\delta-3)}{4\delta-3}_{L^2}
			\|B^n\|^\frac{4\delta}{4\delta-3}_{\dot{H}^\delta}
			\|v^n\|^\frac{6}{4\delta-3}_{\dot{H}^\frac{3}{2}}
			\|v^n\|^\frac{4\delta-6}{4\delta-3}_{\dot{H}^{\delta+\frac{3}{2}}}
			\\
			&\leq \frac{\epsilon}{\sigma}\|j^n\|^2_{\dot{H}^\delta} + \epsilon\nu\|v^n\|^2_{\dot{H}^{\delta+\frac{3}{2}}} + C(\epsilon,\delta,\sigma)\|j^n\|^\frac{2\delta-3}{\delta}_{L^2}\|v^n\|^\frac{3}{\delta}_{\dot{H}^\frac{3}{2}}\|B^n\|^2_{\dot{H}^\delta}
			\\
			&\leq \frac{\epsilon}{\sigma}\|j^n\|^2_{\dot{H}^\delta} + \epsilon\nu\|v^n\|^2_{\dot{H}^{\delta+\frac{3}{2}}} + C(\epsilon,\delta,\sigma)\left(\|j^n\|^2_{L^2} + \|v^n\|^2_{\dot{H}^\frac{3}{2}}\right)\|B^n\|^2_{\dot{H}^\delta};
			\\
			I_2 &= \int_{\mathbb{R}^3} \Lambda^{\delta-\frac{1}{2}}(v^n \otimes v^n) \cdot \Lambda^{\delta+\frac{3}{2}} v^n \,dx
			\\
			&\leq C(\delta)\|\Lambda^{\delta-\frac{1}{2}} v^n\|_{L^3}\|v^n\|_{L^6}\|\Lambda^{\delta+\frac{3}{2}} v^n\|_{L^2}
			\\
			&\leq \epsilon\nu\|v^n\|^2_{\dot{H}^{\delta+\frac{3}{2}}} + C(\epsilon,\delta,\nu) \left(\|v^n\|^2_{L^2} + \|v^n\|^2_{\dot{H}^\frac{3}{2}}\right)\|v^n\|^2_{\dot{H}^\delta};
			\\
			I_3 &\leq C(\delta)\left(\|v^n\|_{\dot{H}^\delta}\|B^n\|_{L^\infty} + \|v^n\|_{L^\infty}\|B^n\|_{\dot{H}\delta}\right)\|j^n\|_{\dot{H}^\delta} =: I_{31} + I_{32},
			\\
			I_{31} &\leq  \frac{\epsilon}{\sigma}\|j^n\|^2_{\dot{H}^\delta} + \epsilon \nu \|v^n\|^2_{\dot{H}^{\delta+\frac{3}{2}}}  + C(\epsilon,\delta,\sigma)\|B^n\|^\frac{2(2\delta-3)}{3}_{L^2}\|B^n\|^2_{\dot{H}^\delta}\|v^n\|^2_{\dot{H}^\frac{3}{2}},
			\\
			I_{32} &\leq \frac{\epsilon}{\sigma}\|j^n\|^2_{\dot{H}^\delta} + C(\epsilon,\delta,\sigma)\|v^n\|^2_{L^\infty}\|B^n\|^2_{\dot{H}^\delta},
		\end{align*}
		where we used Lemma \ref{lem-Agmon}. Thus, we can use Step 3a to bound $\|v^n\|_{L^2_tL^\infty_x}$, which closes the main estimate.
		
		\textbf{Step 4: The case $0 < \delta \leq s < \frac{3}{2}$.} In this case, by using Step 3a, we need to estimate only $(E^n,B^n)$, with $I_3$ is bounded as follows
		\begin{align*}
			I_3 &\leq C(s)\left(\|\Lambda^s v^n\|_{L^\frac{6}{2s}}\|B^n\|_{L^\frac{6}{3-2s}} + \|v^n\|_{L^\infty}\|B^n\|_{\dot{H}^s}\right)\|j^n\|_{\dot{H}^s}
			\\
			&\leq \frac{\epsilon}{\sigma}\|j^n\|^2_{\dot{H}^s} +  C(\epsilon,\sigma,s)\left(\|v^n\|^2_{\dot{H}^\frac{3}{2}} + \|v^n\|^2_{H^{\delta+\frac{3}{2}}}\right)
			\|B^n\|^2_{\dot{H}^s}.
		\end{align*}
		
		\textbf{Step 5: The case $0 < \delta \leq s = \frac{3}{2}$.} Similar to the previous step, 		
		\begin{align*}
			I_3 &\leq C(s)\left(\|\Lambda^s v^n\|_{L^\frac{6}{3-2\delta}}\|B^n\|_{L^\frac{6}{2\delta}} + \|v^n\|_{L^\infty}\|B^n\|_{\dot{H}^s}\right)\|j^n\|_{\dot{H}^s}
			\\
			&\leq \frac{\epsilon}{\sigma}\|j^n\|^2_{\dot{H}^s} +  C(\epsilon,\delta,\sigma,s)\|v^n\|^2_{\dot{H}^{s+\delta}}\|B^n\|_{\dot{H}^\delta} + C(\epsilon,\delta,\sigma,s)\|v^n\|^2_{H^{\delta+\frac{3}{2}}}
			\|B^n\|^2_{\dot{H}^s}
			\\
			&\leq \frac{\epsilon}{\sigma}\|j^n\|^2_{\dot{H}^s} +  C(\epsilon,\delta,\sigma,s)\|v^n\|^2_{H^{\delta+\frac{3}{2}}} \|B^n\|^2_{H^s}.
		\end{align*}
		
		\textbf{Step 6: The case $s > \frac{3}{2}$ and $s - \frac{3}{2} \leq \delta < s$.} In this case, we find that
		\begin{align*}
			I_3 &\leq C(s) \left(\|v^n\|_{\dot{H}^s}\|B^n\|_{L^\infty} + \|v^n\|_{L^\infty}\|B^n\|_{\dot{H}^s}\right)\|j^n\|_{\dot{H}^s}
			\\
			&\leq \frac{\epsilon}{\sigma}			\|j^n\|^2_{\dot{H}^s} + C(\epsilon,\sigma,s) \|v^n\|^2_{H^{\delta+\frac{3}{2}}}\|B^n\|^2_{H^s}.
		\end{align*}
		
		\textbf{Step 7: The case $s > \frac{3}{2}$ and $s < \delta < s + \frac{3}{2}$.} In this case, by using Step 3c, we need to estimate only $v^n$. We write $\delta = s + \epsilon_0$ for some $\epsilon_0 \in (0,\frac{3}{2})$, and bound $I_2$ as in Step 3c and $I_1$ as follows
		\begin{align*}
			I_1 &= \int_{\mathbb{R}^3} \Lambda^{\delta-\frac{3}{2}} (j^n \times B^n) \cdot \Lambda^{\delta+\frac{3}{2}} v^n \,dx
			\\
			&\leq C(\delta)\left(\|\Lambda^{\delta-\frac{3}{2}} j^n\|_{L^\frac{6}{2\epsilon_0}}\|B^n\|_{L^\frac{6}{3-2\epsilon_0}} + \|j^n\|_{L^\frac{6}{3-2\epsilon_0}} \|B^n\|_{L^\frac{6}{2\epsilon_0}}\right)\|v^n\|_{\dot{H}^{\delta+\frac{3}{2}}}
			\\
			&\leq C(\delta) \left(\|j^n\|_{\dot{H}^s}\|B^n\|_{\dot{H^{\epsilon_0}}} + \|j^n\|_{\dot{H}^{\epsilon_0}} \|B^n\|_{\dot{H}^s}\right)\|v^n\|_{\dot{H}^{\delta+\frac{3}{2}}}
			\\
			&\leq \epsilon\nu \|v^n\|^2_{\dot{H}^{\delta+\frac{3}{2}}} + C(\epsilon,\delta,\nu) \|j^n\|^2_{H^s}\|B^n\|^2_{H^s}.
		\end{align*}
		
		\textbf{Step 8: The case $s > \frac{3}{2}$ and $\delta = s + \frac{3}{2}$.} Similarly, $I_2$ is bounded as in Step 3c and
		\begin{equation*}
			I_1 = \int_{\mathbb{R}^3} \Lambda^s(j^n \times B^n) \cdot \Lambda^{\delta+\frac{3}{2}} v^n \,dx \leq \epsilon\nu \|v^n\|^2_{\dot{H}^{\delta+\frac{3}{2}}} + C(\epsilon,\delta,\nu,s)\|j^n\|^2_{H^s}\|B^n\|^2_{H^s}.
		\end{equation*}
		
		\textbf{Step 9: The case $s = \frac{3}{2}$ and $\frac{3}{2} < \delta < 3$.} Since $\delta-\frac{3}{2} \in (0,\frac{3}{2})$, we bound $I_2$ as in Step 3c and
		\begin{align*}
			I_1 &= \int_{\mathbb{R}^3} \Lambda^{\delta-\frac{3}{2}}(j^n \times B^n) \cdot \Lambda^{\delta+\frac{3}{2}} v^n \,dx 
			\\
			&\leq C(\delta)\left(\|\Lambda^{\delta-\frac{3}{2}} j^n\|_{L^\frac{6}{2\left(\delta-\frac{3}{2}\right)}}\|B^n\|_{L^\frac{6}{3-2\left(\delta-\frac{3}{2}\right)}} + \|j^n\|_{L^\frac{6}{3-2\left(\delta-\frac{3}{2}\right)}}\|\Lambda^{\delta-\frac{3}{2}} B^n\|_{L^\frac{6}{2\left(\delta-\frac{3}{2}\right)}}\right)\|v^n\|_{\dot{H}^{\delta+\frac{3}{2}}}
			\\
			&\leq \epsilon\nu \|v^n\|^2_{\dot{H}^{\delta+\frac{3}{2}}} + C(\epsilon,\delta,\nu)\|j^n\|^2_{\dot{H}^\frac{3}{2}}\|B^n\|^2_{\dot{H}^{\delta-\frac{3}{2}}} + C(\epsilon,\delta,\nu)\|B^n\|^2_{\dot{H}^\frac{3}{2}}\|j^n\|^2_{\dot{H}^{\delta-\frac{3}{2}}}
			\\
			&\leq \epsilon\nu \|v^n\|^2_{\dot{H}^{\delta+\frac{3}{2}}} + C(\epsilon,\delta,\nu)\|j^n\|^2_{H^s}\|B^n\|^2_{H^s}.
		\end{align*}
		
		\textbf{Step 10: The case $\frac{3}{4} < s < \frac{3}{2}$ and $\frac{3}{2} < \delta \leq 2s$.} Similar to the previous case, it can be seen that
		\begin{align*}
			I_1 &\leq C(s)\left(\|\Lambda^{\delta-\frac{3}{2}} j^n\|_{L^\frac{6}{2s}}\|B^n\|_{L^\frac{6}{3-2s}} + \|j^n\|_{L^\frac{6}{3-2s}}\|\Lambda^{\delta-\frac{3}{2}} B^n\|_{L^\frac{6}{2s}}\right)\|v^n\|^2_{\dot{H}^{\delta+\frac{3}{2}}}
			\\
			&\leq C(s) \left(\|j^n\|_{\dot{H}^{\delta-s}}\|B^n\|_{\dot{H}^s} + \|j^n\|_{\dot{H}^s}\|B^n\|_{\dot{H}^{\delta-s}}\right)\|v^n\|^2_{\dot{H}^{\delta+\frac{3}{2}}}
			\\
			&\leq \epsilon\nu \|v^n\|^2_{\dot{H}^{\delta+\frac{3}{2}}} + C(\epsilon,\delta,\nu)\|j^n\|^2_{H^s}\|B^n\|^2_{H^s}.
		\end{align*}
		
		\textbf{Step 11: The case $\frac{3}{4} \leq s < \frac{3}{2}$ and $s < \delta \leq \frac{3}{2}$.} In this case, $I_2$ can bounded as in Steps 3a and 3b. Since $\frac{3}{2} - s \leq s$ and $2\delta \leq \delta + \frac{3}{2}$
		\begin{align*}
			I_1 &= \int_{\mathbb{R}^3} (j^n \times B^n) \cdot \Lambda^{2\delta} v^n \,dx 
			\\
			&\leq C(s)\left(\|j^n\|_{L^\frac{6}{2s}}\|B^n\|_{L^\frac{6}{3-2s}} + \|j^n\|_{L^\frac{6}{3-2s}}\|B^n\|_{L^\frac{6}{2s}}\right)\|v^n\|_{\dot{H}^{2\delta}}
			\\
			&\leq \epsilon\nu\left(\|v^n\|^2_{L^2} +  \|v^n\|^2_{\dot{H}^{\delta+\frac{3}{2}}}\right) + C(\epsilon,\delta,\nu)\|j^n\|^2_{H^s}\|B^n\|^2_{H^s}.
		\end{align*}
		
		\textbf{Step 12: The case $0 < s < \frac{3}{4}$ and $s < \delta \leq 2s$.} In this case, $I_2$ is estimated as in Step 3a. We only need to focus on the estimate of $I_1$. In addition, for some $\epsilon \in (0,1)$, since $s \in (0,\frac{3}{4})$, $s < \delta \leq 2s$ and $\frac{3}{2}-(\delta-s),\frac{\delta-s}{2} \in (0,\frac{3}{2})$, by using \eqref{KPV}
		\begin{align*}
			I_1 &= \int_{\mathbb{R}^3} [(\Lambda^s(j^n \times B^n) - \Lambda^s j^n \times B^n - j^n \times \Lambda^s B^n) + \Lambda^s j^n \times B^n + j^n \times \Lambda^s B^n] \cdot \Lambda^{2\delta-s} v^n \,dx =: \sum^3_{k=1} I_{1k},
			\\
			I_{11} &\leq \|\Lambda^s(j^n \times B^n) - \Lambda^s j^n \times B^n - j^n \times \Lambda^s B^n\|_{L^\frac{6}{3+2\left(\frac{3}{2}-(\delta-s)\right)}} \|\Lambda^{2\delta-s} v^n\|_{L^\frac{6}{3-2\left(\frac{3}{2}-(\delta-s)\right)}}
			\\
			&\leq C(\delta,s)\|\Lambda^{\frac{6s-2\delta}{4}} j^n\|_{L^\frac{12}{3+2\left(\frac{3}{2}-(\delta-s)\right)}}\|\Lambda^{\frac{\delta-s}{2}} B^n\|_{L^\frac{12}{3+2\left(\frac{3}{2}-(\delta-s)\right)}} \|v^n\|_{\dot{H}^{\delta+\frac{3}{2}}}
			\\
			&\leq \epsilon\nu \|v^n\|^2_{\dot{H}^{\delta+\frac{3}{2}}} + C(\epsilon,\delta,\nu,s)\|\Lambda^s j^n\|^2_{L^2}\|\Lambda^{\delta-s}B^n\|^2_{L^2}
			\\
			&\leq \epsilon\nu \|v^n\|^2_{\dot{H}^{\delta+\frac{3}{2}}} + C(\epsilon,\delta,\nu,s)\|j^n\|^2_{H^s}\|B^n\|^2_{H^s},
			\\
			I_{12} &\leq \|j^n\|_{\dot{H}^s}\|B^n\|_{L^\frac{6}{3-2(\delta-s)}}\|\Lambda^{2\delta-s} v^n\|_{L^\frac{6}{2(\delta-s)}} \\
			&\leq \epsilon\nu \|v^n\|^2_{\dot{H}^{\delta+\frac{3}{2}}} + C(\epsilon,\delta,s)\|j^n\|^2_{\dot{H}^s}\|B^n\|^2_{\dot{H}^{\delta-s}},
			\\
			&\leq \epsilon\nu \|v^n\|^2_{\dot{H}^{\delta+\frac{3}{2}}} + C(\epsilon,\delta,\nu,s)\|j^n\|^2_{H^s}\|B^n\|^2_{H^s},
			\\
			I_{13} &\leq \epsilon\nu \|v^n\|^2_{\dot{H}^{\delta+\frac{3}{2}}} + C(\epsilon,\delta,\nu,s)\|j^n\|^2_{H^s}\|B^n\|^2_{H^s}.
		\end{align*}
		
		\textbf{Step 13: Conclusion from Step 2 to Step 12.} From Step 2 to Step 12, we can close the $H^\delta-H^s$ estimate of $(v^n,E^n,B^n)$, which yields $T^n_* = \infty$ and uniform bounds in terms of $n$ with replacing $T^n_*$ by any $T \in (0,\infty)$.
		
		\textbf{Step 14: Pass to the limit.} This step can be done by applying Steps 16 and 17a for $s,\delta > \frac{3}{2}$ and Step 17b for either $s \in (0,\frac{3}{2}]$ or $\delta \in (0,\frac{3}{2}]$, in the proof of Theorem \ref{theo1}. We only mention that in the case $\delta = 0$ and $s \in (0,\frac{3}{2})$, we have $\partial_t v^n$ is uniformly bounded in $L^2_tH^{-\frac{3}{2}}_x$. Thus, we will use the injections $H^1 \hookrightarrow L^2 \hookrightarrow H^{-\frac{3}{2}}$ for $v^n$ instead of the previous one in two dimensions. Therefore, we can pass to the limit in the same way. We omit further details.
		
		\textbf{Step 15: Uniqueness.} It is enough to prove the uniqueness in the case $\delta = 0$ and $s \in (0,\frac{3}{2})$. Similar to Step 18 in the proof of Theorem \ref{theo1}, the usual energy method does not work here for $s \in (0,1)$. Indeed,
		\begin{equation*}
			\frac{1}{2}\frac{d}{dt}\|v-\bar{v}\|^2_{L^2} + \nu\|v-\bar{v}\|^2_{\dot{H}^\frac{3}{2}} 
			=: \sum^3_{k=1} \bar{I}_k,
		\end{equation*}
		where for some $\epsilon \in (0,1)$ and for any $s' \in (0,s]$
		\begin{align*}
			\bar{I}_1 &= -\int_{\mathbb{R}^3} (v-\bar{v}) \cdot \nabla v \cdot (v-\bar{v}) \,dx 
			\\
			&\leq C\|v\|_{\dot{H}^\frac{3}{2}}\|v-\bar{v}\|_{L^2}\|v-\bar{v}\|_{\dot{H}^1}
			\\
			&\leq \epsilon\nu\|v-\bar{v}\|^2_{\dot{H}^\frac{3}{2}} + C(\epsilon,\nu)\|v\|^\frac{3}{2}_{\dot{H}^\frac{3}{2}}\|v-\bar{v}\|^2_{L^2};
			\\
			\bar{I}_2 &= \int_{\mathbb{R}^3} (j-\bar{j}) \times B \cdot (v-\bar{v}))\,dx
			\\
			&\leq C(s')\|j-\bar{j}\|_{L^2}\|B\|_{\dot{H}^{s'}}\|v-\bar{v}\|_{\dot{H}^{\frac{3}{2}-s'}}
			\\
			&\leq C(s')\|j-\bar{j}\|_{L^2}\|B\|_{\dot{H}^{s'}}\|v-\bar{v}\|^\frac{2s'}{3}_{L^2}\|v-\bar{v}\|^\frac{3-2s'}{3}_{\dot{H}^\frac{3}{2}}
			\\
			&\leq \epsilon\nu \|v-\bar{v}\|^2_{\dot{H}^\frac{3}{2}} + C(\epsilon,\nu,s') \|B\|^\frac{6}{2s'+3}_{\dot{H}^{s'}} \|j-\bar{j}\|^\frac{6}{2s'+3}_{L^2} \|v-\bar{v}\|^\frac{4s'}{2s'+3}_{L^2};
			\\
			\bar{I}_3 &= \int_{\mathbb{R}^3} \bar{j} \times (B-\bar{B}) \cdot (v-\bar{v})\,dx
			\\
			&\leq C(s')\|\bar{j}\|_{L^2}\|B-\bar{B}\|_{\dot{H}^{s'}}\|v-\bar{v}\|_{\dot{H}^{\frac{3}{2}-s'}}
			\\
			&\leq \epsilon\nu \|v-\bar{v}\|^2_{\dot{H}^\frac{3}{2}} + C(\epsilon,\nu,s') \|B-\bar{B}\|^\frac{6}{2s'+3}_{\dot{H}^{s'}} \|\bar{j}\|^\frac{6}{2s'+3}_{L^2} \|v-\bar{v}\|^\frac{4s'}{2s'+3}_{L^2}.
		\end{align*}
		Therefore, by choosing $\epsilon = \frac{1}{6}$ and taking $T_* \in (0,T]$
		\begin{equation*}
			\|v-\bar{v}\|^2_{L^\infty(0,T_*;L^2)} + \nu \int^{T_*}_0  \|v-\bar{v}\|^2_{\dot{H}^\frac{3}{2}} \,d\tau \leq \sum^3_{k=1} \bar{J}_k,
		\end{equation*}
		where 
		\begin{align*}
			\bar{J}_1 &:= C(\nu,s') \int^{T_*}_0 \|v\|^\frac{3}{2}_{\dot{H}^\frac{3}{2}}\|v-\bar{v}\|^2_{L^2} \,d\tau \leq C(\nu)T^\frac{1}{4}_* \|v\|^\frac{3}{2}_{L^2(0,T_*;\dot{H}^\frac{3}{2})} \|v-\bar{v}\|^2_{L^\infty(0,T_*;L^2)};
			\\
			\bar{J}_2 &:= C(\nu,s') \int^{T_*}_0  \|B\|^\frac{6}{2s'+3}_{\dot{H}^{s'}} \|j-\bar{j}\|^\frac{6}{2s'+3}_{L^2} \|v-\bar{v}\|^\frac{4s'}{2s'+3}_{L^2} \,d\tau \leq \sum^3_{k=1} \bar{J}_{2k},
			\\
			\bar{J}_{21} &:= C(c,\nu,\sigma,s') \int^{T_*}_0  \|B\|^\frac{6}{2s'+3}_{\dot{H}^{s'}} \|E-\bar{E}\|^\frac{6}{2s'+3}_{L^2} \|v-\bar{v}\|^\frac{4s'}{2s'+3}_{L^2} \,d\tau
			\\
			&\leq C(c,\nu,\sigma,s') T_*  \|B\|^\frac{6}{2s'+3}_{L^\infty(0,T_*;\dot{H}^{s'})} \left(\|E-\bar{E}\|^2_{L^\infty(0,T_*;L^2)} + \|v-\bar{v}\|^2_{L^\infty(0,T_*;L^2)}\right),
			\\
			\bar{J}_{22} &:= C(\nu,\sigma,s') \int^{T_*}_0  \|B\|^\frac{6}{2s'+3}_{\dot{H}^{s'}} \|(v-\bar{v}) \times B\|^\frac{6}{2s'+3}_{L^2} \|v-\bar{v}\|^\frac{4s'}{2s'+3}_{L^2} \,d\tau
			\\
			&\leq C(\nu,\sigma,s') \int^{T_*}_0  \|B\|^\frac{6}{2s'+3}_{\dot{H}^{s'}} \|v-\bar{v}\|^\frac{6}{2s'+3}_{\dot{H}^{\frac{3}{2}-s}} \|B\|^\frac{6}{2s'+3}_{\dot{H}^s} \|v-\bar{v}\|^\frac{4s'}{2s'+3}_{L^2} \,d\tau
			\\
			&\leq C(\nu,\sigma,s') \int^{T_*}_0  \|B\|^\frac{12}{2s'+3}_{\dot{H}^{s'}} \|v-\bar{v}\|^\frac{4s'}{2s'+3}_{L^2} \|v-\bar{v}\|^\frac{2(3-2s')}{2s'+3}_{\dot{H}^\frac{3}{2}}  \|v-\bar{v}\|^\frac{4s'}{2s'+3}_{L^2} \,d\tau
			\\
			&\leq C(\nu,\sigma,s') T^\frac{4s'}{2s'+3}_* \|B\|^\frac{12}{2s'+3}_{L^\infty(0,T_*;\dot{H}^{s'})}\left(\|v-\bar{v}\|^2_{L^2(0,T_*;L^2)} + \nu \|v-\bar{v}\|^2_{L^2(0,T_*;\dot{H}^\frac{3}{2})}\right),
			\\
			\bar{J}_{23} &:= C(\nu,\sigma,s') \int^{T_*}_0  \|B\|^\frac{6}{2s'+3}_{\dot{H}^{s'}} \|\bar{v} \times (B-\bar{B})\|^\frac{6}{2s'+3}_{L^2} \|v-\bar{v}\|^\frac{4s'}{2s'+3}_{L^2} \,d\tau
			\\
			&\leq C(\nu,\sigma,s') \int^{T_*}_0  \|B\|^\frac{6}{2s'+3}_{\dot{H}^{s'}} \|\bar{v}\|^\frac{4s'}{2s'+3}_{L^2} \|\bar{v}\|^\frac{2(3-2s')}{2s'+3}_{\dot{H}^\frac{3}{2}} \|B-\bar{B}\|^\frac{6}{2s'+3}_{\dot{H}^{s'}} \|v-\bar{v}\|^\frac{4s'}{2s'+3}_{L^2} \,d\tau
			\\
			&\leq C(\nu,\sigma,s') T^\frac{4s'}{2s'+3}_* \|B\|^\frac{6}{2s'+3}_{L^\infty(0,T_*;\dot{H}^{s'})} \|\bar{v}\|^\frac{4s'}{2s'+3}_{L^2(0,T_*;L^2)} \|\bar{v}\|^\frac{2(3-2s')}{2s'+3}_{L^2(0,T_*;\dot{H}^\frac{3}{2})}
			\\
			&\quad \times \left(\|B-\bar{B}\|^2_{L^\infty(0,T_*;\dot{H}^{s'})} + \|v-\bar{v}\|^2_{L^\infty(0,T_*;L^2)}\right);
			\\
			\bar{J}_3 &:= C(\nu,s') \int^{T_*}_0 \|B-\bar{B}\|^\frac{6}{2s'+3}_{\dot{H}^{s'}} \|\bar{j}\|^\frac{6}{2s'+3}_{L^2} \|v-\bar{v}\|^\frac{4s'}{2s'+3}_{L^2} \,d\tau 
			\\
			&\leq C(\nu,s')T^\frac{2s'}{2s'+3}_* \|\bar{j}\|^\frac{6}{2s'+3}_{L^2(0,T_*;L^2)} \left(\|B-\bar{B}\|^2_{L^\infty(0,T_*;\dot{H}^{s'})} + \|v-\bar{v}\|^2_{L^\infty(0,T_*;L^2)}\right).
		\end{align*}
		In addition, by using Lemma \ref{lem_M0}, it follows that
		\begin{align*}
			\|(E-\bar{E},B-\bar{B})\|^2_{L^\infty(0,T_*;H^{s'})} \leq C(c)\|j-\bar{j}\|^2_{L^1(0,T_*;H^{s'})} 
			=: \sum^6_{k=4}\bar{J}_k,
		\end{align*}
		where for any $s' \in (0,s)$
		\begin{align*}
			\bar{J}_4 &= C(c,\sigma)\|E-\bar{E}\|^2_{L^1(0,T_*;H^{s'})} \leq  C(c,\sigma)T^2_*\|E-\bar{E}\|^2_{L^\infty(0,T_*;H^{s'})};
			\\
			\bar{J}_5 &=  C(c,\sigma)\|(v-\bar{v}) \times B\|^2_{L^1(0,T_*;H^{s'})} \leq \bar{J}_{51} + \bar{J}_{52},
			\\
			\bar{J}_{51} &:= C(c,\sigma) \|(v-\bar{v}) \times B\|^2_{L^1(0,T_*;L^2)} 
			\\
			&\leq C(c,\sigma,s')T^\frac{3+2s'}{3}_*\|B\|^2_{L^\infty(0,T_*;\dot{H}^{s'})}\|v-\bar{v}\|^\frac{4s'}{3}_{L^\infty(0,T_*;L^2)}\|v-\bar{v}\|^\frac{2(3-2s')}{3}_{L^2(0,T_*;\dot{H}^\frac{3}{2})},
			\\
			&\leq C(c,\nu,\sigma,s')T^\frac{3+2s'}{3}_*\|B\|^2_{L^\infty(0,T_*;\dot{H}^{s'})} \left(\|v-\bar{v}\|^2_{L^\infty(0,T_*;L^2)} + \nu\|v-\bar{v}\|^2_{L^2(0,T_*;\dot{H}^\frac{3}{2})}\right),
			\\
			\bar{J}_{52} &:= C(c,\sigma) \|(v-\bar{v}) \times B\|^2_{L^1(0,T_*;\dot{H}^{s'})} 
			\\
			&\leq C(\nu,c,\sigma,s')T_*\|B\|^2_{L^\infty(0,T_*;\dot{H}^{s'})} \nu\|v-\bar{v}\|^2_{L^2(0,T_*;\dot{H}^\frac{3}{2})} 
			\\
			&\quad+ C(c,\sigma,s,s')T^{3+2(s-s')}_* \|B\|^2_{L^\infty(0,T_*;\dot{H}^s)}\|v-\bar{v}\|^\frac{4(s-s')}{3}_{L^\infty(0,T_*;L^2)}\|v-\bar{v}\|^\frac{2(3-2(s-s'))}{3}_{L^2(0,T_*;\dot{H}^\frac{3}{2})}
			\\
			&\leq C(c,\nu,\sigma,s')T_*\|B\|^2_{L^\infty(0,T_*;\dot{H}^{s'})} \nu\|v-\bar{v}\|^2_{L^2(0,T_*;\dot{H}^\frac{3}{2})} 
			\\
			&\quad+ C(c,\nu,\sigma,s,s')T^{3+2(s-s')}_* \|B\|^2_{L^\infty(0,T_*;\dot{H}^s)} \left(\|v-\bar{v}\|^2_{L^\infty(0,T_*;L^2)} +  \nu\|v-\bar{v}\|^2_{L^2(0,T_*;\dot{H}^\frac{3}{2})}\right);
			\\
			\bar{J}_6 &= C(c,\sigma)\|\bar{v} \times (B-\bar{B})\|^2_{L^1(0,T_*;H^{s'})} \leq \bar{J}_{61} + \bar{J}_{62},
			\\
			\bar{J}_{61} &:= C(c,\sigma)\|\bar{v} \times (B-\bar{B})\|^2_{L^1(0,T_*;L^2)}
			\\
			&\leq C(c,\sigma,s') T^\frac{3+2s'}{3}_* \|\bar{v}\|^\frac{4s'}{3}_{L^\infty(0,T_*;L^2)} \|\bar{v}\|^\frac{2(3-2s')}{3}_{L^2(0,T_*;\dot{H}^\frac{3}{2})}  \|B-\bar{B}\|^2_{L^\infty(0,T_*;\dot{H}^{s'})},
			\\
			\bar{J}_{62} &:= C(c,\sigma)\|\bar{v} \times (B-\bar{B})\|^2_{L^1(0,T_*;\dot{H}^{s'})}
			\\
			&\leq C(c,\sigma)T_* \left(\|\bar{v}\|^2_{L^2(0,T_*;\dot{H}^\frac{3}{2})} + \|\bar{v}\|^2_{L^2(0,T_*;L^\infty)}\right) \|B-\bar{B}\|^2_{L^\infty(0,T_*;H^{s'})}.
		\end{align*}
		Combining all the above estimates and using Step 1, we find that for sufficiently small $T_*$
		\begin{align*}
			A(v-\bar{v},E-\bar{E},B-\bar{B}) &:= \|v-\bar{v}\|^2_{L^\infty(0,T_*;L^2)} + \nu \|v-\bar{v}\|^2_{L^2(0,T_*;\dot{H}^\frac{3}{2})} +	\|(E-\bar{E},B-\bar{B})\|^2_{L^\infty(0,T_*;H^{s'})} \\
			&\leq \frac{1}{2} A(v-\bar{v},E-\bar{E},B-\bar{B}),
		\end{align*}
		which yields $v = \bar{v},E = \bar{E}$ and $B = \bar{B}$ in $(0,T_*)$. By repeating this process, we obtain the conclusion in the whole time interval $(0,T)$. Finally, we note that only the estimate of $\bar{J}_{52}$ needs $s' < s$ and other ones hold for $s' = s$ as well.
	\end{proof}
	
	\begin{proof}[Proof of Theorem \ref{theo1-3d}-$(ii)$] We will follow the idea in the proof of \cite[Corollary 1.3]{Arsenio-Gallagher_2020}, where the authors considered the case $\nu > 0$, $\alpha = 1$ and $d = 2$, and proved that up to an extraction of a subsequence $(v^c,B^c) \to (v,B)$ as $c \to \infty$ in the sense of distributions. We aim to apply the same idea to the case $\nu > 0$, $\alpha = \frac{3}{2}$ and $d = 3$. It suffices to focus on the case $\delta = 0$ and $s \in (0,\frac{3}{2})$. It can be seen from \eqref{NSM} with $\alpha = \frac{3}{2}$ that
		\begin{equation} \label{NSM-alpha-c}
			\left\{
			\begin{aligned}
				\partial_t v^c + v^c \cdot \nabla v^c + \nabla \pi^c &= -\nu (-\Delta)^\frac{3}{2}v^c + j^c \times B^c, 
				\\
				\frac{1}{c}\partial_t E^c - \nabla \times B^c &= -j^c,
				\\
				\partial_t B^c - \nabla \times (v^c \times B^c) &= -\frac{1}{\sigma} \nabla \times j^c,
				\\
				\textnormal{div}\, v^c = \textnormal{div}\, B^c &=  0.
			\end{aligned}
			\right.
		\end{equation}
		By applying Step 1 in Part $(i)$, we know that $(v^c,E^c,B^c,j^c)$ is uniformly bounded in terms of $c$ for any $T \in (0,\infty)$ in the following spaces
		\begin{equation*}
			v^c \in L^\infty(0,T;L^2) \cap L^2(0,T;H^\frac{3}{2}), \quad (E^c,B^c) \in L^\infty(0,T;H^s) \quad \text{and} \quad j^c \in L^2(0,T;H^s),
		\end{equation*}
		which implies that there exists $(v,E,B,j)$ such that up to an extraction of a subsequence (use the same notation) as $c \to \infty$
		\begin{align*}
			&&(v^c,E^c,B^c) &\overset{\ast}{\rightharpoonup}  (v,E,B) &&\text{in} \quad L^\infty_t(L^2_x \times H^s_x \times H^s_x),&&
			\\
			&&(v^c,j^c) &\rightharpoonup (v,j) &&\text{in} \quad L^2_t(H^\frac{3}{2}_x \times H^s_x).&&
		\end{align*}
		In addition, we find from \eqref{NSM-alpha-c} that
		\begin{align*}
			(\partial_t v^c,\partial_tB^c) \quad \text{is uniformly bounded in} \quad L^2_t(H^{-\frac{3}{2}}_{\textnormal{loc},x} \times H^{-1}_{\textnormal{loc},,x})
		\end{align*}
		and by using the Aubin-Lions lemma as $c \to \infty$
		\begin{equation*}
			(v^c,B^c) \to (v,B) \quad \text{ in}\quad L^2_tL^2_{\textnormal{loc},x}.
		\end{equation*}
		As in Step 17b in the proof of Theorem \ref{theo1}, for $\phi,\varphi \in C^\infty_0([0,T) \times \mathbb{R}^3;\mathbb{R}^3)$ with $\textnormal{div}\,\phi = 0$, the weak form of \eqref{NSM-alpha-c} is given by (similar to those of $a)$, $b)$ and $c)$)
		\begin{align*}
			a')\quad &\int^T_0 \int_{\mathbb{R}^3} v^c \cdot \partial_t\phi + (v^c \otimes v^c) : \nabla \phi - \nu v^c \cdot (-\Delta)^\frac{3}{2} \phi + (j^c \times B^c) \cdot \phi \,dxdt = -\int_{\mathbb{R}^3} v^c(0) \cdot \phi(0) \,dx,
			\\
			b')\quad &\int^T_0 \int_{\mathbb{R}^3} \frac{1}{c}E^c \cdot \partial_t\varphi + B^c \cdot (\nabla \times \varphi) - j^c \cdot \varphi \,dxdt = -\int_{\mathbb{R}^3} \frac{1}{c}E^c(0) \cdot \varphi(0) \,dx,
			\\
			c')\quad &\int^T_0 \int_{\mathbb{R}^3} B^c \cdot \partial_t\varphi + [(v^c \times B^c) - \frac{1}{\sigma} j^c] \cdot (\nabla \times \varphi) \,dxdt  = - \int_{\mathbb{R}^3} B^c(0) \cdot \varphi(0) \,dx.
		\end{align*}
		Therefore, we can pass to the limit by using $(v^c_0,E^c_0,B^c_0) \rightharpoonup (\bar{v}_0,\bar{E}_0,\bar{B}_0)$ in $L^2 \times H^s \times H^s$ and the above strong convergences as $c \to \infty$ to obtain that \eqref{NSM-alpha-c} converges in the sense of distributions to 
		\begin{equation*} 
			\left\{
			\begin{aligned}
				\partial_t v + v \cdot \nabla v + \nabla \pi &=  -\nu (-\Delta)^\frac{3}{2} v + j \times B, 
				\\
				\partial_t B - \nabla \times (v \times B) &= -\frac{1}{\sigma} \nabla \times j,
				\\
				\textnormal{div}\, v = \textnormal{div}\, B &= 0,
			\end{aligned}
			\right.
		\end{equation*}
		where $j = \nabla \times B$ and $(v,B)_{|_{t=0}} = (\bar{v}_0,\bar{B}_0)$. Thus, the proof is finished since $\nabla \times (\nabla \times B) = - \Delta B$.
	\end{proof}
	
	%
	\section{Proof of Theorem \ref{theo-MH}} \label{sec:theo-MH}
	%
	
	In this section, we provide a proof of Theorem \ref{theo-MH} as follows.
	
	\begin{proof}[Proof of Theorem \ref{theo-MH}] 

		We first redefine $F^n$ as in Step 1 in the proofs of Theorems \ref{theo1} and \ref{theo1-3d} as follows for $s > \frac{3}{2}$ 
		\begin{equation*}
			F^n : X^s_n := H^s_{n,0} \times (H^s_n \cap \dot{H}^{-1}) \times (H^s_{n,0} \cap \dot{H}^{-1}) \to X^s_n \quad \text{with} \quad \Gamma^{\sigma,n} \mapsto F^n(\Gamma^{\sigma,n}),
		\end{equation*}
		where $\Gamma^{\sigma,n} := (v^{\sigma,n},E^{\sigma,n},B^{\sigma,n})$. Here, the norm in $X^s_n$ is given by
		\begin{equation*}
			\|(f_1,f_2,f_3)\|^2_{X^s} :=  \|(f_1,f_2,f_3)\|^2_{H^s} + \|(f_2,f_3)\|^2_{\dot{H}^{-1}}.
		\end{equation*}
		Then, we can check that $F^n$ is well-defined and locally Lipschitz on $X^s_n$. In addition, we note that $\dot{H}^{-1}(\mathbb{R}^3)$ is a Hilbert space (see \cite{Bahouri-Chemin-Danchin_2011}). Thus, there exists a unique solution $\Gamma^{\sigma,n} \in C^1([0,T^n_*);X^s_n)$ for some $T^n_* > 0$. Assume that $T^n_* < \infty$. As in Step 3c in the proof of Theorem \ref{theo1-3d}, we find that for $t \in (0,T^n_*)$ and $\Gamma^{\sigma}_0 := (v^{\sigma}_0,E^{\sigma}_0,B^{\sigma}_0)$ 
		\begin{align} \label{vEB_n_L2}
			&\|\Gamma^{\sigma,n}(t)\|^2_{L^2} + \int^t_0 \nu \|v^{\sigma,n}\|^2_{\dot{H}^\frac{3}{2}} + \frac{1}{\sigma} \|j^{\sigma,n}\|^2_{L^2} \,d\tau \leq \|\Gamma^{\sigma}_0\|^2_{L^2},
			\\
			\label{vEB_n_L2Hs}
			& \|\Gamma^{\sigma,n}(t)\|^2_{H^s} + \int^t_0 \nu \|v^{\sigma,n}\|^2_{H^{s+\frac{3}{2}}} + \frac{1}{\sigma} \|j^{\sigma,n}\|^2_{H^s} \,d\tau \leq (T^n_*,\nu,\sigma,s, \Gamma^{\sigma}_0).
		\end{align}
		We now focus on the $\dot{H}^{-1}$ estimate of $(E^{\sigma,n},B^{\sigma,n})$ as follows
		\begin{equation*}
			\frac{d}{dt} \|(E^{\sigma,n},B^{\sigma,n})\|^2_{\dot{H}^{-1}} + \frac{1}{\sigma} \|j^{\sigma,n}\|^2_{\dot{H}^{-1}} \leq 
			\sigma \|v^{\sigma,n} \times B^{\sigma,n} \|^2_{\dot{H}^{-1}},
		\end{equation*}
		where by using the embedding $L^{p_0}(\mathbb{R}^3) \hookrightarrow \dot{H}^{s_0}(\mathbb{R}^3)$ (see \cite{Bahouri-Chemin-Danchin_2011}) for $p_0 \in (1,2]$ and $s_0 = \frac{3}{2} - \frac{3}{p_0}$ 
		with $(s_0,p_0) = (-1,\frac{6}{5})$ 
		\begin{align*}
			\|v^{\sigma,n} \times B^{\sigma,n} \|^2_{\dot{H}^{-1}} &\leq 
			C \|v^{\sigma,n} \times B^{\sigma,n} \|^2_{L^\frac{6}{5}} 
			\\
			&\leq C\|v^{\sigma,n}\|^2_{L^3} \|B^{\sigma,n}\|^2_{L^2}
			\\
			&\leq C \|v^{\sigma,n}\|^2_{\dot{H}^\frac{1}{2}} \|B^{\sigma,n}\|^2_{L^2}
			\\
			&\leq C\|v^{\sigma,n}\|^\frac{4}{3}_{L^2} \|v^{\sigma,n}\|^\frac{2}{3}_{\dot{H}^\frac{3}{2}} \|B^{\sigma,n}\|^2_{L^2},
		\end{align*}
		and \eqref{vEB_n_L2}, it follows that for $t \in (0,T^n_*)$
		\begin{align} \label{EB_n_Hsd}
			&\|(E^{\sigma,n},B^{\sigma,n})(t)\|^2_{\dot{H}^{-1}}+ \frac{1}{\sigma} \int^t_0 \|j^{\sigma,n}\|^2_{\dot{H}^{-1}} \,d\tau 
			\leq   C(t,\nu,\sigma,\Gamma^{\sigma}_0),
		\end{align}
		that yields $T^n_* = \infty$ by using further \eqref{vEB_n_L2Hs}. 
		Since $T^n_* = \infty$, we can repeat the above computations for any $T \in (0,\infty)$ to obtain similar bounds as in \eqref{vEB_n_L2Hs}-\eqref{EB_n_Hsd}, i.e., for $t \in (0,T)$ 
		\begin{equation*}
			\|\Gamma^{\sigma,n}(t)\|^2_{X^s} 
			+ \int^t_0 \|v^{\sigma,n}\|^2_{H^{s+\frac{3}{2}}} + \|j^{\sigma,n}\|^2_{H^s \cap \dot{H}^{-1}} \,d\tau 
			\leq  C(T,\nu,\sigma,s,\Gamma^{\sigma}_0),
		\end{equation*}
		which allows us to pass to the limit as $n \to \infty$ in the approximate system to obtain the limiting one \eqref{NSM_P}, with repalcing $\nu \Delta v$ by $-\nu (-\Delta)^\frac{3}{2} v$, in $L^2(0,T;H^{s'-\frac{3}{2}})$ for $\frac{3}{2} < s' < s$ (we omit the details, for example, see \cite{KLN_2024}) and
		\begin{align} \label{vEB_L2}
			&\|\Gamma^{\sigma}(t)\|^2_{L^2} + \int^t_0 \nu \|v^{\sigma}\|^2_{\dot{H}^\frac{3}{2}} + \frac{1}{\sigma} \|j^{\sigma}\|^2_{L^2} \,d\tau \leq \|\Gamma^{\sigma}_0\|^2_{L^2},
			\\ \label{EB_Hsd}
			&\|(E^{\sigma},B^{\sigma})(t)\|^2_{\dot{H}^{-1}} 
			+ \frac{1}{\sigma} \int^t_0 \|j^{\sigma}\|^2_{\dot{H}^{-1}} \,d\tau 
			\leq   C(t,\nu,\sigma,\Gamma^{\sigma}_0),
			\\ \label{vEB_L2Hs}
			&\|\Gamma^{\sigma}(t)\|^2_{X^s} 
			+ \int^t_0 \|v^{\sigma}\|^2_{H^{s+\frac{3}{2}}} + \|j^{\sigma}\|^2_{H^s \cap \dot{H}^{-1}} \,d\tau 
			\leq  C(T,\nu,\sigma,s,\Gamma^{\sigma}_0).
		\end{align}
		Furthermore, by defining $A^{\sigma}$ such that $\nabla \times A^{\sigma} = B^{\sigma}$ and $\textnormal{div}\, A^{\sigma} = 0$ ($A^{\sigma}$ are not unique), using \eqref{vEB_L2}-\eqref{EB_Hsd} and \cite[Proposition 1.36]{Bahouri-Chemin-Danchin_2011}, it follows that for $t \in (0,T)$ 
		\begin{align} \nonumber
			\mathcal{H}^{\sigma}(t) &:= \int_{\mathbb{R}^3} A^{\sigma}(t) \cdot B^{\sigma}(t)  \,dx 
			\leq C\|A^{\sigma}(t)\|_{\dot{H}^{\frac{1}{2}}} \|B^{\sigma}(t)\|_{\dot{H}^{-\frac{1}{2}}} \leq C \|B^{\sigma}(t)\|^2_{\dot{H}^{-\frac{1}{2}}}
			\\ \label{AB_csigma}
			&\leq C  \|B^{\sigma}(t)\|_{\dot{H}^{-1}} \|B^{\sigma}(t)\|_{L^2}
			\leq  C(T,\nu,\sigma,\Gamma^{\sigma}_0).
		\end{align}
		In addition,  \eqref{EB_Hsd}-\eqref{AB_csigma}, Tonelli and Fubini (twice and we need the estimate of $\|B^\sigma\|_{\dot{H}^{-1}}$ here\footnote{In fact, $\mathcal{H}^\sigma(t)$ is well-defined if $\|B^\sigma(t)\|_{\dot{H}^{-\frac{1}{2}}}$ is finite for $t \in (0,\infty)$, which is possible if we consider the $\dot{H}^{-\frac{1}{2}}$ estimate instead of the $\dot{H}^{-1}$ one in \eqref{EB_n_Hsd} for $(E^\sigma_0,B^\sigma_0) \in \dot{H}^{-\frac{1}{2}}$.}) Theorems (see \cite{Brezis_2011}, to compute the weak derivative in the first line below since $s' \in (\frac{3}{2},s)$), 
		and the limiting system for $\Gamma^{\sigma}$ yield
		\begin{align} \nonumber
			\frac{d}{dt} \mathcal{H}^{\sigma}(t) 
			&= \int_{\mathbb{R}^3} \partial_t A^{\sigma} \cdot B^{\sigma} + A^{\sigma} \cdot \partial_t B^{\sigma} \,dx = 2 \int_{\mathbb{R}^3} A^{\sigma} \cdot \partial_t B^{\sigma} \,dx 
			\\\nonumber
			&= -2 \int_{\mathbb{R}^3} B^{\sigma} \cdot cE^{\sigma} \,dx
			\\\nonumber
			&=-2\int_{\mathbb{R}^3} B^{\sigma} \cdot \left(\frac{1}{\sigma} j^{\sigma} -  (v^{\sigma} \times B^{\sigma})\right)\,dx
			\\  \label{AB_csigma_1}
			&
			= -2 \int_{\mathbb{R}^3} \frac{1}{\sigma} j^{\sigma} \cdot B^{\sigma}\,dx.
		\end{align}
		Integrating in time, we find from \eqref{vEB_L2} and \eqref{AB_csigma}-\eqref{AB_csigma_1} that for $\tau \in (0,T)$ (see \cite{Brezis_2011})
		\begin{align*}
			\left|\mathcal{H}^{\sigma}(\tau) - \mathcal{H}^{\sigma}(0)\right| &= \left|\int^{\tau}_0 \frac{d}{dt} \mathcal{H}^{\sigma}(t) \,dt\right| 
			\\
			&= \left|-2 \int^{\tau}_0 \int_{\mathbb{R}^3} \frac{1}{\sigma} j^{\sigma} \cdot B^{\sigma}\,dx \,dt \right|
			\\
			&\leq \sigma^{-\frac{1}{2}} \left(\frac{1}{\sigma} \int^{\tau}_0 \|j^{\sigma}\|^2_{L^2} \,dt + \int^{\tau}_0 \|B^{\sigma}\|^2_{L^2}\,dt \right)
			\\
			&
			\leq \sigma^{-\frac{1}{2}} (\tau + 1)\|\Gamma^{\sigma}_0\|^2_{L^2},
		\end{align*}
		which after taking $\sigma \to \infty$ implies that for a.e. $t \in (0,T)$
		\begin{equation*}
			\lim_{\sigma \to \infty} \int_{\mathbb{R}^3} A^{\sigma}(t) \cdot B^{\sigma}(t)  \,dx 
			= \int_{\mathbb{R}^3} A_0 \cdot B_0 \,dx,
		\end{equation*}
		since as in \eqref{AB_csigma}, by using $B^{\sigma}_0 \to B_0$ in $\dot{H}^{-1}$ as $\sigma \to \infty$ with 
		\begin{align*}
			\left|\int_{\mathbb{R}^3} A^{\sigma}_0 \cdot B^{\sigma}_0  - A_0 \cdot B_0 \,dx \right| &= \left|\int_{\mathbb{R}^3} (A^{\sigma}_0 + A_0) \cdot (B^{\sigma}_0 - B_0) \,dx\right|
			\\
			&\leq C \|B^{\sigma}_0 + B_0\|_{L^2}\|B^{\sigma}_0 - B_0\|_{\dot{H}^{-1}}
			\quad 
			\to 0 \quad \text{as} \quad \sigma \to \infty.
		\end{align*}
		Finally, it follows from the above limit and \eqref{AB_csigma} that if the initial magnetic helicity is positive then there exists an absolute positive constant $C$ such that
		\begin{equation*}
			\liminf_{t\to \infty} \liminf_{\sigma \to \infty} \|B^{\sigma}(t)\|^2_{\dot{H}^{-\frac{1}{2}}} \geq  C\liminf_{t\to \infty} \lim_{\sigma \to \infty} \int_{\mathbb{R}^3} A^{\sigma}(t) \cdot B^{\sigma}(t)  \,dx 
			= C\int_{\mathbb{R}^3} A_0 \cdot B_0 \,dx > 0.
		\end{equation*}
		Thus, the proof is complete.
		
	\end{proof}

	%
	\section{Proof of Theorem \ref{theo-inviscid}} \label{sec:inv}
	%
	
	In this subsection, we focus on giving the standard proof of Theorem \ref{theo-inviscid}, which shares similar ideas as those of Theorems \ref{theo1} and \ref{theo1-3d}.
	
	\begin{proof}[Proof of Theorem \ref{theo-inviscid}] The proof consists of several steps as follows.
		
		\textbf{Step 1: Local existence, $H^s$ estimate and uniform bound.} As the proofs of Theorems \ref{theo1} and \ref{theo1-3d}, we will use \eqref{NSM_app} with $\nu = 0$ as an approximate system. It can be seen from \eqref{NSM_app} with $\nu = 0$ that
		\begin{equation*}
			\frac{1}{2} \frac{d}{dt} \|(v^n,E^n,B^n)\|^2_{H^s} + \frac{1}{\sigma}\|j^n\|^2_{H^s} =: \sum^3_{k=1} J_k,
		\end{equation*}
		where for some $\epsilon \in (0,1)$, since $s > \frac{d}{2} + 1$
		\begin{align*}
			J_1 + J_3 &= \int_{\mathbb{R}^d} J^s(j^n \times B^n) \cdot J^s v^n + J^s j^n \cdot J^s (v^n \times B^n) \,dx
			\leq \frac{\epsilon}{\sigma}\|j^n\|^2_{H^s} + C(d,\epsilon,\sigma,s)\left(\|(v^n,B^n)\|^2_{H^s}\right)^2;
			\\
			J_2 &= - \int_{\mathbb{R}^d} [J^s(v^n \cdot \nabla v^n) - v^n \cdot \nabla J^s v^n] \cdot J^s v^n \,dx
			\leq C(s)\|v^n\|^3_{H^s},
		\end{align*}
		here we used the following well-known Kato-Ponce commutator estimate (see \cite{Kato-Ponce_1988}) 
		\begin{equation*} 
			\|J^r(fg) - f J^r g\|_{L^2} \leq C(d,r)\left(\|J^r f\|_{L^2} \|g\|_{L^\infty} +  \|\nabla f\|_{L^\infty}\|J^{r-1} g\|_{L^2}\right) \qquad \forall r > 0.
		\end{equation*} 
		By choosing $\epsilon = \frac{1}{2}$, we find that
		\begin{equation*}
			\frac{d}{dt}Y_{n,s}  + \frac{1}{\sigma}\|j^n\|^2_{H^s} \leq C(\sigma,s)Y^2_{n,s},
		\end{equation*}
		where $Y_{n,s}(t) := \|(v^n,E^n,B^n)(t)\|^2_{H^s} + 1$ for $t \in (0,T^n_*)$. It can be seen that the above estimate implies an uniform bound in terms of $n$ for $Y_{n,s}$ in $(0,T_0)$ for some $T_0 = T_0(d,\sigma,s,v_0,E_0,B_0) > 0$ (does not depend on $n$) and for $t \in (0,T_0)$
		\begin{equation*}
			\|(v^n,E^n,B^n)(t)\|^2_{H^s} + \int^t_0 \|j^n\|^2_{H^s}\,d\tau \leq
			C(T_0,d,\sigma,s,v_0,E_0,B_0).
		\end{equation*}
		
		\textbf{Step 2: Pass to the limit.} In this case, since $\nu = 0$ and $\delta = s > \frac{d}{2} + 1$, we need to modify the estimates of $I_{42}$ and $I_{53}$ in Step 16 in the proof of Theorem \ref{theo1} in the following way (for $I_{41}, I_{51}, I_{61}$ and $I_{63}$, we replace $\mathbb{R}^2$ by $\mathbb{R}^d$ with using the same estimates)
		\begin{align*}
			I_{42} &= -\int_{\mathbb{R}^d} T_m((v^n-v^m) \cdot \nabla v^n) \cdot (v^n-v^m)\,dx
			\leq \|\nabla v^n\|_{L^\infty}\|v^n-v^m\|^2_{L^2};
			\\
			I_{53} &= \int_{\mathbb{R}^d} T_m(j^m \times (B^n-B^m)) \cdot (v^n-v^m)\,dx \leq \|j^m\|_{L^\infty}\|(v^n-v^m,B^n-B^m)\|^2_{L^2},
		\end{align*}
		which shows that $(v^n,E^n,B^n)$ and $j^n$ are Cauchy sequences in $L^\infty(0,T_0;L^2(\mathbb{R}^d))$ and $L^2(0,T_0;L^2(\mathbb{R}^d))$ by using Step 1. Therefore, we can pass to the limit as in Step 17a in the proof of Theorem \ref{theo1} by replacing $\mathbb{R}^2$ by $\mathbb{R}^d$ with receiving the limit system \eqref{NSM_P} for $\nu = 0$. We skip further details.
		
		\textbf{Step 3: Uniqueness.} Assume that $(v,E,B,\pi)$ and $(\bar{v},\bar{E},\bar{B},\bar{\pi})$ are two solutions to \eqref{NSM} with $\nu = 0$ and the same initial data. As in Step 18 in the proof of Theorem \ref{theo1}, it follows that		
		\begin{align*}
			\frac{1}{2}\frac{d}{dt}\|(v-\bar{v},E-\bar{E},B-\bar{B})\|^2_{L^2} + \frac{1}{\sigma}\|j-\bar{j}\|^2_{L^2} =: \sum^3_{k=1} \bar{I}_k,
		\end{align*}
		where for some $\epsilon \in (0,1)$
		\begin{align*}
			\bar{I}_1 &= -\int_{\mathbb{R}^d} (v-\bar{v}) \cdot \nabla v \cdot (v-\bar{v}) \,dx \leq \|\nabla v\|_{L^\infty}\|v-\bar{v}\|^2_{L^2};
			\\
			\bar{I}_2 &= \int_{\mathbb{R}^d} (\bar{j} \times (B-\bar{B})) \cdot (v-\bar{v})\,dx \leq \|\bar{j}\|_{L^\infty}\|(v-\bar{v},B-\bar{B})\|^2_{L^2};
			\\
			\bar{I}_3 &= \int_{\mathbb{R}^d} (j-\bar{j}) \cdot (\bar{v} \times (B-\bar{B}))\,dx \leq \frac{\epsilon}{\sigma}\|j-\bar{j}\|^2_{L^2} + C(\epsilon,\sigma)\|\bar{v}\|^2_{L^\infty}\|B-\bar{B}\|^2_{L^2},
		\end{align*}
		which yields for $\epsilon = \frac{1}{2}$
		\begin{align*}
			\frac{d}{dt}\|(v-\bar{v},E-\bar{E},B-\bar{B})\|^2_{L^2} + \frac{1}{\sigma}\|j-\bar{j}\|^2_{L^2} \leq C(\sigma)\left(\|(\nabla v,\bar{j})\|_{L^\infty} + \|\bar{v}\|^2_{L^\infty}\right)\|(v-\bar{v},B-\bar{B})\|^2_{L^2}.
		\end{align*}
		Therefore, the uniqueness follows by using Step 1.
		
		\textbf{Step 4: Inviscid limit.} We should remark here in the two-dimensional case that given $T_0 > 0$ then for any $\nu > 0$, there exists a unique solution $(v^\nu,E^\nu,B^\nu)$ to \eqref{NSM} given in Theorem \ref{theo1} in $(0,T_0)$. If $d = 3$ then an application of Steps 1, 2 and 3 above gives us the local existence and uniqueness of $(v^\nu,E^\nu,B^\nu)$ to \eqref{NSM} with $\nu > 0$ and $\alpha = 1$ in the same time interval $(0,T_0)$. Therefore, $T_0$ does not depends on $\nu$. Let $(v^\nu,E^\nu,B^\nu,\pi^\nu)$ and $(v,E,B,\pi)$ be the corresponding solutions to \eqref{NSM} with $\nu > 0$ and $\nu = 0$ satisfying $(v^\nu,E^\nu,B^\nu)_{|_{t=0}} = (v,E,B)_{|_{t=0}} = (v_0,E_0,B_0)$. Similar to the proof of uniqueness in the previous step with replacing $(v,E,B,\pi,j)$ and $(\bar{v},\bar{E},\bar{B},\bar{\pi},\bar{j})$ by $(v^\nu,E^\nu,B^\nu,\pi^\nu,j^\nu)$ and $(v,E,B,\pi,j)$, respectively, there are two additional terms, one on the left-hand side of the energy equality related to the viscosity and the other one on the right-hand side denoted by $\bar{I}_4$.
		We will bound $\bar{I}_4$ and also need to modify the estimate of $\bar{I}_1$ as follows 
		\begin{align*}
			\bar{I}_1 &= -\int_{\mathbb{R}^d} (v^\nu-v) \cdot \nabla v^\nu \cdot (v^\nu-v) \,dx 
			=-\int_{\mathbb{R}^d} (v^\nu-v) \cdot \nabla v \cdot (v^\nu-v) \,dx 
			\leq \|\nabla v\|_{L^\infty}\|v^\nu-v\|^2_{L^2};
			\\
			\bar{I}_4 &:= \nu\int_{\mathbb{R}^d} \Delta v \cdot (v^\nu-v) \,dx \leq \nu^2 \|\Delta v\|^2_{L^2} + \|v^\nu-v\|^2_{L^2},
		\end{align*}
		in which we find that for $Y^\nu(t) := \|(v^\nu-v,E^\nu-E,B^\nu-B)(t)\|^2_{L^2}$ with $t \in (0,T_0)$
		\begin{equation*}
			\frac{d}{dt}Y^\nu + \nu \|\nabla(v^\nu-v)\|^2_{L^2} + \frac{1}{\sigma}\|j^\nu-j\|^2_{L^2} 
			\leq C(\sigma)\left(\|(\nabla v,|v|^2,j)\|_{L^\infty} + 1\right)Y^\nu + \nu^2 \|\Delta v\|^2_{L^2},
		\end{equation*}
		which yields 
		\begin{equation*}
			Y^\nu(t)
			\leq \nu^2 \int^t_0 \|\Delta v\|^2_{L^2} \,d\tau \exp\left\{C(\sigma)\int^t_0 \|(\nabla v,|v|^2,j)\|_{L^\infty}  + 1 \,d\tau \right\}.
		\end{equation*}
		By using Step 1, for $s' \in [0,s)$
		\begin{equation*}
			\|(v^\nu-v)(t)\|_{H^{s'}} \leq \|(v^\nu-v)(t)\|^\frac{s-s'}{s}_{L^2}\|(v^\nu-v)(t)\|^\frac{s'}{s}_{H^s} \leq \nu^\frac{s-s'}{s}C(T_0,d,\sigma,s,v_0,E_0,B_0),
		\end{equation*}
		which is similarly for $(E^\nu-E,B^\nu-B)$ and gives us the conclusion. In addition, the bound conthe right-hand side does not depend on $\nu$ since during the proof we do not use any bounds on $(v^\nu,E^\nu,B^\nu,j^\nu)$, but only ones on $(v,E,B,j)$.
		
		\textbf{Step 5: The limit $c \to \infty$.} 
		It can be seen from \eqref{NSM} with $\nu = 0$ that
		\begin{equation} \label{NSM-c}
			\left\{
			\begin{aligned}
				\partial_t v^c + v^c \cdot \nabla v^c + \nabla \pi^c &= j^c \times B^c, 
				\\
				\frac{1}{c}\partial_t E^c - \nabla \times B^c &= -j^c,
				\\
				\partial_t B^c - \nabla \times (v^c \times B^c) &= -\frac{1}{\sigma} \nabla \times j^c,
				\\
				\textnormal{div}\, v^c = \textnormal{div}\, B^c &= 0.
			\end{aligned}
			\right.
		\end{equation}
		By applying Step 1, we know that the local solution $(v^c,E^c,B^c,j^c)$ is uniformly bounded in terms of $c$ in the following spaces
		\begin{equation*}
			(v^c,E^c,B^c) \in L^\infty(0,T_0;H^s) \quad \text{and} \quad j^c \in L^2(0,T_0;H^s),
		\end{equation*}
		which implies that there exists $(v,E,B,j)$ such that up to an extraction of a subsequence (use the same notation) as $c \to \infty$
		\begin{align*}
			&&(v^c,E^c,B^c) &\overset{\ast}{\rightharpoonup}  (v,E,B) &&\text{in} \quad L^\infty_tH^s_x,&&
			\\
			&&j^c &\rightharpoonup j &&\text{in} \quad L^2_tH^s_x.&&
		\end{align*}
		In addition, we find from \eqref{NSM-c} that
		\begin{align*}
			(\partial_t v^c,\partial_tB^c) \quad \text{is uniformly bounded in} \quad L^2_tH^{s-1}_x
		\end{align*}
		and by using the Aubin-Lions lemma as $c \to \infty$
		\begin{equation*}
			(v^c,B^c) \to (v,B) \quad \text{(locally in space) in}\quad L^2_{t,x}.
		\end{equation*}
		As in Step 17b in the proof of Theorem \ref{theo1}, for $\phi,\varphi \in C^\infty_0([0,T_0) \times \mathbb{R}^d;\mathbb{R}^3)$ with $\textnormal{div}\,\phi = 0$, the weak form of \eqref{NSM-c} is given by (similar to those of $a)$, $b)$ and $c)$)
		\begin{align*}
			a'')\quad &\int^{T_0}_0 \int_{\mathbb{R}^d} v^c \cdot \partial_t\phi + (v^c \otimes v^c) : \nabla \phi + (j^c \times B^c) \cdot \phi \,dxdt = -\int_{\mathbb{R}^d} v^c(0) \cdot \phi(0) \,dx,
			\\
			b'')\quad &\int^{T_0}_0 \int_{\mathbb{R}^d} \frac{1}{c}E^c \cdot \partial_t\varphi + B^c \cdot (\nabla \times \varphi) - j^c \cdot \varphi \,dxdt = -\int_{\mathbb{R}^d} \frac{1}{c}E^c(0) \cdot \varphi(0) \,dx,
			\\
			c'')\quad &\int^{T_0}_0 \int_{\mathbb{R}^d} B^c \cdot \partial_t\varphi + [(v^c \times B^c) - \frac{1}{\sigma} j^c] \cdot (\nabla \times \varphi) \,dxdt  = - \int_{\mathbb{R}^d} B^c(0) \cdot \varphi(0) \,dx.
		\end{align*}
		Therefore, we can pass to the limit by using the weak convergence of $(v^c_0,E^c_0,B^c_0)$ to $(\bar{v}_0,\bar{E}_0,\bar{B}_0)$ in $H^s$ and the above strong convergences as $c \to \infty$ to obtain that \eqref{NSM-c} converges in the sense of distributions to 
		\begin{equation*} 
			\left\{
			\begin{aligned}
				\partial_t v + v \cdot \nabla v + \nabla \pi &=  j \times B, 
				\\
				\partial_t B - \nabla \times (v \times B) &= -\frac{1}{\sigma} \nabla \times j,
				\\
				\textnormal{div}\, v = \textnormal{div}\, B &= 0,
			\end{aligned}
			\right.
		\end{equation*}
		where $j = \nabla \times B$ and $(v,B)_{|_{t=0}} = (\bar{v}_0,\bar{B}_0)$. Thus, the proof is finished since $\nabla \times (\nabla \times B)  
		= - \Delta B$.
	\end{proof}
	
	%
	\section{Proof of Theorem \ref{theo3}} \label{sec:theo3}
	%
	
	In this section, we will provide a proof of Theorem \ref{theo3}. The proof shares similar ideas to those of the previous sections. However, some modifications are needed due to the appearance of new terms, which are related to the constant magnetic vector $B^*$.
	
	\begin{proof}[Proof of Theorem \ref{theo3}-(i)]
		The proof contains several steps as follows.
		
		\textbf{Step 1: Local existence.} As mentioned previously, since $B^*$ is a constant vector in $\mathbb{R}^3$ then $\nabla \times B^* = E^* = \nabla \pi^* = 0$ in \eqref{NSM*}. We will use an approximate system of \eqref{NSM*}, which is a slightly modification of \eqref{NSM_app}, where $j^n$ is replaced by $\bar{j}^n$ and $F^n_1$ is redefined as follows 
		\begin{align*}
			j^n_* &= \sigma T_n(v^n \times B^*), 
			\\
			\bar{j}^n &= j^n + j^n_* = \sigma (cE^n + T_n(v^n \times (B^n+B^*))),  
			\\
			F^n_1 &= -\nu v^n -\mathbb{P}(T_n(v^n \cdot \nabla v^n)) + \mathbb{P}(T_n(\bar{j}^n \times (B^n+B^*))).
		\end{align*}
		Therefore, similar to the proof of Theorem \ref{theo1}, there exists a unique solution $(v^n,E^n,B^n) \in C^1([0,T^n_*),V^s_n \times H^s_n \times V^s_n)$ for some $T^n_* > 0$ satisfying the following property: if $T^n_* < \infty$ then
		\begin{equation*}
			\lim_{t\to T^n_*} \|(v^n,E^n,B^n)(t)\|^2_{H^s} = \infty.
		\end{equation*}
		
		\textbf{Step 2: $H^s$ estimate.} Assume that $T^n_* < \infty$. The energy balance is given by
		\begin{equation*}
			\frac{1}{2}\frac{d}{dt}\|(v^n,E^n,B^n)\|^2_{L^2} + \nu \|v^n\|^2_{L^2} + \frac{1}{\sigma}\|\bar{j}^n\|^2_{L^2} = 0.
		\end{equation*}
		In addition, the $H^s$ estimate is
		\begin{equation*}
			\frac{1}{2}\frac{d}{dt}\|(v^n,E^n,B^n)\|^2_{H^s} + \nu\|v^n\|^2_{H^s} + \frac{1}{\sigma}\|\bar{j}^n\|^2_{H^s} =: \sum^5_{i=1}J_i,
		\end{equation*}
		where for some $\epsilon \in (0,1)$, since $s > \frac{d}{2} + 1$
		\begin{align*}
			J_1 &= \int_{\mathbb{R}^d} J^s(\bar{j}^n \times B^n) \cdot J^s v^n \,dx 
			+ \int_{\mathbb{R}^d} J^s(\bar{j}^n \times  B^*) \cdot J^s v^n\,dx =: J_{11} + J_{12},
			\\
			J_{11} &\leq \frac{\epsilon}{\sigma}\|\bar{j}^n\|^2_{H^s} +   
			C(\epsilon,\sigma,s)\|B^n\|^2_{H^s}\|v^n\|^2_{H^s},
			\\
			J_{12} &= \int_{\mathbb{R}^d} (J^s\bar{j}^n \times  B^*) \cdot J^s v^n\,dx;
			\\
			J_2 &= -\int_{\mathbb{R}^d} [J^s(v^n \cdot \nabla v^n) - v^n \cdot \nabla J^s v^n] \cdot J^s v^n \,dx \leq C(s)\|v^n\|^3_{H^s};
			\\
			J_3 &= \int_{\mathbb{R}^d} J^s \bar{j}^n \cdot  J^s(v^n \times B^n) \,dx + \int_{\mathbb{R}^d} J^s \bar{j}^n \cdot  J^s(v^n \times B^*) \,dx =: J_{31} + J_{32},
			\\
			J_{31} &\leq  \frac{\epsilon}{\sigma}\|\bar{j}^n\|^2_{H^s} + 
			C(\epsilon,\sigma,s)\|v^n\|^2_{H^s}\|B^n\|^2_{H^s},
			\\
			J_{32} &= \int_{\mathbb{R}^d} J^s \bar{j}^n \cdot  (J^sv^n \times B^*) \,dx = -J_{12};
			\\
			J_4 &= \int_{\mathbb{R}^d} J^s(\nabla \times B^n) \cdot J^s(c E^n) \,dx;
			\\
			J_5 &= -\int_{\mathbb{R}^d} J^s(\nabla \times E^n) \cdot J^s(c B^n) \,dx = -J_4.
		\end{align*}
		Therefore, by choosing $\epsilon = \frac{1}{4}$
		\begin{equation*}
			\frac{d}{dt}\|(v^n,E^n,B^n)\|^2_{H^s} + \nu\|v^n\|^2_{H^s} + \frac{1}{\sigma}\|\bar{j}^n\|^2_{H^s} \leq C(s)\|v^n\|^3_{H^s} +  C(s)\|v^n\|^2_{H^s}\|B^n\|^2_{H^s}.
		\end{equation*}
		
		\textbf{Step 3: Bootstrap argument.} By defining the following energy form for $t \geq 0$
		\begin{equation*} 
			E_n(t) := \esssup_{\tau \in [0,t]} \|(v^n,E^n,B^n)(\tau)\|^2_{H^s} + \int^t_0 \nu\|v^n\|^2_{H^s} +\frac{1}{\sigma}\|\bar{j}^n\|^2_{H^s} \,d\tau,
		\end{equation*}
		it follows that for some fixed positive constants $C_1 = C_1(\nu,s)$ and $C_2 = C_2(\nu,\sigma,s)$, and for $t \in (0,T^n_*)$
		\begin{equation} \label{BA}
			E_n(t) \leq E_n(0) + C_1E^\frac{3}{2}_n(t) + C_2E^2_n(t).
		\end{equation}
		To the end of this step, we aim to prove the following property. \textbf{Claim:} Let $S^n := \{t \in (0,T^n_*) : E_n(t) \leq 2\epsilon^2_0\}$. Then $S^n = (0,T^n_*)$ and $T^n_* = \infty$.
		
		\textit{3a) Hypothesis implies conclusion.} Assume that for some $t \in (0,T^n_*)$
		\begin{equation} \label{BA1}
			E_n(t) \leq \min\left\{\frac{1}{16C^2_1},\frac{1}{4C_2}\right\} =: \frac{1}{C^2_0}.
		\end{equation}
		Therefore, by choosing $\epsilon_0 > 0$ such that $2C_0\epsilon_0 \leq 1$, it follows from  \eqref{BA} and \eqref{BA1}  that 
		\begin{equation} \label{BA2}
			E_n(t) \leq 2E_n(0) \leq 2\epsilon^2_0 \leq \frac{1}{2C^2_0}.
		\end{equation}
		
		\textit{3b) Conclusion is stronger than hypothesis.} Assume that \eqref{BA2} holds for some $t_0 \in (0,T^n_*)$. For a given $\delta_0 > 0$, by the continuity in time of $(v^n,E^n,B^n)$ in $H^s$, there exists a small $t_{\delta_0}$ such that
		\begin{equation*}
			E_n(t) < E_n(t_0) + \delta_0 \leq \frac{1}{2C^2_0} + \delta_0 \qquad \forall t \in (t_0-t_{\delta_0},t_0+t_{\delta_0}),
		\end{equation*}
		which yields \eqref{BA1} if we choose $\delta_0 \leq \frac{1}{2C_0^2}$.
		
		\textit{3c) Conclusion is closed.} Let $t_m$ and $t$ in $(0,T^n_*)$ such that $t_m \to t$ as $m \to \infty$. If $E_n(t_m) \leq 2\epsilon^2_0$ for all $m \in \mathbb{N}$ then by the continuity in time of $(v^n,E^n,B^n)$ in $H^s$, $E_n(t) \leq 2\epsilon^2_0$ as well.
		
		\textit{3d) Base case.} By the continuity in time of $(v_n,E_n,B_n)$ in $H^s$, we can find some $T^n_{**} \in (0,T^n_*)$ 
		\begin{equation*} 
			E_n(t) \leq 2E_n(0) \leq 2\epsilon^2_0 \leq \frac{1}{2C^2_0} \qquad \forall t \in (0,T^n_{**}).
		\end{equation*}
		This implies that $S^n$ is a non-empty set. We then apply the abstract bootstrap principle (see \cite[Proposition 1.21]{Tao_2006}) to obtain the first part of the claim, while the second part follows immediately by using Step 1. Moreover, for $t \geq 0$
		\begin{equation} \label{BA3} 
			\|(v^n,E^n,B^n)(t)\|^2_{H^s} + \int^t_0 \nu\|v^n\|^2_{H^s} +\frac{1}{\sigma}\|\bar{j}^n\|^2_{H^s} \,d\tau \leq 2\epsilon^2_0,
		\end{equation}
		which will be used to obtain additional estimate of $(E^n,B^n)$ as follows. We now focus on the Maxwell part in the approximate system. By applying $\Lambda^{s'-1}$ for $s' \in [1,s]$ and testing the result by $-(\Lambda^{s'-1} \nabla \times B^n, \Lambda^{s'-1} \nabla \times E^n)$, we find that for $\tau > 0$
		\begin{align*}
			\int^{\tau}_0 \|\Lambda^{s'-1} \nabla \times B^n\|^2_{L^2} \,dt 
			&= \frac{1}{c} \int^{\tau}_0 \frac{d}{dt} \int_{\mathbb{R}^d} \Lambda^{s'-1} E^n \cdot \Lambda^{s'-1} \nabla \times B^n \,dxdt 
			+  \|\Lambda^{s'-1} \nabla \times E^n\|^2_{L^2(0,t;L^2)} 
			\\
			&\quad  + \int^{\tau}_0 \int_{\mathbb{R}^d} \Lambda^{s'-1} j^2 \cdot \Lambda^{s'-1}\nabla \times B^n \,dxdt.
		\end{align*} 
		It is needed to bound the second term on the right-hand side. In order to do that, we first apply $\Lambda^{s'-1} \nabla \times $ to the Maxwell system, and test the result by $c(\Lambda^{s'-1}\nabla \times E^n, \Lambda^{s'-1}\nabla \times B^n)$, which leads to by using \eqref{BA3} 
		\begin{equation*}
			\|(\Lambda^{s'-1} \nabla \times E^n,\Lambda^{s'-1} \nabla \times B^n)(t)\|^2_{L^2} + c^2\sigma \|\Lambda^{s'-1} \nabla \times E^n\|^2_{L^2(0,t;L^2)} \leq C(B^*,d,\nu,\sigma,s) \epsilon^2_0,
		\end{equation*}
		which yields the following estimate of $B$ for $t > 0$ by using $\textnormal{div}\, B = 0$ and \eqref{BA3} 
		\begin{equation*}
			\|B^n\|^2_{L^2(0,t;\dot{H}^{s'})} \leq (c^{-1} + c^{-2} + 1) C(B^*,d,\nu,\sigma,s) \epsilon^2_0.
		\end{equation*}
		In addition, for $s'' \in [0,s]$ and $t > 0$, by using \eqref{BA3} again, it follows that
		\begin{equation*}
			\|(E^n,B^n)(t)\|^2_{\dot{H}^{s''}} + c^2\sigma \|E^n\|^2_{L^2(0,t;\dot{H}^{s''})} \leq C(B^*,d,\nu,\sigma,s) \epsilon^2_0.
		\end{equation*}

		\textbf{Step 4: Cauchy sequence, pass to the limit and uniqueness.} Assume that $(v^n,E^n,B^n)$ and $(v^m,E^m,B^m)$ for $m,n \in \mathbb{R}$ with $m > n > 0$ are two solutions to the approximate system with the same initial data. Therefore, it follows that 
		\begin{equation*}
			\frac{1}{2}\frac{d}{dt} \|(v^n-v^m,E^n-E^m,B^n-B^m)\|^2_{L^2} + \nu\|v^n-v^m\|^2_{L^2} + \frac{1}{\sigma}\|\bar{j}^n-\bar{j}^m\|^2_{L^2} =: \sum^6_{k=4} I_k,
		\end{equation*}
		where for some $\epsilon \in (0,1)$, since $s > \frac{d}{2} + 1$ 
		\begin{align*}
			I_4 &= \int_{\mathbb{R}^d} (-T_n(v^n \cdot \nabla v^n) + T_m(v^m \cdot \nabla v^m)) \cdot (v^n-v^m)\,dx =: \sum^{3}_{k=1} I_{4k},
			\\
			I_{41} &= -\int_{\mathbb{R}^d} (T_n-T_m)(v^n \cdot \nabla v^n) \cdot (v^n-v^m)\,dx
			\leq C(s)n^{-(s-1)}\|v^n\|^2_{H^s}\|v^n-v^m\|_{L^2},
			\\
			I_{42} &= -\int_{\mathbb{R}^d} T_m((v^n-v^m) \cdot \nabla v^n) \cdot (v^n-v^m)\,dx
			\leq  
			4\epsilon\nu \|v^n-v^m\|^2_{L^2} + C(\epsilon,\nu)\|\nabla v^n\|^2_{L^\infty}\|v^n-v^m\|^2_{L^2},
			\\
			I_{43} &= -\int_{\mathbb{R}^d} T_m(v^m \cdot \nabla (v^n-v^m)) \cdot (v^n-v^m)\,dx = 0;
			\\
			I_5 &= \int_{\mathbb{R}^d} (T_n(\bar{j}^n \times (B^n+B^*)) - T_m(\bar{j}^m \times (B^m+B^*))) \cdot (v^n-v^m)\,dx =: \sum^{4}_{k=1} I_{5k},
			\\
			I_{51} &= \int_{\mathbb{R}^d} (T_n-T_m)(\bar{j}^n\times B^n) \cdot (v^n-v^m)\,dx \leq C(s)n^{-s}\|B^n\|_{H^s}\left(\|v^n-v^m\|^2_{L^2} + \|\bar{j}^n\|^2_{H^s}\right),
			\\
			I_{52} &= \int_{\mathbb{R}^d} T_m((\bar{j}^n-\bar{j}^m)\times B^n) \cdot (v^n-v^m)\,dx,
			\\
			I_{53} &= \int_{\mathbb{R}^d} T_m(\bar{j}^m \times (B^n-B^m)) \cdot (v^n-v^m)\,dx \leq  
			4\epsilon\nu \|v^n-v^m\|^2_{L^2} + C(\epsilon,\nu)\|\bar{j}^m\|^2_{L^\infty}\|B^n-B^m\|^2_{L^2},
			\\
			I_{54} &= \int_{\mathbb{R}^d} (T_n(\bar{j}^n \times B^*) - T_m(\bar{j}^m \times B^*)) \cdot (v^n-v^m)\,dx =: I_{541} + I_{542},
			\\
			I_{541} &= \int_{\mathbb{R}^d} (T_n-T_m)(\bar{j}^n \times B^*) \cdot (v^n-v^m)\,dx \leq C(s) n^{-s}\|B^*\|_{L^\infty}\left(\|\bar{j}^n\|^2_{H^s} + \|v^n-v^m\|^2_{L^2}\right),
			\\
			I_{542} &= \int_{\mathbb{R}^d} T_m((\bar{j}^n-\bar{j}^m) \times B^*) \cdot (v^n-v^m)\,dx; 
			\\
			I_6 &= -\int_{\mathbb{R}^d} (\bar{j}^n - \bar{j}^m) \cdot (-T_n(v^n \times (B^n+B^*)) + T_m(v^m \times (B^m+B^*))) \,dx =: \sum^{4}_{k=1} I_{6k},
			\\
			I_{61} &= \int_{\mathbb{R}^d} (\bar{j}^n-\bar{j}^m) \cdot (T_n-T_m)(v^n \times B^n) \,dx 
			\leq \frac{\epsilon}{\sigma}\|\bar{j}^n-\bar{j}^m\|^2_{L^2} + C(\epsilon,\sigma,s) n^{-2s}\|v^n\|^2_{H^s}\|B^n\|^2_{H^s},
			\\
			I_{62} &= \int_{\mathbb{R}^d} (\bar{j}^n-\bar{j}^m) \cdot T_m((v^n-v^m) \times B^n)\,dx = - I_{52},
			\\
			I_{63} &= \int_{\mathbb{R}^d} (\bar{j}^n-\bar{j}^m) \cdot T_m(v^m \times (B^n-B^m))\,dx \leq \frac{\epsilon}{\sigma}\|\bar{j}^n-\bar{j}^m\|^2_{L^2} + C(\epsilon,\sigma)\|v^m\|^2_{H^s}\|B^n-B^m\|^2_{L^2},
			\\
			I_{64} &= -\int_{\mathbb{R}^d} (\bar{j}^n - \bar{j}^m) \cdot (-T_n(v^n \times B^*) + T_m(v^m \times B^*))\,dx =: I_{641} + I_{642},
			\\
			I_{641} &= -\int_{\mathbb{R}^d} (\bar{j}^n - \bar{j}^m) \cdot (T_m-T_n)(v^n \times B^*)\,dx
			\leq C(s)n^{-s}\|B^*\|_{L^\infty}\left(\|\bar{j}^n - \bar{j}^m\|^2_{L^2} + \|v^m\|^2_{H^s}\right),
			\\
			I_{642} &= \int_{\mathbb{R}^d} (\bar{j}^n - \bar{j}^m) \cdot T_m((v^n-v^m) \times B^*)\,dx = -I_{542}.
		\end{align*}
		As Step 16 in the proof of Theorem \ref{theo1}, by choosing $\epsilon = \frac{1}{8}$, it follows that $(v^n,E^n,B^n)$ and $(v^n,\bar{j}^n)$ are Cauchy sequences in $L^\infty(0,\infty;L^2(\mathbb{R}^d))$ and $L^2(0,\infty;L^2(\mathbb{R}^d))$, respectively. Therefore, we can pass to the limit as in Step 17a in the proof of Theorem \ref{theo1} to obtain a limiting system, which is similar to \eqref{NSM_P} with replacing $\mathbb{P}(j \times B)$ and $j$ by $\mathbb{P}(\bar{j} \times (B+B^*))$ and $\bar{j}$, respectively, where $\bar{j} = \sigma(cE + v \times (B+B^*))$, i.e.,
		\begin{equation} \label{NSM_P1}
			\left\{
			\begin{aligned}
				\partial_t v + \mathbb{P}(v \cdot \nabla v) &= \nu \Delta v +  \mathbb{P}(\bar{j} \times (B+B^*)), 
				\\
				\frac{1}{c}\partial_t E - \nabla \times B  &=  -\bar{j}
				\\
				\frac{1}{c}\partial_t B + \nabla \times E &= 0
				\\
				\bar{j} &= \sigma(cE + v \times (B+B^*)),
				\\
				\textnormal{div}\, v = \textnormal{div}\, B &= 0.
			\end{aligned}
			\right.
		\end{equation}
		Moreover, the limiting solution $(v,E,B)$ satisfies for $t > 0$, $s'' \in [0,s]$ and $s' \in [1,s]$
		\begin{align} \label{BA3'} 
			&\|(v,E,B)(t)\|^2_{H^s} + \int^t_0 \nu\|v\|^2_{H^s} +\frac{1}{\sigma}\|\bar{j}\|^2_{H^s} \,d\tau \leq 2\epsilon^2_0,
			\\ \label{BA3''}
			& \int^t_0 \|E\|^2_{\dot{H}^{s''}} + \|B\|^2_{\dot{H}^{s'}} \,d\tau \leq \epsilon^2_0 
			\begin{cases}
				C(B^*,c,d,\nu,\sigma,s) &\text{if} \quad c \in (0,1),
				\\
				C(B^*,d,\nu,\sigma,s) &\text{if} \quad c  \geq 1.
			\end{cases}
		\end{align}
		We can also prove the uniqueness as in Step 3 in the proof of Theorem \ref{theo-inviscid}. We omit further details.

		\textbf{Step 5: Large-time behavior.} It can be seen from the Ohm's law in \eqref{NSM_P1} that for some $\epsilon \in (0,1)$
		\begin{align*}
			\sigma c \int^\infty_0 \|E\|^2_{L^2} \,d\tau &= \int^\infty_0 \int_{\mathbb{R}^d} (\bar{j} -\sigma(v \times (B+B^*))) \cdot E \,dxd\tau
			\\
			&\leq \frac{1}{c}C(\epsilon,\sigma) \int^\infty_0 (1 + \|(B,B^*)\|^2_{L^\infty} )\|(v,\bar{j})\|^2_{L^2} \,d\tau + 3\epsilon\sigma c \int^\infty_0 \|E\|^2_{L^2}\,d\tau,
		\end{align*}
		which yields by using \eqref{BA3'} and choosing $\epsilon = \frac{1}{6}$
		\begin{equation} \label{LTB1}
			\int^\infty_0 \|(v,E)\|^2_{L^2} \,d\tau \leq \frac{1}{c^2}C(\epsilon_*,\nu,\sigma,s)(\epsilon^2_0 + \epsilon^4_0) + C(\nu)\epsilon^2_0.
		\end{equation}
		In addition, we observe that $(v,E) \in C([0,T];L^2)$\footnote{It is after possibly being redefined on a set of measure zero.} since $(\partial_t v, \partial_t E) \in L^2(0,T;H^{-1})$ for any $T \in (0,\infty)$ (see \cite{Evans_2010,Temam_2001}), which implies by using \eqref{BA3'}  that
		\begin{equation*}
			\frac{1}{2}\frac{d}{dt} \|(v,E)\|^2_{L^2} + \nu\|v\|^2_{L^2} + \frac{1}{\sigma}\|\bar{j}\|^2_{L^2} = \int_{\mathbb{R}^d} (\nabla \times B) \cdot cE \,dx \leq 2c\epsilon^2_0.
		\end{equation*}
		Therefore, for $0 \leq t' < t < \infty$
		\begin{equation} \label{LTB2}
			\|(v,E)(t)\|^2_{L^2} - \|(v,E)(t')\|^2_{L^2} \leq 4c\epsilon^2_0 (t-t').
		\end{equation}
		By using \eqref{LTB1}-\eqref{LTB2}, it follows that $\|(v,E)(t)\|_{L^2} \to 0$ as $t \to \infty$ (see \cite[Lemma 2.3]{Lai-Wu-Zhong_2021}). As a consequence, we find that 
		\begin{align*}
			\|\bar{j}(t)\|^2_{L^2} &= \sigma\int_{\mathbb{R}^d} \bar{j}(t) \cdot (cE + v \times (B+B^*))(t) \,dx 
			\\
			&\leq  \frac{1}{2}\|\bar{j}(t)\|^2_{L^2} + \sigma^2 c^2\|E(t)\|^2_{L^2} + C(s)\sigma^2\|v(t)\|^2_{L^2}
			\|B(t)\|^2_{H^s} + C\sigma^2\|B^*\|^2_{L^\infty}\|v(t)\|^2_{L^2},
		\end{align*}
		which yields $\|\bar{j}(t)\|_{L^2} \to 0$ as $t \to \infty$. By using again \eqref{BA3'} and the above $L^2$ decay properties, for $s' \in [0,s)$ and $f \in \{v,E,\bar{j}\}$
		\begin{equation*}
			\|f(t)\|_{H^{s'}} \leq C(s)\|f(t)\|^\frac{s-s'}{s}_{L^2}\|f(t)\|^\frac{s'}{s}_{H^s} \quad \to 0 \quad \text{as} \quad t \to \infty,
		\end{equation*}
		where if $f \equiv \bar{j}$ then we also used the following estimate
		\begin{align*}
			\|\bar{j}(t)\|_{H^s} &\leq C(c,\sigma,s)(\|E(t)\|_{H^s} + \|v(t)\|_{H^s}\|B(t)\|_{H^s} + \|B^*\|_{L^\infty}\|v(t)\|_{H^s}) \leq C(c,\epsilon_0,\epsilon_*,\sigma,s).
		\end{align*}
		In addition, for $s' \in [0,s)$
		\begin{equation*}
			\|j(t)\|_{H^{s'}} \leq \|\bar{j}(t)\|_{H^{s'}} + \|j_*(t)\|_{H^{s'}} \quad\to 0 \quad \text{as} \quad t \to \infty.
		\end{equation*}
		As a consequence, for $f \in \{E,B,\bar{j}\}$ the following quantities for $s' \in [0,s)$
		\begin{equation*}
			\left|\int_{\mathbb{R}^d} J^{s'}(v(t)) \cdot J^{s'}(f(t)) \,dx\right| \leq \|v(t)\|_{H^{s'}}\|f(t)\|_{H^{s'}} \quad \to 0 \quad \text{as} \quad t \to \infty.
		\end{equation*}
		Furthermore, similar to the case of $(v,E)$ above, we also have $B \in C([0,T],L^2)$ for any $T \in (0,\infty)$, it follows that
		\begin{equation*}
			\frac{1}{2}\frac{d}{dt} \|(v,E,B)\|^2_{L^2} + \nu\|v\|^2_{L^2} + \frac{1}{\sigma}\|\bar{j}\|^2_{L^2} = 0,
		\end{equation*}
		which implies that $0 < f(t) := \|(v,E,B)(t)\|^2_{L^2} < \epsilon^2_0$ and $f(t)$ is a strictly decreasing function. By using the $L^2$ decay in time property of $(v,E)$, we find that $\|B(t)\|_{L^2} \to b_0$ as $t \to \infty$ for some constant $b_0 \in [0,\epsilon_0)$. Since $\partial_t B = -c\nabla \times E$ then $\|\partial_tB(t)\|_{H^{r-1}} = c\|\nabla \times E(t)\|_{H^{r-1}} \to 0$ as $t \to \infty$ for $r \in [1,s)$. In the sequel, we aim to prove that 
		\begin{equation} \label{d_t_E}
			\|\partial_t E(t)\|_{L^2} \to 0 \quad \text{as}\quad t \to \infty.
		\end{equation}
		Indeed, for any $T \in (0,\infty)$ we find from \eqref{BA3'} that 
		\begin{equation*}
			\left\{
				\begin{aligned}
					&& \bar{j} &= \sigma(cE + v \times (B+B^*)) &&\in L^2(0,\infty;H^s),&&
					\\
					&&\partial_t v &= -\mathbb{P}(v \cdot \nabla v) - \nu v + \mathbb{P}(\bar{j} \times (B+B^*)) &&\in L^2(0,\infty;H^{s-1}),&&
					\\
					&&\frac{1}{c}\partial_t E &=  \nabla \times B - \bar{j} &&\in L^2(0,T;H^{s-1}),&&
					\\
					&&\frac{1}{c}\partial_t B &= - \nabla \times E &&\in L^2(0,T;H^{s-1}),&&
					\\
					&&\frac{1}{c}\partial_{tt} E &= \nabla \times \partial_t B - c\sigma \partial_t E - \sigma(\partial_t v \times (B + B^*) + v \times \partial_t B) &&\in L^2(0,T;H^{s-2}).&&
				\end{aligned}
			\right.
		\end{equation*}	
		Furthermore, since $B \in L^2(0,T;H^s)$ and $\partial_t B \in L^2(0,T;H^{s-1})$ with $s - 1 > \frac{d}{2} \geq 1$ for any $T \in (0,\infty)$, it follows from \cite[Lemma 1.2, Chapter 3]{Temam_2001} that $B \in C([0,T];H^1)$ and since $v,E,B \in C([0,T];L^2)$
		\begin{equation} \label{d_t_E_1}
			\partial_t E(t) = c \nabla \times B(t) - c^2\sigma E(t) - c\sigma v(t) \times (B(t) + B^*) \quad \in C([0,T];L^2),
		\end{equation}
		which gives us the meaning for the value of $\partial_t E$ at $t = 0$ and suggests us to take
		\begin{equation} \label{d_t_E_2}
			\partial_t E_{|_{t=0}} = \left(c \nabla \times B - c^2\sigma E - c\sigma v \times (B + B^*)\right)_{|_{t=0}}.
		\end{equation}
		In addition, it can be seen that 
		\begin{equation*}
			\frac{1}{2}\frac{d}{dt} \|\partial_t E\|^2_{L^2} + c^2\sigma \|\partial_t E\|^2_{L^2} =: \sum^4_{k=1} R_k,
		\end{equation*}
		where for some $\epsilon \in (0,1)$, since $s > \frac{d}{2} + 1$
		\begin{align*}
			R_1 &= c\int_{\mathbb{R}^d} \nabla \times \partial_t B \cdot \partial_t E \,dx = -c\int_{\mathbb{R}^d} \nabla \times (\nabla \times cE) \cdot \partial_t E \,dx 
			\\
			&=  -c\int_{\mathbb{R}^d} \nabla \times (\nabla \times (\frac{1}{\sigma} \bar{j} - v \times (B + B^*))) \cdot \partial_t E \,dx
			\\
			&\leq 3\epsilon c^2\sigma \|\partial_t E\|^2_{L^2} + C(\epsilon,\sigma,s)\left(\|\bar{j}\|^2_{H^s} + \|v\|^2_{H^s}(\|B\|^2_{H^s} + \|B^*\|^2_{L^\infty})\right);
			\\
			R_2 &= -c\sigma \int_{\mathbb{R}^d} (\partial_t v \times B) \cdot \partial_t E \,dx
			\\
			&\leq \epsilon c^2\sigma \|\partial_t E\|^2_{L^2} + C(\epsilon,\sigma,s) \|B\|^2_{H^s} \|\partial_t v\|^2_{L^2}
			\\
			&\leq \epsilon c^2\sigma \|\partial_t E\|^2_{L^2} + C(\epsilon,\nu,\sigma,s)\|B\|^2_{H^s} \left(\|v\|^4_{H^s} +  \|v\|^2_{H^s} + \|\bar{j}\|^2_{H^s}(\|B\|^2_{H^s} + \|B^*\|^2_{L^\infty})\right);
			\\
			R_3 &= -c\sigma \int_{\mathbb{R}^d} (\partial_t v \times B^*) \cdot \partial_t E \,dx
			\\
			&\leq \epsilon c^2\sigma \|\partial_t E\|^2_{L^2} + C(\epsilon,\nu,\sigma,s)\|B^*\|^2_{L^\infty} \left(\|v\|^4_{H^s} +  \|v\|^2_{H^s} + \|\bar{j}\|^2_{H^s}(\|B\|^2_{H^s} + \|B^*\|^2_{L^\infty})\right);
			\\
			R_4 &= -c\sigma \int_{\mathbb{R}^d} (v \times \partial_t B) \cdot \partial_t E \,dx = c\sigma \int_{\mathbb{R}^d} (v \times (\nabla \times cE)) \cdot \partial_t E \,dx 
			\\
			&= c\sigma \int_{\mathbb{R}^d} (v \times (\nabla \times (\frac{1}{\sigma}\bar{j} - v \times (B + B^*)))) \cdot \partial_t E \,dx 
			\\
			&\leq 3\epsilon c^2\sigma \|\partial_t E\|^2_{L^2} + C(\epsilon,\sigma,s)\left(\|v\|^2_{H^s}\|\bar{j}\|^2_{H^s} + \|v\|^4_{H^s}(\|B\|^2_{H^s} + \|B^*\|^2_{L^\infty})\right),
		\end{align*}
		here in the estimates of $R_2$ and $R_3$, we employed the following fact
		\begin{equation*}
			\|\partial_t v(t)\|^2_{H^{s-1}} \leq C(\nu)\left(\|v(t)\|^4_{H^s} +  \|v(t)\|^2_{H^s} + \|\bar{j}(t)\|^2_{H^s}(\|B(t)\|^2_{H^s} + \|B^*\|^2_{L^\infty})\right) \qquad \text{for} \quad t > 0.
		\end{equation*}
		Therefore, by choosing $\epsilon = \frac{1}{16}$ and using \eqref{BA3'}, \eqref{d_t_E_1}-\eqref{d_t_E_2}, it follows that for $0 \leq t' < t < \infty$
		\begin{align*}
			\|\partial_t E(t)\|^2_{L^2} - \|\partial_t E(t')\|^2_{L^2} &\leq C(c,\epsilon_0,\epsilon_*,\sigma,s)(t-t'),
			\\
			\int^t_0 \|\partial_t E\|^2_{L^2} \,d\tau &\leq 
			C(c,\epsilon_0,\epsilon_*,\sigma,s),
		\end{align*}
		where we also used the following estimate
		\begin{align*}
			\|\bar{j}(t)\|^2_{H^s} &\leq C(c,\sigma,s)\left(\|E(t)\|^2_{H^s} + \|v(t)\|^2_{H^s}(\|B(t)\|^2_{H^s} + \|B^*\|^2_{L^\infty})\right) \leq C(c,\epsilon_*,\sigma,s)(\epsilon^2_0 + \epsilon^4_0).
		\end{align*}
		Thus, \eqref{d_t_E} follows by using \cite[Lemma 2.3]{Lai-Wu-Zhong_2021} again. Combining \eqref{d_t_E} and the decay in time of $\|\bar{j}\|_{L^2}$, we find that $\|\nabla \times B(t)\|_{L^2} = \|\nabla B(t)\|_{L^2} \to 0$ as $t \to \infty$. Therefore, for any $s' \in [0,s-1)$ and for some suitable $s'' > \frac{d}{2}$, by using \eqref{BA3'}, interpolation inequalities and Lemma \ref{lem-Agmon}, it follows that as $t \to \infty$
		\begin{align*}
			\|B(t)\|_{\dot{H}^{s'+1}} &\leq C(s') \|\nabla B(t)\|^\frac{s-1-s'}{s-1}_{L^2} \|\nabla B(t)\|^\frac{s'}{s-1}_{\dot{H}^{s-1}}  \quad \to 0,
			\\
			\|B(t)\|_{L^\infty} &\leq C(d,s'') \|B(t)\|^{1-\frac{d}{2s''}}_{L^2} \|B(t)\|^\frac{d}{2s''}_{\dot{H}^{s''}}  \qquad \to 0,
			\\
			\|B(t)\|_{L^p} &\leq \|B(t)\|^\frac{2}{p} \|B(t)\|^\frac{p-2}{p}_{L^\infty} \quad \to 0.
		\end{align*}
		As a consequence, $\|B(t)\|_{L^q_{\textnormal{loc}}} \to 0$ as $t \to \infty$ for $q \in [1,\infty]$. Finally, the convergences in time of $(\partial_t v,\partial_t E)$ follows from those of $(v,E,B,\bar{j})$, \eqref{BA3'} and the following estimates for $r \in [0,s-1)$
		\begin{align*}
			\|\partial_t v(t)\|_{H^r} &\leq \|v(t) \cdot \nabla v(t)\|_{H^r} + \nu \|v(t)\|_{H^r} + \|\bar{j}(t) \times (B(t) + B^*)\|_{H^r},
			\\
			\|\partial_t E(t)\|_{H^r} &\leq \|\nabla \times B(t)\|_{H^r} + \|\bar{j}(t)\|_{H^r}.
		\end{align*}
	\end{proof}
	
	\begin{proof}[Proof of Theorem \ref{theo3}-(ii)] We now focus on the limit as $c \to \infty$. It can be seen from  \eqref{NSM*} with  $\alpha = 0$ that 
		\begin{equation} \label{NSM-GO-c}
			\left\{
			\begin{aligned}				
				\partial_t v^c + v^c \cdot \nabla v^c + \nabla \pi^c &= -\nu v^c + \bar{j}^c \times (B^c+B^*), 
				\\
				\frac{1}{c}\partial_t E^c - \nabla \times B^c &= -\bar{j}^c,
				\\
				\partial_t B^c - \nabla \times (v^c \times (B^c+B^*)) &= -\frac{1}{\sigma} \nabla \times \bar{j}^c,
				\\
				\textnormal{div}\, v^c = \textnormal{div}\, B^c &= 0.
			\end{aligned}
			\right.
		\end{equation}
		By using the smallness condition on the initial data $(v^c_0,E^c_0,B^c_0)$, an application of Part $(i)$ gives us the existence of a sequence of global solutions $(v^c,E^c,B^c)$ to \eqref{NSM-GO-c} with some uniform bounds in terms of $c$. Therefore, this step can be done as Step 5 in the proof of Theorem \ref{theo-inviscid}, which by using the weak convergence of $(v^c_0,E^c_0,B^c_0)$ in $H^s$ and some mentioned vector identities implies that \eqref{NSM-GO-c} converges to \eqref{HMHD*} with $\kappa = 0$ in the sense of distributions as $c \to \infty$. Here, we note in addition that $\nabla \times (B + B^*) = \nabla \times B = \bar{j}$, where $B$ and $\bar{j}$ are the corresponding limits of  $B^c$ and $\bar{j}^c$ as $c \to \infty$, respectively. We skip further details and end the proof.
	\end{proof}
	
	%
	\section{Proof of Theorem \ref{pro-HMHD}} \label{sec:HMHD}
	%
	
	In this section, we will provide a simple proof of Theorem \ref{pro-HMHD}, which shares a similar idea as that of Theorem \ref{theo3}.
	
	\begin{proof}[Proof of Theorem \ref{pro-HMHD}]
		The proof is quite similar to that of Theorem \ref{theo3}, which contains several steps as follows. In this case, since $B^*$ is a constant vector in $\mathbb{R}^3$, we have $j^* = \nabla \times B^* = 0$. Thus, \eqref{HMHD*} is reduced to the following system
		\begin{equation*} 
			\left\{
			\begin{aligned}
				\partial_t v + v \cdot \nabla v + \nabla \pi &= -\nu v + j \times (B+B^*),
				\\
				\partial_t B - \nabla \times (v \times (B+B^*)) &= \frac{1}{\sigma}\Delta B - \frac{\kappa}{\sigma} \nabla \times (j \times (B+B^*)),
				\\
				\textnormal{div}\, v = \textnormal{div}\, B &= 0,
			\end{aligned}
			\right.
		\end{equation*} 
		which can be further equivalently rewritten as follows for $p^* := \pi + \frac{1}{2}|B+B^*|^2$
		\begin{equation} \label{HMHD**}
			\left\{
			\begin{aligned}
				\partial_t v + v \cdot \nabla v &= - \nabla p^* -\nu v + B \cdot \nabla B + B^* \cdot \nabla B,
				\\
				\partial_t B + v \cdot \nabla B &= \frac{1}{\sigma} \Delta B + B \cdot \nabla v + B^* \cdot \nabla v + \frac{\kappa}{\sigma}(j \cdot \nabla B - B \cdot \nabla j - B^* \cdot \nabla j),
				\\
				\textnormal{div}\, v = \textnormal{div}\, B &= 0,
			\end{aligned}
			\right.
		\end{equation}
		by using the following vector indentities for $f \in \{v,j\}$
		\begin{align*}
			(\nabla \times (B+B^*)) \times (B+B^*) &= (B+B^*) \cdot \nabla (B+B^*) - \frac{1}{2} \nabla |B+B^*|^2,
			\\
			\nabla \times (f \times (B+B^*)) &= - f \cdot \nabla (B+B^*) + (B+B^*) \cdot \nabla f.
		\end{align*}
		\textbf{Step 1: Local existence.} We will consider the following approximate system to \eqref{HMHD**}
		\begin{equation} \label{HMHD_app}
			\frac{d}{dt} (v^n,B^n) = (F^n_1,F^n_2)(v^n,B^n),
			\quad
			\textnormal{div}\, v^n = \textnormal{div}\, B^n = 0,
			\quad 
			(v^n,B^n)_{|_{t=0}} = T_n(v_0,B_0),
		\end{equation}
		where 
		\begin{align*}
			F^n_1 &= -T_n(\mathbb{P}(v^n \cdot \nabla v^n)) - \nu v^n + T_n(\mathbb{P}(B^n \cdot \nabla B^n)) + T_n(\mathbb{P}(B^* \cdot \nabla B^n)),
			\\
			F^n_2 &= -T_n(v^n \cdot \nabla B^n) + T_n(B^n \cdot \nabla v^n) + T_n(B^* \cdot \nabla v^n) + \frac{1}{\sigma}\Delta B^n + \frac{k}{\sigma}T_n(j^n \cdot \nabla B^n - B^n \cdot \nabla j^n - B^* \cdot \nabla j^n).
		\end{align*}
		By defining $F^n : V^s_n \times H^s_n \to V^s_n \times H^s_n$ with $F^n := (F^n_1,F^n_2)$, the existence of a unique local solultion $(v^n,B^n) \in C^1([0,T^n_*); V^s_n \times H^s_n)$ to  \eqref{HMHD_app} follows as previously. 
		
		\textbf{Step 2: $H^s$ estimate and global uniform  bound.}
		It can be seen from \eqref{HMHD_app} that 
		\begin{equation*}
			\frac{1}{2}\frac{d}{dt}\|(v^n,B^n)\|^2_{H^s} + \nu \|v^n\|^2_{H^s} + \frac{1}{\sigma}\|\nabla B^n\|^2_{H^s} =: \sum^9_{k = 1} I_k,
		\end{equation*}
		where for some $\epsilon \in (0,1)$, since $s > \frac{d}{2} + 1$ 
		\begin{align*}
			I_1 &= - \int_{\mathbb{R}^d} [J^s(\mathbb{P}(T_n(v^n \cdot \nabla v^n)) 
			- v^n \cdot \nabla J^sv^n] 
			\cdot J^s v^n \,dx \leq C(s)\|v^n\|^3_{H^s};
			\\
			I_2 &= \int_{\mathbb{R}^d} J^s(\mathbb{P}(T_n(B^n \cdot \nabla B^n))   
			\cdot J^s v^n \,dx \leq \frac{\epsilon}{\sigma}\|\nabla B^n\|_{H^s} + C(\epsilon,\sigma,s)\|v^n\|^2_{H^s}\|B^n\|^2_{H^s};
			\\
			I_3 &= \int_{\mathbb{R}^d} J^s(\mathbb{P}(T_n(B^* \cdot \nabla B^n)) \cdot J^s v^n \,dx = \int_{\mathbb{R}^d} B^* \cdot \nabla J^s B^n \cdot J^s v^n \,dx; 
			\\
			I_4 &= -\int_{\mathbb{R}^d} J^s(T_n(v^n \cdot \nabla B^n)) \cdot J^s B^n \,dx \leq \frac{\epsilon}{\sigma}\|\nabla B^n\|_{H^s} + C(\epsilon,\sigma,s)\|v^n\|^2_{H^s}\|B^n\|^2_{H^s};
			\\
			I_5 &= \int_{\mathbb{R}^d} J^s(T_n(B^n \cdot \nabla v^n)) 
			\cdot J^s B^n \,dx \leq \frac{\epsilon}{\sigma}\|\nabla B^n\|_{H^s} + C(\epsilon,\sigma,s)\|v^n\|^2_{H^s}\|B^n\|^2_{H^s};
			\\
			I_6 &= \int_{\mathbb{R}^d} J^s(T_n(B^* \cdot \nabla v^n)) \cdot J^s B^n \,dx = \int_{\mathbb{R}^d} B^* \cdot \nabla J^sv^n \cdot J^s B^n \,dx = -I_3;
			\\
			I_7 &= \frac{\kappa}{\sigma} \int_{\mathbb{R}^d} J^s(T_n(j^n \cdot \nabla B^n)) \cdot J^s B^n \,dx \leq \frac{\kappa}{\sigma}C(s)\|\nabla B^n\|^2_{H^s}\|B^n\|_{H^s};
			\\
			I_8 &= -\frac{\kappa}{\sigma} \int_{\mathbb{R}^d} J^s(T_n(B^n \cdot \nabla j^n)) \cdot J^s B^n \,dx \leq \frac{\kappa}{\sigma}C(s)\|\nabla B^n\|^2_{H^s}\|B^n\|_{H^s};
			\\
			I_9 &= -\frac{\kappa}{\sigma} \int_{\mathbb{R}^d} J^s(T_n(B^* \cdot \nabla j^n)) \cdot J^s B^n \,dx = -\frac{\kappa}{\sigma} \int_{\mathbb{R}^d} J^s(-j^n \cdot \nabla B^* + B^* \cdot \nabla j^n) \cdot J^s B^n \,dx
			\\
			&= -\frac{\kappa}{\sigma} \int_{\mathbb{R}^d} J^s(\nabla \times(j^n \times B^*)) \cdot J^s B^n \,dx 
			= -\frac{\kappa}{\sigma} \int_{\mathbb{R}^d} J^s(j^n \times B^*) \cdot J^s j^n \,dx = 0.
		\end{align*}
		By chossing $\epsilon = \frac{1}{6}$
		\begin{equation*}
			\frac{d}{dt}\|(v^n,B^n)\|^2_{H^s} + \nu \|v^n\|^2_{H^s} + \frac{1}{\sigma}\|\nabla B^n\|^2_{H^s} \leq C(s)\|v^n\|^3_{H^s} + C(\sigma,s)\|v^n\|^2_{H^s}\|B^n\|^2_{H^s} + C(\kappa,\sigma,s)\|B^n\|_{H^s}\|\nabla B^n\|^2_{H^s}.
		\end{equation*}
		Similar to the previous parts, under the smallness assumption on initial data, it follows that for $t \geq 0$
		\begin{equation*}
			\|(v^n,B^n)(t)\|^2_{H^s} + \int^t_0 \nu \|v^n\|^2_{H^s} + \frac{1}{\sigma}\|\nabla B^n\|^2_{H^s} \,d\tau \leq 2\epsilon^2_0.
		\end{equation*}
		
		\textbf{Step 3: Cauchy sequence, pass to the limit and uniqueness.} Thanks to the uniform bound obtained in the prevous step, we then can prove that $(v^n,B^n)$ and $(v^n,\nabla B^n)$ are Cauchy sequences in $L^\infty(0,\infty;L^2(\mathbb{R}^d))$ and $L^2(0,\infty;L^2(\mathbb{R}^d))$, respectively. That allows us to pass to the limit and to obtain 
		\begin{equation} \label{HMHD*_bound}
			\|(v,B)(t)\|^2_{H^s} + \int^t_0 \nu \|v\|^2_{H^s} + \frac{1}{\sigma}\|\nabla B\|^2_{H^s} \,d\tau \leq 2\epsilon^2_0.
		\end{equation}
		The proof of the uniqueness is standard.  This step is similar to the previous parts then we omit further details.
		
		\textbf{Step 4a: Large-time behavior: implicit rate.}  It can be seen from \eqref{HMHD**} that
		\begin{align*}
			\frac{1}{2}\frac{d}{dt}\|v\|^2_{L^2} + \nu \|v\|^2_{L^2} &=: R_5, 
			\\
			\frac{1}{2}\frac{d}{dt}\|B\|^2_{\dot{H}^1} + \frac{1}{\sigma} \|B\|^2_{\dot{H}^2} &=: \sum^{11}_{k=6} R_k,
		\end{align*}
		where for some $\epsilon \in (0,1)$, since $s > \frac{d}{2} + 1$
		\begin{align*}
			R_5 &= \int_{\mathbb{R}^d} (B+B^*) \cdot \nabla B \cdot v \,dx \leq C(s)\left(\|B\|_{H^s} + \|B^*\|_{L^\infty}\right)\left(\|v\|^2_{L^2} + \|B\|^2_{\dot{H}^1}\right);
			\\
			R_6 &= \int_{\mathbb{R}^d} v \cdot \nabla B \cdot \Delta B \,dx \leq \frac{\epsilon}{\sigma} \|B\|^2_{\dot{H}^2} + C(\epsilon,\sigma,s) \|v\|^2_{H^s}\|B\|^2_{\dot{H}^1};
			\\
			R_7 &= -\int_{\mathbb{R}^d} B \cdot \nabla v \cdot \Delta B \,dx \leq \frac{\epsilon}{\sigma} \|B\|^2_{\dot{H}^2} + C(\epsilon,\sigma,s) \|v\|^2_{H^s}\|B\|^2_{H^s};
			\\
			R_8 &= -\int_{\mathbb{R}^d} B^* \cdot \nabla v \cdot \Delta B \,dx \leq \frac{\epsilon}{\sigma} \|B\|^2_{\dot{H}^2} + C(\epsilon,\sigma,s) \|v\|^2_{H^s}\|B^*\|^2_{L^\infty};
			\\
			R_9 &= -\frac{\kappa}{\sigma} \int_{\mathbb{R}^d} j \cdot \nabla B \cdot \Delta B \,dx \leq \frac{\epsilon}{\sigma} \|B\|^2_{\dot{H}^2} + C(\epsilon,\sigma,s) \|B\|^2_{H^s}\|B\|^2_{\dot{H}^1};
			\\
			R_{10} &= \frac{\kappa}{\sigma} \int_{\mathbb{R}^d} B \cdot \nabla j \cdot \Delta B \,dx \leq \frac{\epsilon}{\sigma} \|B\|^2_{\dot{H}^2} + C(\epsilon,\sigma,s) \|B\|^2_{H^s}\|B\|^2_{\dot{H}^2}; 
			\\
			R_{11} &= \frac{\kappa}{\sigma} \int_{\mathbb{R}^d} B^* \cdot \nabla j \cdot \Delta B \,dx \leq \frac{\epsilon}{\sigma} \|B\|^2_{\dot{H}^2} + C(\epsilon,\sigma,s) \|B^*\|^2_{L^\infty}\|B\|^2_{\dot{H}^2}.
		\end{align*}
		Therefore, similar to Step 5 in the proof of Theorem \ref{theo3}, by choosing $\epsilon = \frac{1}{12}$, and using the energy estimate and \eqref{HMHD*_bound}, we have $\|(v,\nabla B)(t)\|_{L^2} \to 0$ as $t \to \infty$. In addition, for $s' \in [0,s)$ and $s'' \in [1,s)$, interpolation inequalities and Lemma \ref{lem-Agmon} yield $\|v(t)\|_{H^{s'}}$, $\|B(t)\|_{L^p}$ for $p \in (2,\infty]$, $\|B(t)\|_{L^q_{\textnormal{loc}}}$ for $q \in [1,\infty]$ and $\|B(t)\|_{\dot{H}^{s''}} \to 0$  as $t \to \infty$. The convergence in time of $(\partial_t v, \partial_t B)$ follows from that of $(v,B)$ and the following estimates for $r \in [0,s-2)$
		\begin{align*}
			\|\partial_t v(t)\|_{H^r} &\leq \|v(t) \cdot \nabla v(t)\|_{H^r} + \nu \|v(t)\|_{H^r} + \|B(t) \cdot \nabla B(t)\|_{H^r} + \|B^* \cdot \nabla B(t)\|_{H^r},
			\\
			\|\partial_t B(t)\|_{H^r} &\leq \|v(t) \cdot \nabla B(t)\|_{H^r} + \frac{1}{\sigma} \|\Delta B(t)\|_{H^r} + \|B(t) \cdot \nabla v(t)\|_{H^r} + \|B^* \cdot \nabla v(t)\|_{H^r} 
			\\
			&\quad + \frac{\kappa}{\sigma}\left(\|j(t) \cdot \nabla B(t)\|_{H^r} + \|B(t) \cdot \nabla j(t)\|_{H^r} + \|B^* \cdot \nabla j(t)\|_{H^r}\right).
		\end{align*}
		We only show how to deal with the most difficult term as follows, other ones can be done similarly. Indeed,  since $s > \frac{d}{2} + 1$ then for some suitable $s' \in (\frac{d}{2} + 1,s)$ and $r \geq 0$ with $r + 1 \leq s' - 1$
		\begin{align*}
			\|j(t) \cdot \nabla B(t)\|_{H^r} &\leq C(s)\|j(t) \otimes B(t)\|_{H^{r+1}} \leq C(s) \|j(t) \otimes B(t)\|_{H^{s'-1}} 
			\\
			&\leq C(s) \|j(t)\|_{H^{s'-1}} \|B(t)\|_{H^{s'-1}} \\
			&\leq C(s) \left(\|\nabla B(t)\|_{L^2} + \|B(t)\|_{\dot{H}^{s'}}\right)\|B(t)\|_{H^s} \quad \to 0 \quad \text{as} \quad t \to \infty.
		\end{align*}

		\textbf{Step 4b: Large-time behavior: explicit rate.} If in addition $(v_0,B_0) \in L^1$ then an explicit rate of convergence in suitable norms can be established. More precisely, we can follow closely the ideas in \cite{Schonbek_1985,Schonbek_1986} (for the Navier-Stokes equations) by applying the Fourier-splitting method  to obtain the $L^2$ decay in time of $(v,B)$. Indeed, there is a new difficulty, which is related to the perturbation terms (see $(S_3,S_6,S_9)$ below) since at some point we need to control $L^\infty$ norm of such a bab term $\mathcal{F}(B^* \otimes B)$ and it seems leading to the $L^1$ estimate of $B$, which has not been obtained yet. Thus, the techniques in \cite{Schonbek_1985,Schonbek_1986} can not be applied directly and new ideas should be suggested. To overcome this new issue, we will estimate more carefully the bad term, especially using the velocity damping kernel, which allows us to gain more good factors. For fixed $\nu > 0$, by defining for $(x,t) \in \mathbb{R}^d \times (0,\infty)$ and for some $m \in \mathbb{N}$ will be chosen later
		\begin{equation*}
			(v_\nu,B_\nu)(x,t) := \frac{m}{\nu}(v,B)\left(x,\frac{mt}{\nu}\right), \quad p^*_\nu(x,t) := \frac{m^2}{\nu^2}p^*\left(x,\frac{mt}{\nu}\right), \quad B^*_\nu := \frac{m}{\nu}B^* \quad \text{and}\quad j_\nu := \nabla \times B_\nu,
		\end{equation*}
		we reduce \eqref{HMHD**} to 
		\begin{equation} \label{HMHD*nu} 
			\left\{
			\begin{aligned}
				\partial_t v_\nu + v_\nu \cdot \nabla v_\nu &= -mv_\nu + B_\nu \cdot \nabla B_\nu + B^*_\nu \cdot \nabla B_\nu - \nabla p^*_\nu, 
				\\
				\partial_t B_\nu + v_\nu \cdot \nabla B_\nu  &= B_\nu \cdot \nabla v_\nu + B^*_\nu \cdot \nabla v_\nu + \frac{m}{\nu\sigma}\Delta B_\nu +  \frac{\kappa}{\sigma} (j_\nu \cdot \nabla B_\nu - B_\nu \cdot \nabla j_\nu - B^*_\nu \cdot \nabla j_\nu),
				\\
				\textnormal{div}\, v_\nu = \textnormal{div}\, B_\nu &= 0,
			\end{aligned}
			\right.
		\end{equation}
		with the initial data is given by $(v_\nu,B_\nu)_{|_{t=0}} = m\nu^{-1}(v_0,B_0)$. From the previous step, we know the existence and uniqueness of solutions $(v_\nu,B_\nu)$ to \eqref{HMHD*nu} satisfying $(v_\nu,B_\nu) \in L^\infty(0,\infty;H^s(\mathbb{R}^d))$, $(v_\nu,\nabla B_\nu) \in L^2(0,\infty;H^s(\mathbb{R}^d))$ for $s > \frac{d}{2} + 1$ and the estimate \eqref{HMHD*_bound} with $C(\nu,m)\epsilon^2_0$ instead of $2\epsilon^2_0$.  
		It can be seen from the energy balance of \eqref{HMHD*nu} that
		\begin{equation*}
			\frac{d}{dt}\left(h_\nu(t) := \int_{\mathbb{R}^d} |\mathcal{F}(v_\nu)(t)|^2 + |\mathcal{F}(B_\nu)(t)|^2\,d\xi\right) = - 2m \int_{\mathbb{R}^d} |\mathcal{F}(v_\nu)(t)|^2 \,d\xi - \frac{2m}{\nu\sigma} \int_{\mathbb{R}^d} |\xi|^2|\mathcal{F}(B_\nu)(t)|^2\,d\xi.
		\end{equation*}
		As in \cite{Schonbek_1986}, for some $\beta > 0$ to be determined later, we define for $t > 0$
		\begin{align*}
			S(t) := \{\xi \in \mathbb{R}^d : |\xi| \leq g(t)\} \quad \text{and} \quad \tilde{g}(t) := \exp\left\{\beta\int^t_0 g^2 \,d\tau\right\}  \quad \text{with}\quad g^2(t) := \frac{m}{\beta(e+t)\log(e+t)}.
		\end{align*}
		For $S^c(t) := \mathbb{R}^d \setminus S(t)$, by choosing $\beta = \frac{2m}{\nu\sigma}$ and using the fact that $\beta g^2 \leq m$, it follows that
		\begin{align*}
			\frac{d}{dt}\left(\tilde{g}(t)h_\nu(t)\right) &= \frac{d}{dt}(\tilde{g}(t)) h_\nu(t) + \tilde{g}(t)\frac{d}{dt}h_\nu(t)
			\\
			&\leq \beta g^2(t)\tilde{g}(t)h_\nu(t) + \tilde{g}(t) \left(-2\int_{S^c(t)} |\mathcal{F}(v_\nu)(t)|^2 \,d\xi - \frac{2}{\nu\sigma}g^2(t) \int_{S^c(t)} |\mathcal{F}(B_\nu)(t)|^2\,d\xi\right)
			\\
			&\leq \beta g^2(t) \tilde{g}(t) \left(\int_{S(t)} |\mathcal{F}(v_\nu)(t)|^2 + |\mathcal{F}(B_\nu)(t)|^2\,d\xi =: S_{v_\nu}(t) + S_{B_\nu}(t)\right).
		\end{align*}
		It remains to control the integral on the right-hand side. It can be seen that
		\begin{align*}
			\partial_t \mathcal{F}(v_\nu) + m\mathcal{F}(v_\nu) &= \mathcal{F}(\mathbb{P}(-v_\nu \cdot \nabla v_\nu + B_\nu \cdot \nabla B_\nu + B^*_\nu \cdot \nabla B_\nu)),
			\\
			\partial_t \mathcal{F}(B_\nu) + \frac{m}{\nu\sigma}|\xi^2| \mathcal{F}(B_\nu) &=  \mathcal{F}(-v_\nu \cdot \nabla B_\nu + B_\nu \cdot \nabla v_\nu + B^*_\nu \cdot \nabla v_\nu +  \frac{\kappa}{\sigma}(j_\nu \cdot \nabla B_\nu - B_\nu \cdot \nabla j_\nu - B^*_\nu \cdot \nabla j_\nu)),
		\end{align*}
		it implies that for $t > 0$ 
		\begin{align*}
			S_{v_\nu}(t) &\leq Cg^d(t) \|v_0\|^2_{L^1} + \int_{S(t)} \left(\int^t_0 \exp\{-m(t-\tau)\}|\mathcal{F}(\mathbb{P}(-v_\nu \cdot \nabla v_\nu + B_\nu \cdot \nabla B_\nu + B^*_\nu \cdot \nabla B_\nu))| \,d\tau \right)^2 d\xi 
			\\
			&\leq Cg^d(t)\|v_0\|^2_{L^1} +C\sum^3_{k=1} S_k.
		\end{align*}
		We are going to estimate each term on the right-hand side by using \eqref{HMHD*_bound} as follows 
		\begin{align*}
			S_1 &:= \int_{S(t)}\left(\int^t_0 \exp\{-m(t-\tau)\} |\mathcal{F}(\mathbb{P}(v_\nu \cdot \nabla v_\nu))| \,d\tau\right)^2\,d\xi 
			\\
			&\leq  \int_{S(t)} \int^t_0 \exp\{-2m(t-\tau)\} \,d\tau \int^t_0 |\mathcal{F}(v_\nu \cdot \nabla v_\nu)|^2 \,d\tau d\xi
			\\
			&\leq C(m) g^2(t) \int^t_0 \int_{S(t)} |\mathcal{F}(v_\nu \otimes v_\nu)|^2 \,d\xi d\tau 
			\\
			&\leq C(m) g^{d+2}(t) \int^t_0 \|\mathcal{F}(v_\nu \otimes v_\nu)\|^2_{L^\infty}  \,d\tau 
			\\
			&\leq C(m) g^{d+2}(t) \int^t_0 \|v_\nu \otimes v_\nu\|^2_{L^1}  \,d\tau 
			\\
			&\leq C(\epsilon_0,\nu,m) g^{d+2}(t);
			\\
			S_2 &:= \int_{S(t)}\left(\int^t_0 \exp\{-m(t-\tau)\} |\mathcal{F}(\mathbb{P}(B_\nu \cdot \nabla B_\nu))| \,d\tau\right)^2 d\xi \leq C(\epsilon_0,\nu,m) tg^{d+2}(t);
			\\
			S_3 &:= \int_{S(t)}\left(\int^t_0 \exp\{-m(t-\tau)\} |\mathcal{F}(\mathbb{P}(B^*_\nu \cdot \nabla B_\nu))| \,d\tau\right)^2 d\xi
			\\
			&\leq \int_{S(t)} \int^t_0 \exp\{-2m(t-\tau)\}\,d\tau \int^t_0 |\mathcal{F}(B^*_\nu \cdot \nabla B_\nu)|^2 \,d\tau d\xi
			\\
			&\leq C(m) \|B^*\|^2_{L^\infty} g^2(t) \int^t_0 \int_{\mathbb{R}^d}  |\mathcal{F}(B_\nu)|^2 \,d\xi d\tau 
			\\
			&\leq C(\epsilon_0,\epsilon_*,\nu,m) t g^2(t).
		\end{align*}
		Similarly, we find that 
		\begin{align*}
			S_{B_\nu}(t) &\leq Cg^d(t) \|B_0\|^2_{L^1}  + C\int_{S(t)} \left(\int^t_0 \exp\left\{-\frac{m|\xi|^2}{\nu \sigma}(t-\tau)\right\} |\mathcal{F}(-v_\nu \cdot \nabla B_\nu + B_\nu \cdot \nabla v_\nu + B^*_\nu \cdot \nabla v_\nu)| \,d\tau \right)^2d\xi 
			\\
			&\quad +
			C(\kappa,\sigma)\int_{S(t)} \left(\int^t_0 \exp\left\{-\frac{m|\xi|^2}{\nu \sigma}(t-\tau)\right\} |\mathcal{F}(j_\nu \cdot \nabla B_\nu - B_\nu \cdot \nabla j_\nu - B^*_\nu \cdot \nabla j_\nu)| \,d\tau \right)^2 d\xi
			\\
			&\leq Cg^d(t) \|B_0\|^2_{L^1} + C(\kappa,\sigma)\sum^9_{k=4} S_k,
		\end{align*}
		where each term on the right-hand side is bounded by
		\begin{align*}
			S_4 &:= \int_{S(t)} \left(\int^t_0 \exp\left\{-\frac{m|\xi|^2}{\nu \sigma}(t-\tau)\right\} |\mathcal{F}(v_\nu \cdot \nabla B_\nu)| \, d\tau\right)^2 d\xi \leq C(\epsilon_0,\nu,m) tg^{d+2}(t);
			\\
			S_5 &:= \int_{S(t)} \left(\int^t_0 \exp\left\{-\frac{m|\xi|^2}{\nu \sigma}(t-\tau)\right\} |\mathcal{F}(B_\nu \cdot \nabla v_\nu)| \, d\tau\right)^2 d\xi \leq C(\epsilon_0,\nu,m) tg^{d+2}(t);
			\\
			S_6 &:= \int_{S(t)} \left(\int^t_0 \exp\left\{-\frac{m|\xi|^2}{\nu \sigma}(t-\tau)\right\} |\mathcal{F}(B^*_\nu \cdot \nabla v_\nu)| \, d\tau\right)^2 d\xi 
			\leq C(\epsilon_0,\epsilon_*,\nu,m) tg^2(t);
			\\
			S_7 &:= C(\kappa,\sigma)\int_{S(t)} \left(\int^t_0 \exp\left\{-\frac{m|\xi|^2}{\nu \sigma}(t-\tau)\right\} |\mathcal{F}(j_\nu \cdot \nabla B_\nu)| \, d\tau\right)^2 d\xi 
			\\
			&\leq C(\kappa,\sigma,m) t g^{2+d}(t) \int^t_0 \|j_\nu \otimes B_\nu\|^2_{L^1}  \,d\tau 
			\\
			&\leq C(\kappa,\sigma,m) t g^{2+d}(t) \int^t_0 \|\nabla B\|^2_{L^2} \|B_\nu\|^2_{L^2}  \,d\tau 
			\\
			&\leq C(\epsilon_0,\kappa,\nu,m,\sigma) tg^{d+2}(t);
			\\
			S_8 &:= C(\kappa,\sigma)\int_{S(t)} \left(\int^t_0 \exp\left\{-\frac{m|\xi|^2}{\nu \sigma}(t-\tau)\right\} |\mathcal{F}(B_\nu \cdot \nabla j_\nu)| \, d\tau\right)^2 d\xi \leq C(\epsilon_0,\kappa,\nu,m,\sigma) tg^{d+2}(t);
			\\
			S_9 &:= C(\kappa,\sigma)\int_{S(t)} \left(\int^t_0 \exp\left\{-\frac{m|\xi|^2}{\nu \sigma}(t-\tau)\right\} |\mathcal{F}(B^*_\nu \cdot \nabla j_\nu)| \, d\tau\right)^2 d\xi  
			\leq C(\epsilon_0,\epsilon_*,\kappa,\nu,m,\sigma) tg^2(t). 
		\end{align*}
		Therefore, 
		\begin{equation} \label{HMHD-L2-decay1}
			\frac{d}{dt}\left(\tilde{g}(t) h_\nu(t)\right) 
			\\
			\leq C(\epsilon_0,\epsilon_*,\kappa,\nu,m,\sigma,v_0,B_0) \tilde{g}(t) (e+t)g^4(t).
		\end{equation}
		Moreover, it can be easily checked that
		\begin{align*}
			\tilde{g}(t) = \exp\left\{\beta \int^t_0 g^2\,d\tau \right\} = \log^m(e+t),
		\end{align*}
		which by choosing $m = 2$ yields for $s > 0$
		\begin{equation*}
			G := \int^s_0 \tilde{g}(t) (e+t)g^4(t) \,dt = \int^s_0 \frac{1}{(e+t) \log^{2-m}(e+t)}\,dt \leq \log(e+s). 
		\end{equation*}
		Thus, \eqref{HMHD-L2-decay1} implies that for $t > 0$
		\begin{equation} \label{HMHD-L2-decay2}
			\|(v_\nu,B_\nu)(t)\|^2_{L^2} \leq  C(\epsilon_0,\epsilon_*,\kappa,\nu,m,\sigma,v_0,B_0) \log^{-1}(e+t).
		\end{equation} 
		By using \eqref{HMHD-L2-decay2}, we observe that for $C = C(\epsilon_0,\epsilon_*,\kappa,\nu,m,\sigma,v_0,B_0)$
		\begin{equation*}
			\int^t_0 \|(v_\nu,B_\nu)\|^2_{L^2} \,d\tau \leq C \int^t_0 \log^{-1}(e+\tau) \,d\tau = C\int^{\log(e+t)}_1 e^u u^{-1} \,du 
			\leq C(e+t)\log^{-1}(e+t),
		\end{equation*}
		which will improve the estimates of $S_3$, $S_6$ and $S_9$ as follows
		\begin{equation*}
			S_3,S_6,S_9 \leq C(\epsilon_0,\epsilon_*,\kappa,\nu,m,\sigma) (e+t) g^2(t) \log^{-1}(e+t),
		\end{equation*}
		that leads to a multiplication by the factor $\log^{-1}(e+t)$ to the right-hand side of \eqref{HMHD-L2-decay1}. Therefore, by changing the estimate of $G$ with choosing again $m = 3$ as follows
		\begin{equation*}
			\tilde{G} := \int^s_0 \tilde{g}(t) (e+t)g^4(t) \log^{-1}(e+t)\,dt = \int^s_0 \frac{1}{(e+t) \log^{3-m}(e+t)}\,dt \leq  \log(e+s),
		\end{equation*}
		the estimate \eqref{HMHD-L2-decay2} can be replaced by
		\begin{equation} \label{HMHD-L2-decay3}
			\|(v_\nu,B_\nu)(t)\|^2_{L^2} \leq  C(\epsilon_0,\epsilon_*,\kappa,\nu,m,\sigma,v_0,B_0) \log^{-2}(e+t).
		\end{equation}
		By repeating this iteration, it can be seen that \eqref{HMHD-L2-decay3} can be improved  for each $m \in \mathbb{N}, m \geq 3$
		\begin{equation*} 
			\|(v_\nu,B_\nu)(t)\|^2_{L^2} \leq  C(\epsilon_0,\epsilon_*,\kappa,\nu,m,\sigma,v_0,B_0) \log^{-(m-1)}(e+t).
		\end{equation*}
		That finishes the proof by combining the above inequality, \eqref{HMHD-L2-decay2}, a change of variables from $(v_\nu,B_\nu)$ to $(v,B)$ and interpolation inequalities.
	\end{proof}
	
	\begin{remark} \label{re-3D} (The case $d = 3$)
		In Step 4 above, we condsidered both cases $d = 2$ and $d = 3$ at the same time. However, in the three-dimensional case, it would be expected to obtain a faster $L^2$ decay rate such as $(t+1)^{-\frac{3}{4}}$, which is known in the case of either the Navier-Stokes \cite{Schonbek_1985,Schonbek_1986} or the Hall-MHD equations \cite{Chae-Schonbek_2013}. We now give a remark on this case, where it seems to be difficult to obtain polynomial decay in time compared to the case of the Navier-Stokes and Hall-MHD equations (see \cite{Chae-Schonbek_2013,Schonbek_1985,Schonbek_1986}), unless new estimates of $S_3,S_6$ and $S_9$ are provided. Similarly, for fixed $\nu > 0$, by defining for $(x,t) \in \mathbb{R}^3 \times (0,\infty)$
		\begin{equation*}
			(v_\nu,B_\nu)(x,t) := \frac{3}{\nu}(v,B)\left(x,\frac{3t}{\nu}\right), \quad p_\nu(x,t) := \frac{9}{\nu^2}p\left(x,\frac{3t}{2\nu}\right), \quad B^*_\nu := \frac{3}{\nu} B^* \quad \text{and}\quad j_\nu := \nabla \times B_\nu,
		\end{equation*}
		we reduce \eqref{HMHD**} to 
		\begin{equation} \label{HMHD*nu2} 
			\left\{
			\begin{aligned}
				\partial_t v_\nu + v_\nu \cdot \nabla v_\nu &= -3v_\nu + B_\nu \cdot \nabla B_\nu + B^*_\nu \cdot \nabla B_\nu + \nabla p_\nu, 
				\\
				\partial_t B_\nu + v_\nu \cdot \nabla B_\nu  &= B_\nu \cdot \nabla v_\nu + B^*_\nu \cdot \nabla v_\nu + \frac{3}{\nu\sigma}\Delta B_\nu +  \frac{\kappa}{\sigma} (j_\nu \cdot \nabla B_\nu - B_\nu \cdot \nabla j_\nu - B^*_\nu \cdot \nabla j_\nu),
				\\
				\textnormal{div}\, v_\nu = \textnormal{div}\, B_\nu &= 0,
			\end{aligned}
			\right.
		\end{equation}
		with the initial data is given by $(v_\nu,B_\nu)_{|_{t=0}} = 3\nu^{-1}(v_0,B_0)$. From Step 3 above, we know that $(v_\nu,B_\nu)$ to \eqref{HMHD*nu} satisfying $(v_\nu,B_\nu) \in L^\infty(0,\infty;H^s(\mathbb{R}^3))$, $(v_\nu,\nabla B_\nu) \in L^2(0,\infty;H^s(\mathbb{R}^3))$ for $s > \frac{5}{2}$ and the estimate \eqref{HMHD*_bound} with $C(\nu)\epsilon^2_0$ instead of $2\epsilon^2_0$.  Therefore, the energy balance of \eqref{HMHD*nu2} is given by
		\begin{equation*}
			\frac{d}{dt}\left(h_\nu(t) := \int_{\mathbb{R}^3} |\mathcal{F}(v_\nu)(t)|^2 + |\mathcal{F}(B_\nu)(t)|^2\,d\xi\right) = - 6 \int_{\mathbb{R}^3} |\mathcal{F}(v_\nu)(t)|^2 \,d\xi - \frac{6}{\nu\sigma} \int_{\mathbb{R}^3} |\xi|^2|\mathcal{F}(B_\nu)(t)|^2\,d\xi =: F_1 + F_2.
		\end{equation*}
		For some $\beta > 0$ to be determined later, by defining
		\begin{align*}
			S(t) := \{\xi \in \mathbb{R}^3 : |\xi| \leq g(t)\} \quad \text{with}\quad g^2(t) := \frac{3}{\beta(t+1)} \quad\text{and}\quad  S^c(t) := \mathbb{R}^3 \setminus S(t),
		\end{align*}
		and using $\beta g^2 \leq 3$, we find that
		\begin{align*}
			F_1 &\leq -3\int_{S(t)} |\mathcal{F}(v_\nu)(t)|^2 \,d\xi -3 \int_{S^c(t)} |\mathcal{F}(v_\nu)(t)|^2 \,d\xi  
			\\
			&\leq -3\int_{S(t)} |\mathcal{F}(v_\nu)(t)|^2 \,d\xi -\frac{3}{t+1} \int_{S^c(t)} |\mathcal{F}(v_\nu)(t)|^2 \,d\xi  \pm \frac{3}{t+1}\int_{S(t)} |\mathcal{F}(v_\nu)(t)|^2 \,d\xi 
			\\
			&\leq 
			-\frac{3}{t+1} \int_{\mathbb{R}^3} |\mathcal{F}(v_\nu)(t)|^2 \,d\xi  
			+ \frac{3}{t+1}\int_{S(t)} |\mathcal{F}(v_\nu)(t)|^2 \,d\xi, 
		\end{align*}
		and by choosing $\beta := \frac{3}{\nu\sigma}$, we obtain
		\begin{align*}
			F_2 &\leq - \frac{3}{\nu\sigma} \int_{S(t)} |\xi|^2|\mathcal{F}(B_\nu)(t)|^2\,d\xi - \frac{3}{\nu\sigma} \int_{S^c(t)} |\xi|^2|\mathcal{F}(B_\nu)(t)|^2\,d\xi
			\\
			&\leq - \frac{3}{\nu\sigma} \int_{S(t)} |\xi|^2|\mathcal{F}(B_\nu)(t)|^2\,d\xi - \frac{3}{t+1} \int_{S^c(t)} |\mathcal{F}(B_\nu)|^2 \,d\xi \pm \frac{3}{t+1}\int_{S(t)} |\mathcal{F}(B_\nu)(t)|^2 \,d\xi 
			\\
			&\leq
			- \frac{3}{t+1} \int_{\mathbb{R}^3} |\mathcal{F}(B_\nu)(t)|^2 \,d\xi + \frac{3}{t+1}\int_{S(t)} |\mathcal{F}(B_\nu)(t)|^2 \,d\xi.
		\end{align*}
		Therefore,
		\begin{equation*}
			\frac{d}{dt} h_\nu(t) + \frac{3}{t+1}h_\nu(t) \leq \frac{3}{t+1}\int_{S(t)} |\mathcal{F}(v_\nu)(t)|^2 + |\mathcal{F}(B_\nu)(t)|^2 \,d\xi 
		\end{equation*}
		and by multiplying $(t+1)^3$ both sides
		\begin{equation*}
			\frac{d}{dt}(h_\nu(t)(t+1)^3) = (t+1)^3\frac{d}{dt} h_\nu(t) + 3(t+1)^2h_\nu(t) \leq 3(t+1)^2 \int_{S(t)} |\mathcal{F}(v_\nu)(t)|^2 + |\mathcal{F}(B_\nu)(t)|^2 \,d\xi.
		\end{equation*}
		It remains to bound the integral on the right-hand side. Similar to the previous case, it follows that 
		\begin{align*}
			\partial_t \mathcal{F}(v_\nu) + 3\mathcal{F}(v_\nu) &= \mathcal{F}(\mathbb{P}(-v_\nu \cdot \nabla v_\nu + B_\nu \cdot \nabla B_\nu + B^*_\nu \cdot \nabla B_\nu)),
			\\
			\partial_t \mathcal{F}(B_\nu) + \frac{3}{\nu\sigma}|\xi^2| \mathcal{F}(B_\nu) &=  \mathcal{F}(-v_\nu \cdot \nabla B_\nu + B_\nu \cdot \nabla v_\nu + B^*_\nu \cdot \nabla v_\nu +  \frac{\kappa}{\sigma}(j_\nu \cdot \nabla B_\nu - B_\nu \cdot \nabla j_\nu - B^*_\nu \cdot \nabla j_\nu)),
		\end{align*}
		which by using the same notation for $S_i$ with $i \in \{1,...,9\}$ with a small modification in the exponential factors and assuming $(v_0,B_0) \in L^1$ yields
		\begin{align*}
			\int_{S(t)} |\mathcal{F}(v_\nu)(t)|^2 + |\mathcal{F}(B_\nu)(t)|^2 \,d\xi &\leq Cg^3(t)\|(v_0,B_0)\|^2_{L^1} + C(\kappa,\sigma)\sum^9_{k=1} S_k,
		\end{align*}
		where the estimates of $S_i$ for $i \notin \{3,6,9\}$ can be given as in Step 4 above, for the remaining terms, if we use again those estimates, i.e., $S_3, S_6, S_9 \leq Ctg^2 \leq C$ for some positive constant $C$ depending on the parameters, then we will find that
		\begin{equation*}
			\frac{d}{dt}(h_\nu(t)(t+1)^3) \leq C(t+1)^2,
		\end{equation*}
		which unfortunately does not provide us a decay in time after integrating in time. So as mentioned previously, new ideas should be suggested to overcome this issue. For instance, 
		\begin{align*}
			S_3 &\leq C g^{2+2\gamma_0}(t) \int^t_0 \int_{S(t)}  |\mathcal{F}(B^*_\nu \otimes B_\nu)|^2 |\xi|^{-2\gamma_0} \,d\xi d\tau  \qquad \text{for some}\quad \gamma_0 \in \left(0,\frac{3}{2}\right)
			\\
			&\leq C \|B^*\|^2_{L^\infty} g^{2+2\gamma_0}(t) \int^t_0 \int_{\mathbb{R}^3}  |\mathcal{F}(B_\nu)|^2 |\xi|^{-2\gamma_0} \,d\xi d\tau 
			\\
			&= C \|B^*\|^2_{L^\infty} g^{2+2\gamma_0}(t) \int^t_0 \|B_\nu\|^2_{\dot{H}^{-\gamma_0}} \,d\tau 
			\\
			&\leq C(\epsilon_*,\gamma_0,\nu) g^{2+2\gamma_0}(t) \int^t_0 \|B_\nu\|^2_{L^\frac{6}{3+2\gamma_0}} \,d\tau, 
		\end{align*}
		which leads to the study of $L^p$ estimate of $B_\nu$ for some $p \in (1,2)$. We hope that the above time integral can be controlled nicely, for example, $S_3 \leq Ct^{\alpha_0} g^{2+2\gamma_0}$ for some constant $\alpha_0$ such that $\alpha_0 < 1 + \gamma_0$. It seems hardly to be the case since the standard dinemsional analysis shows that this integral has $3+2\gamma_0$ dimensions. However, if it is the case then that will lead to a polynomial decay rate as $(t+1)^{-(1+\gamma_0 - \alpha_0)}$ at the first level. We leave it as an open question for the interested reader.
	\end{remark}
	
	%
	\section*{Acknowledgements} 
	%
	
	K. Kang’s work is supported by RS-2024-00336346. J. Lee’s work is supported by NRF-2021R1A2C1092830. D. D. Nguyen’s work is supported by NRF-2019R1A2C1084685 and NRF-2021R1A2C1092830. 
	
	%
	\section{Appendix} \label{sec:app}
	%
	
	\subsection{Appendix A: Besov spaces}
	
	Let us quickly recall the definitions of the standard nonhomogeneous and homogeneous Besov spaces, see more details in \cite{Bahouri-Chemin-Danchin_2011}. There exist two smooth radial functions $\chi,\varphi : \mathbb{R}^d \to [0,1]$ for $d \geq 1$ such that 
	\begin{align*}
		&\text{supp}(\chi) \subset \left\{\xi \in \mathbb{R}^d : |\xi| \leq \frac{4}{3}\right\}, && \chi = 1 \quad \text{in} \quad \left\{\xi \in \mathbb{R}^d : |\xi| \leq \frac{3}{4}\right\},&&
		\\
		&\text{supp}(\varphi) \subset \left\{\xi \in \mathbb{R}^d : \frac{3}{4} \leq |\xi| \leq \frac{8}{3}\right\}, &&\varphi(\xi) := \chi\left({\frac{\xi}{2}}\right) - \chi(\xi),&&
		\\
		&\chi(\xi) + \sum_{0 \leq j \in \mathbb{Z}} \varphi(2^{-j}\xi) = 1 \quad \forall \xi \in \mathbb{R}^d, &&\sum_{j \in \mathbb{Z}} \varphi(2^{-j}\xi) = 1 \quad \forall \xi \in \mathbb{R}^d \setminus \{0\}.&&
	\end{align*}
	Defining $\tilde{h} := \mathcal{F}^{-1}(\chi)$ and $h := \mathcal{F}^{-1}(\varphi)$, where $\mathcal{F}^{-1}$ denotes the usual inverse Fourier transform. The nonhomogeneous and homogeneous dyadic blocks are defined by 
	\begin{equation*}
		\Delta_j f := 
		\begin{cases}
			0 &\text{if} \quad j \leq -2,
			\\ 
			\tilde{h}*f &\text{if} \quad j = -1,
			\\
			2^{jd} h(2^j\cdot)*f &\text{if} \quad j \geq 0,
		\end{cases}
		\qquad \text{and} \qquad 
		\dot{\Delta}_j f := 2^{jd} h(2^j\cdot)*f  \quad \forall j \in \mathbb{Z},
	\end{equation*}
	where $*$ stands for the usual convolution operator. Then formally the nonhomogeneous and homogeneous low-frequency cut-off operators are set for any $k \in \mathbb{Z}$ by (see also \cite{Bahouri_2019})
	\begin{equation*}
		S_k f := \sum_{-1 \leq j \leq k-1} \Delta_j f \qquad \text{and} \qquad \dot{S}_k f := 2^{kd} \tilde{h}(2^k\cdot)*f = \sum_{j \leq k-1, j \in \mathbb{Z}} \dot{\Delta}_j f.
	\end{equation*}
	For $s \in \mathbb{R}$ and $p,q \in [1,\infty]$, the nonhomogeneous and homogeneous Besov spaces are established as follows
	\begin{align*}
		B^s_{p,q}(\mathbb{R}^d) &:= \left\{f \in \mathcal{S}'(\mathbb{R}^d) : \|f\|_{B^s_{p,q}(\mathbb{R}^d)} := \|2^{sj}\|\Delta_j f\|_{L^p(\mathbb{R}^d)}\|_{\ell^q(\mathbb{Z})} < \infty\right\},
		\\
		\dot{B}^s_{p,q}(\mathbb{R}^d) &:= \left\{f \in \mathcal{S}'_h(\mathbb{R}^d) : \|f\|_{\dot{B}^s_{p,q}(\mathbb{R}^d)} := \|2^{sj}\|\dot{\Delta}_j f\|_{L^p(\mathbb{R}^d)}\|_{\ell^q(\mathbb{Z})} < \infty\right\},
	\end{align*}
	where $\mathcal{S}'(\mathbb{R}^d)$ denotes the dual space of the usual Schwartz class $\mathcal{S}(\mathbb{R}^d)$, the so-called the space of tempered distributions and 
	\begin{equation*}
		\mathcal{S}'_h(\mathbb{R}^d) := \left\{f \in \mathcal{S}'(\mathbb{R}^d): \lim_{\lambda \to \infty} \|g(\lambda D) f\|_{L^\infty} = 0 \quad \forall g \in C^\infty_0(\mathbb{R}^d)\right\},
	\end{equation*}
	here for any measurable function $g$ on $\mathbb{R}^d$ with at most polynomial growth at infinity, $g(D) f := \mathcal{F}^{-1}(g(\xi)\mathcal{F}(f)(\xi))$. It is also convenient to use the identities $B^s_{2,2}(\mathbb{R}^d) \approx H^s(\mathbb{R}^d)$  and $\dot{B}^s_{2,2}(\mathbb{R}^d) \approx \dot{H}^s(\mathbb{R}^d)$ for $s \in \mathbb{R}$. In addition, the Littlewood–Paley decompositions are given by 
	\begin{equation*}
		f = \sum_{-1 \leq j\in \mathbb{Z}} \Delta_j f \qquad \text{in} \quad \mathcal{S}'(\mathbb{R}^d) \qquad \text{and} \qquad f = \sum_{j\in \mathbb{Z}} \dot{\Delta}_j f \qquad \text{in}\quad  \mathcal{S}'(\mathbb{R}^d) \quad \forall f \in \mathcal{S}'_h(\mathbb{R}^d).
	\end{equation*} 
	We also recall a product rule in homogeneous Besov spaces (see \cite[Corollary 2.55]{Bahouri-Chemin-Danchin_2011}) for $s_1,s_2 \in (-\frac{d}{2},\frac{d}{2})$ and $s_1 + s_2 > 0$
	\begin{equation} \label{paraproduct}
		\|fg\|_{\dot{B}^{s_1+s_2-\frac{d}{2}}_{2,1}(\mathbb{R}^d)} \leq C(d,s_1,s_2) \|f\|_{\dot{H}^{s_1}(\mathbb{R}^d)}\|g\|_{\dot{H}^{s_2}(\mathbb{R}^d)}. 
	\end{equation}
	An application of \eqref{paraproduct} is the following Sobolev product estimate for $s_1,s_2 \in (0,\frac{d}{2})$
	\begin{equation} \label{So-paraproduct}
		\|fg\|_{H^{s_1+s_2-\frac{d}{2}}(\mathbb{R}^d)} \leq C(d,s_1,s_2) \|f\|_{H^{s_1}(\mathbb{R}^d)}\|g\|_{H^{s_2}(\mathbb{R}^d)}. 
	\end{equation}
	Indeed, if $s_1+s_2 \in (0,\frac{d}{2})$ then for $h \in H^{\frac{d}{2}-(s_1+s_2)}$
	\begin{equation*}
		\int_{\mathbb{R}^d} fg h\,dx \leq \|f\|_{L^\frac{2d}{d-2s_1}}\|g\|_{L^\frac{2d}{d-2s_2}}\|h\|_{L^\frac{2d}{d-2\left(\frac{d}{2} -(s_1+s_2)\right)}} \leq C(d,s_1,s_2) \|f\|_{H^{s_1}(\mathbb{R}^d)}\|g\|_{H^{s_2}(\mathbb{R}^d)} \|h\|_{H^{\frac{d}{2}-(s_1+s_2)}}.
	\end{equation*}
	On the other hand, if $s_1+s_2 \geq \frac{d}{2}$ then \eqref{paraproduct} yields
	\begin{align*}
		\|fg\|_{L^2} &\leq \|f\|_{L^\frac{2d}{d-2s_1}}\|g\|_{L^\frac{2d}{d-2\left(\frac{d}{2}-s_1\right)}} \leq C(d,s_1)\|f\|_{\dot{H}^{s_1}} \|g\|_{\dot{H}^{\frac{d}{2}-s_1}} \leq C(d,s_1)\|f\|_{H^{s_1}} \|g\|_{H^{s_2}},
		\\
		\|fg\|_{\dot{H}^{s_1+s_2-\frac{d}{2}}} &\leq C(d,s_1,s_2)\|fg\|_{\dot{B}^{s_1+s_2-\frac{d}{2}}_{2,2}} \leq C(d,s_1,s_2)\|fg\|_{\dot{B}^{s_1+s_2-\frac{d}{2}}_{2,1}} \leq C(d,s_1,s_2)\|f\|_{H^{s_1}} \|g\|_{H^{s_2}}.
	\end{align*}
	It is also convenient to recall the time-space Besov spaces. For $T > 0$, $s \in \mathbb{R}$, $r_0,p_0,q_0 \in [1,\infty]$, the Chemin-Lerner spaces $\tilde{L}^{r_0}(0,T;\dot{B}^s_{p_0,q_0}(\mathbb{R}^d))$ and $\tilde{L}^{r_0}(0,T;B^s_{p_0,q_0}(\mathbb{R}^d))$ were introduced in \cite{Chemin-Lerner_1995} (see \cite{Bahouri-Chemin-Danchin_2011} for more details) and are given as follows
	\begin{align*}
		\tilde{L}^{r_0}(0,T;\dot{B}^s_{p_0,q_0}(\mathbb{R}^d)) &:= \left\{f \in \mathcal{S}'_0(\mathbb{R}^d) :  \|f\|_{\tilde{L}^{r_0}(0,T;\dot{B}^s_{p_0,q_0}(\mathbb{R}^d))} := \|2^{sq}\|\dot{\Delta}_q f\|_{L^{r_0}(0,T;L^{p_0}(\mathbb{R}^d))}\|_{\ell^{q_0}(\mathbb{Z})}  < \infty \right\},
		\\
		\mathcal{S}'_0 (\mathbb{R}^d) &:= \left\{f \in \mathcal{S}'(\mathbb{R}^d) : \lim_{k \to -\infty} \|\dot{S}_k f\|_{L^{r_0}(0,T;L^{p_0}(\mathbb{R}^d))} = 0 \right\},
		\\
		\tilde{L}^{r_0}(0,T;B^s_{p_0,q_0}(\mathbb{R}^d)) &:= \left\{ f \in \mathcal{S}'(\mathbb{R}^d) : \|f\|_{\tilde{L}^{r_0}(0,T;B^s_{p_0,q_0}(\mathbb{R}^d))} := \|2^{sq}\|\Delta_q f\|_{L^{r_0}(0,T;L^{p_0}(\mathbb{R}^d))}\|_{\ell^{q_0}(\mathbb{Z})} < \infty \right\}.
	\end{align*}
	By using Minkowski inequality for integrals, the following relations hold 
	\begin{align*}
		\tilde{L}^{r_0}(0,T;\dot{B}^s_{p_0,q_0}(\mathbb{R}^d)) \subset L^{r_0}(0,T;\dot{B}^s_{p_0,q_0}(\mathbb{R}^d)) \quad \text{and} \quad \tilde{L}^{r_0}(0,T;B^s_{p_0,q_0}(\mathbb{R}^d)) \subset L^{r_0}(0,T;B^s_{p_0,q_0}(\mathbb{R}^d)) \quad \text{if} \quad r_0 \geq q_0,
		\\
		L^{r_0}(0,T;\dot{B}^s_{p_0,q_0}(\mathbb{R}^d)) \subset \tilde{L}^{r_0}(0,T;\dot{B}^s_{p_0,q_0}(\mathbb{R}^d)) \quad \text{and} \quad 
		L^{r_0}(0,T;B^s_{p_0,q_0}(\mathbb{R}^d)) \subset \tilde{L}^{r_0}(0,T;B^s_{p_0,q_0}(\mathbb{R}^d))
		\quad \text{if} \quad r_0 \leq q_0.
	\end{align*}

	\subsection{Appendix B: Homogeneous Sobolev inequalities and proof of \eqref{L-infty}}
	
	There is a proof of \eqref{L-infty} in \cite{Kukavica-Wang-Ziane_2016} in a more general $L^p$ framework. In Hilbert spaces, the proof is much more simpler. However, we do not find a specific reference for the proof of the three-dimensional case, especially for the homogeneous Sobolev norm version, so for the sake of completeness, we provide a standard proof of \eqref{L-infty} and its three-dimensional version as follows. Since we used both versions in the previous proofs. We note that for the nonhomogeneous Sobolev norm version, it is a consequence of a result in \cite{Brezis-Mironescu_2018}.
	
	\begin{lemma} \label{lem-Agmon} Assume that $f \in 
	H^s(\mathbb{R}^d)$ with $s > \frac{d}{2}$ and $d \in \{2,3\}$ then
		\begin{equation} \label{Agmon}
			\|f\|_{L^\infty} \leq C(s)
			\begin{cases}
				\|f\|^{\frac{s-1}{s}}_{L^2}\|f\|^\frac{1}{s}_{\dot{H}^s} &\text{if} \quad d = 2,
				\\
				\|f\|^{\frac{2s-3}{2s}}_{L^2}\|f\|^\frac{3}{2s}_{\dot{H}^s} &\text{if} \quad d = 3.
			\end{cases} 
		\end{equation}
	\end{lemma}
	
	\begin{proof}[Proof of Lemma \ref{lem-Agmon}]
		It can be seen that for $x \in \mathbb{R}^d$ 
		\begin{align*}
			f(x) = \int_{|\xi| \leq M} \exp\{ix \cdot \xi\} \mathcal{F}(f)(\xi) \,d\xi + \int_{|\xi| > M} \exp\{ix \cdot \xi\} \mathcal{F}(f)(\xi) \,d\xi =: F_1 + F_2,
		\end{align*}
		where $M$ is a positive constant to be determined later and 
		\begin{equation*}
			|f(x)| \leq |F_1| + |F_2| \leq 
			\begin{cases}
				CM\|f\|_{L^2} + C(s)M^{1-s}\|f\|_{\dot{H}^s} &\text{if} \quad d = 2,
				\\
				CM^\frac{3}{2}\|f\|_{L^2} + C(s)M^{\frac{3}{2}-s}\|f\|_{\dot{H}^s} &\text{if} \quad d = 3.
			\end{cases}
		\end{equation*}
		Thus, \eqref{Agmon} follows by choosing $M = \|f\|^{-\frac{1}{s}}_{L^2}\|f\|^\frac{1}{s}_{\dot{H}^s}$. 
	\end{proof}
	
	\subsection{Appendix C: A logarithmic Gronwall inequality}
	
	In this subsection, we provide a simple proof of a lograrithmic Gronwall inequality, which is used several times before.
	
	\begin{lemma} \label{lem-gronwall}
		Assume that $h_1,h_2,y \geq 0$ satisfying $h_1,h_2 \in L^1_{\textnormal{loc}}(0,\infty)$, $y(0) \geq 0$ and for $\alpha \geq 1$
		\begin{equation} \label{log-gron}
			\frac{d}{dt} y(t) \leq h_1(t) y(t) + h_2(t)\log(\alpha+y(t))y(t) \qquad \text{for}\quad t > 0,
		\end{equation}
		then 
		\begin{equation} \label{log-gron1}
			y(t_2) \leq \exp\left\{\left(\log(\alpha+y(t_1)) + \int^{t_2}_{t_1} h_1 \,d\tau \right)  \exp\left\{\int^{t_2}_{t_1} h_2 \,d \tau \right\}\right\} \qquad \text{for} \quad 0 \leq t_1 \leq t_2 < \infty.
		\end{equation}
	\end{lemma}
	
	\begin{proof}[Proof of Lemma \ref{lem-gronwall}] By setting $v(t) := \log(\alpha+y(t))$, it can be seen from \eqref{log-gron} that
		\begin{equation*}
			\frac{d}{dt} v(t) \leq h_1(t) + h_2(t) v(t).
		\end{equation*}
		Therefore, for $0 \leq t_1 \leq t_2 < \infty$
		\begin{equation*}
			v(t_2)\exp\left\{-\int^{t_2}_0 h_2 \,d \tau \right\} \leq v(t_1) \exp\left\{-\int^{t_1}_0 h_2 \,d \tau \right\} + \int^{t_2}_{t_1} h_1 \exp\left\{-\int^\tau_0 h_1 \,ds\right\}\,d\tau,
		\end{equation*}
		which implies that
		\begin{equation*}
			\log(\alpha+y(t_2)) \leq 
			\left(\log(\alpha+y(t_1)) + \int^{t_2}_{t_1} h_1 \,d\tau \right)  \exp\left\{\int^{t_2}_{t_1} h_2 \,d \tau \right\}
		\end{equation*}
		and \eqref{log-gron1} follows.	
	\end{proof}
	
	\subsection{Appendix D: Parabolic regularity and proof of \eqref{Maximal_Regularity}}
	
	For the sake of completeness, we will provide here a proof of \eqref{Maximal_Regularity}, which is a special case of a more general situation below. Let us consider a fractional heat equation given in the following form for suitable force $f$, initial data $w_0$, $\alpha  \in [0,\infty)$, $\nu \in (0,\infty),T \in (0,\infty]$ and $d \geq 1$
	\begin{equation} \label{F_heat}
		\partial_t w + \nu(-\Delta)^\alpha w = f \qquad \text{in}\quad  (0,T) \times \mathbb{R}^d \qquad \text{and} \qquad w_{|_{t= 0}} = w_0.
	\end{equation}
	It is well-known that the solution to \eqref{F_heat} can be represented by the following Duhamel formula
	\begin{equation*}
		w(t) = \exp\{t\nu(-\Delta)^\alpha\} w_0 + \int^t_0 \exp\{(t-\tau)\nu(-\Delta)^\alpha\} f(\tau) \,d\tau \qquad \text{for} \quad t \in (0,T),
	\end{equation*}
	where we have been used the notation 
	\begin{equation*}
		\exp\{t \nu(-\Delta)^\alpha\} f := \mathcal{F}^{-1}(\exp\{-\nu t|\xi|^{2\alpha}\}\mathcal{F}(f)(\xi)).
	\end{equation*}
	In the sequel, we aim to prove the following result, which mostly follows the ideas in \cite[Proposition 3.1]{Arsenio-Gallagher_2020}, where the authors focused  on the case $\alpha = 1$. We note that for a similar result in form of Chemin-Lerner spaces, see \cite[Proposition 2]{Chae-Lee_2004}.

	\begin{proposition} \label{pro-F_heat} 
		Let $d \geq 1$ and $w$ be a solution to \eqref{F_heat} with $w_{|_{t= 0}} = w_0$, $\alpha \in [0,\infty)$, and $\nu \in (0,\infty)$. Assume that $\delta_0 \in \mathbb{R}$, $p \in [1,\infty]$, $1 < r \leq m < \infty$, $1 \leq q \leq m$, $T \in (0,\infty]$, $w_0 \in \dot{B}^{\delta_0 + 2\alpha}_{p,q}(\mathbb{R}^d))$ and $f \in L^r(0,T;\dot{B}^{\delta_0 + \frac{2\alpha}{r}}_{p,q}(\mathbb{R}^d)$. Then there are some positive constants $C_1 = C_1(\alpha,d,\delta_0,m,\nu,p,q,r)$ and $C_2 = C_2(\alpha,d,\delta_0,m,\nu,p,q)$ such that
		\begin{equation} \label{F_heat_0}
			\|w\|_{L^m(0,T;\dot{B}^{\delta_0 + 2\alpha + \frac{2\alpha}{m}}_{p,q}(\mathbb{R}^d))} \leq C_1 \|f\|_{L^r(0,T;\dot{B}^{\delta_0 + \frac{2\alpha}{r}}_{p,q}(\mathbb{R}^d))} + C_2 \|w_0\|_{\dot{B}^{\delta_0 + 2\alpha}_{p,q}(\mathbb{R}^d)}.
		\end{equation}
	\end{proposition}
	Once the above proposition is established, \eqref{Maximal_Regularity} follows by choosing $\alpha = \frac{3}{2}$, $\delta_0 = s - 3$, $m = r = p = 2$, $q = 1$ and $w = v^2$ with $w_0 = v^2_0 = 0$. Before going to the proof of  Proposition \ref{pro-F_heat}, we need to establish the following technical lemma, which follows the ideas in \cite[Lemma 2.4]{Bahouri-Chemin-Danchin_2011}, \cite[Lemma 2.1]{Chemin_1999}, where the authors considered  the case $\alpha = 1$. See also \cite[Lemma 1]{Chae-Lee_2004} with a similar proof in the case $d = 3$ and $\alpha \geq 0$.
	
	\begin{lemma} \label{lem-F_heat}
		Let $d \geq 1$ and $\mathcal{C}(c_1,c_2)$ be an annulus with the smaller radius $c_1 > 0$ and the bigger radius $c_2 > 0$. There exist positive constants $C_3 = C_3(\alpha,c_1,c_2,d)$ and $C_4 = C_4(\alpha,c_1,d)$ such that for any $\alpha \in [0,\infty)$, $p \in [1,\infty]$ and any pair $(t, \lambda)$ of positive real numbers the following property holds. If $\textnormal{supp}(\mathcal{F}(u)) \subset \lambda \mathcal{C}$ then 
		\begin{equation} \label{F_heat_estimate}
			\|\exp\{t\nu(-\Delta)^\alpha\} u\|_{L^p(\mathbb{R}^d)} \leq C_4 \exp\{-C_3\nu t\lambda^{2\alpha}\}\|u\|_{L^p(\mathbb{R}^d)}.
		\end{equation}
	\end{lemma}
	
	\begin{proof}[Proof of Lemma \ref{lem-F_heat}] It can be seen that in the case $p = 2$, \eqref{F_heat_estimate} follows immediately by using the Plancherel’s identity. For $p \in [1,\infty]$, we will closely follow the idea in \cite[Lemma 2.4]{Bahouri-Chemin-Danchin_2011} by focusing mainly on the case $\lambda = 1$. Indeed, the case $\lambda \neq 1$ can be transformed to the case $\lambda = 1$ as follows. Assume that \eqref{F_heat_estimate} holds in the case $\lambda = 1$ for any $t' \in (0,\infty)$ and for any $f$ satisfying $\textnormal{supp}(\mathcal{F}(f)) \subset \mathcal{C}$, i.e., 
		\begin{equation} \label{F_heat_estimate_1}
			\|\exp\{t'\nu(-\Delta)^\alpha\} f\|_{L^p} \leq C_3 \exp\{-C_4\nu t'\}\|f\|_{L^p}.
		\end{equation}
		We now fix $(t,\lambda) \in (0,\infty)$. Let $u$ be a function such that $\textnormal{supp}(\mathcal{F}(u)) \subset \lambda \mathcal{C}$. We define for $x \in \mathbb{R}^d$
		\begin{equation*}
			v(x) := \frac{1}{\lambda^d}u\left(\frac{x}{\lambda}\right) \qquad \text{with} \qquad \mathcal{F}(v)(\xi) = \mathcal{F}(u)(\lambda \xi) \quad \forall \xi \in \mathbb{R}^d,
		\end{equation*}
		which yields $\textnormal{supp}(\mathcal{F}(v)) \subset \mathcal{C}$. An application of \eqref{F_heat_estimate_1} to the case $f = v$ and $t' = t\lambda^{2\alpha}$ gives us
		\begin{equation} \label{F_heat_estimate_2}
			\|\exp\{t\lambda^2\nu(-\Delta)^\alpha\} v\|_{L^p} \leq C_3 \exp\{-C_4\nu t\lambda^2\}\|v\|_{L^p} = \lambda^{\frac{d}{p}-d} C_3 \exp\{-C_4\nu t\lambda^{2\alpha}\}\|u\|_{L^p}.
		\end{equation}
		Furthermore, it can be verified that
		\begin{align*}
			\lambda^{d-\frac{d}{p}}\|\exp\{t\lambda^2\nu(-\Delta)^\alpha\} v\|_{L^p} 
			&= \lambda^{d-\frac{d}{p}}\|\mathcal{F}^{-1}( \exp\{-\nu t|\lambda \xi|^{2\alpha}\}\mathcal{F}(v)(\xi))(\cdot)\|_{L^p}
			\\
			&= \|\lambda^{d-\frac{d}{p}} \mathcal{F}^{-1}(\exp\{-\nu t|\lambda \xi|^{2\alpha}\}\mathcal{F}(u)(\lambda \xi))(\cdot)\|_{L^p}
			\\
			&= \|\lambda^{-\frac{d}{p}} \mathcal{F}^{-1}(\exp\{-\nu t|\xi|^{2\alpha}\}\mathcal{F}(u)(\xi))(\lambda^{-1}\cdot)\|_{L^p}
			\\
			&= \|\mathcal{F}^{-1}(\exp\{-\nu t|\xi|^{2\alpha}\}\mathcal{F}(u)(\xi))(\cdot)\|_{L^p} = \|\exp\{t\nu(-\Delta)^\alpha\} u\|_{L^p},
		\end{align*}
		which combines with \eqref{F_heat_estimate_2} leading to \eqref{F_heat_estimate}. Therefore, it remains to check \eqref{F_heat_estimate} in the case $\lambda = 1$.
		By choosing $\phi \in C^\infty_0(\mathbb{R}^d \setminus \{0\})$ with $0 \leq \phi \leq 1$, $\phi = 1$ in $\mathcal{C}(c_1,c_2)$ and $\phi = 0$ outside of $\mathcal{C}(\frac{1}{2}c_1,\frac{3}{2}c_2)$. 
		Since $\textnormal{supp}(\mathcal{F}(u)) \subset \mathcal{C}$ and $\phi = 1$ in $\mathcal{C}(c_1,c_2)$ then by using Young inequality for convolution
		\begin{equation*}
			\|\exp\{t\nu(-\Delta)^\alpha\} u\|_{L^p} 
			= \|\mathcal{F}^{-1}(\phi(\xi) \exp\{-\nu t|\xi|^{2\alpha}\}\mathcal{F}(u)(\xi))\|_{L^p} 
			\leq \|G(t,\cdot)\|_{L^1} \|u\|_{L^p},
		\end{equation*}
		where for $x \in \mathbb{R}^d$
		\begin{equation*}
			G(t,x) := (2\pi)^{-d} \int_{\mathbb{R}^d} \exp\{ix\cdot \xi\} \phi(\xi) \exp\{-\nu t|\xi|^{2\alpha}\} \,d\xi.
		\end{equation*}
		It remains to bound $\|G(t,\cdot)\|_{L^1}$. By using integration by parts, $G$ can be rewritten by
		\begin{equation*}
			G(t,x) = (2\pi)^{-d} (1 + |x|^2)^{-d} \int_{\mathbb{R}^d} \exp\{ix \cdot \xi\} (\textnormal{Id}-\Delta_\xi)^d \left(\phi(\xi) \exp\{-\nu t|\xi|^{2\alpha}\}\right) \,d\xi.
		\end{equation*}
		We need to control the second term inside of the above integral. It can be checked that
		\begin{align*}
			(\textnormal{Id}-\Delta_\xi)^d(\phi(\xi)\exp\{-\nu t |\xi|^{2\alpha}\}) = \sum_{0 \leq j \leq d} \sum_{0 \leq |\alpha_0| \leq 2j}
			C(\alpha_0,j)\partial^{|\alpha_0|}(\phi(\xi)) \partial^{2j-|\alpha_0|} (\exp\{-\nu t |\xi|^{2\alpha}\}).
		\end{align*}
		In addition, since $\textnormal{supp}(\phi) \subset \mathcal{C}(\frac{1}{2}c_1,\frac{3}{2}c_2)$ then for $\xi \in \textnormal{supp}(\phi)$, we find that  $|\partial^{|\alpha_0|} (\phi(\xi))| \leq
		C(c_1,c_2,d)$ and 
		\begin{align*}
			|\partial^{2j-|\alpha_0|} \exp\{-\nu t |\xi|^{2\alpha}\}| 
			&\leq C(\alpha,d) \sum_{0 \leq i \leq 2j-|\alpha_0|} (\nu t |\xi|^{2\alpha})^i |\xi|^{-2j + |\alpha_0|} \exp\{-\nu t |\xi|^{2\alpha}\}
			\\
			&\leq C(\alpha,d)  |\xi|^{-2j + |\alpha_0|}  (1 + \nu t |\xi|^{2\alpha})^{2d} \exp\{-\nu t |\xi|^{2\alpha}\} 
			\\
			&\leq C(\alpha,c_1,c_2,d)\exp\{-C(\alpha,c_1,d)\nu t\},
		\end{align*}
		where we also used another fact that $s\exp\{-s\} \leq \exp\{1\}\exp\{-\frac{1}{2}s\}$ for any $s \in \mathbb{R}, s \geq  0$, which leads to $(1 + s)^{2d}\exp\{-s\} \leq C(d)\exp\{-c(d)s\}$ as well. Therefore,
		\begin{align*}
			|G(t,x)| &\leq C(\alpha,c_1,c_2,d) (1 + |x|^2)^{-d} \exp\{-C(\alpha,c_1,d)\nu t\},
		\end{align*}
		which implies that
		\begin{equation*}
			\|G(t,\cdot)\|_{L^1} \leq C(\alpha,c_1,c_2,d) \exp\{-C(\alpha,c_1,d)\nu t\}.
		\end{equation*}
		Thus, the proof is complete.
	\end{proof}
	
	\begin{proof}[Proof of Proposition \ref{pro-F_heat}] The proof consists of the following steps.
		
		\textbf{Step 1: Parabolic regularity estimate.} We aim to obtain the following standard estimate 
		\begin{equation} \label{F_heat_1}
			\|w\|_{L^m(0,T;\dot{B}^{\delta_0 + 2\alpha + \frac{2\alpha}{m}}_{p,q}(\mathbb{R}^d))} \leq C_1 \|f\|_{L^r(0,T;\dot{B}^{\delta_0 + \frac{2\alpha}{r}}_{p,q}(\mathbb{R}^d))} + C_2\|w_0\|_{\dot{B}^{\delta_0 + 2\alpha}_{p,q}(\mathbb{R}^d)}.
		\end{equation}
		for the same range of parameters $\alpha,\delta_0,m,p,r$ and similar constants $C_1,C_2$ as in \eqref{F_heat_0}, but only for $r \leq q \leq m$. In order to prove \eqref{F_heat_1}, as the usual case $\alpha = 1$, we decompose $\dot{\Delta}_k w := \dot{\Delta}_k w_1 + \dot{\Delta}_k w_2$ for each $k \in \mathbb{Z}$, where
		\begin{align*}
			&&\partial_t \dot{\Delta}_k w_1 + \nu(-\Delta)^\alpha \dot{\Delta}_k w_1 &= 0,  &&\dot{\Delta}_k w_1(t=0) = \dot{\Delta}_k w_0,&&
			\\
			&&\partial_t \dot{\Delta}_k w_2 + \nu(-\Delta)^\alpha \dot{\Delta}_k w_2 &= \dot{\Delta}_k f,  &&\dot{\Delta}_k w_2(t=0) = 0,&&
		\end{align*}
		and estimate $w_1$ and $w_2$ in the desired norms. We begin with the bound of $w_1$ by using the fact $\dot{\Delta}_k w_1 = \exp\{t\nu(-\Delta)^\alpha (\dot{\Delta}_k w_0)\}$ in which Lemma \ref{lem-F_heat} yields
		\begin{equation*}
			\|\dot{\Delta}_k w_1\|_{L^p} \leq C_3 \exp\{-C_4\nu t 2^{2\alpha k}\} \|\dot{\Delta}_k w_0\|_{L^p},
		\end{equation*}
		and a direct calculation, which implies that for $m,p,q \in [1,\infty]$
		\begin{equation*}
			\|w_1\|_{\tilde{L}^m(0,T;\dot{B}^{\delta_0 + 2\alpha + \frac{2\alpha}{m}}_{p,q})} \leq C(\alpha,d,\nu,m) \left(\sum_{k \in \mathbb{Z}} 2^{(\delta_0 + 2\alpha)kq} 
			\|\dot{\Delta}_k w_0\|^q_{L^p}\right)^\frac{1}{q} = C(\alpha,d,\nu,m) \|w_0\|_{\dot{B}^{\delta_0 + 2\alpha}_{p,q}}.
		\end{equation*}
		Similarly, since
		\begin{equation*}
			\dot{\Delta}_k w_2(t) = \int^t_0 \exp\{(t-\tau)\nu(-\Delta)^\alpha \dot{\Delta}_k f \,d\tau,
		\end{equation*}
		by using Lemma \ref{lem-F_heat} and Minkowski inequality for integrals with $r,m,p \in [1,\infty]$ and $r \leq m$
		\begin{align*}
			\|\dot{\Delta}_k w_2\|_{L^m(0,T;L^p)} 
			&\leq C_3\left(\int^T_0 \left(\int^t_0 \exp\{-C_4 \nu (t-\tau) 2^{2\alpha k}\} \|\dot{\Delta}_k f\|_{L^p} d\tau\right)^m \,dt\right)^\frac{1}{m}
			\\
			&= C_3\left(\int^T_0 \left(\int^t_0 \exp\{-C_4 \nu (t-\tau) 2^{2\alpha k}\}  \mathbbm{1}_{t\geq\tau}(t) \|\dot{\Delta}_k f(\tau)\|_{L^p} \,d\tau\right)^m \,dt\right)^\frac{1}{m}
			\\
			&\leq C_3\int^T_0 \left(\int^T_{\tau} \exp\{-C_4 m\nu (t-\tau) 2^{2\alpha k}\}  \,dt\right)^\frac{1}{m}  \|\dot{\Delta}_k f(\tau)\|_{L^p} \,d\tau
			\\
			&\leq C(C_3,C_4,m,\nu) \int^T_0 2^{-\frac{2\alpha k}{m}} \exp\{-C_4 \nu \tau 2^{2\alpha k}\} \|\dot{\Delta}_k f(\tau)\|_{L^p} \,d\tau
			\\
			&\leq C(C_3,C_4,m,\nu) 2^{-\frac{2\alpha k}{m}}  \left(\int^T_0 \exp\left\{-C_4 \nu \tau 2^{2\alpha k} \frac{r}{r-1}\right\} \,d\tau\right)^\frac{r-1}{r} \|\dot{\Delta}_k f\|_{L^r(0,T;L^p)} 
			\\
			&\leq C(C_3,C_4,m,\nu,r) 2^{-\frac{2\alpha k}{m}}  2^{-{2\alpha k} \frac{r-1}{r}} \|\dot{\Delta}_k f\|_{L^r(0,T;L^p)},
		\end{align*}
		which leads to for $1 \leq r \leq m$ and for $m,p,q \in [1,\infty]$
		\begin{align*}
			\|w_2\|_{\tilde{L}^m(0,T;\dot{B}^{\delta_0 + 2\alpha + \frac{2\alpha}{m}}_{p,q})} 
			\leq C(\alpha,d,\nu,m,r) \|f\|_{\tilde{L}^r(0,T;\dot{B}^{\delta_0 + \frac{2\alpha}{r}}_{p,q})}.
		\end{align*}
		Therefore, for $m,r,p,q \in [1,\infty]$ with $1 \leq r \leq m$
		\begin{equation*}
			\|w\|_{\tilde{L}^m(0,T;\dot{B}^{\delta_0 + 2\alpha + \frac{2\alpha}{m}}_{p,q})} \leq C_1 \|f\|_{\tilde{L}^r(0,T;\dot{B}^{\delta_0 + \frac{2\alpha}{r}}_{p,q})} + C_2\|w_0\|_{\dot{B}^{\delta_0 + 2\alpha}_{p,q}}.
		\end{equation*}
		Furthermore, by using the properties of Chemin-Lerner spaces given in Appendix A, we find that
		\begin{align*}
			&&L^r(0,T;\dot{B}^{\delta_0 + \frac{2\alpha}{r}}_{p,q}) &\subset \tilde{L}^r(0,T;\dot{B}^{\delta_0 + \frac{2\alpha}{r}}_{p,q}) &&\text{if} \quad r \leq q,&&
			\\
			&&\tilde{L}^m(0,T;\dot{B}^{\delta_0 + 2\alpha + \frac{2\alpha}{m}}_{p,q}) &\subset L^m(0,T;\dot{B}^{\delta_0 + 2\alpha + \frac{2\alpha}{m}}_{p,q}) &&\text{if} \quad q \leq m.&&
		\end{align*}
		Thus, \eqref{F_heat_1} follows by the previous estimate.
		 
		\textbf{Step 2: The case $w_0 = 0$, $m = r$ and $q = 1$.} Similar to \cite[Proposition 3.1]{Arsenio-Gallagher_2020} by using the duality argument, for all $g \in L^{r'}(0,T)$ with $\frac{1}{r} + \frac{1}{r'} = 1$ and for $C = C(\alpha,d,\delta_0,\nu,p,r)$, it is enough to prove that 
		\begin{equation*}
			I := \sum_{k \in \mathbb{Z}} \int^T_0 g(t) 2^{k(\delta_0 + 2\alpha + \frac{2\alpha}{r})}\|\dot{\Delta}_k w(t)\|_{L^p} \,dt = \int^T_0 g(t) \|w(t)\|_{\dot{B}^{\delta_0 + 2\alpha + \frac{2\alpha}{r}}_{p,1}} \,dt \leq C \|f\|_{L^r(0,T;\dot{B}^{\delta_0 + \frac{2\alpha}{r}}_{p,1})} \|g\|_{L^{r'}(0,T)}.
		\end{equation*}
		It can be seen from the representation formula and Lemma \ref{lem-F_heat} that for $t \in (0,T)$ and $k \in \mathbb{Z}$
		\begin{align*}
			\|\dot{\Delta}_k w(t)\|_{L^p} \leq  \int^t_0 \|\exp\{(t-\tau)\nu(-\Delta)^\alpha\} \dot{\Delta}_k f(\tau)\|_{L^p} \,d\tau \leq  C(\alpha,d)\int^t_0 \exp\{-C(d)\nu(t-\tau)2^{2\alpha k}\} \|\dot{\Delta}_k f(\tau)\|_{L^p} \,d\tau,
		\end{align*}
		where we also used the property $\textnormal{supp}(\mathcal{F}(\dot{\Delta}_k w)) \subset \mathcal{C}(\frac{3}{4}2^k,\frac{8}{3}2^k)$, the annulus with the smaller radius $\frac{3}{4}2^k$ and the bigger radius $\frac{8}{3}2^k$, which yields
		\begin{align*}
			I &\leq C(\alpha,d)\sum_{k \in \mathbb{Z}} \int^T_0 \int^t_0  2^{2\alpha k} \exp\{-C(\alpha,d)\nu(t-\tau)2^{2\alpha k}\} |g(t)| 2^{(\delta_0 + \frac{2\alpha}{r})k} \|\dot{\Delta}_k f(\tau)\|_{L^p} \,d\tau dt
			\\
			&= C(\alpha,d,\nu)\sum_{k \in \mathbb{Z}} \int^T_0 \int^T_{\tau} \nu 2^{2\alpha k} 
			\exp\{-C(\alpha,d)\nu(t-\tau)2^{2\alpha k}\} |g(t)|  \,2^{(\delta_0 + \frac{2\alpha}{r})k} \|\dot{\Delta}_k f(\tau)\|_{L^p}  \,dt d\tau 
			\\
			&= C(\alpha,d,\nu)\sum_{k \in \mathbb{Z}} \int^T_0 \int^T_0 \nu 2^{2\alpha k} \mathbbm{1}_{t \geq \tau}(t)  \exp\{-C(\alpha,d)\nu(t-\tau)2^{2\alpha k}\} |g(t)|  \,2^{(\delta_0 + \frac{2\alpha}{r})k} \|\dot{\Delta}_k f(\tau)\|_{L^p}  \,dtd\tau 
			\\
			&\leq C(\alpha,d,\nu)\sum_{k \in \mathbb{Z}} \int^T_0  Mg(\tau) 2^{(\delta_0 + \frac{2\alpha}{r})k} \|\dot{\Delta}_k f(\tau)\|_{L^p}  \,d\tau
			\\
			&= C(\alpha,d,\nu) \int^T_0  Mg(\tau) \|f(\tau)\|_{\dot{B}^{\delta_0 + \frac{2\alpha}{r}}_{p,1}}  \,d\tau
			\\
			&\leq C(\alpha,d,\nu) \|Mg\|_{L^{r'}(0,T)} \|f\|_{L^r(0,T;\dot{B}^{\delta_0 + \frac{2\alpha}{r}}_{p,1})}
			\\
			&\leq C(\alpha,d,\nu,r') \|g\|_{L^{r'}(0,T)} \|f\|_{L^r(0,T;\dot{B}^{\delta_0 + \frac{2\alpha}{r}}_{p,1})},
		\end{align*}
		where 
		\begin{equation*}
			Mg(\tau) := 
			\begin{cases}
				\sup_{\rho > 0} \int^T_0 \rho \mathbbm{1}_{t \geq \tau}(t) \exp\{-C(\alpha,d)(t-\tau) \rho\}  |g(t)| \,dt &\text{if} \quad \tau \in (0,T),
				\\
				0 &\text{if} \quad \tau \notin (0,T).
			\end{cases}
		\end{equation*}
		It remains to check the last inequality in the previous estimate. Indeed, it can be seen that if $\tau \in (0,T)$ then $Mg(\tau)$ can be rewritten by
		\begin{equation*}
			Mg(\tau) = \sup_{\frac{1}{\rho} > 0} (K_{\frac{1}{\rho}} *  \tilde{g})(\tau) = \int_{\mathbb{R}} K_{\frac{1}{\rho}}(\tau-t)  \tilde{g}(t) \,dt,
		\end{equation*}
		where for $t \in \mathbb{R}$
		\begin{equation*}
			K(t) := \exp\{-C(\alpha,d)|t|\}, \quad K_{\frac{1}{\rho}}(t) := \rho K(t \rho) 
			\quad \text{and}\quad \tilde{g}(t) :=  \mathbbm{1}_{t \geq \tau}(t) \mathbbm{1}_{(0,T)}(t) |g(t)|.
		\end{equation*}
		In addition, we can verify that $K$ satisfies all conditions in \cite[Theorem 2.1.10]{Grafakos_2014}, which yields 
		\begin{equation*}
			Mg(\tau) = \sup_{\frac{1}{\rho} > 0} (K_{\frac{1}{\rho}} *  \tilde{g})(\tau) \leq \|K\|_{L^1(\mathbb{R})} \mathcal{M}(\tilde{g})(\tau) \leq C(\alpha,d) \mathcal{M}(\tilde{g})(\tau),
		\end{equation*}
		where the centered Hardy–Littlewood maximal function of $\tilde{g}$ is defined by
		\begin{equation*}
			\mathcal{M}(\tilde{g})(\tau) := \sup_{r > 0} \frac{1}{2r} \int_{B(\tau,r)} |\tilde{g}(t)|\,dt \qquad \text{with} \qquad B(\tau,r) := \{s \in \mathbb{R} : |s-\tau| < r\}. 
		\end{equation*}
		Finally, an application of \cite[Theorem 2.1.6]{Grafakos_2014}, which is on the boundedness of the maximal operator $\mathcal{M}$ from $L^{p_0}$ to $L^{p_0}$ for $p_0 \in (1,\infty)$, implies that
		\begin{equation*}
			\|Mg\|_{L^{r'}(\mathbb{R})} \leq C(\alpha,d) \|\mathcal{M}\tilde{g}\|_{L^{r'}(\mathbb{R})} \leq C(\alpha,d,r') \|\tilde{g}\|_{L^{r'}(\mathbb{R})} = C(\alpha,d,r') \|g\|_{L^{r'}(0,T)}.
		\end{equation*}
		
		\textbf{Step 3: The case $w_0 = 0$, $1 < r \leq m < \infty$ and $1 \leq q \leq m$.} We use exactly the argument in \cite{Arsenio-Gallagher_2020}. More precisely, it follows from \eqref{F_heat_0} for $m = r$, $q = 1$ and from \eqref{F_heat_1} for $q = r = 1$, respectively,  that
		\begin{align*}
			\|w\|_{L^m(0,T;\dot{B}^{\delta_0 + 2\alpha + \frac{2\alpha}{m}}_{p,1})} &\leq C_1 \|f\|_{L^m(0,T;\dot{B}^{\delta_0 + \frac{2\alpha}{m}}_{p,1})},
			\\
			\|w\|_{L^m(0,T;\dot{B}^{\delta_0 + 2\alpha + \frac{2\alpha}{m}}_{p,1})} &\leq C_1 \|f\|_{L^1(0,T;\dot{B}^{\delta_0 + 2\alpha}_{p,1})},
		\end{align*}
		 which combines with interpolation theory in \cite[Theorems 5.1.2 and 6.4.5]{Bergh-Lofstrom_1976} yielding for $1 < r \leq m < \infty$
		 \begin{equation*}
		 	\|w\|_{L^m(0,T;\dot{B}^{\delta_0 + 2\alpha + \frac{2\alpha}{m}}_{p,1})} \leq C_1 \|f\|_{L^r(0,T;\dot{B}^{\delta_0 + \frac{2\alpha}{r}}_{p,1})}.
		 \end{equation*}
		In addition, \eqref{F_heat_1} with $q = m$, which  gives us
		\begin{equation*}
			\|w\|_{L^m(0,T;\dot{B}^{\delta_0 + 2\alpha + \frac{2\alpha}{m}}_{p,m})} \leq C_1 \|f\|_{L^r(0,T;\dot{B}^{\delta_0 + \frac{2\alpha}{r}}_{p,m})}.
		\end{equation*}
		Thus, combining the two previous estimates finishes the proof of this step. Therefore, \eqref{F_heat_0} with $w_0 = 0$ follows.
		
		\textbf{Step 4: The case $w_0 \neq 0$, $1 < r \leq m < \infty$ and $1 \leq q \leq m$.} In this step, in order to prove \eqref{F_heat_0} in the case $w_0 \neq 0$,  we can repeat Step 1 (use the same estimate of $w_1$) with using \eqref{F_heat_0} for $w_0= 0$ (in the estimate of $w_2$) to obtain the desired result. We omit further details. Thus, the proof of the proposition now is finished.
	\end{proof}
	
	%
	\subsection{Appendix E: Remarks on the Maxwell equations}
	%
	
	In this subsection, we show that under suitable assumptions on the velocity, the existence and uniqueness of $L^2$ weak solutions to \eqref{M} can be provided. We first recall the following result.
	
	\begin{lemma} \label{lem_M0} \textnormal{(\cite[Lemma 2.2]{Giga-Ibrahim-Shen-Yoneda_2018}, \cite[Lemma 1.10]{Masmoudi_2010})} If $(E_0,B_0) \in H^s(\mathbb{R}^d)$ with $s \in \mathbb{R}$ 
		and  $j \in L^1_{\textnormal{loc}}(0,\infty;H^s(\mathbb{R}^d))$ then any solution $(E,B)$  to \eqref{M} satisfying for any $T \in (0,\infty)$
		\begin{equation*}
			\|(E,B)\|^2_{L^\infty(0,T;H^s)} \leq  2\|(E_0,B_0)\|^2_{H^s} + 2c^2\|j\|^2_{L^1(0,T;H^s)}.
		\end{equation*}
	\end{lemma}
	
	\begin{lemma} \label{lem-M} Let $d \in \{2,3\}$ and $E,B,v : \mathbb{R}^d \times (0,\infty) \to \mathbb{R}^3$ satisfying \eqref{M} with $(E,B)_{|_{t=0}} = (E_0,B_0)$.
		\begin{enumerate} 
			\item[(i)] \textnormal{(Global well-posedness)} If $v \in L^\infty_{\textnormal{loc}}(0,\infty;L^2(\mathbb{R}^d)) \cap  L^2_{\textnormal{loc}}(0,\infty;H^\frac{d}{2}(\mathbb{R}^d) \cap L^\infty(\mathbb{R}^d))$ and $(E_0,B_0) \in H^s(\mathbb{R}^d)$ with $s \in [0,\frac{d}{2})$ then there exists a unique global weak solution $(E,B)$ to \eqref{M} satisfying $(E,B) \in L^\infty(0,T;H^s)$ and $E \in L^2(0,T;H^s)$ for any $T \in (0,\infty)$.
			
			\item[(ii)] \textnormal{(The limit as $c \to \infty$)} Let $v$ be given as in Part $(i)$. Let $c > 0$ and $(E^c_0,B^c_0) \in H^s(\mathbb{R}^d)$ with $s \in [0,\frac{d}{2})$ satisfying $\textnormal{div}\, B^c_0 = 0$ and as $c \to \infty$
			\begin{equation*}
				(E^c_0,B^c_0) \rightharpoonup (\bar{E}_0,\bar{B}_0) \qquad \text{in} \quad H^s
			\end{equation*}
			for some $(\bar{E}_0,\bar{B}_0)$ with $\textnormal{div}\, \bar{B}_0 = 0$. Then there exists a sequence of global solutions $(E^c,B^c)$ to \eqref{M} with $(E^c,B^c)_{|_{t=0}} = (E^c_0,B^c_0)$ given as in Part (i). In addition,  up to an extraction of a subsequence, $B^c$ converges to $B$ in the sense of distributions as $c \to \infty$, where
			$B$ satisfies 
			\begin{equation*} 
				\partial_t B - \nabla \times (v \times B) = \frac{1}{\sigma} \Delta B, \qquad \textnormal{div}\, B = 0 \qquad  \text{and}\qquad  B_{|_{t=0}} = \bar{B}_0.
			\end{equation*}
		\end{enumerate}
	\end{lemma} 
	
	\begin{proof}[Proof of Lemma \ref{lem-M}] The proof is very simple, which shares the ideas as those of Theorems \ref{theo1} and \ref{theo1-3d} and can be done as follows. 
		
		\textbf{Step 1: The existence of Part $(i)$.} We first consider an approximate system to \eqref{M} by
		\begin{equation*}
			\frac{1}{c}\frac{d}{dt}(E^n,B^n) = F^n(E^n,B^n), \qquad \textnormal{div}\, B^n = 0 \qquad \text{and}\qquad (E^n,B^n)_{|_{t=0}} = T_n(E_0,B_0),
		\end{equation*}
		where $F^n = (F^n_1,F^n_2)$ with $F^n_1 = \nabla \times B^n - j^n$, $j^n = \sigma (c E^n + T_n(v \times B^n))$ and $F^n_2 = - \nabla \times E^n$. Furthermore, for $s \in [0,\frac{d}{2})$, $F^n : H^s_n \times V^s_n \to H^s_n \times V^s_n$ is well-defined and is a locally Lipschitz function as well. Therefore, there exists a unique solution $(E^n,B^n) \in C^1([0,T^n_*);H^s_n \times V^s_n)$ for some $T^n_* \in (0,\infty]$ satisfying in addition if $T^n* < \infty$ then 
		\begin{equation*}
			\lim_{t\to T^n_*} \|(E^n,B^n)\|^2_{H^s} = \infty.
		\end{equation*}
		Assume that $T^n_* < \infty$ then the energy balance 
		\begin{align*}
			\frac{d}{dt}\|(E^n,B^n)\|^2_{L^2} + \frac{1}{\sigma} \|j^n\|^2_{L^2} = \int_{\mathbb{R}^d} T_n(v \times B^n) \cdot j^n \,dx,
		\end{align*}
		which implies for $t \in (0,T^n_*)$
		\begin{equation*}
			\|(E^n,B^n)(t)\|^2_{L^2} \leq \|(E_0,B_0)\|^2_{L^2}\exp\left\{C(\sigma)\int^{T^n_*}_0 \|v\|^2_{L^\infty} \,d\tau\right\}.
		\end{equation*}
		Similarly, for $s \in (0,\frac{d}{2})$
		\begin{equation*}
			\|(E^n,B^n)(t)\|^2_{\dot{{H}^s}} \leq \|(E_0,B_0)\|^2_{\dot{H}^s}\exp\left\{C(\sigma,s)\int^{T^n_*}_0 \|v\|^2_{\dot{H}^\frac{d}{2}} + \|v\|^2_{L^\infty} \,d\tau\right\}.
		\end{equation*}
		The above estimates give us a contradiction to the assumption $T^n_* < \infty$ and yield $T^n_* = \infty$. Replacing $T^n_*$ by any $T \in (0,\infty)$, we obtain uniform bounds (in terms of $n$) of $(E^n,B^n)$ in $L^\infty(0,T;H^s)$ and $E^n$ in $L^2(0,T;H^s)$. That leads to the existence of $(E,B)$ such that up to an extraction of a subsequence 
		\begin{align*}
			&&(E^n,B^n) &\overset{\ast}{\rightharpoonup} (E,B)  &&\text{in} \quad L^\infty(0,T;H^s(\mathbb{R}^d))
			\\
			&& E^n &\rightharpoonup  E  &&\text{in} \quad L^2(0,T;H^s(\mathbb{R}^d)). &&
		\end{align*}
		Moreover, as Step 17b in the proof of Theorem \ref{theo1}, by using the following weak formulation for  $\varphi \in C^\infty_0([0,T) \times \mathbb{R}^d;\mathbb{R}^3)$
		\begin{align*}
			\int^T_0 \int_{\mathbb{R}^d} \frac{1}{c}E^n \cdot \partial_t\varphi + B^n \cdot (\nabla \times \varphi) - \sigma (c E^n + T_n(v \times B^n)) \cdot \varphi\,dxdt = -\int_{\mathbb{R}^d} \frac{1}{c}E^n(0) \cdot \varphi(0) \,dx,
			\\
			\int^T_0 \int_{\mathbb{R}^d} \frac{1}{c}B^n \cdot \partial_t\varphi - E^n \cdot (\nabla \times \varphi) \,dxdt  = -\int_{\mathbb{R}^d} \frac{1}{c}B^n(0) \cdot \varphi(0) \,dx,
		\end{align*}
		we can pass to the limit as $n \to \infty$ easily by using the assumptions of $v$. 
		
		\textbf{Step 2: The uniqueness of Part $(i)$.} For two solutions $(E,B)$ and $(\bar{E},\bar{B})$ to \eqref{M} with the same initial data $(E_0,B_0)$ and for $s' \in [0,s)$, it follows from Lemma \ref{lem_M0} that 
		\begin{equation*}
			\|(E-\bar{E},B-\bar{B})\|^2_{L^\infty(0,T;H^{s'})} \leq c^2 \|j-\bar{j}\|^2_{L^1(0,T;H^{s'})} \leq \sum^6_{k=4} \bar{J}_k.
		\end{equation*}
		By repeating either Step 18 in the proof of Theorem \ref{theo1} for $d = 2$ or Step 15 in the proof of Theorem \ref{theo1-3d} for $d = 3$, we find that for $v = \bar{v}$ and for sufficiently small $T_* \in (0,T)$
		\begin{align*}
			\|(E-\bar{E},B-\bar{B})\|^2_{L^\infty(0,T_*;H^{s'})} \leq \frac{1}{2} \|(E-\bar{E},B-\bar{B})\|^2_{L^\infty(0,T_*;H^{s'})},
		\end{align*}
		which yields $E = \bar{E}$ and $B = \bar{B}$ in $(0,T_*)$. By repeating this process, we obtain the conclusion in the whole time interval $(0,T)$. 
		
		\textbf{Step 3: Proof of Part $(ii)$.} Since the main estimates in Step 1 are independent on $c$, then the proof follows as that of Part $(ii)$ in Theorem \ref{theo1-3d} by using the above weak formulation ($n$ is replaced by $c$ and without $T_n$), the fact $j^c = \sigma(cE^c + v \times B^c)$ and the weak convergence of $(E^c_0,B^c_0)$. Thus, the proof now is complete.
	\end{proof}
	
	%
	\subsection{Appendix F: Proof of Proposition \ref{pro-H}}
	%

	For the sake of completeness, we will give a proof of Proposition \ref{pro-H} below.
	
	\begin{proof}[Proof of Proposition \ref{pro-H}-The case $d = 2$] The proof consits of the following steps.
		
		\textbf{Step 1: Local and global existence.} As previous parts, an approximate system to \eqref{H} is given by 
	\begin{equation} \label{H_app} 
		\partial_t B^n = -\frac{1}{\sigma}(-\Delta)^\beta B^n - \frac{\kappa}{\sigma} \nabla \times T_n(j^n \times B^n),
		\quad \textnormal{div}\, B^n = 0, \quad B^n_{|_{t=0}} = T_n(B_0), \quad j^n := \nabla \times B^n,
	\end{equation}
	and there exists a unique solution $B^n$ to \eqref{H_app} with $B^n \in C^1([0,T^n_*);V^s_n)$ for some $T^n_* > 0$. It is sufficient to focus on the case $\beta = \frac{3}{2}$. The case $\beta > \frac{3}{2}$ can be done in the same way in which we will omit the details. It can be seen from \eqref{H_app} that for $s > 0$
		\begin{align*}
			\frac{1}{2}\frac{d}{dt}\|B^n\|^2_{L^2} + \frac{1}{\sigma}\|B^n\|^2_{\dot{H}^\frac{3}{2}} &= 0,
			\\
			\frac{1}{2}\frac{d}{dt}\|B^n\|^2_{\dot{H}^s} + \frac{1}{\sigma}\|B^n\|^2_{\dot{H}^{s+\frac{3}{2}}} &= -\frac{\kappa}{\sigma} \int_{\mathbb{R}^2} \Lambda^s(j^n \times B^n) \cdot \Lambda^s j^n \,dx =: H.
		\end{align*}
		$\bullet$ If $s \in (0,\frac{1}{2})$ then for some $\epsilon \in (0,1)$ the $\dot{H}^s$ estimate is closable as follows
		\begin{align*}
			H &= -\frac{\kappa}{\sigma} \int_{\mathbb{R}^2} (j^n \times B^n) \cdot \Lambda^{2s} j^n \,dx \leq \frac{\kappa}{\sigma}C(s)\|j^n\|_{L^4}\|B^n\|_{L^\frac{2}{1-s}}\|\Lambda^{2s}j^n\|_{L^\frac{2}{1-\left(\frac{3}{2}-(s+1)\right)}} 
			\\
			&\leq \frac{\epsilon}{\sigma}\|B^n\|^2_{\dot{H}^{s+\frac{3}{2}}} + C(\epsilon,\kappa,\sigma,s)\|B^n\|^2_{\dot{H}^\frac{3}{2}}\|B^n\|^2_{\dot{H}^s},
		\end{align*}
		$\bullet$ If $s = \frac{1}{2}$ then $2s + 1 = s+\frac{3}{2} = 2$ and
		\begin{align*}
			H &\leq \frac{\kappa}{\sigma}\|j^n\|_{L^4}\|B^n\|_{L^4}\|\Lambda^{2s}j^n\|_{L^2}\leq \frac{\epsilon}{\sigma}\|B^n\|^2_{\dot{H}^2} + C(\epsilon,\kappa,\sigma) \|B^n\|^2_{\dot{H}^\frac{3}{2}}\|B^n\|^2_{\dot{H}^\frac{1}{2}}.
		\end{align*}
		$\bullet$ If $s > \frac{1}{2}$ then
		\begin{equation*}
			\frac{1}{2}\frac{d}{dt}\|B^n\|^2_{H^s} + \frac{1}{\sigma}\|\Lambda^\frac{3}{2}B^n\|^2_{H^s} = -\frac{\kappa}{\sigma} \int_{\mathbb{R}^2} [J^s(j^n \times B^n) - J^s j^n \times B^n] \cdot J^s j^n \,dx =: H,
		\end{equation*}
		where for some $\epsilon \in (0,1)$, the Kato-Ponce commutator estimate gives
		\begin{align*}
			H &\leq C(s) \frac{\kappa}{\sigma}\left(\|j^n\|_{H^{s-1}}\|\nabla B^n\|_{L^\infty} + \|j^n\|_{L^\infty}\|B^n\|_{H^s}\right)\|B^n\|_{H^{s+1}}
			\\
			&\leq \frac{\epsilon}{\sigma}\|B^n\|^2_{H^{s+\frac{3}{2}}} + C(\epsilon,\kappa,\sigma,s) \|\nabla B^n\|^2_{L^\infty}\|B^n\|^2_{H^s},
		\end{align*}
		which by choosing $\epsilon  = \frac{1}{2}$ and using \eqref{BGW} with $d = 2$, $f = \nabla B^n$ and $s_0 = s - \frac{1}{2} > 0$ implies that 
		\begin{equation*}
			\frac{d}{dt}\|B^n\|^2_{H^s} + \frac{1}{\sigma}\|B^n\|^2_{H^{s+\frac{3}{2}}} \leq  \frac{1}{\sigma}\|B^n\|^2_{L^2} + \left[\frac{1}{2}C(\kappa,\sigma,s) \|B^n\|_{H^s}\|\nabla B^n\|_{H^1}\left(1+\log^\frac{1}{2}\left(\frac{\|B^n\|_{H^{s+\frac{3}{2}}}}{\|\nabla B^n \|_{H^1}}\right)\right)\right]^2.
		\end{equation*}
		By using the previous case to bound $\|B^n\|_{L^2_tH^2_x}$, the conclusion follows as Step 3 in the proof of Theorem \ref{theo1}. 
		
		\textbf{Step 2: Pass to the limit.} This step can be done as either Step 16a for $s > 1$ or Step 16b for $s \in [0,1]$ in the proof of Theorem \ref{theo1}. We omit further details.
		
		\textbf{Step 3: Uniqueness.} It is enough to consider the case $s = 0$. Let $B$ be the limit in Step 2. It can be seen that $B \in L^2(0,T;H^\frac{3}{2})$ and $\partial_t B \in L^2(0,T;H^{-\frac{3}{2}})$, which implies that $B \in C([0,T];L^2)$ (see \cite{Temam_2001}) after possibly being redefined on a set of measure zero. Assume that $B_1$ and $B_2$ are two solutions to \eqref{H} with the same initial data $B_0 \in L^2$ and $j_i = \nabla \times B_i$ for $i \in \{1,2\}$. Thus, we find that
		\begin{equation*}
			\frac{1}{2}\frac{d}{dt}\|B_1-B_2\|^2_{L^2} + \frac{1}{\sigma}\|B_1-B_2\|^2_{\dot{H}^\frac{3}{2}}= - \int_{\mathbb{R}^2} (j_1 \times B_1 - j_2 \times B_2) \cdot (j_1-j_2)\,dx =: H_1
		\end{equation*}
		where for some $\epsilon \in (0,1)$
		\begin{align*}
			H_1 
			&\leq \|j_1\|_{L^4}\|B_1-B_2\|_{L^2}\|j_1-j_2\|_{L^4}
			\leq \frac{\epsilon}{\sigma}\|B_1-B_2\|^2_{\dot{H}^\frac{3}{2}} + C(\epsilon,\sigma)\|B_1\|^2_{\dot{H}^\frac{3}{2}}\|B_1-B_2\|^2_{L^2},
		\end{align*}
		which yields $B_1 = B_2$ and ends the proof.
	\end{proof}
	
	\begin{proof}[Proof of Proposition \ref{pro-H}-The case $d = 3$] The proof is divided into several steps as follows.
		
		\textbf{Step 1: Local and global existence.}  Similar to the previous case, we will focus on \eqref{H_app} with $\beta = \frac{7}{4}$. In addition, for $s > 0$
		\begin{align*}
			\frac{1}{2}\frac{d}{dt}\|B^n\|^2_{L^2} + \frac{1}{\sigma}\|B^n\|^2_{\dot{H}^\frac{7}{4}} &= 0,
			\\
			\frac{1}{2}\frac{d}{dt}\|B^n\|^2_{\dot{H}^s} + \frac{1}{\sigma}\|B^n\|^2_{\dot{H}^{s+\frac{7}{4}}} &= -\frac{\kappa}{\sigma} \int_{\mathbb{R}^3} \Lambda^s(j^n \times B^n) \cdot \Lambda^s j^n \,dx =: H.
		\end{align*}
		
		$\bullet$ If $s \in (0,\frac{3}{4})$ then 
		\begin{align*}
			H &\leq \frac{\kappa}{\sigma}C(s)\|j^n\|_{L^4}\|B^n\|_{L^\frac{6}{3-2s}}\|\Lambda^{2s}j^n\|_{L^\frac{6}{3-2\left(\frac{3}{2}-\left(s+\frac{3}{4}\right)\right)}} 
			\\
			&\leq \frac{\epsilon}{\sigma}\|B^n\|^2_{\dot{H}^{s+\frac{7}{4}}} + C(\epsilon,\kappa,\sigma,s)\|B^n\|^2_{\dot{H}^\frac{7}{4}}\|B^n\|^2_{\dot{H}^s}.
		\end{align*}
		
		$\bullet$ If $s = \frac{3}{4}$ then $2s + 1 = s + \frac{7}{4}$ and 
		\begin{equation*}
			H \leq \|j^n\|_{L^4}\|B^n\|_{L^4}\|\Lambda^{2s}j^n\|_{L^2}
			\leq \frac{\epsilon}{\sigma}\|B^n\|^2_{\dot{H}^{s+\frac{7}{4}}} + C(\epsilon,\kappa,\sigma,s)\|B^n\|^2_{\dot{H}^\frac{7}{4}}\|B^n\|^2_{\dot{H}^\frac{3}{4}}.
		\end{equation*}
		
		$\bullet$ If $s > \frac{3}{4}$ then
		\begin{equation*}
			\frac{1}{2}\frac{d}{dt}\|B^n\|^2_{H^s} + \frac{1}{\sigma}\|\Lambda^\frac{7}{4}B^n\|^2_{H^s} = -\frac{\kappa}{\sigma} \int_{\mathbb{R}^3} [J^s(j^n \times B^n) - J^s j^n \times B^n] \cdot J^s j^n \,dx =: H,
		\end{equation*}
		where for some $\epsilon \in (0,1)$, the Kato-Ponce commutator estimate gives
		\begin{align*}
			H &\leq C(s) \frac{\kappa}{\sigma}\left(\|j^n\|_{H^{s-1}}\|\nabla B^n\|_{L^\infty} + \|j^n\|_{L^\infty}\|B^n\|_{H^s}\right)\|B^n\|_{H^{s+1}}
			\\
			&\leq \frac{\epsilon}{\sigma}\|B^n\|^2_{H^{s+\frac{7}{4}}} + C(\epsilon,\kappa,\sigma,s) \|\nabla B^n\|^2_{L^\infty}\|B^n\|^2_{H^s},
		\end{align*}
		which by choosing $\epsilon  = \frac{1}{2}$ and using \eqref{BGW} with $d = 3$, $f = \nabla B^n$ and $s_0 = s - \frac{1}{4} > \frac{1}{2}$ implies that 
		\begin{equation*}
			\frac{d}{dt}\|B^n\|^2_{H^s} + \frac{1}{\sigma}\|B^n\|^2_{H^{s+\frac{7}{4}}} \leq  \frac{1}{\sigma}\|B^n\|^2_{L^2} + \left[\frac{1}{2}C(\kappa,\sigma,s) \|B^n\|_{H^s}\|\nabla B^n\|_{H^\frac{3}{2}}\left(1+\log^\frac{1}{2}\left(\frac{\|B^n\|_{H^{s+\frac{7}{4}}}}{\|\nabla B^n \|_{H^\frac{3}{2}}}\right)\right)\right]^2.
		\end{equation*}
		Therefore, the conclusion follows.
		
		\textbf{Step 2: Pass to the limit and uniqueness.} This step follows as that of in the previous case. Indeed, the uniqueness in the case $s = 0$ is proceeded with
		\begin{align*}
			H_1 
			\leq \|j_1\|_{L^4}\|B_1-B_2\|_{L^2}\|j_1-j_2\|_{L^4}
			\leq \frac{\epsilon}{\sigma}\|B_1-B_2\|^2_{\dot{H}^\frac{7}{4}} + C(\epsilon,\sigma)\|B_1\|^2_{\dot{H}\frac{7}{4}}\|B_1-B_2\|^2_{L^2},
		\end{align*}
		which finishes the proof.
	\end{proof}
	
	%

\end{document}